\documentclass[12pt]{amsart}

\usepackage{amsmath} 
\usepackage{amsthm} 
\usepackage{amsfonts} 
\usepackage{amssymb} 
\usepackage{stmaryrd} 
\usepackage{mathtools} 
\usepackage[all]{xy} 
\usepackage{tikz-cd} 
\usepackage{array} 
\usepackage{enumerate} 
\usepackage[shortlabels]{enumitem} 
\usepackage{fancyhdr} 
\usepackage{imakeidx}

\usepackage[marginratio=1:2,paper=letterpaper,tmargin=52pt,lmargin=74pt,rmargin=74pt,includehead]{geometry} 

\usepackage[pagebackref=true,
]{hyperref} 
\hypersetup{colorlinks=true,linkcolor=blue,filecolor=magenta,urlcolor=cyan,final} 

\usepackage{bm}

\usepackage{comment}
\usepackage{extarrows}

\makeatletter

\def\chaptermark#1{}

\def\chapter{%
	\if@openright\cleardoublepage\else\clearpage\fi
	\thispagestyle{plain}\global\@topnum\z@
	\@afterindenttrue \secdef\@chapter\@schapter}

\def\@chapter[#1]#2{\refstepcounter{chapter}%
	\ifnum\c@secnumdepth<\z@ \let\@secnumber\@empty
	\else \let\@secnumber\thechapter \fi
	\typeout{\chaptername\space\@secnumber}%
	\def\@toclevel{0}%
	\ifx\chaptername\appendixname \@tocwriteb\tocappendix{chapter}{#2}%
	\else \@tocwriteb\tocchapter{chapter}{#2}\fi
	\chaptermark{#1}%
	\addtocontents{lof}{\protect\addvspace{10\p@}}%
	\addtocontents{lot}{\protect\addvspace{10\p@}}%
	\@makechapterhead{#2}\@afterheading}
\def\@schapter#1{\typeout{#1}%
	\let\@secnumber\@empty
	\def\@toclevel{0}%
	\ifx\chaptername\appendixname \@tocwriteb\tocappendix{chapter}{#1}%
	\else \@tocwriteb\tocchapter{chapter}{#1}\fi
	\chaptermark{#1}%
	\addtocontents{lof}{\protect\addvspace{10\p@}}%
	\addtocontents{lot}{\protect\addvspace{10\p@}}%
	\@makeschapterhead{#1}\@afterheading}
\newcommand\chaptername{Chapter}

\def\@makechapterhead#1{\global\topskip 7.5pc\relax
	\begingroup
	\fontsize{\@xivpt}{18}\bfseries\centering
	\ifnum\c@secnumdepth>\m@ne
	\leavevmode \hskip-\leftskip
	\rlap{\vbox to\z@{\vss
			\centerline{\normalsize\mdseries
				\uppercase\@xp{\chaptername}\enspace\thechapter}
			\vskip 3pc}}\hskip\leftskip\fi
	#1\par \endgroup
	\skip@34\p@ \advance\skip@-\normalbaselineskip
	\vskip\skip@ }
\def\@makeschapterhead#1{\global\topskip 7.5pc\relax
	\begingroup
	\fontsize{\@xivpt}{18}\bfseries\centering
	#1\par \endgroup
	\skip@34\p@ \advance\skip@-\normalbaselineskip
	\vskip\skip@ }
\def\appendix{\par
	\c@chapter\z@ \c@section\z@
	\let\chaptername\appendixname
	\def\thechapter{\@Alph\c@chapter}}

\newcounter{chapter}

\newif\if@openright

\makeatother

\newtheorem{thm}{Theorem}[chapter]
\newtheorem{lem}[thm]{Lemma}
\newtheorem{prop}[thm]{Proposition}
\newtheorem{cor}[thm]{Corollary}

\theoremstyle{definition}
\newtheorem{dfn}[thm]{Definition}
\newtheorem{ex}[thm]{Example}

\newtheorem{set}[thm]{Setting}

\theoremstyle{remark}
\newtheorem{remk}[thm]{Remark}
\newtheorem{quest}[thm]{Question}
\newtheorem{assumption}[thm]{Assumption}

\newtheorem{open}{Question}

\numberwithin{section}{chapter}
\numberwithin{equation}{chapter}

\DeclareMathOperator{\Aut}{Aut}

\DeclareMathOperator{\Tors}{Tors}
\DeclareMathOperator{\Ext}{Ext}
\DeclareMathOperator{\Hom}{Hom}
\DeclareMathOperator{\Res}{Res}
\DeclareMathOperator{\im}{im}
\DeclareMathOperator{\Id}{Id}
\DeclareMathOperator{\lc}{\centerdot}
\DeclareMathOperator{\Tot}{Tot}
\DeclareMathOperator{\Ad}{Ad}
\DeclareMathOperator{\gen}{gen}
\DeclareMathOperator{\Gdm}{Gdm}
\DeclareMathOperator{\UCT}{UCT}
\DeclareMathOperator{\Homm}{\mathcal{H}om}
\DeclareMathOperator{\Extt}{\mathcal{E}xt}
\DeclareMathOperator{\res}{res}

\DeclareMathOperator{\Cone}{Cone}

\DeclareMathOperator{\Log}{Log}

\newcommand{\bH}{\mathbb{H}}
\newcommand{\bP}{{\mathbb P}}
\newcommand{\cB}{{\mathcal B}}
\newcommand{\cL}{{\mathcal L}}
\newcommand{\cM}{{\mathcal M}}
\newcommand{\cF}{{\mathcal F}}

\newcommand{\cE}{{\mathcal E}}
\newcommand{\cA}{{\mathcal A}}
\newcommand{\uR}{\underline{\mathbb R}}

\newcommand{\ul}[1]{\underline{#1}}
\newcommand{\wt}[1]{\widetilde{#1}}
\newcommand{\wh}[1]{\widehat{#1}}
\newcommand{\ov}[1]{\overline{#1}}
\newcommand{\onarrow}[1]{\overset{#1}{\longrightarrow}}
\newcommand{\darg}{\Im\frac{df}{f}}

\newcommand{\hyp}[1]{\hyperref[#1]{\ref{#1}}} 

\newcommand{\bi}[1]{\textbf{\textit{#1}}} 

\newcommand{\C}{\mathbb{C}} 
\newcommand{\N}{\mathbb{N}} 
\newcommand{\Q}{\mathbb{Q}} 
\newcommand{\R}{\mathbb{R}} 
\newcommand{\Z}{\mathbb{Z}} 

\newcommand{\calA}{\mathcal{A}}
\newcommand{\calB}{\mathcal{B}}
\newcommand{\calC}{\mathcal{C}}

\newcommand{\calE}{\mathcal{E}}
\newcommand{\calF}{\mathcal{F}}
\newcommand{\calG}{\mathcal{G}}
\newcommand{\calH}{\mathcal{H}}

\newcommand{\calK}{\mathcal{K}}
\newcommand{\calL}{\mathcal{L}}
\newcommand{\calM}{\mathcal{M}}

\newcommand{\calO}{\mathcal{O}}

\newcommand{\logdr}[2]{\Omega^\bullet_{#1}(\log #2)} 

\newcommand{\Gr}{\mathrm{Gr}} 
\newcommand{\id}{\mathrm{id}} 
\newcommand{\Hdg}[2]{\calH dg^\bullet(#1\,\log #2)} 

\newcommand{\f}[5]{
	\begin{array}{rcl}
		#1\colon #2 & \longrightarrow & #3 \\
		#4 & \longmapsto & #5 \\
	\end{array}
}

\newcommand{\ff}[4]{
	\begin{array}{rcl}
		#1 & \longrightarrow & #2 \\
		#3 & \longmapsto & #4 \\
	\end{array}
}

\title[]{Mixed Hodge Structures on Alexander Modules}

\author{Eva Elduque}
\address{Department of Mathematics, University of Michigan-Ann Arbor, 530 Church St, Ann Arbor, MI 48109, USA}
\curraddr{Departamento de Matem\' aticas, Universidad Aut\' onoma de Madrid, 28049 Madrid, Spain}
\urladdr{https://matematicas.uam.es/~eva.elduque/}
\email{eva.elduque@uam.es}
\thanks{E. Elduque was partially supported by an AMS-Simons Travel Grant}

\author{Christian Geske}
\address{Department of Mathematics, Northwestern University, 2033 Sheridan Rd, Evanston, IL 60208, USA}
\urladdr{https://sites.math.northwestern.edu/~cgeske/}
\email{christian.geske@northwestern.edu}
\thanks{}

\author{Mois\'es Herrad\'on Cueto}
\address{Department of Mathematics, Louisiana State University, 303 Lockett Hall, Baton Rouge, LA 70803, USA}
\curraddr{Departamento de Matem\' aticas, Universidad Aut\' onoma de Madrid, 28049 Madrid, Spain}
\email{moises@lsu.edu}
\urladdr{https://math.lsu.edu/~moises/}
\thanks{M. Herrad\'on Cueto was partially supported by an AMS-Simons Travel Grant}

\author{Lauren\c{t}iu Maxim}
\address{Department of Mathematics, University of Wisconsin-Madison, 480 Lincoln Drive, Madison WI 53706-1388, USA}
\email{maxim@math.wisc.edu}
\urladdr{https://people.math.wisc.edu/~maxim/}
\thanks{L. Maxim is partially supported by the Simons Foundation (Collaboration Grant \#567077) and by the Romanian Ministry of National Education (grant PN-III-P4-ID-PCE-2020-0029).}

\author{Botong Wang}
\address{Department of Mathematics, University of Wisconsin-Madison, 480 Lincoln Drive, Madison WI 53706-1388, USA}
\email{wang@math.wisc.edu}
\urladdr{https://people.math.wisc.edu/~wang/}
\thanks{B. Wang is partially supported by the NSF grant DMS-1701305 and by a Sloan Fellowship}

\keywords{infinite cyclic cover, Alexander module, mixed Hodge structure, thickened complex, limit mixed Hodge structure, semisimplicity}

\subjclass[2020]{Primary 14C30, 14D07, 32S35. Secondary 14F45, 32S30, 32S40, 55N30}

\dedicatory{Dedicated to the memory of Prof. \c Stefan Papadima}

	\makeindex

\begin{document}
	
	\date{\today}
	
	\maketitle
	
		\begin{abstract}
		Motivated by the limit mixed Hodge structure on the Milnor fiber of a hypersurface singularity germ, we construct a natural mixed Hodge structure on the torsion part of the Alexander modules of a smooth connected complex algebraic variety. More precisely, let $U$ be a smooth connected complex algebraic variety and let $f\colon U\to \C^*$ be an algebraic map inducing an epimorphism in fundamental groups. The pullback of the universal cover of $\C^*$ by $f$ gives rise to an infinite cyclic cover $U^f$ of $U$. The action of the deck group $\Z$ on $U^f$ induces a $\Q[t^{\pm 1}]$-module structure on $H_*(U^f;\Q)$. We show that the torsion parts $A_*(U^f;\Q)$ of the Alexander modules $H_*(U^f;\Q)$ carry canonical $\Q$-mixed Hodge structures. We also prove that the covering map $U^f \to U$ induces a mixed Hodge structure morphism on the torsion parts of the Alexander modules. As applications, we investigate the semisimplicity of $A_*(U^f;\Q)$, as well as possible weights of the constructed mixed Hodge structures. 
		Finally, in the case when $f\colon U\to \C^*$ is proper, we prove the semisimplicity and purity of $A_*(U^f;\Q)$, and we compare our mixed Hodge structure on $A_*(U^f;\Q)$ with the limit mixed Hodge structure on the generic fiber of $f$. 
	\end{abstract}

\newpage

\tableofcontents


\chapter{Introduction}

Let $U$ be a connected topological space of finite homotopy type, and let \begin{align*} \xi\colon\pi_1(U) \twoheadrightarrow \mathbb{Z}\end{align*} be an epimorphism. Denote by $U^\xi$ the infinite cyclic cover of $U$ corresponding to $\ker \xi$. 
Let $k$ be a subfield of $\mathbb{R}$, and denote by $R=k[t^{\pm 1}]$ the ring of Laurent polynomials in variable $t$ with $k$-coefficients.
The group of covering transformations of $U^\xi$ is isomorphic to $\mathbb{Z}$, and it induces an $R$-module structure on each group $H_i(U^\xi;k)$. By analogy with knot theory, the $R$-module $H_i(U^\xi;k)$ is called the $i$-th (homology) {\it $k$-Alexander module} of the pair $(U,\xi)$. Since $U$ is homotopy equivalent to a finite CW-complex, $H_i(U^\xi;k)$ is a finitely generated $R$-module, for each integer $i$.

Note that $\xi\colon \pi_1(U) \to \mathbb{Z}$ can be regarded as an element in $H^1(U;\mathbb{Z})$ via the canonical identification:
\[
\Hom\big(\pi_1(U),\mathbb{Z}\big) \cong \Hom\big(H_1(U;\mathbb{Z}),\mathbb{Z}\big) \cong H^1(U;\mathbb{Z}).
\]
 Moreover, any such class in $H^1(U;\mathbb{Z})$ is represented by a homotopy class of continuous maps $U \to S^1$. Whenever such a representative $f\colon U \to S^1$ for $\xi$ is fixed (that is, $\xi=f_{*}$), we will also use the notation $U^f$ for the corresponding infinite cyclic cover of $U$.

For example, if $f\colon U \to S^1$ is a fiber bundle with connected fiber $F$ a finite CW-complex, then $\xi=f_*\colon \pi_1(U) \to \pi_1(S^1)=\mathbb{Z}$ is surjective, and the corresponding infinite cyclic cover $U^f$ is homeomorphic to $F \times \mathbb{R}$ and hence homotopy equivalent to $F$. The deck group action on $H_i(U^f;k)$ is isomorphic 
(up to a choice of orientation on $S^1$, as described in Lemma \ref{lem:fiberMonodromy}) to the 
monodromy action on $H_i(F;k)$, which gives the latter vector spaces $R$-module structures. Therefore $H_i(U^f;k)\cong H_i(F;k)$ is a torsion $R$-module for all $i \geq 0$. This applies in particular to the following geometric situations:

\begin{enumerate}[(a)]
\item a smooth proper surjective submersion (with connected fibers) $f\colon U \to \Delta^*$ from an open set of a smooth complex algebraic variety to a punctured disc.
\item the Milnor fibration $f\colon U \to S^1$ associated to a reduced complex hypersurface singularity germ, with $F$ the corresponding Milnor fiber.
\item \label{item:milnor} the global (affine) Milnor fiber $f\colon U=\mathbb{C}^n \setminus \{f=0\} \to \mathbb{C}^*$, where $f$ is a square-free homogeneous polynomial in $n$ complex variables and $F=f^{-1}(1)$.
\end{enumerate}

In a different vein, it was shown in \cite{LMW} that if $U$ is a smooth quasi-projective variety of complex dimension $n$, admitting a proper semi-small map (e.g., a finite map or a closed embedding) to some complex semiabelian variety, then for any {\it generic} epimorphism $\xi\colon \pi_1(U) \to \mathbb{Z}$ the corresponding Alexander modules $H_i(U^\xi; k)$ are torsion $R$-modules for all $i \neq n$. 

However, for an arbitrary topological space $U$ of finite homotopy type, the Alexander modules $H_i(U^\xi; k)$ are not torsion $R$-modules in general. One then considers the torsion part 
\[
A_i(U^\xi;k)\coloneqq \Tors_R H_i(U^\xi;k)	
\]
of the $R$-module $H_i(U^\xi;k)$. This is a 
$k$-vector space of finite dimension on which a generating covering transformation (i.e., $t$-multiplication) acts as a linear automorphism. Moreover, if $U$ is a smooth complex algebraic variety, then all eigenvalues of the $t$-action on $A_i(U^\xi;k)$ are roots of unity, for any integer $i$ (see Proposition \ref{eig1}). 

\medskip

In this paper, we assume that $U$ is a smooth complex algebraic variety and we investigate the existence of mixed Hodge structures on $A_i(U^\xi;k)$ for $k=\mathbb{Q}$ or $\mathbb{R}$. Note that $U^{\xi}$ is not in general a complex algebraic variety, so the classical Deligne theory does not apply. Specifically, we address the following question, communicated to the authors by \c{S}tefan Papadima:

\begin{quest}\label{conj}
Let $U$ be a smooth connected complex algebraic variety. Let $\xi\colon \pi_1(U) \to \mathbb{Z}$ be an epimorphism with corresponding infinite cyclic cover $U^{\xi}$. Is there a natural $\Q$-mixed Hodge structure on the torsion part $A_i(U^\xi;k)$ of the $\Q[t^{\pm 1}]$-module $H_i(U^{\xi};\Q)$, for all $i \geq 0$?
\end{quest}

The purpose of this paper is to give a positive answer to Question \ref{conj} in the case when the epimorphism $\xi\colon \pi_1(U) \to \mathbb{Z}$ is realized by an {\it algebraic} map $f\colon U \to \mathbb{C}^*$. We remark here that a homomorphism $\xi\colon \pi_1(U)\to \Z$ is induced by an algebraic map $f\colon U\to \C^*$ if and only if, when considered as an element in $H^1(U; \Z)$, $\xi$ is of type $(1, 1)$, that is 
$\xi\in F^1H^1(U; \C)\cap \overline{F^1H^1(U; \C)}.$
This is a consequence of Deligne's theory of $1$-motives (cf. \cite[(10.I.3)]{De3}). 
In this algebraic context, Question \ref{conj} has already been answered positively in the following special situations:
\begin{enumerate}
\item\label{hain} When $H_i(U^{\xi};\Q)$ is $\Q[t^{\pm 1}]$-torsion for all $i \geq 0$ and the $t$-action is unipotent, see \cite{hain1987rham}; 
\item \label{dlt}When $f\colon U=\mathbb{C}^n \setminus \{f=0\}\to \mathbb{C}^*$ is induced by a reduced complex polynomial $f\colon \mathbb{C}^n \to \mathbb{C}$ which is {\it transversal at infinity} (i.e., the hyperplane at infinity in $\mathbb{C}P^n$ is transversal in the stratified sense to the projectivization of $\{f=0\}$), for $\xi=f_*$ and $i<n$; see \cite{DL, Liu}. In fact, in this case it was shown in \cite{Max06} that the corresponding Alexander modules $H_i(U^\xi;\mathbb{Q})$ are torsion $\mathbb{Q}[t^{\pm 1}]$-modules for $i<n$, while $H_n(U^\xi;\mathbb{Q})$ is free and $H_i(U^\xi;\mathbb{Q})=0$ for $i>n$. Furthermore, the $t$-action on $H_i(U^\xi;\mathbb{C})$ is diagonalizable (semisimple) for $i<n$, and the corresponding eigenvalues
are roots of unity of order $d=\deg(f)$.
\item\label{lkk} When $f\colon U=\mathbb{C}^n \setminus \{f=0\}\to \mathbb{C}^*$ is induced by a complex polynomial $f\colon \mathbb{C}^n \to \mathbb{C}$ which has at most isolated singularities, including at infinity, in the sense that both the projectivization of $\{f=0\}$ and its intersection with the hyperplane at infinity have at most isolated singularities. In this case, and with $\xi=f_*$, there is only one interesting Alexander module, $H_{n-1}(U^\xi;\mathbb{Q})$, which is torsion (see \cite[Theorem 4.3, Remark 4.4]{Lib94}), and a mixed Hodge structure on it was constructed in \cite{Lib96}; see also \cite{KK} for the case of plane curves under some extra conditions.
\end{enumerate}

In this paper we prove the following general statement (see Corollary \ref{halexandermhs}):

\begin{thm}\label{mhsexistence} 
Let $U$ be a smooth connected complex algebraic variety, with an algebraic map $f\colon U \to \mathbb{C}^*$. Assume that $\xi=f_*\colon\pi_1(U) \to \mathbb{Z}$ is an epimorphism, and denote by $U^f = \{(x,z)\in U\times \C \mid f(x) = e^z \}$ the corresponding infinite cyclic cover. Then the torsion part $A_i(U^f;\Q)$ of the $\mathbb{Q}[t^{\pm 1}]$-module 
$H_i(U^f;\mathbb{Q})$ carries a canonical $\Q$-mixed Hodge structure for any $i\geq 0$.

Suppose $N$ is a positive integer, chosen such that $t^N$ acts unipotently on $A_i(U^f;\Q)$. Let $\log(t^N)$ denote the Taylor series centered at $t^N=1$. Then the action of $\log(t^N)$ is a mixed Hodge structure morphism $A_i(U^f;\Q)\to A_i(U^f;\Q)(-1)$, where $(-1)$ denotes the Tate twist of mixed Hodge structures.
\end{thm}

In Section~\ref{ss:DLandLiu} we prove that, if $U$ and $f$ are as in case (\ref{dlt}) above, the mixed Hodge structure of Theorem \ref{mhsexistence} recovers the mixed Hodge structure obtained by different means in both \cite{DL, Liu}. We expect, but have not proven, the same to be true for cases (\ref{hain}) and (\ref{lkk}). Moreover, in Section~\ref{ss:genFiber}, we also prove that, in the global affine Milnor fibration case \ref{item:milnor} mentioned above, the mixed Hodge structure of Theorem~\ref{mhsexistence} recovers Deligne's mixed Hodge structure on the fiber (Corollary \ref{cor:quasihom}).

The proof of Theorem~\ref{mhsexistence} makes use of a sequence of reductions (e.g., after pulling back to a finite cover, one may assume that the $t$-action on $A_*(U^\xi;\Q)$ is unipotent, it also suffices to work with cohomological Alexander modules, etc.), and it relies on the construction of a suitable {\it thickening} of the Hodge-de Rham complex (see Chapter \ref{mhsal} for details).
There are choices made in the construction, but we prove that the resultant mixed Hodge structure is {independent} of them (see Theorems \ref{indcompactification} and \ref{indN}) and further that it behaves functorially with respect to algebraic maps over $\C^*$ (see Theorem~\ref{functorial}). The only choice that affects the mixed Hodge structure is the choice of infinite cyclic cover $U^f$. This is because $\Z$ acts on this space as deck transformations and the mixed Hodge structure is {\bf not} preserved in general (see Proposition~\ref{prop:depf}). However, this is consistent with the behavior of the limit mixed Hodge structure (see Theorem~\ref{comp}), which is also not preserved by deck transformations.

We proceed to relate our mixed Hodge structures on the torsion parts of the Alexander modules to known mixed Hodge structures, then derive consequences of our construction and these relations. Centrally, in Chapter~\ref{sec:main} we prove that the infinite cyclic covering map induces a morphism of mixed Hodge structures:

\begin{thm}\label{geoIntro}
In the setting of Theorem~\ref{mhsexistence}, the vector space map $A_i(U^f; \Q) \to H_i(U;\Q)$ induced by the covering $U^f \rightarrow U$ is a morphism of mixed Hodge structures for all $i \geq 0$, where $H_i(U;\Q)$ is equipped with (the dual of) Deligne's mixed Hodge structure. 
\end{thm}

We use this theorem and our construction to obtain several results, including a bound on the weight filtrations of the mixed Hodge structures on the torsion parts of the Alexander modules (Theorem \ref{thm:boundedWeights}). This bound coincides with the known bound for the homology of smooth algebraic varieties of the same dimension as the generic fiber of $f$ (cf. \cite[Corollaire 3.2.15]{De2}).

\begin{thm}\label{boundsIntro}
Assume the setting of Theorem~\ref{mhsexistence}. Let $i \geq 0$.
If $k\notin[i,2i]\cap[i,{2\dim_\C(U)-2]}$, then
\[
\Gr^W_{-k} A_i(U^f;\Q) = 0
\]
where $\Gr^W_{-k}$ denotes the $-k$th graded piece of the weight filtration.
\end{thm}

Other consequences of our construction and Theorem~\ref{geoIntro} are related to the $t$-action on the torsion parts of the Alexander modules. 
For example, we apply it to determine bounds on the size of the Jordan blocks of this $t$-action (see Corollary \ref{cor:jordan}), which nearly cut in half existing bounds, as in \cite[Proposition 1.10]{BudurLiuWang}. We also address conditions under which this $t$-action is a mixed Hodge structure morphism. We prove that this is the case if and only if the $t$-action is semisimple (see Corollary \ref{cor:t} and Proposition \ref{prop:semisimpleIsEasyConverse}).

\begin{thm}\label{tsemisimpleIntro}
	Assume the setting of Theorem \ref{mhsexistence}. Let $i \geq 0$. The $t$-action on $A_i(U^f; \Q)$ is a mixed Hodge structure morphism if and only if it is semisimple.
\end{thm}
 
Semisimplicity has other pleasing consequences.
We prove that, when the $t$-action is semisimple, the mixed Hodge structure on the torsion parts of the Alexander modules can be constructed directly using a finite cyclic cover, which, unlike an infinite cyclic cover, is always a complex algebraic variety.
This bypasses our rather abstract general construction of the mixed Hodge structure. 
We present two different viewpoints.
In the first, we utilize cap product with the pullback of a generator of $H^1(\C^*;\Q)$. 
In the second, we utilize a generic fiber of the algebraic map, which is always a complex algebraic variety. For the following result see Corollary \ref{cor:cup} and Corollary \ref{cor:fiber}.

\begin{thm}\label{finiteIntro}
	Assume the setting of Theorem \ref{mhsexistence}. 
	Let $i \geq 0$ and assume that the $t$-action on $A_i(U; \Q)$ is semisimple. 
	Let $N$ be such that the action of $t^N$ on $A_i(U^f;\Q)$ is unipotent, and let $U_N = \{(x,z) \in U \times \C^*\mid f(x) = z^N\}$ denote the corresponding $N$-fold cyclic cover.
	Equip the rational homology of $U_N$ with the (dual of) Deligne's mixed Hodge structure.
	
	\begin{enumerate}[(A)]
	\item Let $f_N\colon U_N \rightarrow \C^*$ denote the algebraic map induced by projection onto the second component, and let $\gen \in H^1(\C^*;\Q)$ be a generator. Then $A_i(U^f;\Q)$ is isomorphic as a mixed Hodge structure to the image of the mixed Hodge structure morphism induced by cap product with $f^*_N(\gen)$
	\[(-) \frown f^*_N(\gen) \colon H_{i+1}(U_N;\Q)(-1) \rightarrow H_i(U_N;\Q) ,\]
	where $(-1)$ denotes the $-1$th Tate twist of a mixed Hodge structure. 
	
	\item Let $F \hookrightarrow U$ be the inclusion of any generic fiber of $f$ and let $F \hookrightarrow U_N$ be any lift of this inclusion. 
	Then $A_i(U^f;\Q)$ is isomorphic as a mixed Hodge structure to the image of the mixed Hodge structure morphism
	\[H_i(F; \Q) \rightarrow H_i(U_N; \Q)\] 
	induced by the inclusion, where $H_i(F;\Q)$ is equipped with Deligne's mixed Hodge structure.
	\end{enumerate}
\end{thm}

Theorem \ref{finiteIntro}, when it applies, brings the mixed Hodge structures on the torsion parts of the Alexander modules down to earth, and reinforces the significance of semisimplicity. 
The first viewpoint granted by semisimplicity, in terms of cap products (Theorem \ref{finiteIntro}A), is suggested by the thickened complexes that play the central role in our construction (see Section \ref{ss:cupcap}). Regarding the second viewpoint, note that the homologies of different choices of generic fibers in the same degree may have different mixed Hodge structures, but \emph{any} choice is allowed in Theorem \ref{finiteIntro}B.
This shows that the mixed Hodge structures on the torsion parts of the Alexander modules are common quotients of the homologies of all generic fibers, when semisimplicity holds.

Our results in this paper show that semisimplicity is not a rare occurrence. In fact, we have proven that semisimplicity holds in many situations: it always holds on the torsion part of the first Alexander module (see Corollary \ref{cor:semisimple}) and, when $f$ is proper, it necessarily holds on the torsion parts of all Alexander modules (see Theorem \ref{thmsimple} and Corollary \ref{corss}). In fact, we do not know of any example where semisimplicity does not hold--see the open questions at the end of Chapter~\ref{sec:examples}. This lack of examples is mainly due to the fact that higher Alexander modules are harder to compute than the first. An interesting point is that we have used the mixed Hodge structure constructed in this paper to prove the semisimplicity of the first Alexander module for arbitrary $U$ and $f$. This was previously known in the case of certain Alexander modules associated to affine curve complements, as is explained in \cite[Corollary 1.7]{DL}, which combines results from \cite{KK,DN,dimca-jag} (see Remark~\ref{remk:A1Semisimple} for details).

We use the above-mentioned semisimplicity in the case when $f$ is proper to show the following (see Corollary \ref{cor:pure}).
\begin{thm}\label{thm:pur} If $f\colon U \to \C^*$ is a proper algebraic map, then $A_i(U^f;\Q)$ carries a pure Hodge structure of weight $-i$.
\end{thm}

\medskip 
There is a well-known and motivating (for our construction) mixed Hodge structure that we would be remiss not to explicitly address.
In the situations (a)--(c) mentioned above, the mixed Hodge structure on $A_i(U^\xi;\Q)$ coincides with the {\it limit} mixed Hodge structure on $H_i(F;\Q)$. In the situation considered in this paper, the epimorphism $\xi$ is realized by an algebraic map $f\colon U \to \mathbb{C}^*$. Let $D^*$ be a sufficiently small punctured disk centered at $0$ in $\C$, such that $f\colon f^{-1}(D^*)\rightarrow D^*$ is a fibration, and let $T^*=f^{-1}(D^*)$. The infinite cyclic cover $(T^*)^f$ is homotopy equivalent to $F$, where $F$ denotes any fiber of the form $f^{-1}(c)$, for $c\in D^*$. With this in mind, $(T^*)^f$ can be regarded as the canonical fiber of $f\colon f^{-1}(D^*)\rightarrow D^*$. If $f$ is proper, $H_i((T^*)^f;\Q)$ is also endowed with a limit mixed Hodge structure, which we compare with the one we construct on $A_i(U^f;\Q)$. In Chapter~\ref{rellim} we show the following:

\begin{thm}\label{comp} In the setup of Theorem \ref{mhsexistence}, assume $f$ is proper, and let $D^*$ be a small enough punctured disk centered at $0$ in $\C$ such that $f\colon f^{-1}(D^*)\rightarrow D^*$ is a fibration. Let $T^*=f^{-1}(D^*)$, and let $(T^*)^f$ be its corresponding infinite cyclic cover induced by $f$. Then, for all $i\geq 0$, the inclusion $(T^*)^f\subset U^f$ induces an epimorphism of $\Q$-mixed Hodge structures
\begin{align*}
H_i((T^*)^f;\Q) \twoheadrightarrow A_i(U^f;\Q),	
\end{align*}
where $H_i((T^*)^f;\Q)$ is endowed with its limit mixed Hodge structure. If, moreover, $f$ is a fibration, then the two mixed Hodge structures are isomorphic.
\end{thm}

\medskip

The paper is structured as follows. In Chapter~\ref{sec2}, we recall the relevant background and set notations for the rest of the paper. We introduce here the Alexander modules (homological and cohomological versions, and their descriptions in terms of local systems), differential graded algebras (dga), mixed Hodge complexes, and recall how the latter are used to get mixed Hodge structures on smooth varieties. We also summarize here the construction of the limit mixed Hodge structure associated to a family of projective manifolds. Chapter~\ref{sec3} is devoted to the theory of thickened complexes, which already made an appearance in \cite{BudurLiuWang}. The purpose of Chapter~\ref{sec4} is to show that, under certain technical assumptions, the thickened complex of a multiplicative mixed Hodge complex is again a mixed Hodge complex (Theorem \ref{mhsthickened}). In Chapter~\ref{mhsal}, we perform the construction of a suitable thickening of the Hodge-de Rham complex (see Theorem \ref{logthickenedmhs}) and prove Theorem \ref{mhsexistence} (see Corollary \ref{halexandermhs}). We also show that, fixing $f$, the mixed Hodge structure constructed here is independent of all choices (finite cover used to make the monodromy unipotent, and good compactification) and is functorial. 

The content of the remaining sections can largely be classified either as (1) comparing our mixed Hodge structures on the torsion parts of the Alexander modules with known mixed Hodge structures or (2) investigating applications. In Chapter~\ref{sec:main}, we prove that the infinite cyclic covering map induces morphisms of mixed Hodge structures from the torsion parts of the Alexander modules into the homology vector spaces of the base space (see Theorem \ref{geoIntro}). The consequences of this fact are discussed in Chapter~\ref{sec:consequences}, in particular we obtain sharp bounds on the weight filtrations (Theorem~\ref{boundsIntro}) and explore how the $t$-action interacts with the mixed Hodge structures.
It is here that we hit on the significance of semisimplicity, and prove theorems to which it relates (see Theorems \ref{tsemisimpleIntro} and \ref{finiteIntro}). 
We also show that our construction coincides with the mixed Hodge structure of case (\ref{dlt}) above, when the latter is defined.
In Chapter~\ref{sect:ss}, we prove the semisimplicity of $A_i(U^f; \Q)$ in the case when $f\colon U \to \C^*$ is a proper map (see Corollary \ref{corss}), as well as the purity statement of Theorem~\ref{thm:pur}. 
In Chapter~\ref{rellim}, we compare the mixed Hodge structure we constructed on the torsion part of the Alexander modules with the limit mixed Hodge structure on the generic fiber of the algebraic map realizing the infinite cyclic cover. In particular, Theorem \ref{comp} is proved in this chapter.
We apply our results to specific examples in Chapter~\ref{sec:examples}, more exactly to the first Alexander modules of hyperplane arrangements. We conclude by listing open questions regarding our constructions, as well as some of our expectations for what the answers will be.

The interested reader may also benefit from consulting the survey \cite{EGHMW}, which summarizes the main results, constructions and ideas of proofs from this paper.

\medskip

\textbf{Acknowledgements.} We thank Lingquan Ma and Mircea Mus\-ta\-\c{t}\u{a} for helpful discussions. We are also grateful to the referee for carefully reading the paper and for useful comments. 


	\chapter{Preliminaries}\label{sec2}

In this chapter, we recall some standard results about Alexander modules, mixed Hodge complexes, differential graded algebras and limit mixed Hodge structures. We closely follow the presentation of \cite{peters2008mixed}. However, we choose different conventions in several places. 

\section{Denotations and Assumptions}\label{ss:notations}
In this section, we fix the notations for the rest of the paper. 
Let $k$\index{k@$k$} denote a field of characteristic zero.
Unless otherwise stated, $\otimes$\index{$\otimes$} denotes $\otimes_k$.
If $V$ is a graded vector space, we denote by $V^i$ its $i$-th graded component.

Let $R$\index{R@$R$} denote the ring $k[t^{\pm 1}]$ of Laurent polynomials in the variable $t$\index{t@$t$!ring element} over the field $k$. Let $R_\infty$\index{R@$R$!R_infty@$R_\infty$} denote the ring $k[[s]]$ of formal power series in the variable $s$\index{s@$s$} over the field $k$. We identify $R$ with a distinguished subring of $R_\infty$ by setting $t = 1+s$. For $m \geq 1$ let $R_m$\index{R@$R$!R_m@$R_m$} denote the quotient ring $R/((t-1)^m) = R_\infty/(s^m)$. Throughout, $M^\vee$\index{${\cdot}^\vee$} will denote the (sometimes derived) dual of $M$ as an $R$-module, and we will use $M^{\vee_k}$\index{${\cdot}^{\vee_k}$} for the dual as a $k$-vector space.

If $U$ is a smooth manifold, let $\calE^\bullet_U$\index{E@$\calE^\bullet$} denote the real de Rham complex of sheaves on $U$.
If $U$ is moreover a complex manifold, let $\Omega^\bullet_U$\index{Omega@$\Omega^\bullet$} denote the holomorphic de Rham complex of sheaves on $U$.
If $X$ is a complex manifold, and $D \subset X$ is a simple normal crossing divisor, let $\logdr{X}{D}$\index{OmegaXlogD@$\logdr{X}{D}$} denote the log de Rham complex of sheaves on $X$.

When working with sheaves, we follow the notations of \cite{dimca2004sheaves}. In particular, the pullback is denoted $f^{-1}$\index{f-1@$f^{-1}$}. For a module $M$ (over $R$ or $k$) and a space $X$, we will use $\ul M_X$\index{Runderline@$\uR$} to denote the constant sheaf on $X$ with stalk $M$. $\Hom$ will denote the vector space or module of global sheaf morphisms, and $\Homm$ will be used to denote the sheaf of locally defined morphisms\index{Hom@$\Hom$}\index{Homm@$\Homm$}.

If the differential of a cochain complex is not specified, it is assumed to be denoted by $d$. Given any double complex $\calA^{\bullet,\bullet}$, we denote the total complex by $\Tot \calA^{\bullet, \bullet}$\index{Tot@$\Tot$}. We use the sign conventions that agree with \cite[III \S 3.2]{gelfandmanin}. Most importantly, given a map of complexes $f\colon A^\bullet\to B^\bullet$, its cone is $\Cone(f)^i = B^i\oplus A^{i+1}$\index{cone}, with differential $\begin{pmatrix}
d_B & f\\
0 & -d_A
\end{pmatrix}$. The maps $B^\bullet\to \Cone(f)\to A^\bullet[1]$ are induced by $\Id_B$ and $\Id_A$.

\section{Alexander Modules}
\label{ssAlex}

Let $U$\index{U@$U$} be a smooth connected complex algebraic variety, and let $f\colon U\rightarrow \C^*$\index{f@$f$} be an algebraic map inducing an epimorphism $f_*\colon\pi_1(U)\twoheadrightarrow \Z$ on fundamental groups. Let $\exp\colon\C\to \C^*$\index{exp@$\exp$} be the infinite cyclic cover\index{cover!infinite cyclic}\index{infinite cyclic cover}\index{infinite cyclic cover|see{Cover}}, and let $U^f$\index{Uf@$U^f$} be the following fiber product:
\begin{equation}\label{eq:fiberProductIntro}
\begin{tikzcd}
U^f\subset U\times \C\arrow[r,"f_\infty"] \arrow[d,"\pi"]\arrow[dr,phantom,very near start, "\lrcorner"]&
 \C \arrow[d,"\exp"] \\
 U\arrow[r,"f"] &
 \C^*.
\end{tikzcd}
\end{equation}
Under this presentation $U^f$ is embedded in $U\times \C$:
\[
U^f = \{ (x,z)\in U\times \C\mid f(x) = e^z\}.
\]
Let $f_\infty$\index{f_infty@$f_\infty$} be the restriction to $U^f$ of the projection $U\times \C\to \C$. Since $\exp$ is an infinite cyclic cover, $\pi\colon U^f\to U$\index{pi@$\pi$} is the infinite cyclic cover induced by $\ker f_*$. The group of covering transformations of $U^f$ is isomorphic to $\mathbb{Z}$, and it induces an $R$-module structure on each group $H_i(U^f;k)$, with $1\in \mathbb{Z}$ corresponding to $t \in R = k[t^{\pm 1}]$\index{t@$t$!deck transformation}. We will also say $t$ acts on $U^f$ as the deck transformation $(x,z)\mapsto (x,z+2\pi i)$.

\begin{dfn}
The \bi{$i$-th homological Alexander module}\index{Alexander module!homological} of $U$ associated to the algebraic map $f\colon U\rightarrow \C^*$ is the $R$-module
$$
H_i(U^f;k).
$$
\end{dfn}

\begin{remk}\label{rem:ACHTUNG}
The deck transformation group, generated by $t$, induces an automorphism of $H_i(U^f;k)$, but in principle it does not preserve the mixed Hodge structure that we will define on its torsion. We will show in Corollary~\ref{cor:t} and Proposition~\ref{prop:semisimpleIsEasyConverse} that this is the case if and only if the $t$-action is semisimple on $\Tors_R H_i(U^f;k)$. Therefore, it is important that $f_\infty$ is fixed, as choosing a map of the form $f_\infty\circ t^m = f_\infty + 2\pi i m$ for any $m\in \Z$ would still identify $U^f$ with the fiber product $U\times_{\C^*}\C$, but it could give rise to an isomorphic, but not equal, mixed Hodge structure on the same vector space.

This is to be expected: the limit mixed Hodge structure exhibits the same behavior, see \cite[11.2]{peters2008mixed}.
\end{remk}

For our purposes, it is more convenient to realize the Alexander modules as homology groups of a certain local system on $U$, which we now define. Let $\calL = \pi_! \ul k_{U^f}$\index{L@$\cL$}. The action of $t$ on $U^f$ as deck transformations induces an automorphism of $\calL$, making $\cL$ into a local system of rank $1$ free $R$-modules. For any point $x\in U$, the stalks are given by $
\cL_x = \bigoplus_{x'\in \pi^{-1}(x)} k$. The monodromy action of a loop $\gamma$ on $\cL_x$ interchanges the summands according to the monodromy action of $\gamma$ on $\pi^{-1}(x)$. Therefore, the monodromy of $\cL$ is the representation:
\[
\begin{array}{ccc}
 \pi_1(U) & \longrightarrow & \Aut_R(R)\\
 \gamma & \mapsto & \left(1\mapsto t^{f_*(\gamma)}\right).
\end{array}
\]
\begin{remk}\label{remk:isoLocal-Uf}
By \cite[Theorem 2.1]{KL}, there are natural isomorphisms of $R$-modules for all $i$: $
H_i(U;\calL)\cong H_i(U^f;k)$. These come from an isomorphism at the level of chain complexes.
\end{remk}
\begin{remk}\label{homologyAM}
The definition of the homological Alexander modules can be applied to every connected space $Y$ of finite homotopy type and any epimorphism $\pi_1(Y)\twoheadrightarrow \Z$. See for example \cite{DN, BudurLiuWang, Max06, MaxTommy}. Note that every complex algebraic variety has the homotopy type of a finite CW-complex \cite[p.27]{dimca1992hypersurfaces}.
\end{remk}

\begin{remk}\label{rem:conjugate}
If $M$ is a (sheaf of) $R$-modules, we will use $\ov{M}$ to denote $M$ with the conjugate $R$-module structure, where $t$ acts by $t^{-1}$. This is the same as saying $\ov{M} = M\otimes_R R$, where the tensor is via the map $R\to R$ sending $t$ to $t^{-1}$. Note that every (sheaf of) $R$-modules has a canonical $k$-linear isomorphism $M\cong \ov{M}$, namely $m\mapsto \ov m\coloneqq m\otimes 1$.
\end{remk}

\begin{dfn}\label{def:2.2.6}\index{Alexander module!cohomological}
The \bi{$i$-th cohomology Alexander module} of $U$ associated to the algebraic map $f$ is the $R$-module
$$
\overline{H^i(U;\calL)}.
$$
Note that by flatness there is an isomorphism 
\[
\ov{H^i(U;\cL)} = H^i(U;\cL)\otimes_R R\cong H^i(U;\cL\otimes_R R) = H^i(U;\ov\cL).
\]
\index{L@$\cL$!Lbar@$\ov\cL$}
\end{dfn}

\begin{remk}\label{UCT}
The Universal Coefficients Theorem (UCT)\index{universal coefficients theorem}\index{universal coefficients theorem|see{UCT}}\index{UCT} relates the two notions of Alexander modules as follows (for a proof see Lemma~\ref{lem:UCT}). There is a natural short exact sequence of $R$-modules
\[
0\to \Ext^1_R(H_{i-1}(U;\calL),R)\to H^i(U;\Homm_R(\calL,\ul R)) \to \Hom_R(H_i(U;\calL),R)\to 0.
\]
Moreover, this short exact sequence splits, but the splitting is not natural. Taking into account Remark~\ref{rem:oppositeDual} below, the middle term can be identified with $\ov{H^i(U;\cL)}$.
\end{remk}

\begin{remk}\label{rem:oppositeDual}
Both rank 1 local systems $\Homm_R(\cL,\ul R)$ and $\ov\cL$ have monodromy $\gamma\mapsto t^{-f_*(\gamma)}$. Let us make a canonical choice of isomorphism between them. Since $\cL=\pi_!\ul k$, for any $x\in U$, the stalk $\cL_x$ has a $k$-basis parametrized by $\pi^{-1}(x)$. Let us call this basis $\{\delta_{x'} \}_{x'\in \pi^{-1}(x)}$, and note that $t\delta_{x'} = \delta_{tx'}$\index{delta@$\delta_x$}\index{delta@$\delta_x$!deltabar@$\ov\delta_x$}\index{delta@$\delta_x$!deltawedge@$\delta_x^\wedge$} Then we have bases:
\[
\ov\cL_x = k\langle \ov\delta_{x'}\rangle_{x'\in \pi^{-1}(x)} ;\quad 
\Homm_R(\cL,\ul R)_x \cong k\langle \delta_{x'}^\wedge\rangle_{x'\in \pi^{-1}(x)}.
\]
Here $\delta_{x'}^\wedge$ is the element in $\Hom_R(\cL_x,R)$ mapping $\delta_{x'}\mapsto 1$. Let $\gamma$ be a loop acting on the stalks by the monodromy. Then:
\[
\gamma\ov\delta_{x'}=
 t^{-f_*\gamma}\ov\delta_{x'}; \quad 
 t\ov\delta_{x'} = \ov{\delta}_{t^{-1}x'}
; \quad
\gamma\delta_{x'}^\wedge = t^{-f_*\gamma}\delta_{x'}^\wedge;\quad t\delta_{x'}^\wedge = \delta_{t^{-1}x'}^\wedge .
\]
Therefore, mapping $\ov\delta_{x'}\mapsto \delta_{x'}^\wedge$ on every stalk gives an isomorphism
\[
\ov\cL\cong \Homm_R(\cL,\ul R).
\]
We will freely identify these local systems from now on.
\end{remk}

\begin{remk}\label{rem:cohom_not_iso} In general, the cohomological Alexander modules are not isomorphic as $R$-modules to the cohomology of the corresponding infinite cyclic cover. Indeed, $\overline{H^i(U;\calL)}$ is a finitely generated $R$-module for all $i$. However, if $H_i(U^f;k)$ is not a finite dimensional $k$-vector space, then $H^i(U^f;k)$ is not a finitely generated $R$-module. Alternatively, $H^*(U^f;k)$ and $H^*(U;\cL)$ are the hypercohomology of $U$ with coefficients in $R\pi_*\pi^{-1}\ul k_U=\pi_*\ul k_{U^f}$ and $\cL=\pi_!\ul k_{U^f}$ respectively, so they need not be isomorphic.

However, the torsion parts of the cohomological Alexander modules and the cohomology of $U^f$ are related in a way which will be made more precise in Proposition~\ref{propcanon}.
\end{remk}

\section{Universal Coefficient Theorem}
In this paper, we work with local systems of $R$-modules. For an $R$-module $M$ (or a complex thereof), we denote $M^\vee = R\Hom_R^\bullet(M,R)$\index{${\cdot}^\vee$}. Similarly, for a sheaf of $R$-modules $\mathcal{M}$ on a topological space $X$ (or a complex thereof), we denote $\mathcal{M}^\vee = R\Homm_R^\bullet(\mathcal{M},\underline{R}_X)$.

\begin{dfn}\label{dfn:complex}
Let $U$ be a connected locally contractible space, and let $\cF^\bullet$ be a bounded complex of local systems of finitely generated \textit{free} $R$-modules on $U$. Let $x\in U$, and denote $\pi_1\coloneqq\pi_1(U,x)$. Let $\wt U$ be any covering space induced by a normal subgroup $H$ of $\pi_1$ on which the inverse image of all the local systems $\cF^i$ becomes trivial. Following \cite[VI.3]{whitehead}, we define
\begin{itemize}
\item $C_\bullet(\cF^\bullet) \coloneqq C_\bullet(\wt U;R)\otimes_{R[\pi_1/H]} \cF^\bullet_x$
\item $C^\bullet(\cF^\bullet) \coloneqq \Hom_{R[\pi_1/H]}^\bullet(C_\bullet(\wt U;R),\cF^\bullet_x)$
\end{itemize}
Here $C_\bullet(\wt U;R)$ are the singular chains with coefficients in $R$ with a right action of $\pi_1/H$ by deck transformations.\index{C(F)@$C_\bullet(\cF^\bullet)$}\index{C(F)@$C^\bullet(\cF^\bullet)$}
\end{dfn}

\begin{remk}\label{remk:dfncomplex}
Let $\cM$ be a local system of $R$-modules on $U$, and let $\cF^\bullet\rightarrow \cM$ be a resolution of $\cM$, where $U$ and $\cF^\bullet$ are as in Definition \ref{dfn:complex}.

$C_\bullet(\cF^\bullet)$ is a chain complex computing the homology of $\cM$. Indeed, since $C_\bullet(\wt U;R)$ is a complex of free $R[\pi_1/H]$-modules, the tensor product appearing in the formula defining $C_\bullet(\cF^\bullet)$ coincides with the left derived tensor product, so
$
C_\bullet(\cF^\bullet)$ is quasiisomorphic to $C_\bullet(\wt U;R)\otimes_{R[\pi_1/H]} \cM_x= C_{\bullet}(\cM)$ (the quasiisomorphism given by the resolution $\cF^\bullet\rightarrow \cM$),
which is the complex computing the homology of $\cM$ by \cite[VI.3]{whitehead}.

Similarly, $C^\bullet(\cF^\bullet)$ is quasiisomorphic to $\Hom^\bullet_{R[\pi_1/H]}(C_\bullet(\wt U;R), \cM_x)= C^{\bullet}(\cM)$ (the quasiisomorphism given by the resolution $\cF^\bullet\rightarrow \cM$). $C^\bullet(\cM$) is a chain complex computing the cohomology of $\cM$ (by loc. cit.). Again, note that the derived $R\Hom_{R[\pi_1/H]}^\bullet$ coincides with the usual $\Hom_{R[\pi_1/H]}^\bullet$ in the definition of $C^\bullet(\cF^\bullet)$, by the freeness of $C_\bullet(\wt U;R)$. 
\end{remk}

The Universal Coefficients Theorem will be important in Chapter~\ref{sec:main}. Below is the version that we will need to use, which contains the definitions of the maps $\UCT_R$ and $\UCT_{k}$. We include the proof, since we will later make use of the spectral sequence that appears in it, and we have not been able to find this version of the Universal Coefficients Theorem in the literature. When using homological notation, we let $H_j=H^{-j}$. 

\begin{lem}[Definition of the maps $\UCT_R$ and $\UCT_k$]\label{lem:UCT}\index{UCT!$\UCT_k$}\index{UCT!$\UCT_R$}\index{UCT}
Let $U$ be a locally contractible space, and let $\cF^\bullet$ be a bounded complex of local systems of finitely generated \textit{free} $R$-modules. Let $x\in U$, and let $\pi_1\coloneqq\pi_1(U,x)$. Since $R$ is a PID, we have the following natural Universal Coefficients Theorem short exact sequence
\[
0\to \Ext^1_R(H_{i-1}(C_\bullet((\cF^\bullet)^\vee));R) \to H^i(C^\bullet(\cF^\bullet)) \to \Hom_R(H_i(C_\bullet((\cF^\bullet)^\vee)),R)\to 0,
\]
which gives us a natural isomorphism
$$\UCT_R \colon
\Ext^1_R(H_{i-1}(C_\bullet((\cF^\bullet)^\vee));R) \xrightarrow{\sim} \Tors_R H^{i}(C^\bullet(\cF^\bullet)).$$

Now, suppose that $\cF^\bullet$ is a single local system $\cM$ located at degree 0. If the stalk of $\cM$ at $x$ is a finitely generated \textit{torsion} $R$-module, we can consider $\cM$ as a local system of $k$-vector spaces. Analogously to the short exact sequence above, we can replace $R$ by $k$ and $\cF^\bullet$ by $\cM$, to obtain the natural isomorphism $\UCT_k\colon
H^i(U;\cM) \xrightarrow{\sim} \Hom_{k} ( H_i( U;\Homm_{k}(\cM,k)), k)$.
\end{lem}
\begin{proof}
By tensor-hom adjunction, we have the following isomorphism. 
\[
C^\bullet(\cF^\bullet) = \Hom_{R[\pi_1/H]}^\bullet\left(C_\bullet(\wt U;R),\Hom^\bullet_R((\cF^\bullet)^\vee_x,R)\right) \cong\]
\[
\cong \Hom_{R}^\bullet\left(C_\bullet(\wt U;R)\otimes_{R[\pi_1/H]}(\cF^\bullet)^\vee_x ,R\right) = \Hom_{R}^\bullet\left(C_\bullet((\cF^\bullet)^\vee) ,R\right).
\]
Note that we are using $(\cF^\bullet)^\vee = R\Homm_R^\bullet(\cF^\bullet,\underline R)$, but since $\cF^\bullet$ is a complex of local systems of free $R$-modules, then we have $(\cF^\bullet)^\vee = \Homm_R^\bullet(\cF^\bullet,\underline R)$, which is again a bounded complex of local systems of free $R$-modules. We are also using that for bounded complexes of local systems, $\Homm_R^\bullet$ commutes with taking stalks.

Now, since $R$ is a PID, we can proceed as in the proof of the Universal Coefficient Theorem in \cite[Chapter 3.1]{hatcher}, to arrive at the natural Universal Coefficients short exact sequence from the statement of the lemma.
\begin{equation}\label{eq:UCTses}
0\to \Ext^1_R(H_{i-1}(C_\bullet((\cF^\bullet)^\vee));R) \to H^i(C^\bullet(\cF^\bullet)) \to \Hom_R(H_i(C_\bullet((\cF^\bullet)^\vee)),R)\to 0.
\end{equation}
Alternatively, we can obtain this short exact sequence using a spectral sequence, namely the one obtained from the Grothendieck spectral sequence for the composition of two derived functors $RF\circ RG$, taking $G = \Id$ and $F = \Hom_R(\bullet, R)$. Later on we will need to worry about what the maps are, so let us recall how it is obtained. Let $K=k(t)$ be the field of fractions of $R$, viewed as an $R$-module. Let $I^\bullet$ be the following injective resolution of $R$, with the obvious maps:
\[
R\to \underset{(=I^0)}{K}\to \underset{(=I^1)}{K/R.}
\]
Consider the map of complexes induced by the map $I^0\to I^1$:
\[
\Hom_R^\bullet(C_\bullet((\cF^\bullet)^\vee),I^0) \to \Hom_R^\bullet(C_\bullet((\cF^\bullet)^\vee),I^1).
\]
We can turn this map into a double complex where each of the complexes above is a row, and the map between them (with some sign changes) becomes the vertical differential. We consider the spectral sequences associated to this double complex.

Consider the spectral sequence that starts by taking the vertical cohomology. Since $R\to I^\bullet$ is an injective resolution, the cohomology of the columns is $\Ext^i(C_\bullet((\cF^\bullet)^\vee),R)$, which vanishes if $i\neq 0$ since $C_\bullet((\cF^\bullet)^\vee)$ is a complex of free $R$-modules. We are left with only one nonzero row, namely the complex $\Hom^{\bullet}_R(C_\bullet((\cF^\bullet)^\vee),R)$. Therefore, the spectral sequence converges in the second page, to the cohomology of $\Hom^{\bullet}_R(C_\bullet((\cF^\bullet)^\vee),R)$. In other words, the cohomology of the total complex is the cohomology of $C^\bullet(\cF^\bullet)$.

Now we consider the other spectral sequence, starting with the horizontal cohomo\-logy. Since $I^i$ is injective, there are natural isomorphisms:
\[
H^j(\Hom_R^\bullet(C_\bullet((\cF^\bullet)^\vee),I^i)) \cong \Hom_R(H_j(C_\bullet((\cF^\bullet)^\vee)),I^i). 
\]
Since $I^\bullet$ is an injective resolution of $R$, we have isomorphisms as follows:
\[
H^i(\Hom_R(H_j(C_\bullet((\cF^\bullet)^\vee)),I^\bullet)) \cong \Ext^i_R(H_j(C_\bullet((\cF^\bullet)^\vee)),R ) .
\]
Further, the map $R\to I^\bullet$ induces maps $R\to K$ and $\beta\colon K/R[-1]\to R$, which means that the isomorphisms above are induced by the inclusion $R\to K$ if $i=0$ and by $\beta$ if $i=1$ (we will need to use this fact in Proposition~\ref{prop:mapsAreEqual}). At this point, the spectral sequence has converged (just by looking at the degrees), yielding a filtration of $H^{i}(C^\bullet(\cF^\bullet))$ by the cohomology groups above. In other words, we obtain the desired short exact sequence~(\ref{eq:UCTses}).

The isomorphism $\UCT_k$ is obtained by replacing $R$ by $k$, and $\cF^\bullet$ by $\cM$ in the proof above.
\end{proof}

\begin{remk}\label{remk:UCT}
If we take $\cF^\bullet=\ov\cL$ seen as a complex in degree $0$, we get that $(\cF^\bullet)^\vee\cong \cL$. Taking into account Remark \ref{remk:dfncomplex}, the UCT short exact sequence in Lemma \ref{lem:UCT} (in $R$) becomes
$$
0\to \Ext^1_R(H_{i-1}(U;\cL);R) \to H^i(U;\ov\cL) \to \Hom_R(H_i(U;\cL),R)\to 0,
$$
that is, the UCT short exact sequence in Remark~\ref{UCT}.
\end{remk}

\section{Cohomology Alexander Modules and Duality}\label{ss:duality}

The relation between the cohomology Alexander modules and the corresponding infinite cyclic cover can be made more precise as follows.

\begin{prop}
\label{propcanon}\index{functoriality}
There is a natural $R$-module isomorphism $$\Tors_R H^i(U;\ov\calL)\cong (\Tors_R H_{i-1}(U^f;k))^{\vee_k},$$ where $^{\vee_k}$ denotes the dual as a $k$-vector space. 

Moreover, this isomorphism is functorial in the following sense. Let $U_1$ and $U_2$ be smooth connected complex algebraic varieties, with algebraic maps $f_j\colon U_j\rightarrow \C^*$ such that $f_j$ induces an epimorphism in fundamental groups for $j=1,2$, and assume that there exists an algebraic map $g\colon U_1\rightarrow U_2$ that makes the following diagram commutative.

\begin{center}
\begin{tikzcd}[row sep = 1em]
U_1\arrow[rd,"f_1"']\arrow[rr, "g"] & \ & U_2\arrow[ld,"f_2"] \\
\ & \C^* & \
\end{tikzcd}
\end{center}

Let $\pi_j\colon U_j^{f_j}\to U_j$ be the corresponding infinite cyclic covers, and let $\calL_j = \pi_! \ul k_{U_j^{f_j}}$ be the local systems of $R$-modules induced by $f_j$ for $j=1,2$. Then, $g$ induces a commutative diagram of $R$-modules compatible with the above isomorphism as follows
\begin{center}
\begin{tikzcd}[row sep = 1.4em]
\Tors_RH^i(U_2;\ov{\calL_2})\arrow[r,"\cong"]\arrow[d, "g^*"]& (\Tors_R H_{i-1}(U_2^{f_2},k))^{\vee_k}\arrow[d,"(g_*)^{\vee_k}"] \\
\Tors_R H^i(U_1;\ov{\calL_1})\arrow[r,"\cong"]& (\Tors_R H_{i-1}(U_1^{f_1},k))^{\vee_k}.
\end{tikzcd}
\end{center}
\end{prop}
\begin{proof}
By the UCT (Remark~\ref{remk:UCT}), there is a short exact sequence of $R$-modules
\[
0\to \Ext^1_R\big(H_{i-1}(U;\calL),R\big)\to H^i(U;\ov\calL) \to \Hom_R\big(H_i(U;\calL),R\big)\to 0.
\]
Since $R$ is a PID, the $R$-module $\Hom_R(H_i(U;\calL),R)$ is free and the $R$-module $\Ext^1_R(H_{i-1}(U;\calL),R)$ is  torsion. Moreover, there is a natural isomorphism of $R$-modules
\[\Ext^1_R(H_{i-1}(U;\calL),R)\cong\Ext^1_R(\Tors_R H_{i-1}(U^f;k),R).\]
Therefore, we have a natural isomorphism of $R$-modules
\begin{equation}
\Tors_R H^i(U;\ov \calL)\cong \Ext^1_R(\Tors_R H_{i-1}(U^f;k),R).
\label{eqn1}
\end{equation}

Using equation (\ref{eqn1}) and applying Lemma \ref{lemcanon} (see below) to the finitely generated torsion $R$-module $\Tors_R H_{i-1}(U^f;k)$ yields a natural isomorphism of $R$-modules
$$
\Tors_R H^i(U;\ov\calL)\xrightarrow{\cong}(\Tors_R H_{i-1}(U^f;k))^{\vee_k},
$$
concluding our proof of the first part of this proposition.

The functoriality follows from the functoriality of the UCT and Lemma \ref{lemcanon}.
\end{proof}

\begin{lem}
Let $A$ be a finitely generated torsion $R$-module. Then, there exists a natural isomorphism of $R$-modules
$$
\Res\colon\Ext^1_R(A,R)\xrightarrow{\cong} A^{\vee_k},
$$\index{residue!Res@$\Res$}
where $^{\vee_k}$ denotes the dual as a $k$-vector space.
\label{lemcanon}
\end{lem}

This lemma is essentially a special case of the relative local duality theorem of Smith, see \cite[Theorem 1.5]{localduality}. Since we are only using a very special version of the local duality theorem, we sketch a more elementary proof.

We first make the following definition, which we will use in the proof.

\begin{dfn}\label{dfn:Res}
Let $K = k(t)$, and let $\beta\colon K/R[-1]\to R$ be the map in the derived category of $R$-modules given by the floor diagram:\index{beta@$\beta$}
\[
\begin{tikzcd}[row sep = 1.4em]
0 \arrow[d]& K\arrow[r,"1",leftarrow]\arrow[l,leftarrow]\arrow[d,"1"]\arrow[dr,phantom,"\sim"] & R\arrow[d]\\
K/R & K/R \arrow[l,"1",leftarrow]\arrow[r,leftarrow]& 0.
\end{tikzcd}
\]
For an $R$-module $M$, this induces a map
\[
R^1\Hom_R^\bullet(M,\beta)\colon \Hom_R(M,K/R)\to \Ext^1_R(M,R).
\]
Let $\res_*$\index{residue!res@$\res_*$} be the composition $\Hom_R(\cdot,K/R)\to \Hom_k(\cdot,K/R)\to \Hom_k(\cdot,k)$, where the first arrow is just the forgetful arrow and the second arrow is the residue of a rational function (see below for details). For every torsion $R$-module $A$, the isomorphism of $R$-modules $\Res$ is defined as $$\Res=\res_*\circ (R^1\Hom_R^\bullet(A,\beta))^{-1}\colon\Ext^1_R(A,R)\rightarrow \Hom_k(A,k),$$ i.e., it is the composition of the inverse of the map $\Hom_R(A,K/R)\rightarrow \Ext^1_R(A,R)$ in the long exact sequence corresponding to $R\Hom_R(A,\cdot)$ applied to the short exact sequence of $R$-modules $0\rightarrow R\rightarrow K\rightarrow K/R\rightarrow 0$ with $\res_*$.
\end{dfn}

\begin{proof}[Sketch of proof of Lemma~\ref{lemcanon}]

First, we prove that $\Res$ is an isomorphism assuming that $k$ is algebraically closed. Since $R$ is a PID, the following is an injective resolution of $R$:
\[
0\rightarrow R\rightarrow K \rightarrow K/R\rightarrow 0.
\]
By definition, there is a natural isomorphism 
$$
\Ext^1_R(A, R)\cong \Hom_R(A,K/R)/\Hom_R(A,K).
$$ 
Since $A$ is a torsion $R$-module, $\Hom_R(A,K)=0$ and hence we have a natural isomorphism
\begin{equation}\label{eqiso} 
\Ext^1_R(A, R)\cong \Hom_R(A,K/R).
\end{equation}

We next define a $k$-linear map $\mathrm{res}\colon K/R\to k$ as follows\index{residue!res@$\res$}. 
Using the division algorithm and partial fraction decomposition, every element $b\in K$ can be written uniquely as a finite sum
\[
b=v(t)+\sum_{i=1}^r\sum_{j=1}^{q_i}\frac{\alpha_{i,j}}{(t-\beta_i)^j},
\]
where $v(t)\in R$, and $\alpha_{i,j}\in k, \beta_i\in k\setminus \{0\}$ are constants. We define $\mathrm{res}(b)=\sum_{i=1}^{r}\alpha_{i, 1}$. Composing with the $k$-linear map $\mathrm{res}$, we get a $k$-linear map 
\[
\mathrm{res}_*\colon \Hom_R(A,K/R)\to \Hom_k(A, k)=A^{\vee_k}.
\]
One can check that this is a homomorphism of $R$-modules. In fact, it is easy to check this after localizing at each $\beta_i\in k\setminus \{0\}=\mathrm{mSpec}(R)$. Since $A$ is a finitely generated torsion $R$-module, both $\Hom_R(A,K/R)$ and $\Hom_k(A, k)$ decompose into direct sums of the localizations.

Next, we prove that $\mathrm{res}_*$ is an isomorphism of $R$-modules. 
We claim that $\ker(\mathrm{res})\subset K/R$ does not contain any non-zero $R$-submodule $M$. 
In fact, without loss of generality, we can assume that $M$ is supported at one point $\beta_1$. Then a nonzero element of $M$ is of the form $\big[\frac{w(t)}{(t-\beta_1)^q}\big]$, where $q\geq 1$ and $w(\beta_1)\neq 0$. Now, 
\[
(t-\beta_1)^{q-1}\cdot \Big[\frac{w(t)}{(t-\beta_1)^q}\Big]=\Big[\frac{w(t)}{t-\beta_1}\Big]\in M
\]
and $\mathrm{res}(\frac{w(t)}{t-\beta_1})=w(\beta_1)\neq 0$. Thus, the above claim follows, which then implies that $\mathrm{res}_*$ is injective. 

Using the primary decomposition of finitely generated $R$-modules, it is easy to see that as $k$-vector spaces, $\Hom_R(A,K/R)$ and $\Hom_k(A, k)$ have the same dimensions. Therefore $\mathrm{res}_*$ is an isomorphism. Hence, by equation \eqref{eqiso}, we have a functorial $R$-module isomorphism
\[
\Ext^1_R(A,R)\xrightarrow{\cong} A^{\vee_k}.
\]

When $k$ is not algebraically closed, we let $\bar{k}$ be its algebraic closure. The above arguments show that there exists a natural isomorphism
\begin{equation}\label{eqkbar}
\Ext^1_{R\otimes_k \bar{k}}(A\otimes_k \bar{k},R\otimes_k \bar{k})\xrightarrow{\cong} (A\otimes_k \bar{k})^{\vee_k}.
\end{equation}
Notice that the division algorithm and the partial fraction decomposition are invariant under the Galois action of $\mathrm{Gal}(\bar{k}/k)$. Therefore, the isomorphism \eqref{eqkbar} descends to an isomorphism over $k$. 
\end{proof}

\begin{remk}\label{rem:annoyingSign}
Let $\beta$ be as in Definition~\ref{dfn:Res}. In the derived category, there is an exact triangle:
\[
R\xrightarrow{1} K\xrightarrow{1} K/R \xrightarrow{\beta[1]} R[1].
\]
Here $\beta[1]$ is the following floor diagram. Note the sign:
\[
\begin{tikzcd}[row sep = 1.4em]
0 \arrow[d]& K\arrow[r,"1",leftarrow]\arrow[l,leftarrow]\arrow[d,"-1"] & R\arrow[d]\\
K/R & K/R \arrow[l,"1",leftarrow]\arrow[r,leftarrow]& 0.
\end{tikzcd}
\]
This sign means that the map $\Hom_R(\cdot ,K/R)\to \Ext^1_R(\cdot,R)$ appearing in the obvious Ext long exact sequence is opposite to the map in the derived category induced by $\beta[1]$. We will use the former map, but this sign won't affect any of our computations in this paper.
\end{remk}

\begin{cor} \label{isocohom}
In the above notations, assume moreover that $H_{i}(U^f;k)$ is a torsion $R$-module for some $i\geq 0$. Then, there exists a canonical isomorphism
$$
\Tors_R H^{i+1}(U;\ov\calL)\cong H^{i}(U^f;k).
$$
Moreover, if $H_{i+1}(U^f;k)$ is also a torsion $R$-module, then, so is $H^{i+1}(U;\calL)$. Hence, in that case, $H^{i+1}(U;\ov\calL)$ and $H^{i}(U^f;k)$ are naturally isomorphic.
\end{cor}

\begin{proof}
We have
$$
\Tors_R{H^{i+1}(U;\ov\cL)} 
\overset{(\ref{propcanon})}\cong \Hom_k(\Tors_R H_{i}(U^f;k),k) = \Hom_k( H_{i}(U^f;k),k) \overset{\UCT_k}  
\cong H^i(U^f;k).
$$
The rest of the proof follows from the UCT (written as in Remark~\ref{remk:UCT}).
\end{proof}

The result of the above corollary applies, in particular, to the special case when $U=\C^n \setminus \{f=0\}$ and $i \neq n$, with $f\colon U\to \C^*$ induced by a complex polynomial in general position at infinity; see \cite{DL,Max06}.


\section{Relationship with the Generic Fiber}\label{genfiber}

Let $U,f,\pi$ be as in Section~\ref{ssAlex}. By Verdier's generic fibration theorem\index{generic fibration theorem}\index{Verdier's generic fibration theorem}\index{Verdier's generic fibration theorem|see{generic fibration theorem}} \cite[Corollary 5.1]{verdier}, there exists a finite set of points $\calB\subset\C^*$ such that $$
f\colon f^{-1}(\C^*\setminus\calB)\longrightarrow \C^*\setminus\calB$$ is a locally trivial fibration. Being the pullback of a locally trivial fibration, $f_\infty$ is a locally trivial fibration away from $\exp^{-1}(\calB)$ as well.

\begin{dfn}\label{def:genericFiber}
We will say a fiber $f_\infty^{-1}(c)$ (resp. $f^{-1}(e^c)$) is generic if $c\notin\exp^{-1}(\calB)$ (resp. if $e^c\notin\calB$).\index{fiber!generic}
\end{dfn}

Let $F$ be a generic fiber of $f$. For any $c\in \C$, $\pi$ is an isomorphism between a neighborhood of $f_\infty^{-1}(c)$ and a neighborhood of $f^{-1}(e^c)$ (recall that $\pi$ is the pullback of $\exp\colon\C\to \C^*$). In particular, if $F = f^{-1}(e^c)$, there is a unique lift $i_{\infty}\colon F \hookrightarrow U^f$ of the inclusion $i\colon F\hookrightarrow U$ inducing an isomorphism $F\cong f_\infty^{-1}(c)$.\index{f_infty@$f_\infty$}\index{iinfty@$i_\infty$}

Let us recall the description of the monodromy action of $F$, along with its compatibility with the deck action on $U^f$.
\begin{lem}\label{lem:fiberMonodromy}
For any $j\geq 0$, the map $i_{\infty}$ induces a map:
\[
(i_{\infty})_j\colon H_j(F;k)\longrightarrow H_j(U^f;k).
\]
This map is an $R$-module homomorphism if we let $t\in R$ act on the right hand side as the deck transformation of $U^f$ and on the left-hand side as the \emph{clockwise} monodromy of $f$, i.e. the action of $-1\in \Z\cong \pi_1(\C^*)$.

Further, $H_j(F;k)$ is finite dimensional, so it is necessarily a torsion $R$-module. Therefore, the image of $(i_{\infty})_j$ is contained in $\Tors_R H_j(U^f;k) = A_j(U^f;k)$.\index{t@$t$!monodromy}
\end{lem}
\begin{proof}
Suppose we have any two generic fibers $F_1,F_2$, with $F_i = f_{\infty}^{-1}(c_i)$, and a path $\gamma$ connecting $c_1$ with $c_2$ in $\C\setminus \exp^{-1}(\calB)$. Then, $f_\infty$ restricted to $f^{-1}_\infty(\gamma)$ is a locally trivial fibration over a contractible base, and therefore $f^{-1}_\infty(\gamma)$ deformation retracts to any fiber. In particular, we obtain a homotopy equivalence $F_1\simeq F_2$. Therefore, we are free to choose a fixed fiber $F\subseteq U^f$ such that $\exp(f_\infty(F))$ is contained in a small neighborhood of $0$. Recall that we can use $\pi$ to identify $F\cong \pi(F)$. Let $c_0 = f_\infty(F)$.

Let $\gamma$ be a counterclockwise circle centered at $0$ passing through $e^{c_0}$. By definition, the monodromy action of $\gamma$ on $\pi(F)\cong F$ is given as follows. Lifting $\gamma$ to a line segment $\wt \gamma = [c_0,c_0+2\pi i]\subset \C$, we use the same deformation retraction argument from above to obtain the monodromy map as the following composition from left to right (defined up to homotopy):
\[
\pi(F)\xleftarrow[\cong]{\pi} f_\infty^{-1}(c_0) \xhookrightarrow[\simeq]{} f_\infty^{-1}(\wt\gamma) \xhookleftarrow[\simeq]{} f_\infty^{-1}(c_0+2\pi i) \xrightarrow[\cong]{\pi} \pi(F).
\]
Let us see how this relates to the deck action. The preimage $\exp^{-1}(\gamma)$ is a line, and therefore it is contractible. Therefore, $\wt F\coloneqq f_\infty^{-1}(\exp^{-1}(\gamma))$ deformation retracts to all its fibers. Fix the homotopy equivalence $\alpha\colon\pi(F)\cong f_\infty^{-1}(c_0)\xhookrightarrow{\simeq} \wt F$. The deck transformation $t$ acts on $\wt F$ by restricting the action on $U^f$. Consider the following diagram.
\[
\begin{tikzcd}[row sep = 1.6em]
\pi(F) \arrow[dd,"\alpha"',bend right = 50]\arrow[r,equals]&
\pi(F)\arrow[from=r,"\gamma"']&
\pi(F) \arrow[dd,"\alpha",bend left = 50]\\
 f_\infty^{-1}(c_0)\arrow[u,"\pi","\cong"']\arrow[d,hookrightarrow,"\simeq"] \arrow[r,"t"]&
 f_\infty^{-1}(c_0+2\pi i)\arrow[d,hookrightarrow,"\simeq"]
 \arrow[u,"\pi","\cong"']&
 f_\infty^{-1}(c_0)\arrow[u,"\pi","\cong"']\arrow[d,hookrightarrow,"\simeq"]
\\
 \wt F\arrow[r,"t"]&
 \wt F\arrow[r,equals] &
 \wt F. 
\end{tikzcd}
\]
Here $\gamma$ denotes the homotopy action of $\gamma$, as defined above. The diagram is commutative, except for the right hand rectangle, which only commutes up to homotopy (by definition of the monodromy). From this diagram we see that on $\pi(F)$, $\alpha^{-1}\circ t \circ \alpha = \gamma^{-1}$, as desired. Choosing a different lift $F$ of $\pi(F)$, we could repeat this proof to obtain the same result.
\end{proof}

\begin{prop}\label{genf}
Under the notations of Lemma~\ref{lem:fiberMonodromy}, the homomorphism $$(i_{\infty})_j\colon H_j(F;k)\longrightarrow \Tors_R H_j(U^f;k)$$ is an epimorphism for all $j\geq 0$. In other words, the $R$-module $H_j(F;k)$ maps onto to the torsion part of the $j$-th homology Alexander module.

Moreover, the kernel of $(i_{\infty})_j$ can be described as follows: for any $b\in \C^*$, let $T_b=f^{-1}(D_b)$, where $D_b$ is a small disk around $b$. Using a path to identify $F$ with a general fiber in $T_b$, let
\[
K_b = \ker \left(
H_j(F;k)\to H_j(T_b;k)
\right).
\]
The kernel of $(i_{\infty})_j$ is the $R$-module generated by the subgroups $K_b$ for $b\in \calB$, where $\calB\subset \C^*$ consists of the points over which $f$ is not a locally trivial fibration.
\end{prop}

\begin{proof}
We use an argument similar to \cite[2.3]{DN2}. As in the proof of Lemma \ref{lem:fiberMonodromy}, we may choose $F$ to be a generic fiber of $f_\infty$ such that $|\exp(f_\infty(F))|\ll 1$. Recall that $\pi$ induces an isomorphism between $F$ and a generic fiber of $f$.

For all $c\in \C$, let $D_c$ be a small disk centered at $c$, $D_c^*=D_c\setminus\{c\}$, $T_c\coloneqq f^{-1}(D_c)$ and $T_c^*\coloneqq f^{-1}(D_c^*)$. There is a deformation retraction of $\C^*$ to a subspace we will call $X$, which is the union of $D_0^*$ and $D_b$ for $b\in \calB$, together with a path connecting $D_0^*$ with $D_b$ for each $b$. We choose these paths to be contractible and pairwise disjoint.

Then, $\C$ has a deformation retraction to $\exp^{-1}(X)$, which is the union of $D_{\wt b}$ for $\wt b\in \exp^{-1}(\calB)$, the preimages of all the paths in $X$ and a left half-plane, which is the preimage of $D_0^*$. We can further deformation retract $\C$ to $\wt X$, defined as follows: for every $\wt b\in \exp^{-1}(\calB)$, fix a contractible path from $f_\infty(F)$ to the boundary of $D_{\wt b}$, so that all the paths are disjoint away from $f_\infty(F)$. Let $\wt X$ be the union of all of the $D_{\wt b}$'s and all those paths joining each of the disks to $f_\infty(F)$. It will later be convenient for us to use closed disks to construct $\wt X$.

Recall that $f_\infty$ is a locally trivial fibration away from $\exp^{-1}(\calB)$. Let $Y = f_\infty^{-1}(\wt X)$. Using the fibration, we have that $U^f$ deformation retracts to $Y$. Therefore, the inclusion $Y\to U^f$ induces isomorphisms for all $i$:
\[
H_j(Y, F;k)\cong H_j(U^f, F;k).
\]

Now we consider two contractible neighborhoods $V_1\subsetneq V_2$ of $f_{\infty}(F)$ in $\wt X$: let $V_1$ be the complement of the closed disks around each $\wt b\in \calB$ (i.e. a wedge sum of segments), and let $V_2$ be a contractible open neighborhood of $\ov V_1$ that doesn't intersect $\exp^{-1}(\calB)$. For example, $V_2$ can be taken to be the union of $V_1$ with a small open disk around $\ov V_1\cap D_{\wt b}$ for every $\wt b\in \exp^{-1}(\calB)$.

Since $V_i$ are both contractible and $f_\infty$ is a locally trivial fibration over them, we have that the inclusions are homotopy equivalences $F\simeq f_\infty^{-1}(V_1) \simeq f_\infty^{-1}(V_2)$. Therefore, by these equivalences and by excision, we have isomorphisms induced by inclusions:
\[
H_j(Y,F;k) \cong H_j(Y,f_\infty^{-1}(V_2);k) \cong H_j(Y\setminus f_\infty^{-1}(V_1),f_\infty^{-1}(V_2\setminus V_1);k).
\]
Since $\wt X\setminus V_1$ is a disjoint union of $D_{\wt b}$,
\[
H_j(Y,F;k) \cong \bigoplus_{\wt b\in \exp^{-1}(\calB)} H_j(f^{-1}_{\infty}(D_{\wt b}),f_\infty^{-1}(V_2\cap D_{\wt b}) ;k).\]
Over $V_2\cap D_{\wt b}$, $f_\infty$ is a trivial fibration (since the base is contractible). Therefore, any path connecting $f_\infty (F)$ with some point $c\in V_2\cap D_{\wt b}$ induces a homotopy equivalence $F\simeq f_\infty^{-1}(c)$, which in turn is homotopy equivalent to $f_\infty^{-1}(V_2\cap D_{\wt b})$. If we just denote $F=f_\infty^{-1}(c)$, we have
\[
\bigoplus_{\wt b\in \exp^{-1}(\calB)} H_j(f^{-1}_{\infty}(D_{\wt b}),f_\infty^{-1}(V_2\cap D_{\wt b}) ;k) \cong 
\bigoplus_{\wt b\in \exp^{-1}(\calB)} H_j(f^{-1}_{\infty}(D_{\wt b}),F ;k) .
\]
Now we consider the $R$-module structure on the cohomology of the pair $(U^f,F)$. If $F = f_\infty^{-1}(c)$ for $\Re c \ll 0$ ($\Re$ denotes the real part), $F$ is homotopy equivalent to $\wt F \coloneqq f_{\infty}^{-1}( \{z \mid \Re z = \Re c \})$. Then the deck transformations ($z\mapsto z+2\pi i$) act on the pair $(U^f,\wt F)$, making the maps in the long exact sequence $R$-linear. Recall from Lemma~\ref{lem:fiberMonodromy} that the deck action on $H_j(\wt F;k)$ coincides with the monodromy action on $H_j(F;k)$.

Note that the action of $t\in R$ on $U^f$ restricts to $\bigsqcup_{\wt b\in \exp^{-1}(\calB)}(f^{-1}_{\infty}(D_{\wt b}),F ;k)$: we can choose the disks to be $\Z$-invariant, and we can choose the paths identifying the fibers so that the monodromy action agrees with the action on $F$. Therefore, the homology group above is an $R$-module. Further, it is a free $R$-module, since $f^{-1}_{\infty}(D_{\wt b}) \cong f^{-1}(D_b)\times \Z$:
\[
\bigoplus_{\wt b\in \exp^{-1}(\calB)} H_j(f^{-1}_{\infty}(D_{\wt b}),F ;k) \cong
\bigoplus_{b\in \calB} R\otimes_k H_j(f^{-1}(D_{ b}),F ;k).
\]
We have the long exact sequence of the pair $(U^f,F)$:
\[
H_{j+1}(U^f,F;k)\to H_j(F;k)\to H_j(U^f;k)\to H_j(U^f,F;k) .
\]
Since the last group is free, the restriction $\Tors_R H_j(U^f;k)\to H_j(U^f,F;k)$ must vanish. We have seen above that the inclusion induces an isomorphism $$H_{j+1}(U^f,F;k) \cong \bigoplus_{b\in \calB} H_{j+1}(\Z\times T_b,F;k),$$ so the sequence above restricts to an exact sequence of $R$-modules:
\[
\bigoplus_{b\in \calB} R\otimes_k H_{j+1}(T_b,F;k)\to H_j(F;k)\to \Tors_R H_j(U^f;k)\to 0.
\]
By the exactness, the kernel of $(i_{\infty})_j$ is the image of the first map. This image is the $R$-module generated by $H_{j+1}(T_b,F;k)$ for all $b\in \C^*$ (note that for $b\notin \calB$, this homology group vanishes). Therefore, it is enough to describe the image of: 
\[
H_{j+1}(T_b,F;k) \to H_j(F;k).
\]
By the long exact sequence of the pair $(T_b,F)$, the image of the above map is $K_b$.
\end{proof}

\begin{remk}\label{remk:tube}
The homomorphism $(i_{\infty})_j$ in Proposition~\ref{genf} factors through $H_j((T^*)^f;k)$, where $T^*=f^{-1}(D^*)$ for $D^*$ a punctured disk centered at $0$ in $\C$. We choose $D^*$ to be small enough that the restriction of $f$ to $T^*$ is a fibration, so $(T^*)^f$ is homotopy equivalent to $F$, and the epimorphism $H_j(F; k)\twoheadrightarrow \Tors_{R}H_j(U^f;k)$ factors as
\begin{equation}\label{eqn:tubeTorsion}
H_j(F; k)\xrightarrow{\cong}H_j((T^*)^f;k)\twoheadrightarrow \Tors_{R}H_j(U^f;k),
\end{equation}
where the second morphism doesn't depend on the choice of base points.

Now, if we take the $k$-dual of $(i_\infty)_j$, we get a monomorphism of $R$-modules
$$
\big(\Tors_{R}H_j(U^f;k)\big)^\vee\hookrightarrow \big(H_j(F; k)\big)^\vee.
$$

By Proposition \ref{propcanon} and the isomorphism $(H_j(F; k))^\vee\cong H^j(F;k)$, we get a monomorphism of $R$-modules, only depending on a choice of lift of the base point:
\begin{equation}\label{eqn:monoCohom}
{\Tors_R H^{j+1}(U;\overline\calL)}\hookrightarrow H^j(F;k). 
\end{equation}
By the isomorphism in (\ref{eqn:tubeTorsion}), $H_j((T^*)^f;k)$ is finite dimensional, so it is a torsion $R$-module. Hence, by Corollary~\ref{isocohom}, $H^j((T^*)^f;k)\cong H^{j+1}(T^*;\overline\calL)$ (note that all the results of Section~\ref{ss:duality} are purely topological, so they can be applied to $T^*$). This implies that for $D^*$ small enough, the monomorphism in (\ref{eqn:monoCohom}) factors in this way:
$$
{\Tors_R H^{j+1}(U,\overline\calL)}\hookrightarrow {H^{j+1}(T^*;\overline\calL)}\xrightarrow{\cong}H^j(F;k),
$$
where the first map is given by restriction and does not depend on the choice of base points.
\end{remk}

\begin{remk}
If $f$ is proper, then Proposition~\ref{genf} can be expressed using the nearby cycles of $f$ at $b\in \C^*$, as follows.

By \cite[Part II, Section 6.13]{stratifiedmorsetheory}, $T_b$ has a deformation retraction to $F_b=f^{-1}(b)$, and we can compose with the inclusion of $F$ into $T_b$ to obtain the specialization map in cohomology $H^j(F_b;k)\to H^j(F;k)$. Letting $\psi_{f-b} \ul k$ be the nearby cycles of $f$ supported on $f=b$, there is an isomorphism $H^j(F;k) \cong H^j(F_b;\psi_{f-b} \ul k)$ by \cite[Expos\'e XIII p.103]{SGA72}. Under this isomorphism, the specialization map agrees with the map $H^j(F_b;k)\to H^j(F_b;\psi_{f-b} \ul k)$ induced by the canonical morphism $i^{-1}_b\ul k\to \psi_{f-b} \ul k$, where $i_b\colon F_b\hookrightarrow U$ is the inclusion. Let $K_b^{\vee_k}\subseteq H^j(F;k)$ be the image of the specialization map. Using Proposition~\ref{propcanon}, we have the dual map $\left((i_{\infty})_j\right)^{\vee_k}\colon\Tors_R H^{j+1}(U;\ov\cL)\to H^j(F;k)$. The image of $\left((i_{\infty})_j\right)^{\vee_k}$ is the largest $R$-module contained in $K_b^{\vee_k}$ for all $b\in \calB$.\end{remk}


\section{Monodromy Action on Alexander Modules and Finite Cyclic Covers}\label{sscover}

Let $k=\C$, let $U$ be a smooth connected complex algebraic variety, and let $f\colon U\rightarrow \C^*$ be an algebraic map inducing a local system $\calL$ of $R$-modules as in Section~\ref{ssAlex}. Since $R$ is a PID, we have the primary decomposition
$$
A_i(U^f;\C)=\Tors_R H_i(U;\calL)\cong \bigoplus_{i=1}^r R/\big((t-\lambda_i)^{p_i}\big)
$$
with $p_i\geq 1$ for all $i=1,\ldots, r$. The set $\{\lambda_i \in \C \mid i=1, \ldots, r\}$ is uniquely determined by $A_i(U^f;\C)$.
\begin{prop}\label{eig1}
Every $\lambda_i$ defined above is a root of unity.
\label{monodromy}
\end{prop}
\begin{proof}
When $U$ is quasi-projective, this fact is proved in \cite[Proposition 1.4]{BudurLiuWang} by reducing to a structure theorem for the cohomology jump loci of a smooth quasi-projective variety \cite[Theorem 1.1]{budur2015jump}. The structure theorem of cohomology jump loci is generalized to arbitrary smooth complex algebraic varieties in \cite[Theorem 1.4.1]{budur2017absolute}. So the statement of \cite[Proposition 1.4]{BudurLiuWang} applies to any smooth complex algebraic variety.
\end{proof}
In particular, applying Proposition \ref{propcanon}, one has the following.

\begin{cor}
Let $k=\C$. The eigenvalues of the action of $t$ on $\Tors_R H^*(U;\ov\calL)$ are all roots of unity.
\label{roots}
\end{cor}

The main goal of this paper is to construct a natural mixed Hodge structure on $\Tors_R H^*(U;\ov\calL)$ (for $k=\Q,\R$). For reasons that will become apparent later on, we would like to be able to reduce this problem to the case where the eigenvalues of the action of $t$ on $\Tors_R H^*(U;\ov\calL)$ are not just roots of unity, but equal to $1$. The rest of this section is devoted to justifying such a reduction.

Let $N\in\N$ be such that the $N$-th power of all the eigenvalues of the action of $t$ on $\Tors_R H_*(U;\calL)$ (equiv. on $\Tors_R H^*(U;\ov\calL)$) is $1$ for all $*$. Consider the following pull-back diagram:
$$
\begin{tikzcd}
 U_N = \{(x,z)\in U\times \C^*\mid f(x) = z^N \}\arrow[d,"p"] \arrow[dr,phantom,very near start,"\lrcorner"] \arrow[r,"f_N"] & \C^*\arrow[d,"z\mapsto z^N"] \\
 U \arrow[r,"f"] & \C^*.
\end{tikzcd}
$$\index{cover!finite cyclic}\index{UN@$U_N$}\index{fN@$f_N$}\index{p@$p$}
Here $p$ is an $N$-sheeted cyclic cover. Notice that all of the maps involved in this diagram are algebraic, and that $U_N$ is also a smooth algebraic variety. We can define $(U_N)^{f_N}$, $(f_N)_\infty$, $\pi_N$ and $\cL_N$ as in Section~\ref{ssAlex} for the map $f_N\colon U_N\to \C^*$. We define:\index{UNfN@$U_N^{f_N}$}\index{L@$\cL$!LN@$\cL_N$}\index{piN@$\pi_N$}\index{fNinfty@$(f_N)_\infty$}
\[
\f{\theta_N}{U^f}{U_N^{f_N}}{U\times \C\ni (x,z)}{(x,e^{z/N}, z/N)\in U_N\times \C\subset U\times \C^* \times \C.}
\]\index{thetaN@$\theta_N$}
It fits into the following commutative diagram:
\begin{equation}\label{eq:UN}
\begin{tikzcd}[column sep = 7em]
U^f \arrow[d,"f_\infty"]\arrow[r,"\sim"',"\theta_N", dashed] 
\arrow[rrr,"\pi",rounded corners,to path=
{ --([yshift = 0.7em]\tikztostart.north)
-- ([yshift = 0.7em]\tikztotarget.north)\tikztonodes
-- (\tikztotarget)}] 
&
U^{f_N}_N \arrow[d,"(f_N)_\infty"]\arrow[r,"\pi_N"]\arrow[dr,phantom,very near start,"\lrcorner"] &
U_N \arrow[d,"f_N"] \arrow[r,"p"]\arrow[dr,phantom,very near start,"\lrcorner"] &
U \arrow[d,"f"] \\
\C \arrow[r,"z\mapsto\frac{z}{N}"]\arrow[rrr,"\exp", rounded corners,to path=
{ --([yshift = -0.7em]\tikztostart.south)
-- ([yshift = -0.7em]\tikztotarget.south)\tikztonodes
-- (\tikztotarget)}] &
\C \arrow[r,"\exp"]&
\C^* \arrow[r,"z\mapsto z^N"]&
\C^*.
\end{tikzcd}
\end{equation}
The map $\theta_N$ allows us to identify $U^f$ with $U_N^{f_N}$ in a canonical way, which we will do from now on. In particular, we can also identify the constant sheaves $\ul k_{U_N^{f_N}}$ and $\ul k_{U^f}$ canonically.

On fundamental groups, $p$ induces an isomorphism:
\[
\pi_1(U_N) \cong \{ \gamma\in \pi_1(U) \mid f_*(\gamma)\in N\Z\}.
\]
Let $R(N)\coloneqq k[t^{N},t^{-N}]$\index{R@$R$!R(N)@$R(N)$}. Since $p$ is an $N$-sheeted covering space, $p_*\calL_N =p_!\calL_N$ is a local system of $R(N)$-modules of rank $N$ on $U$. 
Since $R$ is a rank $N$ free $R(N)$-module, we can also consider $\calL$ as a local system of rank $N$ free $R(N)$-modules on $U$. 
In fact, $\theta_N$ induces the following isomorphism of local systems of $R(N)$-modules (we are using equal signs for canonical isomorphisms):
\[
p_*\cL_N= p_!\cL_N =p_! \left((\pi_N)_! \ul k_{U_N^{f_N}} \right)\overset{\theta_N}{\cong} p_!\left((\pi_N \circ \theta_N)_!\ul k_{U^f}\right) = \pi_! \ul k_{U^f} = \cL.
\]
Since $p$ is a finite covering, $p_*=p_!$ is the left and right adjoint to the sheaf inverse image $p^{-1}$. In particular $p_*$ is exact.

As in Remark~\ref{rem:oppositeDual}, there is a basis of the stalk of $\cL$ (resp. $\cL_N$) parametrized by the fiber of $\pi$ (resp. $\pi_N$). If we denote $\delta_{x'}$ the elements in this basis, then $\theta_N$ maps an element of the form $\delta_{(x,e^{z/N},z)}$ (in the stalk of $\cL_N$ at $(x,e^{z/N})\in U_N$) to $\delta_{(x,z)}$. This discussion and some immediate consequences are summarized in the following lemma.

\begin{lem}\label{233}
\label{lemLocal}
In the above notations, $\theta_N$ induces a canonical isomorphism of local systems of $R(N)$-modules $
\theta_{\cL_N}\colon p_*\calL_N \cong \calL$. This induces canonical isomorphisms in the homology of these local systems, which is the same as the map induced by $\theta_N$, i.e. the following diagram commutes:\index{thetaLN@$\theta_{\cL_N}$}
\[
\begin{tikzcd}[row sep = 1em]
H_i(U_N;\cL_N) \arrow[r,"\theta_{\cL_N}","\sim"']\arrow[d,"\sim"] &
H_i(U;\cL) \arrow[d,"\sim"] \\
H_i(U_N^{f_N};k) \arrow[r,"\theta_N","\sim"'] &
H_i(U^f;k).
\end{tikzcd}
\]
Identifying each local system with its conjugate, we obtain another canonical isomorphism $\ov\theta_{\cL_N}\colon\ov{p_*\cL_N}\cong p_*\ov\calL_N \cong \ov\calL$. This induces in cohomology the same map as $\theta_N$, i.e. the following diagram commutes:
\[
\begin{tikzcd}[row sep = 1em]
\Tors_R H^{i+1}(U;\ov\cL) \arrow[from = r,"\ov\theta_{\cL_N}"',"\sim"]\arrow[d,"\sim"] &
\Tors_{R(N)}H^{i+1}(U_N;\ov\cL_N) \arrow[d,"\sim"] \\
(\Tors_R H_i(U^f;k))^{\vee_k} \arrow[from=r,"\theta_N^{\vee_k}"',"\sim"] &
(\Tors_{R(N)} H_i(U_N^{f_N};k))^{\vee_k}.
\end{tikzcd}
\]
The vertical isomorphisms come from Proposition~\ref{propcanon}.
\end{lem}
\begin{proof}
The statement about homology is a straightforward application of the discussion in \cite[VI.3]{whitehead}. For the cohomology, we will outline the proof and leave the details to the reader. First, it is enough to show that this diagram commutes, since the isomorphism in Proposition~\ref{propcanon} is natural:
\begin{equation}\label{eq:thetaN}
{\scriptsize
	\begin{tikzcd}[column sep = 2.4em]
		H^{i+1}(U;\ov\cL) \arrow[d,leftrightarrow] \arrow[r,leftrightarrow,"\ov\theta_{\cL_N}"] &
		H^{i+1}(U;\ov{p_!\cL_N})\arrow[d,leftrightarrow] \arrow[r,leftrightarrow,"(1)" ]&
		H^{i+1}(U_N;\ov{\cL_N})\arrow[d,leftrightarrow]
		\\
		H^{i+1}(U;\Homm_R(\cL,\ul R))\arrow[r,leftrightarrow,"\theta_{\cL_N}^\vee"] &
		H^{i+1}(U;\Homm_{R}(p_!\cL_N,\ul R)) \arrow[r,leftrightarrow,"(2)"]
		&
		H^{i+1}(U_N;\Homm_{R(N)}(\cL_N,\ul {R(N)})) 
		\\
		\Ext^1_R(H_{i}(U;\cL),R) \arrow[u,"\UCT_R"]\arrow[r,leftrightarrow,"\theta_{\cL_N}"]& 
		\Ext^1_R(H_{i}(U;p_!\cL_N),R) \arrow[u,"\UCT_R"]\arrow[r,leftrightarrow,"(3)"]& 
		\Ext^1_{R(N)}(H_{i}(U_N;\cL_N),R(N))\arrow[u,"\UCT_{R(N)}"]
	\end{tikzcd}
}
\end{equation}
The top row of vertical arrows is the identification of the dual of a rank 1 free (local system of) modules with its conjugate, as in Remark~\ref{rem:oppositeDual}. The top left square commutes because this identification is natural, as one can easily verify. By $\theta_{\cL_N}^\vee$ (resp. $\ov\theta_{\cL_N}$) we mean the result of applying $\Homm_R(\cdot, \ul R)$ (resp. the conjugate structure) to $\theta_{\cL_N}$.

The bottom row of vertical arrows are all the maps of Lemma~\ref{lem:UCT}, which are isomorphisms onto the torsion of the codomain. The bottom left square commutes by the naturality of these maps. The arrow labeled (1) comes from the relation between cohomology and pushforward.

According to the statement we want to prove, the arrow labeled (3) must come from the isomorphism $\Tors_RH_i(U;p_!\cL_N)\cong \Tors_{R(N)} H_i(U_N;\cL_N)$, after taking $k$-duals and using Lemma~\ref{lemcanon} together with the fact that $\Tors_R = \Tors_{R(N)}$.

The arrow labeled (2) is a composition of two maps. We start with the adjunction between restriction of scalars and $\Hom$: $\Hom_R(A,\Hom_{R(N)}(R,B))\cong \Hom_{R(N)}(A_{R(N)},B) $. We have an $R$-linear isomorphism $R\cong\Hom_{R(N)}(R,R(N))$ where $g(t)$ goes to the map $f(t)\mapsto \frac{1}{N}\sum_{\xi^N = 1} g(\xi t)f(\xi t) $. This gives us an isomorphism
\[
\begin{array}{rcl}
\Homm_R(p_!\cL_N,\ul R)&\cong &\Homm_{R(N)}(p_!\cL_N,\ul{R(N)})\\
\left(\delta_{(x,e^{z/N},z)}^\wedge\colon\delta_{(x,e^{z/N},z)}\mapsto 1\right) &\mapsto & \left(\delta_{(x,e^{z/N},z)}^\wedge\colon\delta_{(x,e^{z/N},z)}\mapsto 1\right).
\end{array}
\]
Note that on the left-hand side, $\delta_{(x,e^{z/N},z)}^\wedge$ is $R$-linear, so it sends $$\delta_{(x,e^{(z+2\pi i k)/N},z+2\pi i k)}\mapsto t^k.$$ On the right-hand side, its image sends $
\delta_{(x,e^{(z+2\pi i k)/N},z+2\pi i k)}$ to $0$ whenever $k$ is not a multiple of $N$. Verdier duality (\cite{verdierDuality}, see also \cite[V, 7.17]{borel}) then gives the following isomorphism, taking into account that, since $p$ is a finite covering map, $p^!= p^{-1}$:
\[
\begin{array}{rcl}
\Homm_{R(N)}(p_!\cL_N,\ul{R(N)})&\cong & p_*\Homm_{R(N)}(\cL_N,\ul{R(N)}) = p_*\cL_N^\vee\\
\delta_{(x,e^{z/N},z)}^\wedge &\mapsto & \left(\delta_{(x,e^{z/N},z)}^\wedge\colon\delta_{(x,e^{z/N},z)}\mapsto 1\right).
\end{array}
\]
The stalk at $x$ of $p_*\cL_N^\vee$ is the sum of the stalks of $\cL_N^\vee$ at $p^{-1}(x)$, and $\delta_{(x,e^{z/N},z)}^\wedge$ is $0$ on all these stalks except at $(x,e^{z/N})$. We can now see that the top right square above commutes, as it comes from applying cohomology to the following commutative square, where the elements are mapped as in the right-hand square (the notation $\ov\cdot$ is as in Remark~\ref{rem:conjugate}):
\[
{
\begin{tikzcd}[row sep = 1.5em, column sep = 1em]
	\ov{p_!\cL_N} \arrow[r,leftrightarrow]\arrow[d] &
	p_!\ov{\cL_N} \arrow[d] &
	\ov{\delta_{(x,e^{z/N},z/N)}} \arrow[r,mapsto]\arrow[d,mapsto] &
	\ov{\delta_{(x,e^{z/N},z/N)}} \arrow[d,mapsto]
	\\
	\Homm_R(p_!\cL_N,\ul R) \arrow[r,leftrightarrow] &
	p_*\Homm_{R(N)}({\cL_N},\ul{R(N)}) & 
	\delta_{(x,e^{z/N},z/N)}^\wedge \arrow[r,mapsto] &
	\delta_{(x,e^{z/N},z/N)}^\wedge.
\end{tikzcd}
}
\]
It remains to show that the bottom right square of (\ref{eq:thetaN}) commutes. To this end, consider the arrow labeled $(3)$. Up to composing with the isomorphism $ H_i(U;p_!\cL_N)\cong H_i(U_N;\cL_N)$, it is the bottom path in the following diagram.
\[
{
\begin{tikzcd}[row sep =2em, column sep = {7em,between origins}]
	&  
	\makebox[1cm]{$\Ext^1_R(H_i(U;p_!\cL_N),\Hom_R(R,R(N)))$}\arrow[dr]
	\\
\Ext^1_R(H_i(U;p_!\cL_N),R) \arrow[d,"\Res","\sim"'] \arrow[ur]
& &
\Ext^1_{R(N)}(H_i(U;p_!\cL_N),R(N))\arrow[d,"\Res","\sim"'] \\
\Hom_k(\Tors_R H_i(U;p_!\cL_N),k) \arrow[rr,equals]& &
\Hom_k(\Tors_{R(N)} H_i(U;p_!\cL_N),k) 
\end{tikzcd}
}
\]
On the top part of the diagram we have the same $R$-linear isomorphism $R\cong \Hom_R(R,R(N))$ that was used above, composed with (Ext applied to) the adjunction map $\Hom_R(R,R(N))\mapsto R(N)$, sending $\varphi\mapsto \varphi(1)$. We leave to the reader the verification that this diagram commutes, noting that one can use the classification of modules over a PID to reduce to the case where $H_i(U;p_!\cL_N)$ is a cyclic module.

Therefore, to show that the bottom right square in (\ref{eq:thetaN}) commutes, we may use the description of the map above. We show that it commutes by taking the complexes whose cohomologies give rise to this square. Take a base point $x_N\in U_N$, let $x=p(x_N)$ and $\pi_1 \coloneqq \pi_1(U,x)\supset \pi_1(U_N,x_N) \eqqcolon \pi_1'$. Denote the stalk $M \coloneqq (\cL_N)_{x_N}$, seen as a $R(N)[\pi_1']$-module. Note that $(p_*\cL_N)_x\cong k[\pi_1]\otimes_{k[\pi_1']} M$, which becomes an $R$-module by letting $t\in R$ act as $\gamma\in \pi_1$ such that $f_*(\gamma)$ is the counterclockwise generator of $\pi_1(\C^*)$. Let $C_\bullet$ be the singular chain complex of the universal cover of $U$ with $k$ coefficients, seen as a right $k[\pi_1]$-module by the inverse of deck transformations. Then, the square we are interested in arises from the cohomology of the commutative diagram below.
\[
\begin{tikzcd}[column sep=1.5em]
 \Hom_{k[\pi_1]}^\bullet(C_{\bullet},\Hom_{R}(k[\pi_1]\otimes_{k[\pi_1']} M, R)) \arrow[r]
 &
 \Hom_{k[\pi_1']}^\bullet(C_{\bullet},\Hom_{R(N)}(M,R(N))) \\
\Hom^\bullet_R(C_\bullet \otimes_{k[\pi_1]} k[\pi_1]\otimes_{k[\pi_1']} M,R) \arrow[u,"\text{t-h}"]\arrow[r]& 
\Hom^\bullet_{R(N)}(C_\bullet \otimes_{k[\pi_1']} M,R(N)).\arrow[u,"\text{t-h}"]
\end{tikzcd}
\]
To obtain the bottom	 right square in (\ref{eq:thetaN}), we must take cohomology and then restrict to the torsion of the bottom row, applying the universal coefficient theorem. The vertical arrows are tensor-hom adjunction.

We can directly verify that the diagram commutes. Note that $C_\bullet$ is a free $k[\pi_1]$-module, with a basis that we will denote $\{c_j\}_{j\in J}$. $M$ is a rank $1$ free module over $R(N)$. Let us call a basis (for both module structures) $\{m\}$, and note that $m$ also generates $M$ over $k[\pi_1']$. Take $\gamma\in \pi_1$ to be any lift of the generator of $\pi_1/\pi_1'\cong \Z/N\Z$. Then $\{ \gamma^k\}_{0\le k<N}$ is a $k[\pi_1']$-basis of $k[\pi_1]$. Let $P\colon R\to R(N)$ be the projection $f(t) \mapsto \frac{1}{N}\sum_{\xi^N=1} f(\xi t)$. The elements of the groups above have the following form. They are parametrized by picking $a_{j}\in R$ for $j\in J$.
\[
\begin{tikzcd}
\left\{
c_{j} \mapsto (\gamma^{k}\otimes m \mapsto t^ka_{j})
\right\}_{(a_j)_{j\in J}} \arrow[r] &
\left\{
\gamma^{-k}c_{j} \mapsto ( m \mapsto P(t^k a_{j}))
\right\}_{(a_j)_{j\in J}} 
	 \\
	 \left\{
c_{j}\otimes \gamma^{k} \otimes m \mapsto t^ka_{j}
\right\}_{(a_j)_{j\in J}} \arrow[u] \arrow[r]&
\left\{
\gamma^{-k}c_{j}\otimes m \mapsto P(t^ka_{j})
\right\}_{(a_j)_{j\in J}}.\arrow[u]
\end{tikzcd}
\]
We will leave to the reader the verification that the maps are indeed defined as the diagram suggests, so that the diagram commutes as desired.
\end{proof}

\begin{remk}
Notice that the only eigenvalue of the action of $t^N$ on the module \linebreak$\Tors_{R(N)} H^*(U_N;\ov\calL_N)$ is $1$. So Lemma \ref{233} allows us to reduce the problem of constructing a mixed Hodge structure on $\Tors_{R} H^*(U;\ov\calL)$ to the case when the only eigenvalue is $1$.
\label{remEigenvalue}
\end{remk}


\section{Differential Graded Algebras} 
By a \bi{commutative differential graded $k$-algebra (cdga)} is meant a triple $$(A,d,\wedge)$$ such that:\index{cdga}\index{commutative differential graded algebra}\index{commutative differential graded algebra|see{cdga}}
\begin{itemize}
 \item $(A, \wedge)$ is a positively graded unital $k$-algebra.
 \item $a \wedge b = (-1)^{\deg(a)\deg(b)} b \wedge a$ for homogeneous $a,b \in A$.
 \item $(A,d)$ is a cochain complex.
 \item $d(a \wedge b) = da \wedge b + (-1)^{\deg(a)} a \wedge db$ for homogeneous $a,b \in A$.
\end{itemize} 
Notice that when we write cdga, the field $k$ is implicit.
We often write $A$ instead of $(A,d,\wedge)$ when the differential and multiplication are understood.

When we discuss Hodge complexes, we will often work with filtered cdgas whose filtrations are compatible with the differential and the multiplication. 
Suppose $(A,d,\wedge)$ is a cdga.
By an \bi{increasing cdga-filtration} on $(A, d, \wedge)$ we mean an increasing filtration $W_{\lc}$ on $A$ such that 
\[
W_iA \wedge W_jA \subset W_{i+j}A \quad\text{and}\quad d(W_iA)\subset W_iA
\]
for all integers $i$ and $j$.
By a \bi{decreasing cdga-filtration} on $(A, d, \wedge)$ we mean a decreasing filtration $F^{\lc}$ on $A$ such that 
\[
F^iA \wedge F^jA \subset F^{i+j}A \quad\text{and}\quad d(F^iA)\subset F^iA
\]
for all integers $i$ and $j$. 
One defines \bi{cdga-filtrations on a sheaf of cdgas} analogously, by looking at the cdgas of sections over arbitrary open subsets. 

\section{Mixed Hodge Structures and Complexes}\label{ss:MHSsAndComplexes}
\textit{Let $k$ denote a subfield of $\R$ and $X$ denote a topological space in this section.}

The purpose of this section is to compile relevant definitions and to set notations related to mixed Hodge structures and complexes.

\begin{dfn}[{\cite[Definition 2.1]{peters2008mixed}}]\label{dfn:pureHS}\index{Hodge!structure}\index{filtration!Hodge}
Let $m\in\Z$, $V$ be a finite dimensional $k$-vector space, and let $V_{\C}=V\otimes_k \C$ be its complexification. A \bi{$k$-Hodge structure of weight $m$ }on $V$ is a direct sum decomposition (\bi{Hodge decomposition}):
$$
V_{\C}=\bigoplus_{\substack{p,q\in\Z\\p+q=m}}V^{p,q}\quad\text{ such that }V^{p,q}=\overline{V^{q,p}}.
$$
The data of the Hodge decomposition is equivalent to a decreasing filtration $F^p$ on $V_{\C}$ (the \bi{Hodge filtration}) with the property that $V_{\C} = F^p\oplus\ov{F^q}$ for all $p,q\in\Z$ such that $p+q=m+1$. More precisely, one can be obtained from the other as follows:
$$
F^p=\bigoplus_{r\geq p} V^{r,m-r};\quad V^{p,q}=F^p\cap\ov{F^q}.
$$
\end{dfn}

\begin{dfn}[{\cite[Definition 3.1]{peters2008mixed}}]\label{dfn:MHS}\index{mixed Hodge structure}\index{mixed Hodge structure|see{MHS}}\index{MHS}\index{filtration!weight}
Let $V$ be a finite dimensional $k$-vector space, and let $V_{\C}=V\otimes_k \C$ be its complexification. A \bi{$k$-mixed Hodge structure (MHS)} on $V$ is the data of an increasing filtration $W_{\lc}$ on $V$ (\bi{weight filtration}) and a decreasing filtration $F^{\lc}$ on $V_{\C}$ (\bi{Hodge filtration}), with the following property: for every $m\in \Z$, on $\Gr_m^WV\coloneqq \frac{W_{m}}{W_{m+1}}$, the filtration induced $F^p$ on the complexification gives $\Gr_m^WV$ a pure Hodge structure of weight $m$.
\end{dfn}

\begin{dfn}[{\cite[Definition 2.32]{peters2008mixed}}]\index{Hodge!complex}
If $m \in \Z$ then a \bi{$k$-Hodge complex of sheaves on $X$ of weight $m$} is a triple $\calK^\bullet = (\calK^\bullet_k, (\calK^\bullet_{\C},F), \alpha)$
where:
\begin{itemize}
 \item $\calK^\bullet_k$ is a bounded below cochain complex of sheaves of $k$-vector spaces on $X$ such that the $\mathbb{H}^*(X,\calK_k^\bullet)$ are finite-dimensional.
 \item $\calK^\bullet_\C$ is a bounded below cochain complex of sheaves of $\C$-vector spaces on $X$ and $F$ is a decreasing filtration on $\calK^\bullet_\C$.
 \item $\alpha\colon \calK^\bullet_k \dashrightarrow \calK^\bullet_\C$ is a \bi{pseudo-morphism} of complexes of sheaves of $k$-vector spaces on $X$ (i.e.\ a morphism in the derived category) that induces a \bi{pseudo-isomorphism} $\alpha \otimes 1\colon \calK^\bullet_k \otimes \C \dashrightarrow \calK^\bullet_\C$ (i.e.\ an isomorphism in the derived category).\index{pseudo-morphism}\index{pseudo-isomorphism}
 \item the filtration induced by $F$ and $\alpha$ on $\bH^*(X,\calK^\bullet_k) \otimes \C \cong \bH^*(X,\calK^\bullet_\C)$ endows the $k$-vector space $\mathbb{H}^*(X, \calK^\bullet_k)$ with a $k$-Hodge structure of weight $*+m$. 
 \item The spectral sequence $\bH^{p+q}(X,\Gr_F^p \calK^\bullet_{\C})\Rightarrow \bH^{p+q}(X,\calK^\bullet_{\C})$ associated to $(\calK^\bullet_\C, F)$ degenerates at the $E_1$-page (see \cite[Definition 2.32]{peters2008mixed} for more details).
\end{itemize}
\end{dfn}

\begin{dfn}[{\cite[Definition 3.13]{peters2008mixed}}]\index{mixed Hodge complex}\index{Hodge!mixed Hodge complex}\index{Hodge!mixed Hodge complex|see{mixed Hodge complex}}
A \bi{$k$-mixed Hodge complex of sheaves on $X$} is a triple $\calK^\bullet = ((\calK^\bullet_k,W_{\lc}), (\calK^\bullet_\C, W_{\lc},F^{\lc}), \alpha)$ 
where:
\begin{itemize}
 \item $\calK^\bullet_k$ is a bounded below cochain complex of sheaves of $k$-vector spaces on $X$ such that $\mathbb{H}^*(X,\calK_k^\bullet)$ are finite-dimensional, and $W_{\lc}$ is an increasing (weight) filtration on $\calK^\bullet_k$.
 \item $\calK^\bullet_\C$ is a bounded below cochain complex of sheaves of $\C$-vector spaces on $X$; $W_{\lc}$ is an increasing (weight) filtration and $F^{\lc}$ a decreasing (Hodge) filtration on $\calK^\bullet_\C$.
 \item $\alpha\colon (\calK^\bullet_k,W_{\lc}) \dashrightarrow (\calK^\bullet_\C,W_{\lc})$ is a \bi{pseudo-morphism of filtered complexes} of sheaves of $k$-vector spaces on $X$ (i.e.\ a chain of morphisms of bounded-below complexes of sheaves as in \cite[Definition 2.31]{peters2008mixed} except that each complex in the chain is filtered, as are all the morphisms) that induces a filtered pseudo-isomorphism
 \[
 \alpha \otimes 1\colon (\calK^\bullet_k \otimes \C, W_{\lc} \otimes \C) \dashrightarrow (\calK^\bullet_\C,W_{\lc})
 \]
that is, a pseudo-isomorphism on each graded component.
 \item for $m \in \Z$, the $m$-th $W$-graded component 
\[
\Gr_m^W\calK^\bullet = \Big(\Gr_m^W\calK^\bullet_k, \big(\Gr_m^W\calK^\bullet_\C, F^{\lc}\big), \Gr_m^W\alpha\Big)
\]
is a $k$-Hodge complex of sheaves \cite[Definition 2.32]{peters2008mixed} on $X$ of weight $m$, where $F^{\lc}$ denotes the induced filtration.
\end{itemize}
We will sometimes introduce a $k$-mixed Hodge complex of sheaves on $X$ simply as $\calK^\bullet$ and implicitly assume the components of the triple to be notationally the same as in the above definition.
\end{dfn}

\begin{dfn}\index{mixed Hodge complex!multiplicative}
A \bi{multiplicative $k$-mixed Hodge complex of sheaves on $X$} is a $k$-mixed Hodge complex of sheaves $\calK^\bullet$ on $X$ such that the pseudo-morphism $\alpha$ has a distinguished representative given by a chain of morphisms of sheaves of cdgas on $X$ (with all but $\calK^\bullet_k$ being a sheaf of $\C$-cdgas), and such that \textit{all} filtrations (including those in the chain) are cdga-filtrations (over $\C$ except for the weight filtration on $\calK^\bullet_k$). 
\end{dfn}

From a given Hodge complex, one can construct others as follows. Suppose $m \in \Z$ and $\calK^\bullet = (\calK^\bullet_k, (\calK^\bullet_\C, F^{\lc}), \alpha)$ is a (pure) $k$-Hodge complex of sheaves \cite[Definition 2.32]{peters2008mixed} on $X$ of weight $m$.
If $j \in \Z$, the \bi{$j$-th Tate twist of $\calK^\bullet$} is the triple
\[\calK^\bullet(j) = \Big( \calK^\bullet_k, \big(\calK^\bullet_\C, F[j]\big), \alpha \Big)\]
where $F[j]^i = F^{j+i}$ is the shifted filtration. $\calK^\bullet(j)$ is a $k$-Hodge complex of sheaves on $X$ of weight $m-2j$.
For details see \cite[Definition 2.35]{peters2008mixed} and notice the we have changed the convention by selecting not to multiply by $(2 \pi i )^j$.

Tate twists can also be defined for mixed Hodge complexes of sheaves. Suppose $\calK^\bullet$ is a $k$-mixed Hodge complex of sheaves on $X$.
If $j \in \Z$, the \bi{$j$-th Tate twist of $\calK^\bullet$} is the triple\index{Tate twist}
\[
\calK^\bullet(j) = \Big(\big(\calK^\bullet_k,W[2j]_{\lc}\big), \big(\calK^\bullet_\C, W[2j]_{\lc},F[j]^{\lc}\big), \alpha\Big)
\]
where $W[2j]_i = W_{2j+i}$ and $F[j]^i = F^{j+i}$ are shifted filtrations\index{shift!W[j]@$W[j]$}\index{shift!F[j]@$F[j]$}\index{mixed Hodge complex!translation}.
$\calK^\bullet(j)$ is again a $k$-mixed Hodge complex of sheaves on $X$. For details see \cite[Definition 3.14]{peters2008mixed} and notice again that we have changed convention by selecting not to multiply by $(2 \pi i )^j$.

We will also refer to Tate twists of $k$-mixed Hodge structures. These are defined by shifting the weight and Hodge filtrations with the same formula we used for mixed Hodge complexes above. See \cite[Example 3.2(3)]{peters2008mixed} for an explicit definition, except that, as expected, we opt not to multiply by a power of $2\pi i$.

We will be interested in shifting mixed Hodge complexes.
If $\calF^\bullet$ is a complex of sheaves on $X$, then its \bi{translation} is the complex $\calF^{\bullet}[1] = \calF^{\bullet+1}$ with differential $d^\bullet[1] = -d^{\bullet+1}$.
Suppose $\calK^\bullet$ is a $k$-mixed Hodge complex of sheaves on $X$.
The \bi{translation of $\calK^\bullet$} is the triple\index{shift!MHC@$\calK^\bullet[j]$ of a mixed Hodge complex}\index{translation}\index{translation|see{mixed Hodge complex}}
\[
\calK^\bullet[1] = ((\calK^\bullet_k[1], W[1]_{\lc}), (\calK^\bullet_\C[1], W[1]_{\lc}, F[1]^{\lc}), \alpha[1])
\]
where the filtrations are described by:
\begin{align*}
 &(W[1])_i\left(\calK^\bullet_k[1]\right) = \left(W_{i+1}\calK^\bullet_k\right)[1],\ (W[1])_i\left(\calK^\bullet_\C[1]\right) = \left(W_{i+1}\calK^\bullet_\C\right)[1],\ i \in \Z\\
 &\left(F[1]\right)^p\left(\calK^\bullet_\C[1]\right) = \left(F^{p+1}\calK^\bullet_\C\right)[1],\ p \in \Z.
\end{align*}
This is again a $k$-mixed Hodge complex of sheaves on $X$. Note that this agrees with the translation of a pure Hodge complex as defined in \cite[2.35]{peters2008mixed} for pure Hodge complexes, but not with the translation of mixed Hodge complexes implicit in loc. cit. 3.22.

\begin{remk}\label{transvstate}\index{mixed Hodge complex!hypercohomology}
Suppose $\calK^\bullet$ is a $k$-mixed Hodge complex of sheaves on $X$.
By \cite[Theorem 3.18II]{peters2008mixed} the hypercohomology vector spaces $\mathbb{H}^*(X, \calK^\bullet_k)$ inherit $k$-mixed Hodge structures. 
There is a relation between Tate twists and the translation of a $k$-mixed Hodge complex $\calK^\bullet$ of sheaves on $X$. 
Namely it can be shown that:
\[\mathbb{H}^{*}(X, \calK^\bullet_k[1]) \cong \mathbb{H}^{*+1}(X, \calK^\bullet_k)(1)\]
where a Tate twist has been taken on the right-hand side, and the $k$-mixed Hodge structure on the left-hand side has been induced by the translated $k$-mixed Hodge complex $\calK^\bullet[1]$.
\end{remk}

\begin{remk}[Derived direct image of a mixed Hodge complex of sheaves.]\index{mixed Hodge complex!direct image}
\label{directim}
Suppose $\calK^\bullet$ is a $k$-mixed Hodge complex of sheaves on $X$ where the filtrations $W_{\lc}$ and $F^{\lc}$ are biregular (i.e. for all $m$, the filtrations induced on $\calK^m$ are finite). Let $g\colon X\rightarrow Y$ be a continuous map between two topological spaces. The derived direct image of $\calK^\bullet$ via $g$ is again a mixed Hodge complex of sheaves, and it is defined as follows (\cite[B.2.5]{peters2008mixed}).

Let $\Tot[\calC_{\Gdm}^\bullet\calF^\bullet]$ be the Godement resolution of a complex of sheaves $\calF^\bullet$ as defined in \cite[B.2.1]{peters2008mixed}, which is a flabby resolution. Here, $\Tot[\calC_{\Gdm}^\bullet\calF^\bullet]$ denotes the simple complex associated to the double complex $\calC_{\Gdm}^\bullet\calF^\bullet$. We define $Rg_*\calK^\bullet$ to be the following triple
{\footnotesize
\[
\Big(\big(g_* \Tot[\calC_{\Gdm}^\bullet \calK^\bullet_k],g_* \Tot[\calC_{\Gdm}^\bullet W_{\lc}]\big), \big(g_* \Tot[\calC_{\Gdm}^\bullet \calK^\bullet_\C], g_* \Tot[\calC_{\Gdm}^\bullet W_{\lc}],g_* \Tot[\calC_{\Gdm}^\bullet F^{\lc}]\big), g_*\alpha\Big)
\]
}
where $g_*\alpha$ is the pseudo-morphism of filtered complexes of sheaves of $k$-vector spaces induced by $\alpha$ and the functoriality of both $g_*$ and the Godement resolution.
\end{remk}


\section{Real Mixed Hodge Complexes on Smooth Varieties}\label{realmixedhodgecomplexonsmoothvarieties}

\textit{Throughout this section let $k = \R$, $U$ denote a smooth complex algebraic variety, and $X$ denote a good compactification \cite[Definition 4.1]{peters2008mixed} of $U$ with simple normal crossing divisor $D = X\setminus U$.
Let $j\colon U \rightarrow X$ denote the inclusion.} In this section, we remind the reader of the $\R$-mixed Hodge complex of sheaves associated to the log de Rham complex $\logdr{X}{D}$. For a definition of the log de Rham complex and its properties, see \cite[Section 4.1]{peters2008mixed}.

Before describing the particular filtrations on the log de Rham complex, we define a pair of filtrations that can be applied to any cochain complex.
Suppose $\calF^\bullet$ is a cochain complex of sheaves.
Then we have the increasing \bi{canonical filtration} $\tau_{\lc}$, which is described by:\index{filtration!canonical}
\begin{align*}
 \tau_m(\calF^\bullet) = \left\{ \cdots \rightarrow \calF^{m-2} \rightarrow \calF^{m-1} \rightarrow \ker d^m \rightarrow 0 \rightarrow 0 \rightarrow \cdots\right\}
\end{align*}
and we have the decreasing \bi{trivial filtration}, whose $p$-th filtered subcomplex is $\calF^{\geq p}$. \index{filtration!trivial}

We now look at the log de Rham complex in particular.
On it we have the weight filtration $W_{\lc}$, which is described by:
\begin{align*}
 W_m\Omega^p_X(\log D) = \left\{\begin{array}{ll}
 0 & \mbox{if}\ m < 0\\
 \Omega_X^p(\log D) & \mbox{if}\ m \geq p\\
 \Omega_X^{p-m} \wedge \Omega_X^m(\log D) & \mbox{if}\ 0 \leq m \leq p
 \end{array}\right.
\end{align*}
and we have Hodge filtration $F^{\lc}$, defined to be the trivial filtration.

\begin{thm}\label{logmhs}
The triple
\[
\calH dg^\bullet(X\,\log D) \coloneqq \Big(\big(j_*\calE^\bullet_U, \tau_{\lc}\big), \big(\logdr{X}{D}, W_{\lc}, F^{\lc}\big), \alpha \Big)
\]
where $\calE^\bullet_U$ is the real de Rham complex and $\alpha$ is represented by:
\begin{equation}\label{eqmorphisms}
\begin{tikzcd}
 (j_*\calE^\bullet_U, \tau_{\lc}) \ar[r, "j_*(\id \otimes 1)"] & (j_*(\calE^\bullet_U \otimes \C), \tau_{\lc}) & \ar{l}{\simeq}[swap]{\textnormal{incl}} (\logdr{X}{D}, \tau_{\lc}) \arrow{r}{\id}[swap]{\simeq} & (\logdr{X}{D}, W_{\lc})
\end{tikzcd}
\end{equation}
determines an $\R$-mixed Hodge complex of sheaves on $X$. We will refer to it as the \textbf{Hodge-de Rham complex} of $(X,D)$.\index{mixed Hodge complex!Hodge-de Rham}
\end{thm}
\begin{proof}
This is the content of \cite[Proposition-Definition 4.11]{peters2008mixed} except that here we work over $\R$ instead of $\Q$.
Notice that we do not need derived pushforwards, because the objects being pushed forward are complexes of soft sheaves. 
\end{proof}

\begin{remk}\label{multiplicative}
All the complexes of sheaves appearing in \eqref{eqmorphisms} are cdgas over $\C$ except the first one, and the morphisms are morphisms of sheaves of cdgas.
Furthermore, all filtrations are cdga-filtrations.
In other words, $\calH dg^\bullet(X\, \log D)$ is a multiplicative $\R$-mixed Hodge complex of sheaves.
\end{remk}


\section{Rational Mixed Hodge Complexes on Smooth Varieties}\label{rationalmixedhodgecomplexonsmoothvarieties}\index{mixed Hodge complex!rational}
\textit{Throughout this section let $k = \Q$, let $U$ denote a smooth complex algebraic variety, and $X$ denote good compactification of $U$ with simple normal crossing divisor $D = X\setminus U$.
Let $j\colon U \rightarrow X$ denote the inclusion.}
In \cite[Section 4.4]{peters2008mixed} the authors define a multiplicative $\Q$-mixed Hodge complex of sheaves associated to the log de Rham complex whose pseudo-morphism is actually a morphism. 
The rational component of this mixed Hodge complex is perhaps less familiar than the de Rham complex, which is the reason we have separately considered the real case. 

Let us outline the $\Q$-mixed Hodge complex. 
Let $\calO_X$ denote the sheaf of holomorphic functions on $X$.
Let $\calO_U^*$ denote the sheaf of invertible holomorphic functions on $U$.
Let $\calM_{X,D}^{\mathrm{gp}}$ be the sheaf of abelian groups associated to $\calM_{X,D} = \calO_X \cap j_*\calO_U^*$, defined by the following universal property: there is a universal map $c\colon\calM_{X,D}\rightarrow \calM_{X,D}^{\mathrm{gp}}$ such that every homomorphism of monoid sheaves from $\calM_{X,D}$ to a sheaf of groups on $X$ factorizes uniquely over $c$. In other words, $\calM_{X,D}^{\mathrm{gp}}$ is the sheaf of meromorphic functions on $X$ which are holomorphic on $U$ and whose inverse is holomorphic on $U$ as well. For $i \in \Z$ let $\mathrm{Sym}^i_\Q(\calO_X)$ denote the $i$-th symmetric tensor sheaf on $\calO_X$, where $\calO_X$ has been interpreted as a sheaf of $\Q$-vector spaces. 
As in \cite[Section 4.4]{peters2008mixed} for integers $m$ and $p$ we define:
\begin{align*}
 \calK^p_m = \left\{\begin{array}{ll} \mathrm{Sym}^{m-p}_{\Q}(\calO_X) \otimes_\Q \bigwedge_{\Q}^p (\calM_{X,D}^{\mathrm{gp}} \otimes_\Z \Q) & \mbox{if}\ m \geq p \geq 0\\
 0 & \mbox{otherwise},
 \end{array}\right.
\end{align*}
and the differential $d\colon \calK^p_m \rightarrow \calK^{p+1}_m$ is given by:
\begin{align*}
d(f_1 \cdots f_{m-p} \otimes y) = \sum_{k=1}^{m-p} f_1 \cdots f_{k-1} \cdot f_{k+1} \cdots f_{m-p} \otimes \exp(2\pi i f_k) \wedge y	
\end{align*}
for $m \geq p \geq 0$, where $f_1, \dots, f_{m-p}$ are sections of $\calO_X$ and $y$ is a section of $\bigwedge_{\Q}^p (\calM_{X,D}^{\mathrm{gp}} \otimes_\Z \Q)$.
For fixed $m$, the differential $d$ endows $\calK^\bullet_m$ with the structure of a cochain complex of sheaves on $X$.
It can be related to the log de Rham complex via the map $\varphi_m\colon \calK^\bullet_m \rightarrow W_m\logdr{X}{D}$ described by:
\[
 \varphi_m(f_1 \cdots f_{m-p} \otimes y_1 \wedge \cdots \wedge y_p) = \frac{1}{(2\pi i)^p} \left(f_1 \cdots f_{m-p}\right) \frac{dy_1}{y_1} \wedge \cdots \wedge \frac{dy_p}{y_p}
\]
for $0 \leq p \leq m$. 
The map $\varphi_m \otimes 1\colon \calK^\bullet_m \otimes \C \rightarrow W_m \logdr{X}{D}$ is a quasi-isomorphism by \cite[Theorem~4.15]{peters2008mixed}.

For any $m$, we have an inclusion of complexes $\calK^\bullet_m \rightarrow \calK^\bullet_{m+1}$ given by:
\[
	f_1 \cdots f_{m-p} \otimes y \mapsto 1 \cdot f_1 \cdots f_{m-p} \otimes y
\]
for $0 \leq p \leq m$. Therefore we may consider the direct limit 
\[
\calK^\bullet_\infty \coloneqq \varinjlim_m \calK^\bullet_m.
\]
Define a weight filtration $W_{\lc}$ on $\calK^\bullet_\infty$ by declaring $W_m\calK^\bullet_\infty$ to be the image of $\calK^\bullet_m$ in the direct limit.
$\calK^\bullet_\infty$ can be equipped with the structure of a sheaf of cdgas by defining:
\begin{align*}
(f_1 \cdots f_r \otimes y) \wedge (g_1 \cdots g_s \otimes z) = (f_1 \cdots f_r \cdot g_1 \cdots g_s) \otimes (y \wedge z).	
\end{align*}
The weight filtration $W_{\lc}$ is then a cdga-filtration.
The direct limit $\varphi_\infty\colon \calK^\bullet_\infty \rightarrow \logdr{X}{D}$ is a morphism of cdgas and moreover by \cite[Corollary 4.16]{peters2008mixed} is a $W_{\lc}$-filtered quasi-isomorphism after tensoring with $\C$. 
Moreover, by \cite[Corollary 4.17]{peters2008mixed}, the triple 
\[
\Big(\big(\calK^\bullet_\infty, W_{\lc}\big), \big(\logdr{X}{D}, W_{\lc}, F^{\lc}\big), \varphi_\infty\Big)
\]
 is a multiplicative $\Q$-mixed Hodge complex of sheaves on $X$.

Note that the filtration $W_{\lc}$ on $\calK^\bullet_\infty$ is not bounded above, though bounded above filtrations will be important later.
This can be easily corrected by replacing $W_{\lc}$ on $\calK^\bullet_{\infty}$ with the cdga-filtration $\tilde{W}_{\lc}$ described by:
\begin{align*}
\tilde{W}_m \calK^\bullet_\infty = \begin{cases}
	W_m \calK^\bullet_\infty & \textnormal{if}\, m \leq \dim X\\
	\calK^\bullet_\infty & \textnormal{otherwise.}
\end{cases}
\end{align*}
The morphism $\varphi_\infty$ maps $\tilde{W}_{\lc}$ into $W_{\lc}$ because $W_{\dim X} \logdr{X}{D} = \logdr{X}{D}$.
For the same reason, the morphism
\[
\varphi_\infty \otimes 1\colon \big(\calK_\infty^\bullet \otimes \C, \tilde{W}_{\lc} \otimes \C\big) \rightarrow \big(\logdr{X}{D}, W_{\lc}\big)
\]
continues to be a filtered quasi-isomorphism.
And because $\C$ is faithfully flat over $\Q$, the inclusion $(\calK^\bullet_\infty, W_{\lc}) \rightarrow (\calK^\bullet_\infty, \tilde{W}_{\lc})$ is also a filtered quasi-isomorphism. 
Therefore we have: 
\begin{thm}\index{K_infty@$\calK^\bullet_\infty$}
	The triple $((\calK^\bullet_\infty, \tilde{W}_{\lc}), (\logdr{X}{D}, W_{\lc}, F^{\lc}), \varphi_\infty)$ is a multiplicative $\Q$-mixed Hodge complex of sheaves on $X$, and $\calK^\bullet_\infty$ is pseudo-isomorphic to $Rj_*\underline{\Q}$, where $\underline{\Q}$ is the constant sheaf on $U$.
\end{thm}
\begin{proof}
This is the content of \cite[Corollary 4.17]{peters2008mixed} except that $W_{\lc}$ has been replaced by $\tilde{W}_{\lc}$.
By the discussion preceding the theorem, this replacement is inconsequential.
\end{proof}


\section{Limit Mixed Hodge Structure}\label{LimitMixedHodgeStructure}

After we have identified the mixed Hodge structure of Theorem \ref{mhsexistence}, we will compare it with the well-studied limit mixed Hodge structure on the generic fiber of $f$.
We review the construction of this structure over $\Q$, considered in our setting, and under an assumption of properness.
\textit{Throughout this section, let $U$ denote a smooth, connected $n$-dimensional complex algebraic variety and $f\colon U \rightarrow \C^*$ a proper algebraic map inducing an epimorphism on fundamental groups.}

Select a good compactification $X$ of $U$ with simple normal crossing divisor $D = X\setminus U$ such that $f\colon U \rightarrow \C^*$ extends to an algebraic map $\bar{f}\colon X \rightarrow \C P^1$.
By replacing $f\colon U \rightarrow \C^*$ with a finite cyclic cover $f_N\colon U_N \rightarrow \C^*$ if necessary, we may assume that $\bar{f}^{-1}(0)$ is reduced, by \cite[Semi-stable Reduction Theorem]{kempf2006toroidal}.
\textit{Fix $\bar{f}\colon X \rightarrow \C P^1$ as above for the rest of this section. Let $j\colon U \rightarrow X$ and $i\colon E \hookrightarrow X$ denote inclusions, where $E$ denotes the divisor $\bar{f}^{-1}(0)$, which is reduced.}

Select an open disk $\Delta \subset \C$ centered at the origin such that $f$ is submersive over $\Delta^*$, the punctured disk.
The \bi{nearby cycles functor}\index{nearby cycles} $\psi_{\bar{f}}\colon D^+\left(\bar{f}^{-1}\Delta\right) \rightarrow D^+(E)$ is a functor between derived categories of bounded below complexes of sheaves of vector spaces, defined for example in \cite[Section 11.2.3]{peters2008mixed}.
Importantly, $\mathbb{H}^*(E; \psi_{\bar{f}}\underline{\Q}) \cong H^*(F; \Q)$ where $F$ is any fiber of $f$ over $\Delta^*$.
A clockwise loop in $\Delta^*$ determines a monodromy homeomorphism from $F$ to itself and so equips $\mathbb{H}^*(E; \psi_{\bar{f}}\underline{\Q})$ with the structure of a torsion module over $\Q[t^{\pm 1}]$. 
The limit mixed Hodge complex is assigned to $\psi_{\bar{f}}\underline{\Q}$. \index{Hodge!limit MHS}

To define complex weight and Hodge filtrations, Peters and Steenbrink in \cite[Section 11.2.5]{peters2008mixed} work with a double complex of sheaves that is quasi-isomorphic to $\psi_{\ov f} \underline{\Q}$. 
In this spirit, but with minor shifts of convention, we define a double complex $\calA^{\bullet,\bullet}$ of sheaves of $\C$-vector spaces on $X$ as follows.
Let 
\[
\calA^{r,s} = 
\begin{cases}
 \frac{\Omega_X^{r+s+1}(\log D)}{W_r\Omega_X^{r+s+1}(\log D)}, &\text{ if }r\geq 0,\\
 0, &\text{ if }r< 0,
\end{cases}
\]
with differentials $d' = \frac{1}{2\pi i} \frac{df}{f} \wedge - \colon \calA^{r,s} \rightarrow \calA^{r+1,s}$, and $d'' = d\colon \calA^{r,s} \rightarrow \calA^{r,s+1}$ being the differential from $\logdr{X}{D}$. 

According to \cite[Section 11.2.5]{peters2008mixed}, the associated total complex $\Tot\calA^{\bullet,\bullet}$ is, after applying the functor $i^{-1}$, pseudo-isomorphic to $\psi_{\ov f} \underline{\C}$.
The \bi{monodromy weight filtration}\index{filtration!monodromy weight} on $\Tot\calA^{\bullet, \bullet}$ is denoted by $W(M)_{\lc}$ and described for integers $r \geq 0$, $s$, and $m$ by:
\begin{align*}
 W(M)_m\, \calA^{r,s} = \mbox{image of}\ W_{m+2r+1}\Omega_X^{r+s+1}(\log D)\ \mbox{in}\ \calA^{r,s}.
\end{align*}
The Hodge filtration on $\Tot\calA^{\bullet, \bullet}$ is denoted by $F^{\lc}$ and defined for an integer $p$ by:
\begin{align*}
	F^p(\Tot\calA^{\bullet,\bullet}) = \bigoplus_{r} \bigoplus_{s \geq p} \calA^{r,s}.
\end{align*}
This concludes the description of the $\C$-component of the limit mixed Hodge complex.

Following the same blueprint, we also construct a double complex which is pseudo-isomorphic to $\psi_{\ov f} \underline{\Q}$.
We make use of the pair $(\calK^\bullet_\infty, \tilde{W}_{\lc})$ from Section \ref{rationalmixedhodgecomplexonsmoothvarieties}.
Define $\tilde{\calC}^{\bullet,\bullet}$ as follows:
\[
\tilde{\calC}^{r,s} = 
\begin{cases}
\frac{\calK_\infty^{r+s+1}}{\tilde{W}_r \calK_\infty^{r+s+1}}, &\text{ if }r\geq 0,\\
0, &\text{ if }r< 0,
\end{cases}
\]
with differentials $d'=(1 \otimes f) \wedge -\colon \tilde{\calC}^{r,s} \rightarrow \tilde{\calC}^{r+1,s}$, and $d''=d\colon \tilde{\calC}^{r,s} \rightarrow \tilde{\calC}^{r,s+1}$ being the differential from $\calK^\bullet_\infty$. 
Equip the associated total complex $\Tot\calC^{\bullet,\bullet}$ with the \bi{monodromy weight filtration} denoted by $\tilde{W}(M)_{\lc}$ and described for integers $r \geq 0$, $s$, and $m$ by: 
\begin{align*}
 \tilde{W}(M)_m\, \tilde{\calC}^{r,s} = \mbox{image of}\ \tilde{W}_{m+2r+1}\calK_\infty^\bullet\ \mbox{in}\ \tilde{\calC}^{r,s}.
\end{align*}
The map $\varphi_\infty\colon \calK^\bullet_\infty \rightarrow \logdr{X}{D}$ induces a monodromy-filtered morphism $$\varphi_{\infty}\colon \Tot\tilde{\calC}^{\bullet,\bullet} \rightarrow \Tot\calA^{\bullet,\bullet}.$$

\begin{thm}\label{rationallimitmhs}
The restricted triple:
\begin{align*}
\psi_{ f}^{\textnormal{Hdg}} \coloneqq i^{-1}\left(\big[\Tot\tilde{\calC}^{\bullet,\bullet}, \tilde{W}(M)_{\lc}\big], \big[\vphantom{\tilde{A}} \Tot\calA^{\bullet,\bullet}, W(M)_{\lc}, F^{\lc}\big], \varphi_{\infty}\right)
\end{align*}\index{psi-f-Hdg@$\psi_f^{\mathrm{Hdg}}$}
is a $\Q$-mixed Hodge complex of sheaves on $E$, and $i^{-1}\Tot\tilde{\calC}^{\bullet,\bullet}$ is pseudo-isomorphic to $\psi_{\bar{f}}\underline{\Q}$. 
\end{thm}
\begin{proof}
This is established (up to minor modifications owing to moderate changes in the definitions of both double complexes) by \cite[Theorem 11.22]{peters2008mixed} except that in place of $\tilde{\calC}^{\bullet, \bullet}$ and its filtration $\tilde{W}(M)_{\lc}$ the authors opt for $\calC^{\bullet,\bullet}$ and $W(M)_{\lc}$, which are built using the unbounded filtration $W_{\lc}$ on $\calK^\bullet_\infty$ instead of $\tilde{W}_{\lc}$.
Because the inclusion $(\calK^\bullet_\infty, W_{\lc}) \rightarrow (\calK^\bullet_\infty, \tilde{W}_{\lc})$ is a filtered quasi-isomorphism of cgdas, the induced map $(\Tot\tilde{\calC}^{\bullet,\bullet}, \tilde{W}(M)_{\lc}) \rightarrow \left(\Tot\calC^{\bullet,\bullet}, W(M)_{\lc}\right)$ can also be shown to be a filtered quasi-isomorphism. 
\end{proof}

The $\Q[t^{\pm 1}]$-module structure on $\mathbb{H}^*(E; \psi_{\bar{f}}\underline{\Q}) \cong H^*(F; \Q)$ induced by the monodromy action can actually be realized on this mixed Hodge complex of sheaves.
Specifically, for integers $r \geq 0$ and $s$ define:
\begin{align*}
&\Theta\colon \tilde{\calC}^{r,s} \rightarrow \tilde{\calC}^{r+1,s-1},\ \Theta(-) = (-)\, \textnormal{mod}\, \tilde{W}_{r+1} \calK_\infty^{r+s+1}\\
&\Theta\colon \calA^{r,s} \rightarrow \calA^{r+1,s-1},\ \Theta(-) = (-)\, \textnormal{mod}\, W_{r+1} \Omega^{r+s+1}_X(\log D).
\end{align*} \index{Theta@$\Theta$}

\begin{thm}\label{limitmhsmonodromy}
Under the above notations, the map $\Theta$ gives rise to a morphism of mixed Hodge complexes of sheaves:
\begin{align*}
i^{-1}\Theta\colon \psi_f^{\textnormal{Hdg}} \rightarrow \psi_f^{\textnormal{Hdg}}(-1)
\end{align*}
which induces the map $\log t \colon \mathbb{H}^*(E; \psi_{\bar{f}}\underline{\Q}) \rightarrow \mathbb{H}^*(E; \psi_{\bar{f}}\underline{\Q})(-1)$.
Here $\log t$ is interpreted as its power series representation at $t = 1$.
\end{thm}
\begin{proof}
The assertion follows from \cite[Section 11.3.1]{peters2008mixed} with minor modifications, because we have selected moderately different definitions for Tate twists and the pair of double complexes appearing in $\psi_f^{\textnormal{Hdg}}$. We remark that $\log t$ is in fact well-defined, because $E$ being reduced implies that the monodromy action on $\mathbb{H}^*(E; \psi_{\ov f}\underline{\Q})$ is unipotent (for a proof see \cite[Corollary 11.19]{peters2008mixed}).
\end{proof}


	\chapter{Thickened Complexes}\label{sec3}

\section{Thickened Complex of a Differential Graded Algebra}

Suppose $(A,d,\wedge)$ is a cdga {over $k$}, $\eta \in A^1 \cap \ker d$, and $m \geq 1$. {Recall that $R_m$ is the quotient ring $R_\infty/(s^m)=k[[s]]/(s^m)$.
We define the \bi{$m$-thickening of $A$ in the direction $\eta$} to be the cochain complex of $R_m$-modules denoted by \index{thickening}
\begin{align*}
	A(\eta,m) = (A {\otimes_k} R_m, d_\eta)
\end{align*}
 and described by:
\begin{itemize}
\item for $p \in \Z$, the $p$-th graded component of $A(\eta,m)$ is the free $R$-module {$A^p {\otimes_k} R_m$.}
\item for $\omega \in A$ and $\phi \in R_m$, we set $d_\eta(\omega \otimes \phi) = d\omega \otimes \phi + (\eta \wedge \omega) \otimes {s\phi} $.
\end{itemize}
A straightforward calculation verifies that $d_{\eta}^2$ vanishes.

If $i\geq 1,j \geq 0$, then the quotient map $R_{i+j} \twoheadrightarrow R_i$ induces a surjective morphism $\phi_{ji}\colon A(\eta,i+j) \twoheadrightarrow A(\eta,i)$.
These maps endow $\{A(\eta,i)\}$ with the structure of an inverse system of $\{R_i\}$-modules.
Using that $R_\infty = \varprojlim R_i$, we define the \bi{$\infty$-thickening of $A$ in the direction $\eta$} to be the cochain complex of $R_\infty$-modules 
\begin{align*}
A(\eta, \infty) = \varprojlim A(\eta,i).
\end{align*}

\begin{remk}\label{thickenedmultiplication}
The inverse system $\{A \otimes R_i\}$ admits a natural multiplication operation, defined on simple tensors by $(\omega \otimes \phi) \wedge (\omega' \otimes \phi') = (\omega \wedge \omega') \otimes \phi \phi'$.
However, this multiplication does not in general endow $\{A(\eta,i)\}$ with the structure of an inverse system of \textit{cdgas}, because it is not compatible with the differential $d_{\eta}$.
Nevertheless, we will find an opportunity to utilize it shortly.
\end{remk}

\begin{remk}\label{filtration}
The complex $A(\eta,m)$ can be equipped with a decreasing filtration $G$, described by 
\begin{align*}
G^pA(\eta,m) = A(\eta,m)s^p
\end{align*} 
for $p \geq 0$.
Its graded components are, for $p \geq 0$:
\begin{align*}
 \Gr^p_G A(\eta,m) = (A \otimes k\langle s^p \rangle, d \otimes 1).
\end{align*}
In particular, $H^*( \Gr^p_G A(\eta,m)) \cong H^*(A)$ $\otimes k\langle s^p \rangle$.
\end{remk}

We next state two lemmas concerning thickened complexes, which are essentially listed in \cite[Section~3]{BudurLiuWang}. For future reference, we also include their proofs.

\begin{lem}\label{cohomologous}
Suppose $(A,d,\wedge)$ is a cdga and $m \geq 1$.
Suppose that $\eta_1$ and $\eta_2$ are cohomologous elements in $A^1 \cap \ker d$.
Then $A(\eta_1, m)$ and $A(\eta_2, m)$ are isomorphic. More precisely: if $a \in A^0$ is such that $\eta_1-\eta_2 = da$, then the morphism
\[\exp(a \otimes s) \wedge -\colon A(\eta_1,m) \rightarrow A(\eta_2,m)\]
is defined in the proof and is an isomorphism of cochain complexes.
\end{lem}
\begin{proof}
Let $a \in A^0$ be such that $da = \eta_1 - \eta_2$.
The element 
\[
\exp(a \otimes s) = \sum_{n = 0}^\infty \frac{1}{n!}\, a^n \otimes s^n
\]
belongs to $\varprojlim A^0 \otimes R_i$ and determines, via the multiplication described in Remark \ref{thickenedmultiplication}, a morphism 
\[
\exp(a \otimes s) \wedge -\colon A \otimes R_m \rightarrow A \otimes R_m
\]
of $R_m$-modules.
We will show that it in fact determines a morphism 
\[
\exp(a \otimes s) \wedge -\colon A(\eta_1, m) \rightarrow A(\eta_2,m)
\]
by verifying commutativity with the differentials. It suffices to investigate its application to a simple tensor of the form $\omega \otimes 1$:
\begin{align*}
 &d_{\eta_2}\left(\exp(a \otimes s)\wedge (\omega \otimes 1)\right)=\\ &= \sum_{n = 0}^{m-1} \frac{1}{n!}~ d_{\eta_2}\left(a^n \omega \otimes s^n\right)=\\
 &= \sum_{n = 0}^{m-1} \frac{1}{n!}~ d\left(a^n\omega\right) \otimes s^n + \sum_{n = 1}^{m-1} \frac{1}{(n-1)!}a^{n-1}\eta_2 \wedge \omega \otimes s^n=\\ 
 &=\sum_{n = 1}^{m-1}\frac{1}{(n-1)!}~ a^{n-1} da \wedge \omega \otimes s^n + \sum_{n=0}^{m-1}\frac{1}{n!}~a^n \wedge d\omega \otimes s^n + \\ &+ \sum_{n = 1}^{m-1} \frac{1}{(n-1)!}a^{n-1}\eta_2 \wedge \omega \otimes s^n
\end{align*}
and since $da = \eta_1-\eta_2$, the above simplifies to:
\begin{align*}
\sum_{n = 1}^{m-1}\frac{1}{(n-1)!}~ &a^{n-1} \eta_1 \wedge \omega \otimes s^n + \sum_{n=0}^{m-1}\frac{1}{n!}~a^n \wedge d\omega \otimes s^n=\\
&= \exp(a \otimes s) \wedge (\eta_1 \wedge \omega \otimes s) + \exp(a \otimes s) \wedge (d\omega \otimes 1)=\\
&= \exp(a \otimes s) \wedge d_{\eta_1}(\omega \otimes 1).
\end{align*}
Therefore $\exp(a \otimes s) \wedge -$ is a morphism of cochain complexes.
It is in fact an isomorphism, because $\exp(-a \otimes s) \wedge -$ is its inverse. 
\end{proof}

\begin{lem}\label{inducedmap}
Suppose $(A,d,\wedge)$ and $(B, d, \wedge)$ are cdgas, $\eta \in A^1 \cap \ker d$, and $m \geq 1$.
If $F\colon A \rightarrow B$ is a morphism of dgas, then there is a natural induced morphism 
\[
F_{\#}\colon A(\eta, m) \rightarrow B(F(\eta), m).
\]
If $F$ is a quasi-isomorphism, then so is $F_{\#}$.
\end{lem}
\begin{proof}
The morphism $F_\#$ is given by $F \otimes \id\colon A \otimes R_m \rightarrow B \otimes R_m$.
$F_\#$ preserves the decreasing filtration $G$ of Remark \ref{filtration}.
In particular, for $p \geq 0$ it induces commutative diagrams of short exact sequences:
$$
\begin{tikzcd}
 0 \ar[r] & G^{p+1} A(\eta,m) \ar[d] \ar[r] & G^p A(\eta,m) \ar[d] \ar[r] & \Gr^p_G A(\eta,m) \ar[d] \ar[r] & 0\\
 0 \ar[r] & G^{p+1} B(F(\eta),m) \ar[r] & G^p B(F(\eta),m) \ar[r] & \Gr^p_G B(F(\eta),m) \ar[r] & 0.
\end{tikzcd}
$$
From the associated commutative diagram of long exact sequences, the central map is a quasi-isomorphism if the outer two are.
For all $p \geq 0$, the induced map on cohomology of graded components is the isomorphism $F^* \otimes \id\colon H^*(A) \otimes k\langle s^p \rangle \rightarrow H^*(B) \otimes k\langle s^p \rangle$.
Therefore the induced map on graded components is a quasi-isomorphism. 
Since $m$ is finite, both $G^mA(\eta,m)$ and $G^mB(F(\eta), m)$ are zero.
This provides the starting point to an induction concluding that $G^pA(\eta,m) \rightarrow G^pB(F(\eta),m)$ is a quasi-isomorphism for all $p \geq 0$.
Take $p = 0$ to conclude.
\end{proof}

\begin{remk}\label{remk:psi-ij}
Besides the maps associated to the inverse system, there is another natural collection of maps between the thickened complexes.
If $i \geq 1,j\geq 0$ then there is an inclusion $R_i \hookrightarrow R_{i+j}$ mapping $R_i$ isomorphically onto $s^jR_{i+j}$.
We denote the induced inclusion of thickened complexes by \index{psi@$\psi_{ij}$}$\psi_{ij}\colon A(\eta,i) \hookrightarrow A(\eta,i+j)$.
\end{remk}

\begin{remk}\label{pid}
Note that $R_\infty$ is a PID.
Therefore all finitely generated $R_\infty$-modules split into torsion and free parts.
We let $\mathrm{Tors}_{R_{\infty}}$ denote torsion as an $R_\infty$-module.
\end{remk}

\begin{remk}\label{rem:tensorRm}
For all $m \geq 1$ we have $A(\eta,m) \cong A(\eta,\infty) \otimes_{R_\infty} R_m$ as $R_\infty$-modules. Therefore, the ring map $R_\infty\to R_m$ induces a map
\[
A(\eta,\infty) =A(\eta,\infty) \otimes_{R_{\infty}} R_{\infty} \to A(\eta,\infty) \otimes_{R_{\infty}} R_{m} \cong A(\eta,m).
\]
\end{remk}

The following lemma allows us to interpret $\mathrm{Tors}_{R_{\infty}}H^*A(\eta,\infty)$ as the kernel of a map between finitely thickened complexes.

\begin{lem}\label{lem:torsion2}
Suppose $C^\bullet$ is a complex of torsion-free $R_\infty$-modules, and suppose that $m$ is such that $s^m$ annihilates $\Tors_{R_\infty}H^*(C^\bullet)$. Let $\psi_{ij} \colon C^\bullet \otimes_{R_\infty} R_i \to C^\bullet \otimes_{R_\infty} R_{i+j}$ be the map induced by the map $R_i \hookrightarrow R_{i+j}$ as in Remark~\ref{remk:psi-ij}, and let $\psi_{ij}^*$ the map induced in cohomology. Let $\phi_{ji}\colon C^\bullet \otimes_{R_\infty} R_{i+j} \to C^\bullet \otimes_{R_\infty} R_{i}$ be the map induced by the projection $R_{i+j}\twoheadrightarrow R_i$. The natural map $C^\bullet\to C^\bullet\otimes_R R_i$ induces isomorphisms
\index{psi@$\psi_{ij}$}\index{psi@$\psi_{ij}$!psi@$\psi_{ij}^*$}\index{phi@$\phi_{ij}$}\index{phi@$\phi_{ij}$!phi@$\phi_{ij}^*$}
\begin{align*}
\Tors_{R_\infty} H^*(C^\bullet) &\cong \ker \psi_{ij}^* \subseteq H^* (C^\bullet\otimes_R R_i)\quad \text{ if $i,j\ge m$};\\
H^*(C^\bullet)\otimes_{R_\infty} R_i &\cong \im \phi_{ji}^* \subseteq H^* (C^\bullet\otimes_R R_i)\quad \text{ if $j\ge m$}.
\end{align*}
\end{lem}
\begin{proof}
Since $C^\bullet$ is torsion-free, we have the following short exact sequences with maps between them:
\[
{
\begin{tikzcd}[row sep = 1.5em, column sep = 1em]
0 \to
C^\bullet \arrow[r,"s^i"]\arrow[d,"=", shift left = 1.7ex] &
C^\bullet \arrow[r,twoheadrightarrow]\arrow[d,"s^j"] &
C^\bullet \otimes_{R_\infty} R_{i} \arrow[d,hookrightarrow,"\psi_{ij}", shift right = 3ex] \to
0 &
0 \to
C^\bullet \arrow[r,"s^{i+j}"]\arrow[d,"s^j", shift left = 1.7ex] &
C^\bullet \arrow[r,twoheadrightarrow]\arrow[d,"="] &
C^\bullet \otimes_{R_\infty} R_{i+j} \arrow[d,twoheadrightarrow,"\phi_{ji}", shift right = 3ex] \to
0\\
0 \to
C^\bullet \arrow[r,"s^{i+j}"] &
C^\bullet \arrow[r,twoheadrightarrow] &
C^\bullet \otimes_{R_\infty} R_{i+j} \to
0; &
0 \to
C^\bullet \arrow[r,"s^{i}"] &
C^\bullet \arrow[r,twoheadrightarrow] &
C^\bullet \otimes_{R_\infty} R_{i} \to
0.
\end{tikzcd}
}
\]
We take the cohomology long exact sequences corresponding to the sequences above, and we consider the map between them. We obtain the following diagrams, where we denote $\Tors_{s^i} M = \{ a\in M \mid s^ia=0\}$:
\[
\begin{tikzcd}[row sep = 1.3em, column sep = 1.4em]
0 \arrow[r] &
H^*(C^\bullet)\otimes_{R_\infty} R_{i\phantom{+j}} \arrow[r]\arrow[d,"s^j\cdot ","(1)"'] &
H^*\left(C^\bullet \otimes_{R_\infty} R_{i\phantom{+j}}\right) \arrow[r]\arrow[d,"\psi_{ij}^*"] &
\Tors_{s^{i\phantom{+j}}}H^{*+1}(C^\bullet) \arrow[r]\arrow[d,hookrightarrow,"\Id_{H^{*+1}(C^\bullet)}"] &
0\\
0 \arrow[r] &
H^*(C^\bullet)\otimes_{R_\infty} R_{i+j} \arrow[r] &
H^*\left(C^\bullet \otimes_{R_\infty} R_{i+j}\right) \arrow[r] &
\Tors_{s^{i+j}}H^{*+1}(C^\bullet) \arrow[r] &
0;\end{tikzcd}
\]
\[
\begin{tikzcd}[row sep = 1.3em, column sep =1.4em]
0 \arrow[r] &
H^*(C^\bullet)\otimes_{R_\infty} R_{i+j} \arrow[r]\arrow[d,"\Id_{H^{*}(C^\bullet)}",twoheadrightarrow,"(2)"'] &
H^*\left(C^\bullet \otimes_{R_\infty} R_{i+j}\right) \arrow[r]\arrow[d,"\phi_{ji}^*"] &
\Tors_{s^{i+j}}H^{*+1}(C^\bullet) \arrow[r]\arrow[d,"s^j\cdot "] &
0\\
0 \arrow[r] &
H^*(C^\bullet)\otimes_{R_\infty} R_{i\phantom{+j}} \arrow[r] &
H^*\left(C^\bullet \otimes_{R_\infty} R_{i\phantom{+j}}\right) \arrow[r] &
\Tors_{s^{i\phantom{+j}}}H^{*+1}(C^\bullet) \arrow[r] &
0.
\end{tikzcd}
\]
Let $i\ge m$. Since $s^m$ annihilates $\Tors_{R_\infty} H^*(C^\bullet)$, the following composition of obvious maps is an injection:
\[
\Tors_{R_\infty} H^*(C^\bullet)\hookrightarrow H^*(C^\bullet) \to \frac{H^*(C^\bullet)}{s^iH^*(C^\bullet)}.
\]
Seeing $\Tors_{R_\infty} H^*(C^\bullet)$ as a submodule of $\frac{H^*(C^\bullet)}{s^iH^*(C^\bullet)}$ in this way, we have that as long as $j\ge m$, $\Tors_{R_\infty} H^*(C^\bullet)$ is the kernel of the multiplication by the $s^j$ map labeled (1). Applying the snake lemma to the diagram above yields the first statement. For the second statement, note that if $j\ge m$, then the right hand map $s^j$ vanishes. Therefore, the image of $\phi_{ji}^*$ equals the image of the map labeled (2).
\end{proof}

\begin{cor}\label{torsion}
Let $(A,d,\wedge)$ be a cdga and $\eta \in A^1 \cap \ker d$.
Assume $H^*A(\eta,\infty)$ is a finitely generated $R_\infty$-module. Consider the map $H^*A(\eta,\infty)\to H^*A(\eta,i)$ from Remark~\ref{rem:tensorRm}. This map induces the following isomorphisms:
\begin{align*}
\mathrm{Tors}_{R_\infty}\,H^*A(\eta,\infty)&\cong \ker \psi_{ij}^* &\text{if $i,j\gg 0$}\\
H^*A(\eta,\infty)\otimes_{R_\infty} R_i&\cong \im \phi_{ji}^* &\text{if $j\gg 0$.}
\end{align*}
\end{cor}

\begin{remk}\label{torsionrelations}
Suppose $i, j$ are sufficiently large as in Corollary \ref{torsion}. The following statements follow from the proof of Lemma~\ref{lem:torsion2}. If $j' \geq j$ then $\ker\psi_{ij'}^* = \ker\psi_{ij}^*$ and if $i' \geq i$ then $\ker \psi_{i'j}^* \cong \ker \psi_{ij}^*$ via $\phi_{i'-i,i}^*$. Similarly, $\im \phi_{j'i}^* = \im \phi_{ji}^*$, and $\phi_{i'-i,i}^*$ induces the projection map $$\im \phi_{ji'}^* \cong H^*A(\eta,\infty)\otimes_{R_\infty} R_{i'}\twoheadrightarrow H^*A(\eta,\infty)\otimes_{R_\infty} R_{i}\cong \im\phi_{ji}^*.$$ Finally, $\psi_{i,i'-i}$ induces the map coming from the inclusion $R_i\hookrightarrow{s^{i'-i}} R_{i'}$, tensored with $H^*A(\eta,\infty)$.
\end{remk}

\begin{remk}\label{twistedmodule}
Suppose $1 \leq m \leq \infty$.
Later on we will be interested not in the given $R_\infty$-module structure on $A(\eta,m)$ but rather a twisted one.
Namely, if $\tilde{s} \in R_\infty$ satisfies $\tilde{s} = s u$ where $u \in R_\infty$ is some unit, then define $A(\eta,m)_{\tilde{s}}$ to be the vector space cochain complex $A(\eta,m)$ but with $R_\infty$-module multiplication $*_{\tilde{s}}$ described by:
\begin{center}
if $g(s) \in R_\infty$, then $- *_{\tilde{s}} g(s)$ = $- \cdot g(\tilde{s})$. 
\end{center} \index{s@$s$!s@$\tilde s$}
The induced $R_\infty$-module structure on $H^*A(\eta,m)_{\tilde{s}}$ is also given by $- *_{\tilde{s}} g(s)$ = $- \cdot g(\tilde{s})$.
Clearly the maps $\phi_{ji}\colon A(\eta,i+j) \twoheadrightarrow A(\eta,i)$ continue to induce $R_\infty$-module morphisms $\phi_{ji}\colon A(\eta,i+j)_{\tilde{s}} \twoheadrightarrow A(\eta,i)_{\tilde{s}}$.
Similarly, the maps $\psi_{ij}\colon A(\eta,i) \hookrightarrow A(\eta,i+j)$ continue to induce $R_\infty$-module morphisms $\psi_{ij}\colon A(\eta,i)_{\tilde{s}} \hookrightarrow A(\eta,i+j)_{\tilde{s}}$.
Even better, since $\tilde{s}$ and $s$ differ by multiplication with a unit, under the hypotheses of Corollary \ref{torsion}, we have that $H^*A(\eta,\infty)_{\tilde{s}}$ is a finitely generated $R_\infty$-module and for all sufficiently large $i,j$, $\ker\big(\psi_{ij}^*\colon H^*A(\eta,i)_{\tilde{s}} \rightarrow H^*A(\eta,i+j)_{\tilde{s}}\big)$ is naturally isomorphic to $\mathrm{Tors}_{R_\infty}\, H^*A(\eta,\infty)_{\tilde{s}}$ as $R_\infty$-modules.
\end{remk}


\section{Thickened Complexes and Filtrations}\label{thickenedcomplexesandfiltrations} \textit{Fix $m \geq 1$ for this section.}

When we discuss mixed Hodge complexes, we will need to find weight and Hodge filtrations on thickened complexes that are induced by filtrations on the original cdga.
In this section, we describe a general procedure for inducing appropriate filtrations on the thickened complex. 

Suppose $(A,d,\wedge)$ is a cdga with an increasing cdga-filtration $W_{\lc}$ and $\eta \in W_1A^1 \cap \ker d$. Note the requirement that $\eta \in W_1A$. 
We start by defining a weight filtration $w_{\lc}$ on $R_m$.
For $i \geq 0$, define $$w_iR_m = R_m \quad \text{and}\quad w_{-2i}R_m = w_{-2i+1}R_m = s^iR_m.$$

We equip $A(\eta, m)$ with the tensor weight filtration, which we also denote by $W_{\lc}$ since it does not depend on $\eta$.
In other words, for $i \in \Z$, we set:
\begin{align*}
 W_iA(\eta,m) &= \sum_j W_{i+j}A \otimes w_{-j}R_m
\end{align*}
where $\sum$ denotes the internal sum, not a direct sum.
Observe that we can rewrite the expression with a direct sum using the definition of $w_{\lc}$:
\begin{align*}
W_iA(\eta,m) &= \bigoplus_{j=0}^{m-1} W_{i+2j}A \otimes k\langle s^j \rangle.
\end{align*}\index{filtration!thickening}\index{filtration!weight}\index{thickening!filtration}\index{thickening!filtration|see{filtration!thickening}}
Since $\eta \wedge W_i A \subset W_{i+1} A$, we have $s\eta\wedge W_iA(\eta,m)\subset W_iA(\eta,m)$. Therefore, $W_iA(\eta,m)$ is closed under $d_\eta$. 

One technical reason for skipping by twos in the direct sum description is explained by the following lemma.
 
\begin{lem}\label{gradedcomponents}
Suppose $(A,d,\wedge)$ is a cdga with an increasing cdga-filtration $W_{\lc}$ and suppose $\eta \in W_1A^1 \cap \ker d$.
Then,
\begin{align*}
 \Gr_i^W A(\eta,m) &= \bigoplus_{j=0}^{m-1} \Gr_{i+2j}^W A \otimes k\langle s^j \rangle
\end{align*}
with differential given by $d \otimes \id$.
\end{lem}
\begin{proof}
This follows from the definition of the weight filtration. In particular, skipping by twos ensures that the $(\eta \wedge - )\otimes s$ component of the differential vanishes on the graded pieces.
\end{proof}

We can state a filtered analogue of Lemma \ref{cohomologous}.

\begin{lem}\label{cohomologousfiltered}
Suppose $(A,d,\wedge)$ is a cdga with an increasing cdga-filtration $W_{\lc}$. Suppose that $\eta_1$ and $\eta_2$ are elements of $W_1A^1 \cap \ker d$ that are cohomologous in $W_1A$.
If $a \in W_1A^0$ is such that $\eta_1-\eta_2 = da$, then the morphism $\exp(a \otimes s) \wedge -$ of Lemma \ref{cohomologous} gives a filtered isomorphism
\[\exp(a \otimes s) \wedge -\colon \left(A(\eta_1, m), W_{\lc}\right) \rightarrow \left(A(\eta_2,m), W_{\lc}\right)\] that induces the identity map on each graded component associated to $W_{\lc}$.
\end{lem}
\begin{proof}
Let $a \in W_1A^0$ be such that $da = \eta_1-\eta_2$.
By Lemma \ref{cohomologous}, we know that 
\[
\exp(a \otimes s) \wedge - \colon A(\eta_1,m) \rightarrow A(\eta_2,m)
\]
is an isomorphism.
Because $a \in W_1A$, this isomorphism preserves filtrations. Furthermore, wedging by $\exp(a \otimes s) - 1 \otimes 1$ takes $W_i$ into $W_{i-1}$ for all integers $i$.
Consequently, $\exp(a \otimes s) \wedge -$ coincides with $(1 \otimes 1) \wedge -$ on each graded component. 
\end{proof}

We can also expand on Lemma \ref{inducedmap}.

\begin{lem}\label{inducedquasi}
Let $(A, d, \wedge)$ and $(B,d,\wedge)$ be cdgas with increasing cdga-filtrations both denoted by $W_{\lc}$, and $\eta \in W_1A^1 \cap \ker d$.
If $F\colon (A,W_{\lc}) \rightarrow (B,W_{\lc})$ is a filtered morphism of cdgas, then $F_{\#}\colon A(\eta, m) \rightarrow B(F(\eta), m)$ is a filtered morphism with respect to $W_{\lc}$.
Moreover for $i \in \Z$, $F_{\#}$ is given on the $i$-th graded component by:
\begin{align*}
\Gr^W_i(F_{\#}) = \bigoplus_{j=0}^{m-1} \Gr_{i+2j}^W F \otimes \id\colon \bigoplus_{j=0}^{m-1} \Gr_{i+2j}^W A \otimes k\langle s^j \rangle
\rightarrow \bigoplus_{j=0}^{m-1} \Gr_{i+2j}^W B \otimes k\langle s^j \rangle.
\end{align*} 
\end{lem}
\begin{proof}
Because $F_{\#} = F \otimes \id$, the lemma is immediate.	
\end{proof}

Suppose now that $(A,d,\wedge)$ is a cdga with two cdga-filtrations, an increasing filtration $W_{\lc}$ and a decreasing filtration $F^{\lc}$.
Suppose also that $\eta \in W_1A^1 \cap F^1A^1 \cap \ker\, d$.
We next equip $A(\eta, m)$ with a decreasing filtration, which we also denote by $F^{\lc}$ because it does not depend on $\eta$. For $p \in \Z$, set:
\begin{align*}
 F^pA(\eta,m) = \bigoplus_{j = 0}^{m-1} F^{p+j}A \otimes k\langle s^j \rangle.
\end{align*}\index{filtration!Hodge}\index{filtration!thickening}\index{thickening!filtration}\index{thickening!filtration|see{filtration!thickening}}
This is closed under $d_{\eta}$, because $\eta \wedge F^pA \subset F^{p+1}A$ for all $p \in \Z$.
Unlike the filtration $W_{\lc}$ on $A(\eta,m)$, the decreasing filtration $F^{\lc}$ is \textit{not} a filtration by cochain complexes of $R_m$-modules, only by cochain complexes of $k$-vector spaces.
The following lemma describes the interplay of the two filtrations we have defined.

\begin{lem}\label{filtrationinterplay}
Suppose $(A,d,\wedge)$ is a cdga with an increasing cdga-filtration $W_{\lc}$ and a decreasing cdga-filtration $F^{\lc}$. 
Suppose also that $\eta \in W_1A^1 \cap F^1A^1 \cap \ker d$.
Then for $i \in \Z$ the filtration on $\Gr_i^W A(\eta,m)$ induced by the filtration $F^{\lc}$ of $A(\eta,m)$ is described by:
\begin{align*}
 F^p \Gr_i^WA(\eta,m) = \bigoplus_{j=0}^{m-1}\left(F^{p+j}\Gr^W_{i+2j}A\right) \otimes k \langle s^j \rangle. 
\end{align*}	
where on the right hand side we use the filtration on $\Gr_{i+2j}^W A$ induced by the filtration $F^{\lc}$ of $A$.
\end{lem}
\begin{proof}
Using the definition of the two filtrations, for all integers $i$ and $p$ we find:
\begin{align*}
	F^pA(\eta,m)\cap W_iA(\eta,m) = \bigoplus_{j=0}^{m-1}\left(F^{p+j}A \cap W_{i+2j}A\right) \otimes k \langle s^j \rangle
\end{align*}
which, under the quotient map to $\Gr_i^WA(\eta,m)$, maps to: 
\[
	F^p\Gr_i^WA(\eta,m) = \bigoplus_{j=0}^{m-1}\left(F^{p+j}\Gr^W_{i+2j}A\right) \otimes k \langle s^j \rangle.
\qedhere
\]
\end{proof}


\section{Thickened Complexes of Sheaves of Differential Graded Algebras}
Suppose $X$ is a topological space and $(\calA, \wedge, d)$ is a sheaf of cdgas on $X$.
If $m \geq 1$ and $\eta \in \Gamma(X, \calA^1) \cap \ker d$, then we define the \bi{$m$-thickening of $\calA$ in direction $\eta$}, denoted by $\calA(\eta, m)$, to be the cochain complex of sheaves of $R_\infty$-modules on $X$ defined by:
\begin{align*}
 U \mapsto \Gamma(U, \calA)(\eta|_U, m)
\end{align*}
for an open subset $U \subset X$. 
Observe that this does in fact define a cochain complex of sheaves, not just presheaves.
By definition, for all $p \geq 0$ we have $\calA^p(\eta,m) = \calA^p \otimes_k R_m$. 

Recalling Remark \ref{twistedmodule}, if $\tilde{s} \in R_\infty$ differs from $s$ by multiplication with a unit, then we define $\calA(\eta,m)_{\tilde{s}}$ to be the cochain complex of sheaves of $R_\infty$-modules described for open $U \subset X$ by:
\begin{align*}
 U \mapsto \Gamma(U, \calA)(\eta|_U, m)_{\tilde{s}}.
\end{align*}

As before, if $i \geq 1,j \geq 0$ then we have a surjective morphism $\phi_{ji}\colon \calA(\eta,i+j) \twoheadrightarrow \calA(\eta,i)$.
This endows $\{\calA(\eta,i)\}$ with the structure of an inverse system of sheaves of $\{R_i\}$-modules.
We avoid discussing the inverse limit, as it may not coincide with the derived inverse limit.
Again as before, if $i \geq 1,j \geq 0$ then we also have an injective morphism $\psi_{ij}\colon \calA(\eta,i) \hookrightarrow \calA(\eta,i+j)$.

\begin{remk}\label{sheafgeneralities}
The analogues of Lemmas \ref{cohomologous} and \ref{inducedmap} hold for sheaves of cdgas and the dga-sheaf-morphisms between them.
If a sheaf of cdgas $(\calA, \wedge, d)$ is equipped with an increasing cdga-filtration $W_{\lc}$, then as in Section \ref{thickenedcomplexesandfiltrations}, the thickened complex of sheaves $\calA(\eta, m)$ inherits this filtration, provided that $\eta$ is a global section of $W_1\calA$. The same holds if we replace $W_{\lc}$ by a decreasing filtration. The analogues of Lemmas \ref{cohomologousfiltered}, \ref{inducedquasi}, and \ref{filtrationinterplay} are then also valid.
\end{remk}

\begin{remk}\label{remk:torsionForSheaves}
The analogue of Lemma~\ref{lem:torsion2} holds for a complex of sheaves of free $R_\infty$-modules. This can be seen by applying Lemma~\ref{lem:torsion2} to a free resolution of a complex representing the hypercohomology. Therefore, the analogue of Corollary~\ref{torsion} holds for the hypercohomology sheaves of a thickening of a sheaf of cdgas $(\calA, \wedge, d)$.
\end{remk}


	\chapter{Thickened Complexes and Mixed Hodge Complexes}\label{sec4}

The main goal of this chapter is to prove that the thickened complex of a multiplicative mixed Hodge complex is again a mixed Hodge complex, provided that the $1$-forms used to conduct thickenings belong to $W_1$ and $F^1$ when applicable.


\section{Denotations and Assumptions}\label{denotationsandassumptionsthickenedcomplexesandmixedhodgecomplexes}

Let $k$ denote a subfield of $\R$ and $X$ denote a topological space.
Let $\calK^\bullet = ((\calK_k^\bullet, W_{\lc}), \linebreak{(\calK^\bullet_\C, W_{\lc},F^{\lc})}, \alpha)$ denote a multiplicative $k$-mixed Hodge complex on $X$ with distinguished representative for $\alpha$ given by:
$$
{\small
\begin{tikzcd}[column sep = small]
 (\calK^\bullet_k, W_{\lc}) = (\calL^{\bullet}_0, W_{\lc}) \arrow{r}{\alpha_1} & (\calL^\bullet_1, W_{\lc})& \arrow{l}[swap]{\alpha_2} (\calL^\bullet_2, W_{\lc}) \arrow{r}{\alpha_3} & \cdots \arrow{r}{\alpha_{2r+1}} & (\calL^\bullet_{2r+1}, W_{\lc}) = (\calK^\bullet_\C,W_{\lc}),
\end{tikzcd}
}
$$
where $\alpha_{\textrm{odd}}$ are right arrows and $\alpha_{\textrm{even}}$ are left arrows.
To thicken $\calK^\bullet$ as a mixed Hodge complex, we need to thicken each $\calL^{\bullet}_i$ with an element from $\Gamma(X, \calL^1_{i})\cap \ker d$. It is in fact enough to choose such global $1$-cycles from the even-numbered sheaves $\calL^{1}_{\textrm{even}}$, provided their images in the odd-numbered $\calL^{1}_{\textrm{odd}}$ are cohomologous.
Specifically, we need the following assumption on chosen sections $\eta_\ell\in \Gamma(X, \calL^1_{2\ell})\cap \ker d$ to guarantee that the differentials in the lifting are compatible with the filtrations and that the maps $\alpha_i$ extend to pseudo-isomorphisms. 
\begin{assumption}\label{as}
The elements $\eta_\ell\in \Gamma(X, \calL^1_{2\ell})\cap \ker d$ where $0 \leq \ell \leq r$ satisfy: 
\begin{enumerate}
\item $\eta_\ell\in W_1\Gamma(X, \calL^1_{2\ell})\cap \ker d$;
\item $\alpha_{2r+1}\eta_r\in F^1\Gamma(X, \calK^\bullet_\C)$;
\item the images of $\eta_\ell$ in $\calL^1_{\textrm{odd}}$ are cohomologous, i.e., 
\[
 \alpha_1 \eta_0 \simeq_{W_1} \alpha_2 \eta_1,\quad \alpha_3 \eta_1 \simeq_{W_1} \alpha_4 \eta_2,\quad \dots,\quad \alpha_{2r-1} \eta_{r-1} \simeq_{W_1} \alpha_{2r} \eta_r,
\]
where $\simeq_{W_1}$ denotes the relation of being cohomologous in $W_1$.
For $\ell \geq 1$ fix choices of $a_\ell \in W_1\Gamma(X,\calL_{2\ell-1}^0)$ such that $\alpha_{2\ell-1}\eta_{\ell-1} - \alpha_{2\ell} \eta_\ell = da_\ell$. 
\end{enumerate}
\end{assumption}


\section{Thickened Complexes of Mixed Hodge Complexes}\label{thickenedcomplexesofmixedhodgecomplexes}

\textit{Fix $m \geq 1$ for this section.}

Before thickening our complexes, we point out that, if $\ell \geq 1$, then $\calL^\bullet_\ell$ is a sheaf of $\C$-cdgas with filtration(s) defined over $\C$. 
This is part of the definition of a multiplicative mixed Hodge complex. 

Using Remark \ref{sheafgeneralities} and Assumption \ref{as}, we obtain a chain of morphisms of thickened complexes:
$$
{\footnotesize
\begin{tikzcd}[column sep = small]
\calK^\bullet_k(\eta_0,m) \arrow[r,"\alpha_{1\#}", shift left=0.2ex] & \calL^\bullet_1(\alpha_1\eta_0, m) \ar[rr, "\exp(a_1 \otimes s)", "\cong"', shift left=0.2ex] & & \calL^\bullet_1(\alpha_2\eta_1,m) \arrow[from=r,"\alpha_{2\#}"', shift right=0.2ex]& \calL^\bullet_2(\eta_1,m) \arrow[r,"\alpha_{3\#}", shift left=0.2ex] & \cdots \arrow[r,"\alpha_{2r+1\#}", shift left=0.2ex] & \calK^\bullet_\C(\alpha_{2r+1}\eta_r,m) 
\end{tikzcd}
}
$$
where all but the leftmost complex have been thickened over $\C$, not $k$. 
Note that $\alpha_{1\#}$ is technically a composition: first the induced morphism of thickened $k$-complexes, then the map from the thickened $k$-complex to the thickened $\C$-complex induced by $ R_m = R_m \otimes k \xrightarrow{\id \otimes \mathrm{incl}} R_m \otimes \C$.

We define $\alpha_{\#}\colon \calK^\bullet_k(\eta_0,m) \dashrightarrow \calK_\C^\bullet(\alpha_{2r+1}\eta_r,m)$ to be the pseudo-morphism associated to the above chain.

Because the thickenings are conducted with sections of $W_{1}\calL^1$, it follows by Remark \ref{sheafgeneralities} that the thickened complexes of sheaves inherit their own increasing filtrations $W_{\lc}$.
Moreover, $\calK^\bullet_\C(\alpha_{2r+1}\eta_r, m)$ inherits a decreasing filtration $F^{\lc}$, because $\alpha_{2r+1}\eta_r$ belongs to $F^1\Gamma(X, \calK^\bullet_C)$.

\begin{thm}\label{mhsthickened}
Under the above notations and assumptions, the triple 
\[
\calK^\bullet(\eta,m) \coloneqq \Big(\big(\calK_k^\bullet(\eta_0,m),W_{\lc}\big), \big(\calK^\bullet_\C(\alpha_{2r+1}\eta_r,m),W_{\lc},F^{\lc}\big), \alpha_{\#}\Big)
\]
is a $k$-mixed Hodge complex of sheaves on $X$.
\end{thm}
\begin{proof}
We need to verify the following assertions:
\begin{enumerate}[(i)]
\item $\alpha_\#$ preserves the $W_{\lc}$-filtrations; 
\item $
\alpha_\# \otimes 1\colon \calK_k^\bullet(\eta_0,m) \otimes \C \dasharrow \calK^\bullet_\C(\alpha_{2r+1}\eta_r,m)
$
is a  pseudo-isomorphism preserving the $W_{\lc}$-filtrations;
\item the $i$-th $W$-graded component of $\calK^\bullet(\eta,m)$ is a pure Hodge complex of weight $i$. 
\end{enumerate}

To show statements (i) and (ii), we apply the sheaf analogues of Lemma \ref{cohomologousfiltered} and Lemma \ref{inducedquasi}. The chain of morphisms representing $\alpha_\#$ is $W_{\lc}$-filtered, and it induces on the $i$-th $W$-graded component the pseudo-morphism:
{\small
\begin{align*}
\Gr^W_i(\alpha_{\#}) = \bigoplus_{j=0}^{m-1} \Gr_{i+2j}^W \alpha \otimes \mathrm{incl}\colon \bigoplus_{j=0}^{m-1} \Gr_{i+2j}^W \calK^\bullet_k \otimes k\langle s^j \rangle
\dasharrow \bigoplus_{j=0}^{m-1} \Gr_{i+2j}^W \calK^\bullet_\C \otimes_\C \C\langle s^j \rangle.
\end{align*}}
Therefore $\Gr^W_i(\alpha_{\#} \otimes 1) = \bigoplus_{j=0}^{m-1} \Gr_{i+2j}^W (\alpha \otimes 1) \otimes_\C \id$.
Each $\Gr_{i+2j}^W (\alpha \otimes 1)$ is a quasi-isomorphism, because $\calK^\bullet$ is a mixed Hodge complex.
Therefore, as their direct sum, $\Gr^W_i(\alpha_{\#}\otimes 1)$ is also a quasi-isomorphism.

To prove (iii), let us fix $i \in \Z$. 
We have just verified that the $i$-th graded component of $\calK^\bullet(\eta,m)$ is:
\begin{align*}
 \bigoplus_{j = 0}^{m-1}\Big(\Gr^W_{i+2j}\calK^\bullet_k \otimes k\langle s^j \rangle, \left(\Gr^W_{i+2j}\calK^\bullet_\C \otimes \C\langle s^j \rangle, F^{\lc}\right), \Gr_{i+2j}^W\alpha \otimes \mathrm{incl}\Big)
\end{align*}
and we need to show that this is a Hodge complex of weight $i$.
It is enough to check that, for all $0\leq j\leq m-1$, the $j$-th direct summand above is a Hodge complex of weight $i$.
By Lemma \ref{filtrationinterplay}, we have:
\begin{align*}
F^p \left(\Gr_{i+2j}^W\calK^\bullet_\C \otimes \C \langle s^j \rangle\right) = \left(F^{p+j}\Gr_{i+2j}^W\calK^\bullet_\C\right) \otimes \C\langle s^j \rangle	
\end{align*}
for $p\in \Z$. Therefore, the $j$-th direct summand is isomorphic to the $j$-th Tate twist of the following weight $i+2j$ Hodge complex of sheaves:
\begin{align*}
\Big(\Gr^W_{i+2j}\calK^\bullet_k, \left(\Gr^W_{i+2j}\calK^\bullet_\C, F^{\lc}\right), \Gr_{i+2j}^W\alpha\Big).
\end{align*}
The $j$-th Tate twist of a weight $i+2j$ Hodge complex is a weight $i$ Hodge complex.
\end{proof}

\begin{cor}\label{mhshypercohomology}
The $k$-mixed Hodge complex of sheaves $\calK^\bullet(\eta,m)$ induces
a $k$-mixed Hodge structure on the hypercohomology $\mathbb{H}^*(X, \calK^\bullet_k(\eta_0,m))$. 
\end{cor}
\begin{proof}
See Remark \ref{transvstate}.
\end{proof}

Recall that if $ i \geq 1,j \geq 0$, then the natural map $R_{i+j} \twoheadrightarrow R_i$ induces a morphism $\phi_{ji}$ of thickened complexes. Furthermore, the natural map $R_i \hookrightarrow R_{i+j}$ induces a morphism $\psi_{ij}$ of thickened complexes.\index{psi@$\psi_{ij}$}\index{psi@$\psi_{ij}$!psi@$\psi_{ij}^*$}\index{phi@$\phi_{ij}$}\index{phi@$\phi_{ij}$!phi@$\phi_{ij}^*$}

\begin{lem}\label{inducedmhsmaps}
If $i \geq 1,j \geq 0$, then the following are morphisms of $k$-mixed Hodge complexes of sheaves. The numbers $(-j)$ and $(-1)$ denote Tate twists.
\begin{enumerate}
\item $\phi_{ji}\colon \calK^\bullet(\eta,j+i) \twoheadrightarrow \calK^\bullet(\eta,i),$
\item $\psi_{ij}\colon \calK^\bullet(\eta,i) \hookrightarrow \calK^\bullet(\eta,i+j)(-j),$
\item Multiplication by $s$, $S_i:\calK^\bullet(\eta,i) \rightarrow \calK^\bullet(\eta,i)(-1)$.
\end{enumerate}
In particular, the following are morphisms of $k$-mixed Hodge structures:
\begin{enumerate}
\item $\phi_{ji}^*\colon \mathbb{H}^*(X,\calK^\bullet_k(\eta_0,j+i)) \rightarrow \mathbb{H}^*(X,\calK^\bullet_k(\eta_0,i)),$
\item $\psi_{ij}^*\colon \mathbb{H}^*(X,\calK^\bullet_k(\eta_0,i)) \rightarrow \mathbb{H}^*(X,\calK^\bullet_k(\eta_0,i+j))(-j),$
\item Multiplication by $s$, $ \mathbb{H}^*(X,\calK^\bullet_k(\eta_0,i)) \rightarrow \mathbb{H}^*(X,\calK^\bullet_k(\eta_0,i))(-1)$.
\end{enumerate}
\end{lem}
\begin{proof}
Recall the definition of the projection $\phi_{ji}$ (at the beginning of Chapter~\ref{sec3}) and the inclusion $\psi_{ij}$ (Remark~\ref{remk:psi-ij}). Observe that $\phi_{ji}$ and $\psi_{ij}$ do indeed induce morphisms between the pseudo-morphism representatives of the two mixed Hodge complexes.
It remains to check their compatibility with the weight and Hodge filtrations.
The morphism $\phi_{ji}$ clearly preserves both filtrations; furthermore, a morphism of $k$-mixed Hodge complexes of sheaves always induces morphisms of $k$-mixed Hodge structures on hypercohomology \cite[Theorem 3.18III]{peters2008mixed}.
It is also straightforward to check that $\psi_{ij}$ maps $W_{\lc}$ into $W_{{\lc}-2j}$ and $F^{\lc}$ into $F^{{\lc}-j}$. In other words, $\psi_{ij}\colon \calK^\bullet(\eta,i) \rightarrow \calK^\bullet(\eta,i+j)(-j)$ is a morphism of $k$-mixed Hodge complexes of sheaves.
Pass to hypercohomology and use that Tate twists commute with taking hypercohomology (see \cite[Theorem 3.18IIi]{peters2008mixed}) to obtain the second statement.

To show part (3), we observe that $S_i$ equals the composition $\psi_{i-1,1}\circ \phi_{1,i-1}$.
\end{proof}

\begin{remk}\label{translatedlemmas}
	Later in the paper, we will be interested not in $\calK^\bullet(\eta,m)$ itself, but rather its translation $\calK^\bullet(\eta,m)[1]$. The valid analogue of Lemma \ref{mhshypercohomology} is that $\calK^\bullet(\eta,m)[1]$ induces a $k$-mixed Hodge structure on $\mathbb{H}^*(X, \calK^\bullet_k(\eta_0,m)[1])$. Note that, by Remark \ref{transvstate},
	$$\mathbb{H}^*(X, \calK^\bullet_k(\eta_0,m)[1]) \cong \mathbb{H}^{*+1}(X, \calK^\bullet_k(\eta_0,m))(1),$$ where a Tate twist has been taken on the right hand side.
The valid analogue of Lemma \ref{inducedmhsmaps} is that for $ i \geq 1,j\ge 0$, the following are morphisms of mixed Hodge structures:
	\begin{align*}
	\phi_{ji}[1]^*&\colon \mathbb{H}^*(X,\calK^\bullet_k(\eta_0,i+j)[1]) \rightarrow \mathbb{H}^*(X,\calK^\bullet_k(\eta_0,i)[1]);\\
	\psi_{ij}[1]^*&\colon \mathbb{H}^*(X,\calK^\bullet_k(\eta_0,i)[1]) \rightarrow \mathbb{H}^*(X,\calK^\bullet_k(\eta_0,i+j)[1])(-j);\\
	(s\cdot)^*&\colon \mathbb{H}^*(X,\calK^\bullet_k(\eta_0,i)[1]) \rightarrow \mathbb{H}^*(X,\calK^\bullet_k(\eta_0,i)[1])(-1). 
	\end{align*}
\end{remk}


	\chapter{Mixed Hodge Structures on Alexander Modules}\label{mhsal}


\section{Denotations and Assumptions}
Let $k = \R$.
Let $U$ be a smooth connected complex algebraic variety of dimension $n$. Let $f\colon U \rightarrow \C^*$ be an algebraic map such that $f_*\colon \pi_1(U)\to \pi_1(\C^*)$ is surjective.
Let $U^f \rightarrow U$ denote the infinite cyclic cover as in Section~\ref{ssAlex}.
Choose a good compactification $X$ of $U$ with simple normal crossing divisor $D = X\setminus U$, such that $f\colon U \rightarrow \C^*$ extends to an algebraic map $\bar{f}\colon X \rightarrow \C P^1$. 
Let 
\[
\calH dg^\bullet(X\, \log D) = \Big(\big(j_*\calE^\bullet_U, \tau_{\lc}\big), \big(\logdr{X}{D}, W_{\lc}, F^{\lc}\big), \alpha \Big)
\]\index{mixed Hodge complex!Hodge-de Rham}
denote the $\R$-mixed Hodge complex of sheaves on $X$ of Theorem \ref{logmhs}. We will consider a choice of $X$ fixed, until Theorem~\ref{indcompactification}, when we prove that our constructions are independent of this choice.

\section{Local Systems and Thickened Real de Rham Complexes}\label{slocal}

The main result of this section is Proposition~\ref{propLocal}. Throughout this section, we will let $k=\R$. As in Section~\ref{ssAlex}, Let $\calL=\pi_!\uR_{U^f}$, and let $\ov\cL$ have the conjugate $R$ action. $\ov\cL$ has stalks isomorphic to $R = \R[t^{\pm 1}]$ and action $\pi_1(U) \rightarrow \Aut(R)$ given by $\gamma \mapsto \left(t^{-f_*(\gamma)}\right)\, \cdot\, $, multiplication by $t^{-f_*(\gamma)}$.
This is a rank one local system of $R$-modules.

Consider $\ov\calL \otimes_R R_m$, which has stalks isomorphic to $R_m = R/(t-1)^mR$ and action $\pi_1(U) \rightarrow \Aut(R_m)$ given by $\gamma \mapsto \left((1+s)^{-f_*(\gamma)}\right)\, \cdot\, $, multiplication by $(1+s)^{-f_*(\gamma)}$ (recall that $1+s=t$). This is a rank one local system of $R_m$-modules.

Next, consider the thickened complex $\calE^{\bullet}_{U}\left(\Im\, df/f,m\right)$, given by
\begin{align*}
 \calE^0_{U} \otimes R_m \xrightarrow{d^0_m = d +\Im\, df/f \wedge \otimes s}
 \calE^1_{U} \otimes R_m \xrightarrow{d^1_m = d +\Im\, df/f \wedge \otimes s}
 \calE^2_{U} \otimes R_m \rightarrow \cdots 
\end{align*}\index{E(df)@$\cE^\bullet_U(\Im\, df/f,m)$}
with differential given by $d_m\colon \omega \otimes 1 \mapsto d\omega \otimes 1 + \left(\Im\, df/f \wedge \omega\right) \otimes s$. Here, $\Im$ denotes the imaginary part. 

 \begin{prop}
 Let
 \[
 \tilde{s}\coloneqq \exp(2\pi s)-1=\sum_{n=1}^\infty \frac{(2\pi s)^{n}}{n!}\in R_\infty.
 \]
 Then
 \[
 \calE^\bullet_U\left(\Im\, df/f,\infty\right)_{\tilde s}\coloneqq \displaystyle\varprojlim_m \left(\calE^\bullet_U\left(\Im\, df/f,m\right)_{\tilde s}\right)
 \]
 is a resolution of $\ov\calL \otimes_R R_\infty$, via a canonical map $\ov\calL \otimes_R R_\infty\to \calE^0_U\left(\Im\, df/f,\infty\right)_{\tilde s}$ (defined in Remark~\ref{rem:nu}). \label{propLocal}
 \end{prop}\index{s@$s$!s@$\tilde s$}
Note that $\tilde{s}=su$, where $u$ is a unit in $R_\infty$, so we are in the situation described by Remark \ref{twistedmodule}.

We will prove Proposition \ref{propLocal} at the end of this section. In simply connected open sets, one can check locally that $\ker d^k_m= \im\, d^{k-1}_m$ for all $k\geq 1$. Thus, $\calE^\bullet_U\left(\Im\, df/f,m\right)$ is a resolution of $\ker\, d^0_m$, which in turn is a rank one local system of $R_m$-modules.

We give a description of this local system.

\begin{lem}\label{lemKer}
If $V \subset U$ is a small enough connected open subset (in the analytic topology) such that $f(V)$ is simply-connected, then $\ker\, d^0_m(V)$ consists of all functions of the form $\exp\left(-\arg f \otimes s\right)g(s) \in \calE^0_U(V) \otimes R_m$, where $g(s)\in R_m$ and $\arg$ can be taken to be any branch of the argument on $f(V)$.
\end{lem}
\begin{proof}
On $V$, the differential is given by $$d^0_m(V)(\omega \otimes 1) = d\omega \otimes 1+\left(d\arg(f) \wedge \omega\right) \otimes s.$$

 Let $\sum_{k=0}^{m-1} g_k\otimes s^k$ be an element of $\ker d^0_m(V)$. Then,
 \[
 0=d^0_m\left(\sum_{k=0}^{m-1} g_k\otimes s^k\right)=
 d g_0\otimes 1+ \sum_{k=1}^{m-1} \big(d g_k+g_{k-1} d\arg(f)\big)\otimes s^k .
 \]
 Hence, $g_0=c_0\in\R$, $dg_1=-c_0 d\arg f$, so $g_1=-c_0 \arg(f)+c_1$ for some $c_1\in \R$ (choosing a different branch of the argument will result in a different $c_1$). Inductively, we get that
 $$
 g_k=\sum_{j=0}^k \frac{(-1)^j}{j!} c_{k-j} \arg^j(f)
 $$
 for some $c_j\in\R$, $j=0, 1, \ldots, k$. Hence,
 \begin{align*}
 &\sum_{k=0}^{m-1} g_k\otimes s^k
 =\sum_{k=0}^{\infty}\left(\sum_{j=0}^k \frac{(-1)^j}{j!} c_{k-j} \arg^j f\right)\otimes s^k=\\
 & =\sum_{j=0}^{\infty}\sum_{k=j}^{\infty} \frac{(-1)^j}{j!} c_{k-j} \arg^j f \otimes s^k=
 \sum_{j=0}^{\infty}\sum_{l=0}^{\infty} \left(\frac{(-1)^j}{j!} \arg^j f \otimes s^j\right)c_ls^l=\\
 &=\left(\sum_{l=0}^{\infty} c_l s^l\right)\exp\big(-\arg(f)\otimes s\big)=
 \left(\sum_{l=0}^{m-1} c_l s^l\right)\exp\big(-\arg(f)\otimes s\big),
 \end{align*}
 which finishes our proof.
\end{proof}

\begin{remk}\label{rem:IVP}
By the proof of the previous lemma, the elements of $\ker d^0_m$ form a local system of rank $m$ over $\R$. In particular, the following map is an $R$-linear isomorphism for any simply connected open set $V$ and any $x\in V$:
\[
\begin{array}{rcl}
\Gamma(V,\ker d^0_m) & \longrightarrow & R_m \\
\sum_{k=0}^{m-1} g_k\otimes s^k &\longmapsto & \sum_{k=0}^{m-1} g_k(x)s^k.
\end{array}
\]
\end{remk}

\begin{lem}
The $\pi_1(U)$-action on the local system $\ker\, d^0_m$ is given by $$\gamma \mapsto \exp(-2\pi f_*(\gamma) s)\, \cdot\, ,$$ i.e., multiplication by $\exp(-2\pi f_*(\gamma) s)$.
\label{lemPi1}
\end{lem}

\begin{proof}
The stalk of $\ker\, d^0_m$ at any point in $U$ is generated by $\exp(-\arg(f) \otimes s)$.
A loop $\gamma\in\pi_1(U)$ takes $\arg(f)$ to $\arg(f)+2\pi f_*(\gamma) $.
Thus, the action of $\gamma \in \pi_1(U)$ on $\ker\, d^0_m$ is:
\begin{align*}
  \gamma \cdot \exp(-\arg(f) \otimes s) &= \exp\left(-\left(\arg(f) + 2\pi f_*(\gamma) \right) \otimes s\right) \\ & =\exp(-2\pi f_*(\gamma) \otimes s) \cdot \exp(-\arg(f) \otimes s) \\&= \exp(-2\pi f_*(\gamma) s) \cdot \exp(-\arg(f) \otimes s). 
\end{align*}
\end{proof}

Consequently, we have two rank one local systems of $R_m$-modules on $U$:
\begin{itemize}
 \item $\ov\calL \otimes_R R_m$ with action $\gamma \mapsto (1+s)^{-f_*(\gamma)} \, \cdot\, ,$ i.e., multiplication by $(1+s)^{-f_*(\gamma)}$, for all $\gamma\in\pi_1(U)$.
 \item $\ker\, d^0_m$ with action $\gamma \mapsto \exp(-2\pi f_*(\gamma) s)\, \cdot \, ,$ i.e., given by multiplication by $\exp(-2\pi f_*(\gamma) s)$, for all $\gamma\in\pi_1(U)$.
\end{itemize}
Because these actions are distinct, $\ov\calL \otimes_R R_m$ and $\ker\, d^0_m$ are not isomorphic as local systems of $R_m$-modules, but they are if one considers the twisted module structure $(\ker\, d^0_m)_{\tilde s}$. 

The following lemma is a consequence of the discussion above.

\begin{lem}
The kernel of the $0$-th differential of $\calE_{U}^\bullet\left(\Im\, df/f,m\right)_{\tilde s}$ is $(\ker\, d^0_m)_{\tilde s}$, which is isomorphic to $\ov\calL\otimes_R R_m$ as local systems of $R_m$-modules.
\label{lemRm}
\end{lem}

\begin{remk}\label{rem:nu}\index{nu@$\nu$}\index{delta@$\delta_x$!deltabar@$\ov\delta_x$}\index{f_infty@$f_\infty$}
We can fix a canonical isomorphism $\nu\colon \ov\calL\otimes_R R_m\cong (\ker\, d^0_m)_{\tilde s}$ as follows. Recall that at any $x\in U$, $\ov\calL_x$ has an $\R$-basis $\{\ov\delta_{x'}\}$ parametrized by $x'\in \pi^{-1}(x)$ (see Remark~\ref{rem:oppositeDual}). For any such $x'$, let $\nu_x(
\ov\delta_{x'})
$ be the unique germ $g$ in the stalk of $(\ker\, d^0_m)$ at $x$ such that $
g(x) = \exp(-\Im f_\infty(x')\otimes s)$. Recall that according to our definitions of $U^f$ and $\cL$ in Section~\ref{ssAlex}, the map $f_\infty$ is fixed as the projection $U^f\subset U\times \C\to \C$.

Since $\exp\circ f_\infty = f\circ \pi$, we can describe $\nu$ as follows. On any simply connected neighborhood $V$ of $x$, let $\iota\colon V\to U^f$ be the section of $\pi$ such that $\iota(x)=x'$. Then $\Im f_\infty \circ \iota$ is a branch of $\arg f$, so
\[
\nu_x(\ov\delta_{x'}) =\exp(-\Im f_\infty\circ \iota\otimes s).
\]
Let us check that $\nu_x$ is an $R_m$-linear map on stalks.
\begin{align*}
\nu_x(t\ov\delta_{x'})& = \nu_x(\ov\delta_{t^{-1}x'}) = \exp(-\Im f_\infty\circ t^{-1}\circ \iota\otimes s)= \exp(-\Im f_\infty\circ \iota\otimes s + 2\pi \otimes s) = \\ &=\nu_x(\ov\delta_{x'})e^{2\pi s} =\nu_x(\ov\delta_{x'})*_{\widetilde s}t.
\end{align*}
Note that, in the formula above, $t$ is seen both as an element of $R_m$ and as a deck transformation.

We can see that $\nu$, which we have defined on stalks, gives us a morphism of local systems over $R_m$. Indeed, from the definition of $\nu$, it agrees with the monodromy action of any path $\gamma$ from $x$ to any given $y\in U$, i.e. $\gamma\cdot \nu_x(\ov\delta_{x'}) = \nu_y(\ov\delta_{\gamma \cdot x'}) =\nu_y(t^{-1}\ov\delta_{ x'}) = \nu_y(\gamma\cdot \ov\delta_{x'})$.
\end{remk}

\begin{remk}\label{remk:thickenedComplex}
The proof of the claims in this remark are straightforward following the definitions of the objects involved, so we omit them. 

The thickened complex $\cE^\bullet_U\left( \darg ,m\right)$ of Lemma~\ref{lemRm} is isomorphic (as sheaves of $\R$-vector spaces) to the complex $\cE^\bullet_U(m)_{\log}$, which is defined as $\cE^i_U(m)_{\log} = \cE^i_U \otimes_\R R_m$ and whose differential is the following, where $\log(1+s)$ represents the power series at $s=0$:
\[
d_{\log}\left(\alpha\otimes s^j\right) = d\alpha\otimes s^j+\darg \wedge \alpha\otimes \frac{ \log(1+s)}{2\pi}s^j.
\] The ($\R$-linear) isomorphism is determined by:
\[
\ff{\cE_U^i\left( \darg ,m\right)}{\cE^i_U(m)_{\log}}{\alpha\otimes s^j}{\alpha\otimes \left(\frac{\log(1+s)}{2\pi}\right)^j.}
\]
Note that the above isomorphism becomes $R_m$-linear if we give $\cE_U^{\bullet}\left( \darg ,m\right)$ the twisted $R_m$-module structure $\cE_U^{\bullet}\left( \darg ,m\right)_{\wt s}$ of Proposition~\ref{propLocal}. \textit{We will use this $\R$-linear identification between $\cE_U^i\left( \darg ,m\right)$ and $\cE^i_U(m)_{\log}$ from now on whenever it is convenient}. It will simplify our notation later, especially in Chapter~\ref{sec:main}.

Using Lemma~\ref{lemRm}, we see that the composition $\overline{\cL}\otimes_R R_m \xrightarrow{\nu}\cE_U^{\bullet}\left( \darg ,m\right)_{\wt s}\to \cE^{\bullet}_U(m)_{\log}$ is an \bi{$\bm{ R_m}$-linear} quasiisomorphism, i.e. $\cE^{\bullet}_U(m)_{\log}$ is a resolution of $\overline{\cL}\otimes_R R_m$ as sheaves of $R_m$-modules. The resolution $\overline{\cL}/s^m\xrightarrow{\eta_m} \cE^{\bullet}_U(m)_{\log}$ is determined by the map $\eta_m$, analogous to the map $\nu$ in Remark~\ref{rem:nu}: using the same notation, 
$\eta_m$ is given by:
$$
\f{(\eta_m)_x}{(\overline{\cL}/s^m)_x}{(\cE^0_U(m)_{\log})_x}{\ov\delta_{x'}}{\exp(-\Im f_\infty \circ \iota \otimes \frac{\log(1+s)}{2\pi}).}
$$
Manipulating power series we can see the expression above as
\[
\eta_m(\ov\delta_{x'}) = (1+s)^{-\frac{1}{2\pi}\Im f_\infty \circ \iota} = \sum_{i=0}^{m-1} \binom{-\frac{1}{2\pi}\Im f_\infty \circ \iota}{i} \otimes s^i.
\]
\end{remk}

We now prove a result that will help us understand $\displaystyle\varprojlim_m \left(\calE^\bullet_U\left(\Im\, df/f,m\right)_{\tilde s}\right)$.

\begin{lem}
Let $\left(\{\calF_m^{\bullet}\}_{m\in \N},\left\{\alpha^\bullet_{m,j}\right\}_{m\geq j\in \N}\right)$ be an inverse system of complexes of sheaves of $R_{\infty}$-modules on a topological space that has a basis of simply connected open sets, such that the maps $\alpha_{m,j}^k(V)\colon \calF_m^k(V)\rightarrow\calF_j^k(V)$ are surjective for all $m\geq j$ and for every simply connected open set $V$. Suppose that $\calF_m^{\bullet}(V)$ are exact complexes for all $m\in\N$ and for every simply connected open set $V$. Then,
$\varprojlim_m \calF_m^\bullet$ is also an exact complex of sheaves of $R_\infty$-modules.
\label{clExact}
\end{lem}

\begin{proof}
We need to show that the stalk of $\varprojlim_m \calF_m^\bullet$ at every point is exact. Since $U$ has a basis of simply connected open sets, and since
\[
\Big(\displaystyle\varprojlim_m \calF_m^k\Big)(V)\simeq \displaystyle\varprojlim_m \calF_m^k(V)
\]
for every $k$, it suffices to show that $\displaystyle\varprojlim_m \calF_m^\bullet(V)$ is exact for any simply connected open subset $V$ of $U$. 

For a fixed $m$, the exact complex $\calF_m^{\bullet}(V)$ gives rise to short exact sequences
\[
0\to \ker(d^k_m)\to \calF_m^{k}(V) \to \ker (d^{k+1}_m)\to 0,
\]
where $d^k_m\colon \calF_m^{k}(V)\to \calF_m^{k+1}(V)$ are the differentials. By assumption, the homomorphisms $\alpha_{m, j}^k$ in the inverse system $\displaystyle\varprojlim_m \calF_m^k(V)$ are surjective, so the inverse system satisfies the Mittag-Leffler (ML) condition. Using the fact that the spaces $\ker (d^{k+1}_m)$ are quotients of $\calF_m^{k}(V)$ for all $k$, we can see that the restrictions $\alpha_{m, j}^k\colon \ker (d^{k+1}_{m})\to \ker (d^{k+1}_{j})$ are surjective, so the inverse system $\displaystyle\varprojlim_m \ker(d^k_m)$ also satisfies (ML). By \cite[Chapter II, Proposition 9.1 (b)]{Hartshorne}, we have exact sequences,
\[
0\to \displaystyle\varprojlim_m\ker(d^k_m)\to \displaystyle\varprojlim_m\calF_m^{k}(V) \to\displaystyle\varprojlim_m \ker (d^{k+1}_m)\to 0,
\]
for every $k$. The above short exact sequences for all $k$ induce the following long exact sequence
\[
\cdots\to \displaystyle\varprojlim_m\calF_m^{k-1}(V) \to \displaystyle\varprojlim_m\calF_m^{k}(V) \to \displaystyle\varprojlim_m\calF_m^{k+1}(V) \to \cdots,
\]
thus completing the proof of the lemma. 
\end{proof}

Now, we are ready to prove Proposition \ref{propLocal}, the main result of this section.
\begin{proof}[Proof of Proposition~\ref{propLocal}]
The result about the resolution follows from applying Lemma \ref{clExact} to the exact complex of sheaves of $R_m$-modules 
\[
\ov\calL\otimes_R R_m\rightarrow \calE_U^\bullet\left(\Im\, df/f,m\right)_{\tilde s}.
\]
Notice that $\varprojlim \ov\calL \otimes_R R_m$ is naturally isomorphic to $\ov\calL \otimes_R R_{\infty}$.
\end{proof}


\section{Local Systems and Thickened Complexes: Rational Version}
\index{K_infty@$\calK^\bullet_\infty$}
We now state the analogue of the results in Section~\ref{slocal} in the $k=\Q$ case, for the thickened complex $j^{-1}\calK_\infty^\bullet(1\otimes f, m)$ constructed from the complex $\calK_\infty^\bullet$ of Section \ref{rationalmixedhodgecomplexonsmoothvarieties}. In the remainder of this section, $k=\Q$, $\ov\calL$ is a local system of $\Q[t^{\pm 1}]$-modules, $d^i_m$ is the $i$-th differential of $j^{-1}\calK_\infty^\bullet(1\otimes f, m)$, and $\tilde{s}_2\coloneqq \exp(s)-1\in R_\infty$. Abusing notation, we will consider $\tilde{s}_2$ as an element of $R_m$ or $\C[t^{\pm 1}]_m$ for all $m\geq 1$ whenever it makes sense.\index{s@$s$!s@$\tilde s_2$}

As stated in Section \ref{rationalmixedhodgecomplexonsmoothvarieties}, $\varphi_{\infty}\otimes 1\colon (\calK_\infty^\bullet,W_{\lc})\otimes \C\rightarrow (\Omega_X^\bullet(\log D),W_{\lc})$ is a filtered quasi-isomorphism, which takes $1\otimes f$ to $\frac{1}{2\pi i}\frac{df}{f}$. So it induces a quasi-isomorphism $j^{-1}\calK_\infty^\bullet\otimes \C\rightarrow \Omega_U^\bullet$. By Remark \ref{sheafgeneralities}, this induces a quasi-isomorphism
\begin{equation}
F\colon j^{-1}\calK_\infty^\bullet(1\otimes f, m)\otimes \C\rightarrow \Omega_U^\bullet\left(\frac{1}{2\pi i}\frac{df}{f}, m\right)
\label{eqnF}
\end{equation}
for all $m\geq 1$. One can check locally in simply connected open sets that the complex $\Omega_U^\bullet\left(\frac{1}{2\pi i}\frac{df}{f}, m\right)$ is exact at place $k$ for all $k\geq 1$. Therefore, since $\C$ is faithfully flat over $\Q$, $\ker d^{k}_m=\im d_m^{k-1}$ for all $k\geq 1$. Hence, $j^{-1}\calK_\infty^\bullet(1\otimes f, m)$ is a resolution of $\ker\, d^0_m$, which in turn is a rank one local system of $R_m$-modules. Now, we give a description of this local system. 
\begin{lem}[Analogue of Lemma \ref{lemKer}]
If $V \subset U$ is a small enough open subset (in the analytic topology) such that $f(V)$ is simply-connected, then $\ker\, d^0_m(V)$ is generated by $\exp\left(-\frac{\log(f)}{2\pi i}\otimes 1 \otimes s\right) \in \calK^0_\infty(V) \otimes R_m$ as an $R_m$-module, where $\log$ is taken to be a branch of the logarithm function which is defined on $f(V)$.
\end{lem}
\begin{proof}
A direct computation shows that $\exp\left(-\frac{\log(f)}{2\pi i}\otimes 1 \otimes s\right)\in \ker\, d^0_m(V)$. We have that
$$
F\left(\exp\Big(-\frac{\log(f)}{2\pi i}\otimes 1 \otimes s\Big)\otimes 1\right)=\exp\left(-\frac{\log(f)}{2\pi i} \otimes s\right)\in \Omega_U^0\otimes \C[t^{\pm 1}]_m,
$$ where $F$ is the quasi-isomorphism in equation (\ref{eqnF}). Following an argument as in Lemma \ref{lemKer} for the thickened {holomorphic} de Rham complex $\Omega_U^\bullet\left(\frac{1}{2\pi i}\frac{df}{f}, m\right)$ instead of $\calE_U^\bullet\left(\Im\big(\frac{df}{f}\big),m \right)$, we get that $\exp\Big(-\frac{\log(f)}{2\pi i} \otimes s\Big)$ generates the kernel of the $0$-th differential of $\Omega_U^\bullet\left(\frac{1}{2\pi i}\frac{df}{f}, m\right)$.

Let $A$ be the $R_m$-module generated by $\exp\left(-\frac{\log(f)}{2\pi i}\otimes 1 \otimes s\right)$. Using the map $F$, we see that the inclusion $A\hookrightarrow \ker\, d^0_m(V)$ becomes an isomorphism after tensoring by $\C$. Since $\C$ is faithfully flat over $\Q$, it follows that $A\cong \ker\, d^0_m(V)$.
\end{proof}

The proof of the rest of the results in this section follow the same steps as the proof of their analogue results in the real case, and we will omit them.
\begin{lem}[Analogue of Lemma \ref{lemPi1}]
The $\pi_1(U)$-action on the local system $j^{-1}\ker\, d^0_m$ is given by $\gamma \mapsto \exp(-f_*(\gamma) s)\, \cdot\, $, multiplication by $\exp(- f_*(\gamma) s)$.
\end{lem}

\begin{lem}[Analogue of Lemma \ref{lemRm}]\label{lemRmQ}
The kernel of the $0$-th differential of the complex $j^{-1}\calK_{\infty}^\bullet\left(1\otimes f,m\right)_{\tilde{s}_2}$ is $(j^{-1}\ker\, d^0_m)_{\tilde{s}_2}$, which is isomorphic to $\ov\calL\otimes_R R_m$ as local systems of $R_m$-modules.
\end{lem}

\begin{remk}[Analogue of Remark \ref{rem:nu}]\label{rem:nuQ}
We can fix a canonical isomorphism $\nu_{\Q}\colon \ov\calL\otimes_R R_m\cong (\ker\, d^0_m)_{\tilde s_2}$ as follows. Recall that at any $x\in U$, $\ov\calL_x$ has an $\Q$-basis $\{\ov\delta_{x'}\}$ parametrized by $x'\in \pi^{-1}(x)$. For any such $x'$, let $(\nu_{\Q})_x(\ov
\delta_{x'})
$ be the unique germ $g$ in the stalk of $(j^{-1}\ker\, d^0_m)$ at $x$ such that $
g(x) = \exp\left(-\frac{f_\infty(x')}{2\pi i}\otimes 1\otimes s\right)$. On any simply connected neighborhood $V$ of $x$, let $\iota\colon V\to U^f$ be the section of $\pi$ such that $\iota(x)=x'$. Then, $\nu_{\Q}$ can be described by
\[
(\nu_\Q)_x(\ov\delta_{x'}) =\exp\left(-\frac{f_\infty\circ\iota}{2\pi i}\otimes 1\otimes s\right).
\]
Similarly as in Remark \ref{rem:nu}, we can check that $\nu_{\Q}$ is an $R_m$-linear map on stalks (after twisting the $R_m$ module structure by $\widetilde{s_2}$ in the target) and that it defines a morphism of local systems of $R_m$-modules on $U$.
\end{remk}

\begin{prop}[Analogue of Proposition \ref{propLocal}]
\begin{align*}
 j^{-1}\calK^\bullet_\infty\left(1\otimes f,\infty\right)_{\tilde{s}_2}\coloneqq \displaystyle\varprojlim_m \left(j^{-1}\calK^\bullet_\infty\left(1\otimes f,m\right)_{\tilde{s}_2}\right)
 \end{align*}
 is a resolution of $\ov\calL \otimes_R R_\infty$, via a canonical map $\ov\calL \otimes_R R_\infty\to j^{-1}\calK^\bullet_\infty\left(1\otimes f,\infty\right)_{\tilde{s}_2}$ described in Remark~\ref{rem:nuQ}. \label{propLocalQ}
 \end{prop}

\begin{remk}
As stated in \cite[Corollary 4.17]{peters2008mixed}, the adjunction map $\calK_\infty^\bullet\rightarrow Rj_*j^{-1}\calK_\infty^\bullet$ is a quasi-isomorphism. Hence, using Remark \ref{sheafgeneralities} and Proposition \ref{propLocalQ}, for all $1 \leq m \leq \infty$ we get natural isomorphisms
$$
\bH^*\big(X,\calK_\infty^\bullet(1\otimes f,m)_{\tilde{s}_2}\big)\cong \bH^*\big(U,j^{-1}\calK_\infty^\bullet(1\otimes f,m)_{\tilde{s}_2}\big)\cong H^*(U,\ov\calL\otimes_R R_m).
$$
Furthermore for large integers $m$ and $j$, the submodule $\Tors_{R_\infty} H^*(U,\ov\calL\otimes_R R_\infty)$ is shown to be contained in $H^*(U, \ov\calL \otimes_R R_m)$ as the kernel of $\psi_{mj}^*$ applying Lemma \ref{lem:torsion2} to the cochain complexes of these local systems.
\label{remQ}
\end{remk}


\section{Thickened Complex of the Log {de Rham} Complex}
\label{sthick}
It is easy to see that the form $df/f$ is an element of the intersection $\Gamma(U, \Omega^1_U) \cap \ker d $.
In fact, we have a stronger result:
\begin{lem}\label{logarithmicform}
We have 
\[
df/f\in W_1\Gamma\big(X, \Omega^1_X(\log D)\big) \cap F^1\Gamma\big(X, \Omega^1_X(\log D)\big) \cap \ker d.
\]
\end{lem}
\begin{proof}
We will check the stalk condition described in \cite[Section 4.1]{peters2008mixed}.
Let $x \in X$ be given.
Select an open $V$ admitting local holomorphic coordinates $(z_1, \dots, z_n)$ at $x$ in which $D$ is given by $z_1\dots z_\ell = 0$.
By shrinking $V$ if necessary and taking the appropriate affine chart of $\C P^1$, we may assume that $\bar{f}\colon X \rightarrow \C P^1$ determines by restriction a holomorphic map $\bar{f}|_V\colon V \rightarrow \C$, where $\left(\bar{f}|_V\right)^{-1}(0)\subset D\cap V$.
In particular, we can write $\bar{f}|_V = z_1^{a_1} \cdots z_\ell^{a_\ell} \cdot g$ for some $a_1, \dots, a_\ell \geq 0$ and invertible $g\colon V \rightarrow \C^*$.
Application of the Leibniz rule to the following:
\begin{align*}
 \frac{df}{f} = \frac{d(z_1^{a_1} \cdots z_\ell^{a_\ell} \cdot g)}{z_1^{a_1} \cdots z_\ell^{a_\ell} \cdot g}, 
\end{align*}
where we have implicitly further restricted to $U \cap V$,
shows that $df/f$ is of the desired form on $V$.
\end{proof}

Our goal is to conduct a thickening of $\Hdg{X}{D}$ so that the hypotheses of Section \ref{denotationsandassumptionsthickenedcomplexesandmixedhodgecomplexes} are satisfied.
Recall that our pseudo-morphism $\alpha$ is given by the chain:
$$
\begin{tikzcd}
 (j_*\calE^\bullet_U, \tau_{\lc}) \ar[r, "j_*(\id \otimes 1)"] & (j_*(\calE^\bullet_U \otimes \C), \tau_{\lc}) & \ar[l, "\simeq"] (\logdr{X}{D}, \tau_{\lc}) \arrow{r}{\id}[swap]{\simeq} & (\logdr{X}{D}, W_{\lc}) 
\end{tikzcd}
$$
where the central map is the inclusion.
As required by Section \ref{denotationsandassumptionsthickenedcomplexesandmixedhodgecomplexes}, we begin by selecting closed global $1$-forms: $\Im (df/f)$ of $\tau_1j_*\calE^\bullet_U$ (where $\Im$ denotes the imaginary part) and $(1/i)df/f$ of $W_1\logdr{X}{D} \cap F^1\logdr{X}{D}$, allowable by Lemma \ref{logarithmicform}.
We must verify that these two global 1-forms are cohomologous in $\tau_1$:

\begin{lem}\label{imaginarycohomologous}
The forms $\Im(df/f)$ and $(1/i)df/f$ are cohomologous in $\tau_1\Gamma(U, \calE^\bullet_U \otimes \C)$ via:
\[\Im \frac{df}{f}-\frac{1}{i}\frac{df}{f} = d\left(\frac{\Log(|f|)}{-i}\right)\]
where $\Log\colon \R_{>0} \rightarrow \R$ is the real logarithm.
\end{lem}
\begin{proof}
As $\tau_1\Gamma(U, \calE^0_U \otimes \C) = \Gamma(U, \calE^0_U \otimes \C)$ it follows that $\frac{\Log(|f|)}{-i} \in \tau_1\Gamma(U, \calE^\bullet_U \otimes \C)$.
It only remains to verify the presented equality.
Note that:
\[\Im \frac{df}{f}-\frac{1}{i}\frac{df}{f} = -\frac{1}{i}\Re \frac{df}{f}\]
where $\Re$ denotes the real part. 
On the other hand, for any branch $\log$ of the complex logarithm, we have:
\[d\left(\frac{\Log(|f|)}{-i}\right) 
=-\frac{1}{i}d(\Re \log (f)) = -\frac{1}{i}\Re \frac{df}{f}, 
\]
as desired.
\end{proof}

Accounting for Remark \ref{multiplicative}, the assumptions of Section \ref{denotationsandassumptionsthickenedcomplexesandmixedhodgecomplexes} are satisfied (as required by assumption (3) we fix the choice $\Log(|f|)/(-i)$ as witness to $\Im(df/f)$ and $(1/i) df/f$ being cohomologous), and we conclude:

\index{MHS|(}
\begin{thm}\label{logthickenedmhs}\index{mixed Hodge complex!Hodge-de Rham}\index{mixed Hodge complex!Hodge-de Rham@$\calH dg^\bullet(X\,\log D)(\frac{1}{i}\frac{df}{f},m)$}
Consider the $\R$-mixed Hodge complex of sheaves:
\begin{align*}
 \calH dg^\bullet(X\,\log D) = \left((j_*\calE^\bullet_U, \tau_{\lc}), (\logdr{X}{D}, W_{\lc}, F^{\lc}), \alpha \right) 
\end{align*} described in Theorem \ref{logmhs}.
Suppose $m \geq 1$.
Then we have a thickened triple as in Theorem \ref{mhsthickened}:
\begin{align*}
 \calH dg^\bullet&(X\,\log D)\left(\frac{1}{i}\frac{df}{f},m\right)\\
 &\coloneqq \left(\left[j_*\calE^\bullet_U\left(\Im\, \frac{df}{f},m\right), \tau_{\lc}\right], \left[\logdr{X}{D}\left(\frac{1}{i}\frac{df}{f},m\right), W_{\lc}, F^{\lc} \right], \alpha_{\#}\right)
\end{align*}
which is an $\R$-mixed Hodge complex of sheaves on $X$. 
\end{thm}

In this paper, we focus on the torsion part of the Alexander modules, but the next Corollary shows that we have a MHS that involves (a quotient of) the free part as well.
\begin{cor}\label{cor:alex/s^mMHS}
Let $m,j\in \mathbb N$. Recall that we define $R_m= R/(s^m)$, where $s=t-1$. The $\R$-vector spaces $M_m \coloneqq H^*(U; \ov\calL) \otimes_R R_m$ admit natural $\R$-mixed Hodge structures for which the following maps are morphisms of mixed Hodge structures:
\begin{enumerate}
\item The projection $M_{m+j} \twoheadrightarrow M_m$.
\item The map induced by multiplication by $(\log(t))^{j}\colon R_m\hookrightarrow R_{m+j}$ after tensoring, that is $(\log(t))^{j}\colon M_{m} \to M_{m+j}(-j)$.
\item Multiplication by $\log(t)$, $M_m\to M_m(-1)$.
\end{enumerate}
Here $\log(t)$ represents the Taylor series centered at $0$, and $(-j)$, $(-1)$ denote Tate twists.
\end{cor}
\begin{proof}

By the flatness of $R_\infty$ over $R$, we have
\[
M_m 
=H^*(U;\ov\calL)\otimes_R R_m 
\cong H^*(U;\ov\calL)\otimes_{R} R_\infty \otimes_{R_\infty} R_m 
\cong H^*(U;\ov\calL\otimes_{R} R_\infty) \otimes_{R_\infty} R_m .
\]
For $m \geq 1$, it turns out that the translated mixed Hodge complex of sheaves
\[\calH dg^\bullet(X\,\log D)\left(\frac{1}{i}\frac{df}{f},m\right)[1]\]
is better suited for comparison to known mixed Hodge structures, so we use it to endow $M_m$ with a mixed Hodge structure. This will enable us to state Theorems~\ref{geoIntro} and \ref{comp} without involving Tate twists. This translated $\R$-mixed Hodge complex of sheaves endows, by Remark \ref{translatedlemmas}, $\R$-mixed Hodge structures on the hypercohomology:
\[
\mathbb{H}^*(U, \calE^\bullet_U(\Im df/f, m)[1])
\]
which we fix for the rest of the proof. We will also use Lemma~\ref{inducedmhsmaps} for the maps between these spaces.

Let $\tilde s= \exp(2\pi s)-1\in R_\infty$. Then $\calE_U^\bullet\left(\Im\, df/f,m\right)_{\tilde s}$ is a complex of soft sheaves for all $m$, since it is a complex of $\calE_U^0$-modules (\cite[Proposition 2.1.8]{dimca2004sheaves}). Hence, $\calE_U^\bullet\left(\Im\, df/f,m\right)_{\tilde s}$ is $\Gamma$-acyclic and therefore, Proposition \ref{propLocal} provides isomorphisms of $R_\infty$-modules
$$H^*\left(U; \ov\calL\otimes_R R_\infty \right) \cong H^*\Gamma\left(U, \calE^\bullet_U(\Im\, df/f, \infty)\right)_{\tilde s} \cong \mathbb{H}^{*-1}(U, \calE^\bullet_U(\Im df/f, \infty)[1])_{\tilde{s}}.$$

Since every complex algebraic variety has the homotopy type of a finite CW complex \cite[p. 27]{dimca1992hypersurfaces}, the above isomorphism yields that $H^*\Gamma\left(U, \calE^\bullet_U(\Im\, df/f, \infty)\right)$ is a finitely generated $R_\infty$-module. In particular, we can apply Corollary~\ref{torsion}: $M_m$ is isomorphic to the image of $\phi_{m'm}[1]^*$ for $m'\gg 0$. We will use that all the maps in Lemma~\ref{inducedmhsmaps} are MHS morphisms (rather, their translations as in Remark~\ref{translatedlemmas}). Being the image of a MHS morphism, $M_m$ is a sub-MHS of $\mathbb{H}^{*-1}(U, \calE^\bullet_U(\Im df/f, m)[1])_{\tilde{s}}$. The fact that this MHS is independent of $m'$ and $m$ is a consequence of Remark~\ref{torsionrelations}.

Finally, we show that the maps in the statement are MHS morphisms. Note that multiplication by $\frac{\log(t)}{2\pi}$ in $\mathbb{H}^{*-1}\big(U, \calE^\bullet_U(\Im df/f, m)[1]\big)_{\tilde s}$ corresponds to multiplication by $s$ in $\mathbb{H}^{*-1}\big(U, \calE^\bullet_U(\Im df/f, m)[1]\big)$. Using Remark~\ref{torsionrelations}: the map (1) is induced by $\phi_{jm}[1]^*$, the map (2) is induced by $(2\pi)^j \psi_{mj}[1]^*$ and (3) is multiplication by $2\pi s$ (taking into account the different module structures). Since all three maps induce MHS morphisms in cohomology by Lemma~\ref{inducedmhsmaps} and Remark~\ref{translatedlemmas}, we are done.
\end{proof}

\begin{cor}\label{torsionmhs}
Suppose that the action of $t$ on $\Tors_{R}H^*(U; \ov\calL)$ is unipotent. The $\R$-vector spaces $\Tors_{R}H^*(U; \ov\calL)$ admit natural $\R$-mixed Hodge structures for which multiplication by $\log(t)$ determines a morphism of mixed Hodge structures into the $-1$st Tate twist.
\end{cor}

\begin{proof}
Since every complex algebraic variety has the homotopy type of a finite CW complex \cite[p. 27]{dimca1992hypersurfaces}, the modules $H^*(U; \ov\calL)$ are finitely generated over $R$. In particular, there is an $m\ge 0$ such that $s^m$ annihilates $\Tors_{R}H^*(U; \ov\calL)$. For such an $m$, $\Tors_{R}H^*(U; \ov\calL)$ is canonically isomorphic to the image of the following map, since $s$ is nilpotent by hypothesis:
\begin{equation}
\label{eqn:imageTorsion}
\Tors_{s^m} \left(H^*(U; \ov\calL) \otimes_{R} R_{2m}\right) \hookrightarrow H^*(U; \ov\calL) \otimes_{R} R_{2m} \twoheadrightarrow H^*(U; \ov\calL) \otimes_{R} R_{m}.
\end{equation}
Since $s$ and $\log(t)$ differ by a unit, $\Tors_{s^m} H^*(U; \ov\calL)$ is the kernel of the multiplication by $(\log(t))^m$, which is a MHS morphism by Corollary~\ref{cor:alex/s^mMHS}, part (3). Hence, the inclusion in equation (\ref{eqn:imageTorsion}) is a MHS morphism. The second map in equation (\ref{eqn:imageTorsion}) is also a MHS morphism by Corollary~\ref{cor:alex/s^mMHS}, part (1). Therefore, the canonical isomorphism between $\Tors_{R}H^*(U; \ov\calL)$ and the image of the map in equation (\ref{eqn:imageTorsion}) endows $\Tors_{R}H^*(U; \ov\calL)$ with a MHS such that multiplication by $\log(t)$ is a morphism of MHS, as it is the restriction of a MHS morphism. By Corollary~\ref{cor:alex/s^mMHS}, part (1), this MHS is independent of $m$, provided that $s^m$ annihilates $\Tors_{R}H^*(U; \ov\calL)$.
\end{proof}

\begin{cor}\label{alexandermhs}
The $\R$-vector spaces $\mathrm{Tors}_R\, H^*(U;\ov\calL)$ admit canonical $\R$-mixed Hodge structures for which multiplication by $\log (t^N)$ is a morphism of mixed Hodge structures into the $-1$st Tate twist. Here $\log$ is the Taylor series centered at $1$, and $N$ is chosen so that $t^N$ acts unipotently on $\mathrm{Tors}_R\, H^*(U;\ov\calL)$.
\end{cor}
\begin{proof}
If $1$ is not the only eigenvalue of the action of $t$ on $H^*(U;\ov\calL\otimes \C)$, pick $N$ such that $\lambda^N=1$ for all $\lambda$ eigenvalue of the action of $t$ on $H^*(U;\ov\calL\otimes\C)$. We can reduce this case to the one where $1$ is the only eigenvalue by Lemma \ref{lemLocal} and Remark \ref{remEigenvalue}, obtaining a $k$-linear isomorphism $$\mathrm{Tors}_R\, H^*(U;\ov\calL) \cong \Tors_{R(N)} H^*(U_N;\ov\calL_N).$$ We give the left hand side the MHS of the right hand side, which we constructed in Corollary \ref{torsionmhs}. As we will see in Theorem \ref{indN} below, this construction is independent of the choice of suitable $N$.
\end{proof}
\index{MHS|)}
\begin{thm}[Independence of the compactification]\index{compactification}
Suppose that the action of $t$ on $\Tors_{R}H^*(U;\ov\calL)$ is unipotent. The mixed Hodge structure on $\Tors_R H^*(U;\ov\calL)$ obtained in Corollary \ref{torsionmhs} is independent of the compactification $X$ of $U$ such that $D=X\setminus U$ is a simple normal crossing divisor, which we used to construct the mixed Hodge structure.
\label{indcompactification}
\end{thm}

\begin{proof}
Suppose that $X$ and $Y$ are two such compactifications, called \emph{good} compactifications. Let $Z$ be a resolution of the closure of the diagonal $\Delta$ of $U\times U$ inside of $X\times Y$ such that $Z$ is a good compactification of $U$ as well. The two projections $p_X\colon Z\rightarrow X$ and $p_Y\colon Z\rightarrow Y$ induce the identity on $U$. It suffices to show that the mixed Hodge structure on $\Tors_R H^*(U;\ov\calL)$ obtained using $Z$ is the same as the one obtained using $X$.

Let $E=Z\setminus U$. By \cite[Lemma 4.12]{peters2008mixed} (recalling Remark \ref{directim}), we have a canonical morphism
$$
\calH dg^\bullet (X\log D)\rightarrow R(p_X)_*\calH dg^\bullet (Z\log E).
$$
By construction, this morphism factors through $(p_X)_*\calH dg^\bullet (Z\log E)$ as maps of triples. Since $$(p_X)_*\left(\calH dg^\bullet (Z\log E)\left(\frac{1}{i}\frac{df}{f},m\right)\right)=\left((p_X)_*\calH dg^\bullet (Z\log E)\right)\left(\frac{1}{i}\frac{df}{f},m\right),$$Lemma \ref{inducedmap} tells us that this map of triples induces a canonical morphism between their corresponding thickenings
$$
\calH dg^\bullet (X\log D)\left(\frac{1}{i}\frac{df}{f},m\right)\rightarrow (p_X)_*\left(\calH dg^\bullet (Z\log E)\left(\frac{1}{i}\frac{df}{f},m\right)\right).
$$
Composing with $(p_X)_*$ of the canonical map into the Godement resolution, we get a morphism of mixed Hodge complexes of sheaves
$$
\calH dg^\bullet (X\log D)\left(\frac{1}{i}\frac{df}{f},m\right)\rightarrow R(p_X)_*\left(\calH dg^\bullet (Z\log E)\left(\frac{1}{i}\frac{df}{f},m\right)\right).
$$

Let $j_X:U\hookrightarrow X$ and $j_Z:U\hookrightarrow Z$ be the inclusions, and note that $p_X\circ j_Z=j_X$. The $[1]$ translation of the map of mixed Hodge complexes above induces isomorphisms between the $\R$-MHS on $\bH^*(U,\calE_U^\bullet(\Im df/f,m)[1])$ obtained by the compactification $X$ and the $\R$-MHS on $\bH^*(U,\calE_U^\bullet(\Im df/f,m)[1])$ obtained by the compactification $Z$.
Hence, the MHS obtained on $\Tors_{R}H^*(U;\ov\calL)$ using these MHS on $\bH^{*-1}(U,\calE_U^\bullet(\Im df/f,m)[1])$ for large enough $m$ is independent of the compactification.
\end{proof}

\begin{thm}[Independence of $N$]\index{cover!finite cyclic}
\label{indN}
The mixed Hodge structure on the module \linebreak $\Tors_R H^*(U,\ov\calL)$ obtained in the proof of Corollary \ref{alexandermhs} is independent of the choice of cover $U_N$ used to construct it.
\end{thm}

\begin{proof}

It suffices to show that if $1$ is the only eigenvalue of the action of $t$ on $H^*(U;\ov\calL\otimes \C)$, then the mixed Hodge structure defined on $\Tors_R H^*(U;\ov\calL)$ is the same as the mixed Hodge structure obtained on $\Tors_{R(N)} H^*(U_N;\ov\calL_N)$, where $N\geq 1$ and $p:U_N\rightarrow U$ is a degree $N$ cover as in Section \ref{sscover}. We use the same notation from the discussion preceding Lemma \ref{lemLocal}. In particular, we defined an isomorphism $\cL\cong p_*\cL_N$ of sheaves of $R(N)$-modules, which we will refer to as ``the canonical isomorphism'' throughout the proof. By Lemma~\ref{233}, this is the obvious isomorphism appearing in this situation, since it induces in cohomology the map coming from the isomorphism of infinite cyclic covers.

As explained in \cite[Section 4.5.1]{peters2008mixed}, it is possible to find smooth compactifications $X$ of $U$ and $Y$ of $U_N$ such that $D\coloneqq X\backslash U$ and $E\coloneqq Y\backslash U_N$ are simple normal crossing divisors, and such that $p:U_N\rightarrow U$ extends to $p: Y\rightarrow X$, and $p^{-1}(D)=E$ by construction. By Remark~\ref{directim}, applying $Rp_*$ to $\calH dg^\bullet (Y\log E)\left(\frac{1}{i}\frac{df_N}{f_N},m\right)$ yields a mixed Hodge complex of sheaves on $X$.

We need to define a map of $\R$-mixed Hodge complexes of sheaves as below. In Corollary~\ref{cor:alex/s^mMHS} we used a translation of these complexes to define the mixed Hodge structures, which according to Remark~\ref{transvstate} will result in a Tate twist in their cohomology groups. We will omit the translation in this proof for brevity.
\[
\wt p:\Hdg{X}{D}\left(\frac{1}{i}\frac{df}{f},m\right)
 \to
Rp_*\left(
\calH dg^\bullet (Y\log E)\left(\frac{1}{i}\frac{df_N}{f_N},m\right)
\right).\]
Note that the twisted de Rham complexes on the left use the ring $R$, and the twisted de Rham complexes on the right use the ring $R(N)=k[t^{\pm N}]$. We recall the notation $R(N)_m=R(N)/((s_N)^m)\subset R$, with $s_N=t^N-1$. We define $\wt p$ as a composition:
\[
\wt p:\Hdg{X}{D}\left(\frac{1}{i}\frac{df}{f},m\right)
 \xrightarrow{\wh p}
p_*\left(
\calH dg^\bullet (Y\log E)\left(\frac{1}{i}\frac{df_N}{f_N},m\right)
\right) \to\]\[\to
Rp_*\left(
\calH dg^\bullet (Y\log E)\left(\frac{1}{i}\frac{df_N}{f_N},m\right)
\right).\]
The second arrow is the natural transformation $p_*\to Rp_*$ (it comes from the map including a complex of sheaves into its Godement resolution and then applying $p_*$). For our purposes, we do not need the object in the middle to be a mixed Hodge complex, only the first and last. For $\omega\in \cE^*_U $ defined on any open subset of $U$, we let:
\[
\wh p_{\R} ( \omega\otimes s^k) = p^* \omega \otimes \frac{(s_N)^k}{N^k}.
\]
where $p^* \omega \otimes \frac{(s_N)^k}{N^k}\in\cE^*_{U_N}\left(\Im\frac{df_N}{f_N},m\right)$. This commutes with the differential:
\begin{align*}
d_{m}(\wh p_{\R}(\omega\otimes s^k) ) &=
 d_{m}\left(
 p^* \omega \otimes \frac{(s_N)^k}{N^k}\right) =
 dp^* \omega \otimes \frac{(s_N)^k}{N^k} + p^*\omega \wedge \darg \otimes \frac{(s_N)^{k+1}}{N^k};\\
\wh p_{\R} \circ d_{m}(\omega\otimes s^k) &=
 \wh p_{\R} \left(
d\omega\otimes s^k + \omega \wedge \Im\frac{df}{f}\otimes s^{k+1}
\right)
\overset{p^*\frac{df}{f}=N\frac{df_N}{f_N}}{=}\\
&=
 dp^*\omega\otimes \frac{(s_N)^k}{N^k} + p^*\omega \wedge \Im\frac{df_N}{f_N}N\otimes \frac{(s_N)^{k+1}}{N^{k+1}}.
\end{align*}
And similarly we can define $\wh p_{\C}$ for the complex part. Using Remark~\ref{directim}, the filtrations on $\calH dg^\bullet (Y\log E)\left(\frac{1}{i}\frac{df_N}{f_N},m\right)$ induce filtrations on both its underived and derived pushforwards, and the map between them preserves these filtrations. Then $\wt p$ preserves the filtrations, given that $\wh p$ does, which is straightforward to verify.

The morphism of pseudomorphisms is the natural one, constructed using the pullback of forms similarly to the definition of $\wt p_{\R}$. Hence, $\wt p$ is a morphism of mixed Hodge complexes of sheaves.

Taking cohomology and using Remark~\ref{rem:nu}, we obtain a morphism of mixed Hodge structures $H^i(U;\ov\cL\otimes_R R_m)\to H^i(U_N;\ov\cL_N\otimes_{R(N)} R(N)_m)$. We want to show that this morphism is related to the canonical isomorphism $\ov\cL\cong p_*\ov\cL_N$ from Lemma~\ref{233}.

Since $\cE_{U_N}^\bullet(m,\Im df_N/f_N)$ is a complex of soft sheaves, the map $$p_*\cE_{U_N}^\bullet(m,\Im df_N/f_N)\rightarrow Rp_*\cE_{U_N}^\bullet(m,\Im df_N/f_N)$$ is a quasi-isomorphism. Since $\wt p$ is a morphism of mixed Hodge complexes, this means that $\widehat p_\R$ induces the same morphism of MHS $H^i(U;\ov\cL\otimes_R R_m)\to H^i(U_N;\ov\cL_N\otimes_{R(N)} R(N)_m)$ as $\wt p_\R$. We work with $\widehat p_\R$ from now on.

The map $\widehat p_\R$ is split injective. A left inverse is $\widehat p_\R':p_*\cE_{U_N}^\bullet(m,\Im df_N/f_N)\to \cE_U^\bullet(m,\Im df/f)$ defined on open sets as follows: on an open set $V$,
\[
\Gamma(V;p_*\cE_{U_N}^\bullet(m,\Im df_N/f_N)) = \Gamma(p^{-1}(V);\cE_{U_N}^\bullet(m,\Im df_N/f_N)).
\] Considering $\alpha\otimes (s_N)^{a}\in \cE_{U_N}^\bullet(m,\Im df_N/f_N)$ defined over $p^{-1}(V)$, we let $t$ act as the generator of the deck group of $p:U_N\to U$, and define
\[
\ov \alpha = \frac{1}{N}\sum_{k=0}^{N-1} (t^k)^*\alpha.
\]
By construction, $\ov \alpha$ is $t$-invariant. Therefore, we can define $\wh p'_{\R}\alpha$ as the unique form such that
$
p^*\wh p'_{\R}(\alpha) = \ov \alpha$. We extend $\wt p'_\R$ to the thickening by letting $\wt p'_\R (\alpha\otimes (s_N)^a) = (\wt p'_\R \alpha)\otimes N^as^a$. Direct computation shows that $\wh p'_\R \circ \wh p_\R=\Id$. We consider the following commutative diagram, in which the horizontal arrows compose to the identity.
\begin{equation}\label{eq:indepNAttempt2}
{
\begin{tikzcd}[column sep = 4em, row sep = 1.3em]
\cE_U^0\otimes_\R R_m\arrow[r,"\wh p_{\R}"] &
p_*\cE_{U_N}^0\otimes_{\R} R(N)_m \arrow[r,"\wh p_{\R}'"] &
\cE_U^0\otimes_\R R_m
\\
\varprojlim_m\cE_U^0\otimes_\R R_m\arrow[r,"\wh p_{\R}"]\arrow[u] &
\varprojlim_m p_*\cE_{U_N}^0\otimes_{\R} R(N)_m \arrow[r,"\wh p_{\R}'"]\arrow[u] &
\varprojlim_m \cE_U^0\otimes_\R R_m \arrow[u]
\\
\ov\cL_U \otimes_R R_\infty \arrow[r,"\sigma" ]\arrow[u,"\nu"] &
p_*\ov\cL_N \otimes_{R(N)} R(N)_\infty\arrow[u,"p_*\nu_N"]\arrow[r,"\sigma'"]&
\ov\cL_U \otimes_R R_\infty \arrow[u,"\nu"].
\end{tikzcd}
}
\end{equation}
The top row of vertical arrows is induced by the limit. $\nu$ and $\nu_N$ are defined as (the inverse limit of the maps) in Remark~\ref{rem:nu}. 
Since $\nu$ and $\nu_N$ are isomorphisms onto the kernel of the differential (by Proposition~\ref{propLocal}), $\sigma$ and $\sigma'$ are uniquely determined as the restrictions of $\wh p_\R$ and $\wh p_\R'$ in order to make the diagram commute.

One can check that, switching the $R$ and $R(N)$-module structure respectively, the map
$$
\wh p_{\R}:\cE_U^\bullet(\Im df/f,m)_{\wt s}\rightarrow p_*\cE_{U_N}^\bullet(\Im df_N/f_N,m)_{\wt{s_N}}
$$
is $R(N)_\infty$-linear, where $\wt s=\exp(2\pi s)-1$ and $\wt{s_N}=\exp(2\pi s_N)-1$. Since $\nu$ and $\nu_N$ are $R_\infty$ and $R(N)_\infty$-linear respectively after these changes in the $R$ and $R(N)$-module structures on the de Rham complexes, we get that $\sigma$ is $R(N)_{\infty}$-linear as well.

We will need a formula for $\sigma'$. First, recall our notation for the sections of these sheaves. Let $x\in U$. The stalk at $x$ of $\ov\cL$ is generated over $k$ by elements of the form $\ov\delta_{(x,z)}$ for some $(x,z)\in U^f$. The stalk of $p_*\ov\cL_N \otimes_{R(N)} R(N)_m$ at $x$ is the direct sum of the stalks of $p^{-1}(x)$, where the stalk at $(x,e^{z/N})$ is generated by sections of the form $\ov\delta_{(x,e^{z/N},z/N + 2\pi i k)}$ with $k\in \Z$. Each element $\ov\delta_{(x,e^{y},y)}$ must be interpreted as a section of $\ov\cL_N$ around the point $(x,e^{y})\in U_N\subset U\times \C^*$, as in diagram (\ref{eq:UN}).

We claim that $\sigma'$ is the map defined on stalks as the $\R$-linear map satisfying
\[
\sigma'\ov\delta_{(x,e^{z/N},z/N)} = \frac{1}{N} \ov\delta_{(x,z)}.
\]
A straightforward computation shows that this map on stalks gives a well-defined map of local systems of $R(N)$-modules. Another straightforward computation shows that this is the correct formula for $\sigma'$, that is, we gave the formula that makes the bottom right hand square commute.

Now we take hypercohomology of diagram (\ref{eq:indepNAttempt2}), and we restrict to the torsion part of the left half. We obtain the following diagram. Recall that $\sigma$ is $R(N)_{\infty}$-linear, and that $\Tors_{R_\infty} = \Tors_{R(N)_\infty}$.
\[
\begin{tikzcd}
H^j\Gamma(U;\cE_U^\bullet\otimes_\R R_m) \arrow[r,"\wh p_{\R}"] &
H^j\Gamma(U_N;\cE_{U_N}^\bullet\otimes_{\R} R(N)_m ) \\
\Tors_{R_\infty} H^j(U;\ov\cL \otimes_R R_\infty) \arrow[r,"\sigma" ,hookrightarrow ]\arrow[u,"\nu",hookrightarrow] &
\Tors_{R(N)_\infty}H^j(U_N;\ov\cL_{N} \otimes_{R(N)} R(N)_\infty)\arrow[u,"p_*\nu_N",hookrightarrow].
\end{tikzcd}
\]
Suppose $m$ is large enough. The top arrow is a MHS morphism because it comes from a mixed Hodge complex morphism. The vertical arrows are MHS morphisms by definition of the MHS morphism on the domains (in Corollary~\ref{alexandermhs}). Also, they are injective. Therefore, $\sigma$ is a MHS morphism and it is injective, since it comes from a split injective morphism of sheaves. Also, note that by flatness and the isomorphism $p_*\ov\cL_N\cong \ov\cL$:
\[
\Tors_{R_\infty} H^j(U;\ov\cL \otimes_R R_\infty) \cong
\Tors_{R} H^j(U;\ov\cL ) \cong
\Tors_{R(N)} H^j(U;\ov\cL ) \cong\]\[\cong
\Tors_{R(N)} H^j(U_N;\ov\cL_N ) \cong
\Tors_{R(N)_\infty} H^j(U_N;\ov\cL_N \otimes_{R(N)} R(N)_\infty) .
\]
In particular, these two spaces have the same dimension. So $\sigma$, which is an injection, must be an isomorphism of MHS. Since $\sigma'$ is its right inverse, it is also an isomorphism of MHS. Now we just need to observe that the following diagram clearly commutes (up to multiplication by $N$), where the top horizontal arrow is the canonical map as we've defined it in Lemma~\ref{233}:
\[
\begin{tikzcd}
p_*\ov\cL_N \arrow[r,"\sim"]\arrow[d] &
\ov\cL\arrow[d] \\
p_*\ov\cL_N\otimes_{R(N)} R(N)_\infty \arrow[r,"\sigma'"] &
\ov\cL\otimes_R R_\infty.
\end{tikzcd}
\]
This shows that the canonical isomorphism $p_*\ov\cL_N\cong \ov\cL$ coincides with $\sigma'$ up to multiplication by an integer constant, so the canonical isomorphism induces a MHS isomorphism:
$$
\Tors_R H^*(U;\ov\calL)\cong \Tors_{R(N)} H^*(U_N;\ov\calL_N).
$$
This concludes the proof.
\end{proof}

\begin{thm}[Functoriality of the mixed Hodge structure.]\label{functorial}\index{functoriality}
Let $U_1$ and $U_2$ be smooth connected complex algebraic varieties, with algebraic maps $f_i:U_i\rightarrow \C^*$ such that $f_i$ induces an epimorphism in fundamental groups for $i=1,2$, and assume that there exists an algebraic map $g:U_1\rightarrow U_2$ that makes the following diagram commutative.

\begin{center}
\begin{tikzcd}[row sep = 1.4em]
U_1\arrow[rd,"f_1"']\arrow[rr, "g"] & \ & U_2.\arrow[ld,"f_2"] \\
\ & \C^* & \ 
\end{tikzcd}
\end{center}
Let $\ov\calL_i$ be the local system of $R$-modules induced by $f_i$ for $i=1,2$, where $k=\R$. Then, the natural map $\Tors_R H^*(U_2,\ov\calL_2)\rightarrow \Tors_R H^*(U_1,\ov\calL_1)$ induced by $g$ is a morphism of mixed Hodge structures.
\end{thm}

\begin{proof}
Note that $\ov\calL_1=g^{*}\ov\calL_2$. Let $N\in\N$ such that the action of $t^N$ on $\Tors_R H^*(U_i;\ov\calL_i)$ is unipotent, for $i=1,2$. The map $g$ lifts to a map $g_N:(U_1)_N\rightarrow (U_2)_N$ such that $(\ov\calL_1)_N=g_N^{*}(\ov\calL_2)_N$.

As explained in \cite[Section 4.5.1]{peters2008mixed}, it is possible to find smooth compactifications $X_i$ of $(U_i)_N$ such that $D_i\coloneqq X_i\backslash (U_i)_N$ is a simple normal crossing divisor for $i=1,2$ and such that $g_N$ extends to $g_N: X_1\rightarrow X_2$. By \cite[Lemma 4.12]{peters2008mixed}, there is a canonical morphism of mixed Hodge complexes of sheaves
$$
(g_N)^*:\Hdg{X_2}{D_2}\rightarrow R{g_N}_*\Hdg{X_1}{D_1}.
$$
By construction, this morphism factors trough $(g_N)_*\Hdg{X_1}{D_1}$ as maps of triples. Since \small
$$(g_N)_*\left(\Hdg{X_1}{D_1}\left(\frac{1}{i}\frac{d(f_1)_N}{(f_1)_N},m\right)\right)=\left((g_N)_*\Hdg{X_1}{D_1}\right)\left(\frac{1}{i}\frac{d(f_1)_N}{(f_1)_N},m\right),$$\normalsize
and using that $(g_N)^* (f_2)_N=(f_1)_N$, Lemma \ref{inducedmap} tells us that this map of triples induces a canonical morphism between their corresponding thickenings
{\small
\[ \Hdg{X_2}{D_2}\left(\frac{1}{i}\frac{d(f_2)_N}{(f_2)_N},m\right)_{\tilde s_N}\rightarrow {g_N}_*\left(\Hdg{X_1}{D_1}\left(\frac{1}{i}\frac{d(f_1)_N}{(f_1)_N},m\right)\right)_{\tilde s_N}. \]
}

Composing with $(g_N)_*$ of the canonical map into the Godement resolution (and twisting the $R(N)_m$-module structure by $\tilde s_N$), we get a morphism of mixed Hodge complexes of sheaves
{\small
\[ \Hdg{X_2}{D_2}\left(\frac{1}{i}\frac{d(f_2)_N}{(f_2)_N},m\right)_{\tilde s_N}\rightarrow R{g_N}_*\left(\Hdg{X_1}{D_1}\left(\frac{1}{i}\frac{d(f_1)_N}{(f_1)_N},m\right)\right)_{\tilde s_N}.\]
}

Therefore, after $[1]$ translating the above morphism, we see that the canonical morphism induced by $g$:
$$
\Tors_R H^*(U_2,\ov\calL_2)\rightarrow \Tors_R H^*(U_1,\ov\calL_1)
$$
is a morphism of $\R$-mixed Hodge structures.
\end{proof}

\begin{thm}[$\Q$-MHS]\label{Qalexandermhs}\index{mixed Hodge complex!rational}
The mixed Hodge structure on $\Tors_R H^i(U;\ov\cL)$ defined for $k=\R$ in Corollary~\ref{alexandermhs} comes from a (necessarily unique) mixed Hodge structure defined for $k=\Q$.
\end{thm}
\begin{proof}
In this proof, let $k=\Q$ and therefore $R=\Q[t^{\pm 1}]$. We start by proving the theorem in the case where the action of $t$ on $H^*(U;\ov\cL)$ is unipotent. Let $\tilde{s}_2\coloneqq \exp(s)-1=\Sigma_{n=1}^{\infty}\frac{s^n}{n!}$. Similarly one obtains a $\Q$-mixed Hodge complex: recalling the notation of Section \ref{rationalmixedhodgecomplexonsmoothvarieties} for $m \geq 1$, we select the thickened triple:
\index{mixed Hodge complex!HodgeT@$\Hdg{X}{D}(\frac{1}{2\pi i}\frac{df}{f},m)$}\index{K_infty@$\calK^\bullet_\infty$}
\begin{align*}
& \Hdg{X}{D}\left(\frac{1}{2\pi i}\frac{df}{f},m\right)_{\tilde{s}_2}\\
& \coloneqq \left(\left[\calK^\bullet_\infty\left(1 \otimes f,m\right),  \tilde{W}_{\lc}\right], \left[\logdr{X}{D}\left(\frac{1}{2\pi i}\frac{df}{f},m\right), W_{\lc}, F^{\lc}\right], \varphi_{\infty\#}\right)_{\tilde{s}_2}
\end{align*}
which also satisfies the hypotheses of Section \ref{denotationsandassumptionsthickenedcomplexesandmixedhodgecomplexes}, therefore is a $\Q$-mixed Hodge complex of sheaves on $X$.

Using Remark \ref{remQ}, we see that the translated mixed Hodge complex \[\Hdg{X}{D}\left(\frac{1}{2\pi i}\frac{df}{f},m\right)_{\tilde{s}_2}[1]\] induces $\Q$-mixed Hodge structures on $$H^*(U;\ov\calL\otimes_R R_m) \cong \mathbb{H}^{*-1}(X, \calK^\bullet_\infty(1 \otimes f,m)[1])_{\tilde{s}_2}$$ for which multiplication by $\log t$ is a morphism of mixed Hodge structures into the $-1$st Tate twist, because it is induced by $S_m[1]$ (multiplication by $s$) on the complex level (recall Lemma \ref{inducedmhsmaps}). For large enough $m$, this $\Q$-MHS on $H^*(U;\ov\calL\otimes_R R_m)$ induces a $\Q$-MHS on $\mathrm{Tors}_{R}H^*(U;\ov\calL)$ via the map of sheaves of $R$-modules $\ov\calL\rightarrow \ov\calL\otimes_R R_m$, just like we had in the case of $\R$ coefficients (as in the proof of Corollary~\ref{torsionmhs}).

Let $\ov\cL_\R \coloneqq \ov\cL \otimes_{\Q} \R$, seen as a local system of $\R[t^{\pm 1}]$-modules. Our goal is to see that the $\R$-MHS on $H^*(U;\ov\calL\otimes_R R_m)\otimes_{\Q} \R\cong H^*(U;\ov\calL_\R\otimes_{\R[t^{\pm 1}]}\R[t^{\pm 1}]/(s^m))$ induced by the $\Q$-MHS on $H^*(U;\ov\calL\otimes_R R_m)$ coincides with the $\R$-MHS on $H^*(U;\ov\calL_\R\otimes_{\R[t^{\pm 1}]}\R[t^{\pm 1}]/(s^m))$ which is obtained by using the real mixed Hodge complex of sheaves $\Hdg{X}{D}\left(\frac{1}{i}\frac{df}{f},m\right)_{\tilde{s}}[1]$. Indeed, this will imply our claim about the $\Q$ and $\R$-MHS on $\Tors_R H^i(U;\ov\cL)$.

Using the definitions of the filtrations on the thickened complex (Section \ref{thickenedcomplexesandfiltrations}), it is straightforward to check that the following map is an isomorphism of bi-filtered complexes
\begin{equation}
\f{G_m}{\left(\logdr{X}{D}\left(\frac{1}{2\pi i}\frac{df}{f},m\right), W_{\lc}, F^{\lc}\right)}{ \left(\logdr{X}{D}\left(\frac{1}{ i}\frac{df}{f},m\right), W_{\lc}, F^{\lc}\right)}{\omega\otimes s^j}{(2\pi)^j\omega\otimes s^j}
\label{eqnFilter}
\end{equation}
 for all $m\geq 1$. Recalling the definitions of the mixed Hodge complexes of sheaves involved, our claim follows if we show that the following two maps are the same in the derived category. The first map is the composition
\begin{align*}
(Rj_*\ov\calL\otimes_R R_m)\otimes_{\Q}\C&\xrightarrow{Rj_*\nu_{\Q}\otimes_{\Q}\C} Rj_*j^{-1}\calK^{\bullet}_{\infty}(1\otimes f,m)\otimes_{\Q} \C\xleftarrow{\cong} \calK^{\bullet}_{\infty}(1\otimes f,m)\otimes_{\Q} \C\\
&\xrightarrow{\varphi_{\infty}\otimes 1} \Omega_X^\bullet(\log D)\left(\frac{1}{2\pi i}\frac{df}{f},m\right)\xrightarrow{G_m} \logdr{X}{D}\left(\frac{1}{ i}\frac{df}{f},m\right),
\end{align*}
which (after a $[1]$ translation) endows the cohomology of $\ov\calL\otimes_R R_m$ with a $\Q$-MHS. $\varphi_\infty$ is defined in Section~\ref{rationalmixedhodgecomplexonsmoothvarieties} and it induces a map on thickenings via Lemma~\ref{inducedmap}; and $\nu_\Q$ is defined in Remark~\ref{rem:nuQ}. The second morphism is the composition
\begin{multline*}
(j_*\ov\calL_\R\otimes_{\R[t^{\pm 1}]} \R[t^{\pm 1}]/(s^m))\otimes_{\R}\C
\xrightarrow{j_*\nu \otimes_{\R}\C} j_*\calE^{\bullet}_{U}\left(\Im\frac{df}{f},m\right)\otimes_{\R} \C\\ \xrightarrow{\exp\left(\frac{\Log(|f|)}{-i}\right)\wedge} j_*\calE^{\bullet}_{U}\left(\frac{1}{ i}\frac{df}{f},m\right)\otimes_{\R} \C
\xleftarrow{\cong} \logdr{X}{D}\left(\frac{1}{ i}\frac{df}{f},m\right)
\end{multline*}
which (after a $[1]$ translation) endows the cohomology of $\ov\calL\otimes_{\R[t^{\pm 1}]} \R[t^{\pm 1}]/(s^m)$ with an $\R$-MHS. Note that the two domains are canonically identified. Instead of proving that those two maps are the same, we will prove that the (post) composition of them with the quasi-isomorphism given by inclusion $$\logdr{X}{D}\left(\frac{1}{ i}\frac{df}{f},m\right)\xrightarrow{\cong}j_*\calE^{\bullet}_{U}\left(\frac{1}{ i}\frac{df}{f},m\right)\otimes_{\R} \C$$ give us the same map in the derived category. Note that we are now dealing with two maps from $Rj_*(\ov\calL\otimes_R R_m)\otimes_{\Q}\C$ to $Rj_*\calE^{\bullet}_{U}\left(\frac{1}{ i}\frac{df}{f},m\right)$ (the sheaves involved are $j_*$-acyclic). Since $Rj_*$ is fully faithful, it suffices to check that these two maps of sheaves of $\C$-vector spaces are the same on stalks at points in $U$. 
\begin{multline*}
\ov\calL\otimes_R R_m\otimes_{\Q}\C\xrightarrow{\nu_{\Q}\otimes_{\Q}\C} j^{-1}\calK^{\bullet}_{\infty}(1\otimes f,m)\otimes_{\Q} \C
\xrightarrow{\varphi_{\infty}\otimes 1} j^{-1}\Omega_X^\bullet(\log D)\left(\frac{1}{2\pi i}\frac{df}{f},m\right)\\
\xrightarrow{G_m} j^{-1}\logdr{X}{D}\left(\frac{1}{ i}\frac{df}{f},m\right)\xrightarrow{\cong} \calE^{\bullet}_{U}\left(\frac{1}{ i}\frac{df}{f},m\right)\otimes_{\R} \C ,
\end{multline*}
\begin{multline*}
\ov\calL\otimes_{\R[t^{\pm 1}]} \R[t^{\pm 1}]/(s^m))\otimes_{\R}\C \xrightarrow{\nu \otimes_{\R}\C} \calE^{\bullet}_{U}\left(\Im\frac{df}{f},m\right)\otimes_{\R} \C \\ 
\xrightarrow{\exp\left(\frac{\Log(|f|)}{-i}\otimes s\right)\wedge} \calE^{\bullet}_{U}\left(\frac{1}{ i}\frac{df}{f},m\right)\otimes_{\R} \C
\end{multline*}
It is straightforward to check that the image of $\ov\delta_{x'}$ in $\calE^{\bullet}_{U}\left(\frac{1}{ i}\frac{df}{f},m\right)\otimes_{\R}\C$ is via both maps given by $\exp(-\frac{ f_{\infty}\circ \iota}{i}\otimes s)$, where $x\in U$, $x'\in\pi^{-1}(x)$ and $\iota$ is a section of $\pi$ defined on a simply connected neighborhood of $x$ that takes $x$ to $x'$. Hence, we have proved the theorem in the case when $t$ acts unipotently on $H^*(U;\ov\cL)$.

In the case where the action is not unipotent, let $U_N$ be as in Remark~\ref{remEigenvalue}. Then we have the isomorphism $H^i(U;\ov\cL)\otimes_{\Q} \R\cong H^i(U_N;\ov\cL_N)\otimes_{\Q} \R$ which is defined over $\Q$. The right hand side has a MHS defined over $\Q$, and therefore the left hand side's MHS is also defined over $\Q$.
\end{proof}

\begin{cor}
The MHS on $\Tors_R H^i(U;\ov\cL)$ defined for $k=\Q$ is independent of the compactification, independent of $N$, and functorial (as in Theorems \ref{indcompactification}, \ref{indN}, and \ref{functorial}).
\end{cor}
\begin{proof}
This is a consequence of the faithful flatness of $\R$ over $\Q$: the $\R$-MHSs determine the MHSs over $\Q$, and all the maps involved in the statements and proofs are defined over $\Q$ as maps of vector spaces.
\end{proof}

Note that, in view of Proposition \ref{propcanon}, results of this section also yield the following result, which motivated our paper (Theorem~\ref{mhsexistence}):
\begin{cor}\label{halexandermhs}\index{Alexander module!homological}
The $\Q$-vector spaces $A_*(U^f;\Q)\coloneqq \mathrm{Tors}_R\, H_*(U^f;\Q)$ admit natural $\Q$-mixed Hodge structures, so that multiplication by $\log t^N$ is a morphism of mixed Hodge structures into the $-1$st Tate twist. Here $\log$ is the Taylor series centered at $1$, and $N$ is chosen so that $t^N$ acts unipotently on $A_*(U^f;\Q)$.
\end{cor}

\begin{proof}
By Remark~\ref{remk:isoLocal-Uf}, we have a canonical isomorphism $$H_i(U^f;\Q)\cong H_{i}(U;\cL).$$ Moreover, we use Proposition~\ref{propcanon} to identify $\Tors_R H_{i}(U;\cL)$ with the dual MHS of \linebreak $\Tors_R {H^{i+1}(U;\ov\cL)}$. The dual of multiplication by $t$ (resp. $\log(t^N)$) is multiplication by $t$ (resp. $\log(t^N)$), so we obtain the dual MHS morphism:
\begin{multline*}
\log(t^N)\colon A_*(U^f;\Q)(1) = \left( 
\Tors_R {H^{i+1}(U;\ov\cL)}(-1)
\right)^{\vee_k} \\ \longrightarrow \left( 
\Tors_R {H^{i+1}(U;\ov\cL)}
\right)^{\vee_k} = A_*(U^f;\Q).
\end{multline*}
It suffices to apply the Tate twist $(-1)$ to this map.
\end{proof}

\begin{remk}\label{remk:mhsSummary}
Let us give an overview of where the MHS on $\Tors_R H^{i+1}(U;\ov\cL)$ comes from. First, we pass to a finite cover $U_N\to U$ such that the action of $t^N$ on $\Tors_R H^{i+1}(U;\ov\cL)$ is unipotent. Lemma~\ref{233} provides an isomorphism between the Alexander modules of $U_N$ and $U$, so it suffices to give the former a MHS. Let us denote $U_N=U$.

Let $k=\R$. We consider the shift of the thickening of the Hodge-de Rham complex $\calH dg^\bullet(X\,\log D)\left(\frac{1}{i}\frac{df}{f},m\right)[1]$, which is a mixed Hodge complex of sheaves by Theorem~\ref{logthickenedmhs}. Hence, its $i$th cohomology, $H^{i}\Gamma(U; \cE_U^\bullet(\Im df/f,m)[1])$, carries an $\R$-mixed Hodge structure. For every $m$, we have a map $\nu[1]\colon \ov\cL\otimes_R R_m[1]\to \cE_U^0(\Im df/f,m)[1]$ as in Remark~\ref{rem:nu}. This map makes the right hand side into a soft resolution of $\ov\cL \otimes_R R_m[1]$, providing an isomorphism $H^{*}(U;\ov\cL\otimes_R R_m) \cong H^{*-1}\Gamma(U; \cE_U^\bullet(\Im df/f,m)[1])$, which we use to give $H^{*}(U;\ov\cL\otimes_R R_m)$ a MHS.

Following the proof of Corollary~\ref{cor:alex/s^mMHS}, the MHS on $H^*(U; \ov\calL)\otimes R_m$ is the unique one that makes the following injective (for $m\gg 0$) map into a MHS morphism:
\[
H^*(U; \ov\calL)\otimes R_m \hookrightarrow H^*(U; \ov\calL\otimes_R R_m).
\]
Following the proof of Corollary~\ref{torsionmhs}, the MHS on $\Tors_R H^*(U; \ov\calL)$ is the unique one that makes the following injective (for $m\gg 0$) map into a MHS morphism:
\[
\Tors_R H^*(U; \ov\calL)\hookrightarrow H^*(U; \ov\calL)\otimes R_m.
\]
Combining the two, we get that the MHS on $\Tors_R H^*(U; \ov\calL)$ is the unique one that makes the following injective (for $m\gg 0$) map into a MHS morphism, for $k=\R$ and also $k=\Q$ by Theorem~\ref{Qalexandermhs}:
\[
\Tors_R H^*(U; \ov\calL) \hookrightarrow H^*(U; \ov\calL\otimes_R R_m).
\]
Note that the image of $\Tors_{R}H^*(U; \ov\calL)$ in $H^*(U; \ov\calL\otimes_R R_m)$ is the kernel of $\psi_{mm}^*$, by Corollary~\ref{torsion}.
\end{remk}


\section{Dependence on the Function}
We have seen that the construction of the mixed Hodge structure is well-defined, but we have made a somewhat arbitrary choice of the map $f_\infty\colon U^f\to \C$. We will now see that changing this choice (and even changing $f$) gives an isomorphic MHS. Note that these are MHS's on the same vector space, so they can be isomorphic but unequal.

\begin{prop}[Dependence on the function]\label{prop:depf}\index{f_infty@$f_\infty$}
Let $c\in \C$. Consider the function $\wt f= e^{2\pi c}f$, and consider the following diagram:
\[
\begin{tikzcd}[column sep = 6em, row sep = 1.3em]
U^{\wt f}\arrow[rrr,"\wt f_\infty"]\arrow[ddd,"\wt \pi"'] & & &
\C \arrow[ddd,"\exp"]\\
&U^{f} \arrow[r,"f_{\infty}"]\arrow[d,"\pi"]\arrow[ul,"{(x,z)\mapsto (x,z+2\pi c) }"',"\chi"] &
\C \arrow[ur,"+{2\pi c}"']\arrow[d,"\exp"] &
\\
&U \arrow[r,"f"] &
\C^* \arrow[dr,"e^{2\pi c} "] &
\\
U \arrow[rrr,"\wt f"] \arrow[ur,equals]& & & 
\C^*. 
\end{tikzcd}
\]
Let $\cL = \pi_! \ul k_{U^f}, \cL' = \wt\pi_! \ul k_{U^{\wt f}}$. $\chi$ induces an isomorphism $\chi_{\cL}\colon 
\cL = \pi_!\ul k_{U^f} \cong \wt\pi_! \chi_!\ul k_{U^f} = \wt \pi_! \ul k_{U^{\wt f}} = \cL'$. Let $\Tors_R H^i(U;\ov\cL)$ (resp. $\Tors_R H^i(U;\ov\cL')$) be the mixed Hodge structure obtained from Corollary~\ref{alexandermhs} using the function $f$ (resp. $\wt f$). Then $\ov\chi_{\cL}\circ t^{\Im c}$ induces an isomorphism of MHS $\Tors_R H^i(U;\ov \cL) \to \Tors_R H^i(U;\ov\cL')$, where
\[
t^{\Im c} = \sum_{j=0}^{m-1} \binom{\Im c/N}{j} (t^N-1)^j.
\] 
Here $N, m$ are any natural numbers for which the action of $(t^N-1)^m$ is $0$ (in which case the above expression doesn't depend on $N,m$).
\end{prop}

\begin{proof}
On sections, $\ov\chi_{\cL}$ is given by $\ov\delta_{(x,z)} \mapsto \ov\delta_{(x,z+2\pi c)}$. Let us start with the case where the only eigenvalue of the action of $t$ on $\Tors_R H^i(U;\ov\cL)$ is $1$.

Let $m$ be large enough so that $(t-1)^m$ annihilates $\Tors_R H^i(U;\ov\cL)$. First, $\cE^\bullet_U(\darg ,m) = \cE^\bullet_U(\Im \frac{d e^{2\pi c}f}{e^{2\pi c}f} ,m)$ has a canonical structure of a mixed Hodge complex independent of $c$. Let
\[
\nu_f:\ov\cL\otimes_R R_m \to \cE^0_U\left( \darg, m\right);\quad 
\nu_{\wt f}:\ov\cL'\otimes_R R_m \to \cE^0_U\left( \darg, m\right) 
\]
be constructed as in Remark~\ref{rem:nu}. Explicitly: for $(x,z)\in U^f$ and $\iota$ a local section of $\pi$ such that $\iota(x)=(x,z)$, we have that $\chi\circ \iota$ is a section of $\wt\pi$. Note that $\wt f_{\infty} \circ \chi= f_\infty + 2\pi c$, so:
\begin{align*}
\nu_{\wt f} \circ \ov\chi_{\cL} (\ov\delta_{(x,z)}) &= 
\nu_{\wt f} (\ov\delta_{(x,z+2\pi c)}) \\ &= 
\exp(-\Im \wt f_{\infty}\circ \chi\circ \iota\otimes s) \\ &= e^{-2\pi \Im c\otimes s}\exp(-\Im f_{\infty}\circ \iota\otimes s) \\ &= \nu_f(t^{-\Im c}\ov\delta_{ (x,z)}).
\end{align*}
In other words, $\nu_{\wt f}\circ \ov\chi_{\cL} = \nu_f\circ t^{-\Im c}$. Recall that $\nu_f$ is not $R_m$-linear unless we change the $R_m$-module structure. In the notation from Remark~\ref{twistedmodule}, we see that $\ul{\phantom{x}}\cdot e^{2\pi \Im c\otimes s}=\ul{\phantom{x}}*_{\widetilde s} t^{\Im c}$, so, since $\widetilde s=\exp(2\pi s)-1$, multiplication by $e^{-2\pi \Im c\otimes s}$ on $\cE^0_U(\darg,m)$ corresponds to multiplication by $t^{-\Im c}$ on $\ov\cL\otimes_R R_m$.

The maps $\nu_f$ and $\nu_{\wt f}$ are used to give mixed Hodge structures to $H^i(U;\ov\cL\otimes_R R_m)$ and $H^i(U;\ov\cL'\otimes_R R_m)$, respectively. This shows that the following is a MHS isomorphism:
\[
\ov\chi_{\cL}\circ t^{\Im c}\colon H^i(U;\ov\cL\otimes_R R_m) \to H^i(U;\ov\cL'\otimes_R R_m).
\]
By Remark~\ref{remk:mhsSummary}, the MHS on $\Tors_R H^i(U;\ov\cL)$ is constructed as a sub-MHS of \linebreak $H^i(U;\ov\cL\otimes_R R_m)$ by the map induced by $R\to R_m$. This means that $\ov\chi_{\cL}\circ t^{\Im c}$ is a MHS morphism $\Tors_R H^i(U;\ov\cL) \to \Tors_R H^i(U;\ov\cL')$.

We now turn to the case where $t$ is not unipotent. As in Remark~\ref{remEigenvalue}, let $N$ be such that $t^N$ is unipotent, which will correspond to deck transformations for the degree $N$ cyclic cover.

Let $\wt{U_{ N}} = \{(x,z)\in U\times \C^*\mid \wt f(x) = z^N \}$. We can draw the maps in Diagram (\ref{eq:UN}) for both $(U,f)$ and $(U,\wt f)$, and they will be connected by the following isomorphisms:
\[
\begin{tikzcd}[column sep = 6em]
U^{f} \arrow[r,"\theta_N"] \arrow[d,"\chi"',"{z\mapsto z+2\pi c}"]&
U_N^{f_N} \arrow[r,"\pi_N"]\arrow[d,"\chi'"',"{z\mapsto z+2\pi c/N}"] &
U_N\arrow[r,"p"] \arrow[d,"\chi_N"',"{z\mapsto e^{2\pi c/N}z}"]&
U \\
 U^{\wt f} \arrow[r,"\wt \theta_N"] &
\wt{U_{ N}}^{\wt f_N} \arrow[r,"\wt\pi_N"] &
\wt{U_{ N}}\arrow[r,"\wt p"] &
U \arrow[u, equals].
\end{tikzcd}
\]
Let $\cL_N = (\pi_N)_! \ul k_{U_N^{f_N}}, \cL'_N = (\wt\pi_N)_! \ul k_{\wt{U_{ N}}^{\wt f_N}}$. Let $t_N = t^N$ be the deck transformation on the cohomology of $\ov\cL_N$. Applying the already proven result in the unipotent case to $U_N$ and $c' = c/N$, we have that $\chi_{\ov\cL_N} \circ t_N^{-\Im c/N}$ is an isomorphism of MHS between $\Tors_{R(N)} H^i(U_N;\ov\cL_N)$ and $\Tors_{R(N)} H^i(\wt{U_{ N}};\ov\cL'_N)$. Now, $t_N =t^N$ and the horizontal maps $\theta_N$ give isomorphisms of MHS (Lemma~\ref{233} and Theorem~\ref{indN}), so we have the desired result.
\end{proof}

\begin{cor}\label{cor:semisimpleIsEasy}\index{semisimple|(}
If the action of $t$ on $\Tors_R H^i(U;\ov\cL)$ is semisimple, then the MHS on the Alexander module is independent of $f_\infty$. 
\end{cor}

\begin{proof}
If the action of $t$ is semisimple, then in the Taylor series of $t^{\Im c}$ we may take $m=1$, so $t^{\Im c} = \Id$. Proposition~\ref{prop:depf} implies that $\ov\chi_{\cL}$ is a MHS isomorphism for any $c\in \C$.
\end{proof}

\begin{remk}\index{semisimple|)}
In particular, if the action of $t$ on $\Tors_R H^i(U;\ov\cL)$ is semisimple, then taking $c= i$ in Corollary~\ref{cor:semisimpleIsEasy}, we conclude that the MHS is preserved by deck transformations.
\end{remk}

\begin{remk}
The converse of Corollary~\ref{cor:semisimpleIsEasy} holds as well, see Proposition~\ref{prop:semisimpleIsEasyConverse}.
\end{remk}


	\chapter{The Geometric Map is a Morphism of MHS}\label{sec:main}

The goal of this chapter is to prove the following result (Theorem~\ref{geoIntro} in the introduction). We use all the notations of Section~\ref{ssAlex}.

\begin{thm}\label{thm:geomIsMHS}\index{pi@$\pi$|(}
The covering space map $\pi\colon U^f\to U$ induces a map in homology $$H_i(\pi)\colon \Tors_R H_i(U^f;k)\to H_i(U;k)$$ and the $k$-dual of $H_i(\pi)$ is a map in cohomology $$H^i(\pi)\colon H^i(U;k)\to \Tors_R H^{i+1}(U;\ov\cL).$$ Both are morphisms of mixed Hodge structures. 
\end{thm}

The MHS's on $\Tors_R H^{i+1}(U;\ov\cL)$ and $\Tors_R H_i(U;\cL)$ are the ones in Corollary~\ref{alexandermhs} and Corollary~\ref{halexandermhs}, respectively. $H^i(U;k)$ is endowed with Deligne's MHS, and $H_i(U;k)$ is endowed with its dual MHS. Before proving Theorem \ref{thm:geomIsMHS} in Section~\ref{ssec:proofOfGeoThm} we have to do some preliminary work.

\section{Maps between Local Systems}\label{sec:maps}

Let $R$ be as in Section~\ref{ssAlex}, and let $R_m = \frac{R}{s^mR}$. We work with local systems of $R$-modules on $U$. For such a local system $\cF$, we will abbreviate $\cF\otimes_R R_m$ to $\cF/s^m$. We make $\ul{k}$ into a sheaf of $R$-modules by letting $s$ act as $0$. For any $R$-module $M$ (or sheaf of $R$-modules), we will denote by $\ov M$ the conjugate structure as in Remark~\ref{rem:conjugate}. Throughout the whole section, the derived functor $R\Homm_R(\cdot,\ul R)$ (or $R\Hom_R(\cdot,R)$) will be abbreviated $\cdot^\vee$.

The covering map $\pi$ induces a map in homology $H_i(U^f;k)\rightarrow H_i(U;k)$. From Remark~\ref{remk:isoLocal-Uf}, there is a canonical isomorphism of $R$-modules $H_i(U;\cL)\cong H_i(U^f;k)$. As discussed in Remark~\ref{rem:cohom_not_iso}, the same is not true for cohomology in general, see Corollary~\ref{isocohom} for the hypotheses needed to have such isomorphism in cohomology. Restricting the map induced by $\pi$ in homology to the torsion, we get that $\pi$ induces a map
\[
H_i(\pi)\colon \Tors_R H_i(U;\calL)\to H_i(U;k).
\]
Taking the dual as $k$-vector spaces, and using the isomorphism $\Res$ from Definition~\ref{dfn:Res}, we obtain a map $
H^i(U;k)\to \Ext_R^1(\Tors_R H_i(U;\calL),R)$. Finally, the Universal Coefficients Theorem (as in Remark~\ref{remk:UCT}) gives us an isomorphism $$\Ext_R^1(\Tors_R H_i(U;\calL),R)\rightarrow \Tors_R H^{i+1}(U;\ov\cL).$$ Precomposing with the inverse of this isomorphism, we get that the covering space map $\pi$ induces a map
\[
H^i(\pi)\colon H^i(U;k)\to \Tors_R H^{i+1}(U;\ov\cL).
\]
We will use the notation $H_i(\pi)$ and $H^i(\pi)$ from now on to refer to these maps, as well as $H_i(\pi_N)$ and $H^i(\pi_N)$ for the analogous definition using $\pi_N$ (as in Section~\ref{sscover}) instead of $\pi$. The goal of this section is to realize both as the maps in (co)homology arising from maps of local systems (Propositions~\ref{prop:geomMapPoincare} and \ref{prop:mapsAreEqual}).

Corollary~\ref{isocohom} justifies the notation $H^i(\pi)$. Indeed, in its proof, we see that under the identification we are using, and under the condition that $H_i(U^f;k)$ is a torsion $R$-module, $H^i(\pi)$ (as defined by us) is indeed the map $H^i(U;k)\rightarrow H^i(U^f;k)$ induced in cohomology by $\pi$ (the dual of the map induced in homology).\index{pi@$\pi$|)}

\begin{prop}\label{prop:geomMapPoincare}\index{pi@$\pi$!piL@$\pi_\cL$|(}\index{delta@$\delta_x$}
Let $\pi_{\cL}\colon \cL\to \frac{\cL}{s} \xrightarrow{\sim} \ul{k}$ be determined by $\pi_{\cL}(\delta_{(x,z)}) = 1$ for all $(x,z)\in U^f\subset U\times \C$. Up to the identification $H_i(U^f;k)\cong H_i(U;\cL)$ (Remark~\ref{remk:isoLocal-Uf}), the two maps $H_i(U;\cL)\to H_i(U;k)$ induced in homology by $\pi$ and $\pi_{\cL}$ coincide. 
\end{prop}

\begin{proof}
The assertion follows by a straightforward computation, given the isomorphism $C_\bullet(U;\cL)\cong C_\bullet(U^f;k)$ of Remark~\ref{remk:isoLocal-Uf}. From the definitions, all we need is to show that the following diagram is commutative:
\[
\begin{tikzcd}[row sep = 1.3em]
C_\bullet(U^f;k)\arrow[r,"\sim"]\arrow[d,"\pi"] &
C_\bullet(U;\cL) = C_\bullet(U^f;k)\otimes_R R \arrow[d,"\pi_\cL:\cL_x=R\to k"]\\
C_\bullet(U;k) \arrow[r,"\sim"]&
C_\bullet(U;\ul{k}) = C_\bullet(U^f;k)\otimes_R k. 
\end{tikzcd}
\]
We can check directly that a chain $c\in C_i(U^f;k)$ is mapped to $\pi_*(c)\in C_\bullet(U;k)$ via both paths.
\end{proof}

Our next goal is to find a map (in the derived category) of local systems $\ul k[-1]\to \ov\cL$ that will induce $H^i(\pi)$ in cohomology.

Let us fix the free resolution $\cL\xrightarrow{s} \cL$ of $\ul{k}$ (see the maps below), so that $\pi_{\cL}$ becomes a composition as the left hand diagram below (the bottom row sits in degree $0$):
\[
\begin{tikzcd}[row sep = 1.4em]
0 \arrow[r]\arrow[d]& \cL\arrow[r]\arrow[d,"s"] & 0\arrow[d] & &
\ov\cL\arrow[r,"1"]\arrow[d,"s"]\arrow[dr,phantom,"\pi_{\cL}^\vee"] &\ov\cL \arrow[d]
\\
\cL \arrow[r,"1"]\arrow[rr,"\pi_{\cL}"',dashrightarrow, bend right =10]& \cL \arrow[r,"\pi_{\cL}"]& \ul{k}; & &
\ov\cL \arrow[r]&
0.
\end{tikzcd}
\]\index{pi@$\pi$!piLvee@$\pi_\cL^\vee$}
Let $\pi_{\cL}^\vee$ be defined by applying $R\Homm_R(\cdot,\ul R)$ to $\pi_{\cL}$. If we identify $\ul k$ with $\cL\xrightarrow{s}\cL$ as above, then it corresponds to the right hand map of complexes above (recall that we have fixed a way of identifying $\Homm_R(\cL,\ul R)\cong \ov\cL$, as in Remark~\ref{rem:oppositeDual}).

\begin{remk}\label{rem:Rvee}\index{residue!Res@$\Res$}
Let $\pi_{\ov\cL}\colon \ov\cL\to \ul k$ be the result of conjugating $\pi_\cL$. This induces a quasi-isomorphism $\ul k^\vee \coloneqq R\Homm_R(\ul k,\ul R)\cong \ul k[-1]$ by using the resolution $\cL\xrightarrow{s}\cL$ above:
$$
\ul k^\vee = R\Homm_R(\ul k,\ul R)\cong (\ov\cL\xrightarrow{s}\ov\cL )\cong \ul k[-1],
$$
where $\ov\cL\xrightarrow{s}\ov\cL$ sits in degrees $0$ and $1$ and the last $\cong$ is given by the map $\pi_{\ov \cL}$.
Definition~\ref{dfn:Res} provides another such isomorphism:
\[
R\Homm_R(\ul k{,}\ul R) \cong \Extt^1_R(\ul k, \ul R)[-1] \xlongrightarrow[\sim]{\Res}
\Homm_k(\ul k,\ul k)[-1] \cong \ul k[-1].
\]
Since $\Aut(\ul k)\cong k^*$ (both in the abelian and the derived categories), these two isomorphisms differ by multiplication by a nonzero constant (as maps in the derived category of sheaves of $R$-modules on $U$). This constant turns out to be $1$, keeping in mind Remark~\ref{rem:annoyingSign}. For all our purposes, it will not matter which constant this is, so we will omit the computation.
\end{remk}

We have two exact triangles, where we are using $\cdot^\vee$ to denote the functor $R\Homm_R(\cdot,\ul R)$.
\begin{equation}\label{eq:picLTriangles}
{
\begin{tikzcd}[column sep = 3em, row sep = 1.3em]
0 \arrow[r]\arrow[d] &
0 \arrow[r]\arrow[d]\arrow[dr,phantom,"\pi_{\cL}"] &
\cL \arrow[r,"1"]\arrow[d,"s"']\arrow[dr,phantom,"-\pi_{\ov\cL}^\vee{[1]}",shift right = 1ex] &
\cL \arrow[d, shift right = 1.8ex]\quad
0 \arrow[r]\arrow[d, shift left = 1.8ex] &
0 \arrow[r]\arrow[d]\arrow[dr,phantom,"\pi_{\ov\cL}"] &
\ov\cL \arrow[r,"1"]\arrow[d,"s"']\arrow[dr,phantom,"-\pi_{\cL}^\vee{[1]}"] &
\ov\cL \arrow[d]\\
\cL \arrow[r,"s"'] &
\cL \arrow[r,"1"'] &
\cL \arrow[r] &
0;\hspace{-1.2mm} \quad
\ov\cL \arrow[r,"s"'] &
\ov\cL \arrow[r,"1"'] &
\ov\cL \arrow[r] &
0.
\end{tikzcd}
}
\end{equation}
Note that the labeling on the maps is consistent with $\cdot^\vee = R\Homm(\cdot,R)$. Below, the first diagram represents $\pi_{\cL}^\vee$, the second diagram is $\pi_{\cL}^\vee[1]$ and the third diagram is the same map, after the isomorphism $\left(\ov\cL\xrightarrow{s}\ov\cL\right)\cong \left(\ov\cL\xrightarrow{-s}\ov\cL\right)$ that is $\Id_{\ov\cL}$ in degree $0$ and $-\Id_{\ov\cL}$ in degree $-1$ (which is the identity morphism in the derived category).
\[
{
\begin{tikzcd}[column sep = {3.5em,between origins}, row sep = {2.7em,between origins}]
\deg = 0 \arrow[r,dashrightarrow]&
\ov\cL \arrow[r,"1"]\arrow[d,"s"'] &
\ov\cL\arrow[d] & 
 & 
\ov\cL \arrow[r,"1"]\arrow[d,"-s"'] &
\ov\cL \arrow[d]& 
 & 
\ov\cL \arrow[r,"-1"]\arrow[d,"s"'] &
\ov\cL \arrow[d]\\
&ø
\ov\cL \arrow[r] &
0 & 
 \deg = 0 \arrow[r,dashrightarrow]& 
\ov\cL \arrow[r] &
0 & 
\deg = 0 \arrow[r,dashrightarrow] & 
\ov\cL \arrow[r] &
0 
\end{tikzcd}
}
\]
We will identify $\ul k$ with its free resolutions, hence our naming of the maps above, with a slight abuse of notation.

\begin{prop}\label{prop:mapsAreEqual}\index{pi@$\pi$!piL@$\pi_\cL$|)}\index{UCT!$\UCT_k$}\index{UCT!$\UCT_R$}\index{UCT}\index{residue!Res@$\Res$}
	The map $H^i(U;k)\to\Tors_R H^{i+1}(U;\ov\cL)$ induced by $\pi^\vee_{\cL}$ in cohomology coincides with $H^i(\pi)$. In other words, the following diagram commutes. The dotted arrow is the desired map computed in two different ways.
\[
{
\begin{tikzcd}[column sep = {9em,between origins}, row sep = 1.2em]
	&
	H^{i}(U;k)\arrow[dr,"\UCT_k",leftarrow,"\sim"'] \arrow[dddl,dashrightarrow, bend left = 20, end anchor = {[xshift=3ex]}] &
	\\
H^{i+1}(U;\ul{k}^\vee) \arrow[dd,"H^i(\pi_{\cL}^\vee)"'] \arrow[ur,"\Res","\sim"']& 
&
\Hom_k( H_i(U;k),k)\arrow[dd,"\Hom(H_i(\pi){,}k)"]\\
\\
 H^{i+1}(U;\ov \cL)\supseteq
 \Tors_R H^{i+1}(U;\ov \cL) &
 &
\Hom_k(\Tors_R H_i(U;\cL),k)\arrow[dl,"\Res","\sim"',leftarrow, end anchor = {[xshift=2ex]}]
\\
&
\Ext^1_R(H_i(U;\cL){,}R).\arrow[ul,"\UCT_R","\sim"', start anchor = {[xshift=-2ex]}]
\end{tikzcd}
}
\]
Here ``UCT'' stands for the maps defined in \ref{lem:UCT}, and $\Res$ (abusing notation) is the canonical isomorphism induced by the isomorphism of $R$-modules $\Res$ of Definition \ref{dfn:Res}. Here we only assume that $U$ is a topological space and $f:U\to \C^*$ is a continuous map.
\end{prop}

\begin{proof}
Throughout this proof, all homology and cohomology groups will be on the space $U$, so we will omit it for brevity. Also, $k$ is always viewed as the $R$-module $R/sR$, via the map sending $1$ to $1$.

The naturality of tensor-hom adjunction in the definition of $\UCT_R$ in Lemma \ref{lem:UCT} tells us that the following diagram commutes:
\[
\begin{tikzcd}[column sep = 5em, row sep =1.3em]
\Ext^1_R(H_{i}(\ul k),R) \arrow[d,"\Ext^1_R(H_i(\pi_{\cL}){,}R)"]\arrow[r,"\UCT_R","\sim"'] & 
H^{i+1}(\ul{k}^\vee)=H^{i}(\Extt^1_R(\ul{k},\underline R))=H^{i+1}(C^\bullet(\ul{k}^\vee)) \arrow[d,"H^i(\pi_{\cL}^\vee)"]\\
\Ext^1_R(H_{i}( \cL),R)\arrow[r,"\UCT_R"',"\sim"]
 & 
\Tors_R H^{i+1}(\ov \cL).
\end{tikzcd}
\]
Recall Remark \ref{rem:Rvee}. For the purposes of the map $H^i(\pi_{\cL}^\vee)$, $\ul{k}^\vee$ is seen as the bounded complex $\ov\cL\xrightarrow{s}\ov\cL$, sitting in degrees $0$ and $1$. Recall that, by truncation, we also identified $\ul{k}^\vee$ with $\Extt^1_R(\ul{k},\underline R)[-1]$. 

In view of the diagram above, this lemma is equivalent to the commutativity of the following diagram: 
\[
{
\begin{tikzcd}[column sep = {9em,between origins}, row sep = 1.2em]
	&
	H^{i}(\ul k) \arrow[rd,"\UCT_k","\sim"',end anchor={[xshift=-1ex]}] 
	\\
H^{i+1}(\ul{k}^\vee) \arrow[ru,"\Res","\sim"'] 
& 
&
\Hom_k( H_i(\ul k),k)\arrow[dd,"\Hom(H_i(\pi){,}k)"]
\\
\\
\Ext^1_R( H_{i}(\ul{k}),R) \arrow[uu,"\UCT_R","\sim"'] \arrow[rd,"\Ext^1_R(H_i(\pi_{\cL}){,}R)"']
& 
&
\Hom_k(\Tors_R H_i(\cL),k).\arrow[dl,"\Res",leftarrow,"\sim"',start anchor={[xshift=-3ex]}]
\\
&
\hspace{2mm}\Ext^1_R(H_i(\cL){,}R) 
\end{tikzcd}
}
\]
We claim that the following diagram of isomorphisms commutes:
\begin{equation}\label{eq:resUCTresUCT}
\begin{tikzcd}[row sep = 1.3em]
H^{i+1}(\ul{k}^\vee) \arrow[d,"\Res"']\arrow[r,"\UCT_R",leftarrow] & 
\Ext^1_R( H_{i}(\ul k),R)\arrow[d,"\Res"]
\\
H^{i}(\ul k) \arrow[r,"\UCT_k"]& 
\Hom_{k}( H_i(\ul k),k). 
\end{tikzcd}
\end{equation}
Let us write the complexes representing these (co)homology groups, like in Definition \ref{dfn:complex} and Remark \ref{remk:dfncomplex}. We will draw a commutative diagram with the corresponding maps between complexes. Let $C_\bullet(\ul{k})$ be as in Remark \ref{remk:dfncomplex}. Note that the coefficients are constant, so we can take $\wt U=U$ in the definition. Then, $C_\bullet(\ul{k}) =C_\bullet(U,R)\otimes_R k$. The proof amounts to looking at the diagram (\ref{eq:betatensorhom}).
\begin{equation}\label{eq:betatensorhom}
{
\begin{tikzcd}[column sep = {6.8em,between origins}, row sep = {3em,between origins}]
H^{i+1}(\ul{k}^\vee)
\arrow[dr,no head, dashed,"a"]
\arrow[dddddddd,"\Res",rightarrow,,
rounded corners,
start anchor={[xshift=0.1ex]},
to path={ -- ([xshift=-1.8ex,yshift=24em]\tikztotarget.west)
	-- ([xshift=-1.8ex]\tikztotarget.west)\tikztonodes
	-- (\tikztotarget)}]
\arrow[rrrr,"\UCT_R ",leftarrow]
&&&&
\Ext^1_R(H_{i}(\ul k),R) 
\arrow[dl,no head, dashed,"c"]
\\
& 
\Hom_R^\bullet(C_\bullet(U,R),(\cF^\bullet)^\vee_x)[1]
\arrow[rr,"\text{t-h}"]
 &&
\Hom_R^\bullet(C_\bullet(\cF^\bullet),R)[1]
\arrow[dd,"\beta",leftarrow]
&
\\
\\
&
\Hom_R^\bullet(C_\bullet(U,R),\Hom_R^\bullet(\cF^\bullet_x,K/R))
\arrow[uu,"\beta"] 
\arrow[rr,"\text{t-h}"]
 &&
\Hom_{R}^\bullet(C_\bullet(\cF^\bullet),K/R)
 &
 \\
 &&&&
 \Hom_R(H_i(\ul k),K/R)
 \arrow[dl,no head, dashed,"b", near start, start anchor={[xshift=-2ex]}, end anchor={[xshift=1ex]}]
 \arrow[uuuu,"\beta"]
 \arrow[dddd,"\res_*"]
\\
&
\Hom_R^\bullet(C_\bullet(U,R),\Hom_R(k,K/R))
\arrow[rr,"\text{t-h}"]
\arrow[uu,"\rho"]
 \arrow[dd,"\Hom_R^\bullet(C_\bullet(U{,}R){,}\res_*)"] 
 &&
\Hom_{R}^\bullet(C_\bullet(\ul{k}),K/R)
\arrow[uu,"\rho"]
\arrow[dd,"\res_*"']
\\
\\
&
\Hom_R^\bullet(C_\bullet(U,R),\Hom_k(k,k))
\arrow[rr,"\text{t-h}"]
 &&
\Hom_{k}^\bullet(C_\bullet(\ul{k}),k)&
\\
H^i(\ul k) 
\arrow[ur,no head, dashed,"a"]
\arrow[rrrr,"\UCT_{k}"]
&&&&
\Hom_{k}(H_i(\ul k),k)
 \arrow[ul,no head, dashed,"b"].
\end{tikzcd}
}
\end{equation}

The four corners and the arrows between them coincide with those of diagram (\ref{eq:resUCTresUCT}). Also, $\cF^\bullet := \ul R \xrightarrow{s} \ul R$ (in degrees $[-1,0]$) is a free resolution of $\ul{k}$, and $\rho\colon \cF^\bullet\to \ul{k}$ is the corresponding quasi-isomorphism (we've abused notation by labeling $\rho$ the maps that are induced by $\rho$ after applying some functors). Note that $(\cF^{\bullet})^\vee[1]$ is a resolution of $\Extt_R^1(\ul{k},\underline R)=\underline{\Ext_R^1(k,R)}$ (recall that we are using $\ul M$ to denote the constant sheaf with stalk $M$). Similarly, $\beta \colon K/R\to R[1]$ is defined as in Definition~\ref{dfn:Res}, and we've labeled $\beta$ the arrows that are induced by $\beta$. The arrows ``t-h'' represent tensor-hom adjunction.

First, note that the ``inside'' rectangle made of solid straight arrows is commutative. The top two squares commute by the naturality of tensor-hom adjunction, and the bottom square can be seen to commute by a direct calculation. \textbf{The dashed lines are not maps}, rather they represent the cohomology of the corresponding complex. We would like to show that the maps drawn ``outside'' in cohomology are induced by the maps of complexes depicted in the diagram.

The lines labeled ``$a$" are obtained by simply taking the cohomology, by definition (together with identifying $\Hom_{k}(k,k) = k$). The lines labeled ``$b$" and ``$c$" arise in the universal coefficient theorem. The $b$'s come from the obvious map $H^i(\Hom_A^\bullet(B^\bullet,C))\to \Hom_A(H^i(B^\bullet),C)$ for any ring $A$, any $A$-module $C$ and any complex $B^\bullet$ of $A$-modules.

The line ``$c$" is obtained from the universal coefficient theorem. We claim that the following diagram commutes:
\[
\begin{tikzcd}[row sep = 0.8em]
H^{i+1}(\Hom_R^\bullet(C_\bullet(\cF^\bullet),R) &
\Ext^1_R(H_i(C_\bullet(\cF^\bullet)),R) \arrow[l,"c"]\\
H^{i}(\Hom_R^\bullet(C_\bullet(\cF^\bullet),K/R) \arrow[u,"\beta"]\arrow[r,"b"]&
\Hom_R(H_i(C_\bullet(\cF^\bullet)),K/R) \arrow[u,"\beta"]
\end{tikzcd}
\]
Note that $\rho$ induces the identity in cohomology, so we are ignoring its role above. The fact that this diagram commutes follows from following the proof of Lemma~\ref{lem:UCT}.

This, together with the definitions, shows that the maps ``inside'' the diagram (\ref{eq:betatensorhom}) induce in cohomology the maps ``outside'', which correspond to diagram (\ref{eq:resUCTresUCT}). Now, the commutativity of diagram (\ref{eq:resUCTresUCT}) follows from the commutativity of the ``inside'' diagram (\ref{eq:betatensorhom}).

After having proved that the diagram (\ref{eq:resUCTresUCT}) commutes, the desired statement is equivalent to proving that the following diagram commutes:
\[
\begin{tikzcd}[column sep = 5em, row sep = 1.3em]
\Ext^1_R(H_{i}(\ul k),R) \arrow[r,"\Ext^1_R(H_i(\pi_{\cL}){,}R)"]\arrow[d,"\Res","\sim"'] & 
\Ext^1_R(H_{i}( \cL),R)\arrow[d,"\Res","\sim"']\\
\Hom_{k}(H_{i}(\ul k),k) \arrow[r,"\Hom_{k}(H_i(\pi){,}k)"] & 
\Hom_{k}(\Tors_R H_{i}(\cL),k).
\end{tikzcd}
\]
But now using the naturality of the ``Res'' isomorphism, this is equivalent to $
H_i(\pi) = H_i(\pi_{\cL})$, which is the content of Proposition~\ref{prop:geomMapPoincare}.
\end{proof}


\section{Proof of the Theorem}\label{ssec:proofOfGeoThm}

In this section we prove Theorem~\ref{thm:geomIsMHS}. We use all the notations of Section~\ref{sec:maps}: we have the map $\pi_{\cL}$ from Proposition~\ref{prop:geomMapPoincare}, and its dual $\pi_{\cL}^\vee$, where we are still writing $\cdot^\vee = R\Homm_R(\cdot,\ul R)$. We will identify $\ul k^\vee\cong \ul k[-1]$ as in Remark~\ref{rem:Rvee}. For a (sheaf of) $R$-modules $M$, we may abbreviate $\frac{M}{s^m}\coloneqq M\otimes_R R_m$.

\begin{remk}
Since $\R$ is faithfully flat over $\Q$, it suffices to prove the claim when $k=\R$. \textit{For the rest of this section, we will only deal with real coefficients.}
\end{remk}

\begin{remk}
Let $N\in\N$ be a natural number such that the action of $t^N$ on $\Tors_R H^i(U;\calL)$ is unipotent, and consider the $N$-fold cover $U_N$ of $U$ as in Lemma~\ref{233} and Remark~\ref{remEigenvalue}.

It suffices to show that the map in cohomology $$H^i(\pi_N)\colon H^i(U_N;\R)\to\Tors_R H^{i+1}(U;\cL)$$ is a MHS morphism, from which the result follows, since $U_N\to U$ is an algebraic map, and therefore it induces a MHS morphism in cohomology. \textit{Throughout the whole proof we will denote $U:=U_N$, $f:=f_N$, $\pi=\pi_N$ and assume that the $t$-action on $\Tors_R H_i(U;\cL)$ is unipotent for all $i$}.
\end{remk}

\begin{proof}[Proof of Theorem~\ref{thm:geomIsMHS}]
In Corollary~\ref{halexandermhs}, the MHS on $\Tors_R H_i(U^f;\R)$ is defined as the dual of the MHS defined on $\Tors_R H^{i+1}(U;\ov\cL)$, via a series of isomorphisms:
$$
(\Tors_R H^{i+1}(U;\ov\cL) )^\vee \overset{\Res}\cong \Ext^1_R(H^{i+1}(U;\ov\cL),R)  \overset{\UCT_R}\cong \Tors_R H_i(U;\cL) \overset{\text{}}\cong \Tors_R H_i(U^f;\R).
$$

We consider the projection $\phi_m=(\phi_U)_m\colon \ov \cL\to \frac{\ov \cL}{s^m} $, which arises from $R\to R_m$ after tensoring with $\ov\cL$. We are interested in $\phi_m\circ \pi_{\cL}^\vee\colon \uR_U^\vee\to \frac{\ov\cL}{s^m}$, where $\pi_{\cL}$ is defined as in Proposition~\ref{prop:geomMapPoincare}. As in Remark~\ref{rem:Rvee}, we identify $\uR_U^\vee$ with $\uR_U[-1]$. By Proposition~\ref{prop:mapsAreEqual}, it is enough to prove that $H^i(\pi_{\cL}^\vee)\colon H^i(U;\R)\to \Tors_R H^{i+1}(U;\ov\cL)$ is a MHS morphism.

We will prove the theorem in a series of lemmas. We will state them all, then we will prove the theorem using them, and after that we will include the proofs of all the lemmas in order.

\begin{lem}\label{lem:theMapA}
Let $X$ be a good compactification of $U$, such that $D:=X\backslash U$ is a normal crossing divisor, as in Section~\ref{ss:MHSsAndComplexes} and \cite[Definition 4.1]{peters2008mixed}. Let $j\colon U\hookrightarrow X$ be the inclusion.

There is a morphism of mixed Hodge complexes $A_{\calH dg}\colon \Hdg{X}{D}\rightarrow \linebreak\Hdg{X}{D}\left(\frac{1}{i}\frac{df}{f}, m \right)[1]$ given, in the real part, by $j_*A_\R$, where:
$$
\f{A_\R}{\calE_{U}^\bullet}{\calE_{U}^{\bullet+1}\otimes R_m}{\alpha}{\left(\Im\frac{df}{f}\wedge \alpha\right)\otimes 1,}
$$
in the complex part, by
$$
\f{A_\C}{\Omega_{X}^\bullet(\log D)}{\Omega_{X}^{\bullet+1}(\log D)\otimes R_m}{\omega}{\left(\frac{1}{i}\frac{df}{f}\wedge \omega\right)\otimes 1.}
$$
Being a morphism of mixed Hodge complexes, it induces a MHS morphism on the cohomology.
\[
H^i(A)\colon H^i(U;\R)\to H^{i+1}\left(U;\frac{\ov\cL}{s^m}\right).
\]
Here $A$ is the composition:
\[
\uR_U\to \cE^\bullet_U \onarrow{A_{\R}} \cE^\bullet_U\left(\darg, m\right)[1] \cong \cE^\bullet_U(m)_{\log}[1] \onarrow{\eta_m^{-1}} \frac{\ov\cL}{s^m}[1],
\]
where $\eta_m$ is as in Remark \ref{remk:thickenedComplex}. Note that $A$ is in principle just a morphism in the derived category of sheaves of $\R$-vector spaces.
\end{lem}

\begin{lem}\label{lem:proofInCircle}
Let $\phi_{m,\C^*}, \pi_{\cL,\C^*}^\vee, A_{\C^*},\cL_{\C^*}, \ov\cL_{\C^*}$ be defined analogously to $\phi_m$, $\pi_{\cL}^\vee,A$, $\cL$, $\ov\cL$ in the particular case where $U=\C^*$ and $f=\Id$. There is a constant $c\in \R\setminus \{0\}$ such that
\[
(\phi_{m,\C^*}\circ \pi_{\cL,\C^*}^\vee)[1] = c\cdot A_{\C^*}\colon \uR_{\C^*}\longrightarrow \frac{\ov\cL_{\C^*}}{s^m}[1].
\]
These are equal as maps in the derived category of sheaves of $\R$-vector spaces. In particular, $A_{\C^*}$ comes from a map in the derived category of sheaves of $R$-modules, after restriction of scalars.
\end{lem}

\begin{lem}\label{lem:pullbacks}
Here $A_U=A$, $\cL_U=\cL$, etc. We have two commutative diagrams, whose horizontal arrows are isomorphisms:
\[
\begin{tikzcd}
f^{-1}\uR_{\C^*} \arrow[d,"f^{-1}A_{\C^*}"]\arrow[r] &
\uR_{U} \arrow[d,"A_U"] & & &
f^{-1}\uR_{\C^*} \arrow[d,"f^{-1}(\phi_{m,\C^*}\circ \pi_{\cL,\C^*}^\vee){[1]}"']\arrow[r] &
\uR_{U} \arrow[d,"(\phi_{m,U}\circ \pi_{\cL,U}^\vee){[1]}"] \\
f^{-1}\frac{\ov\cL_{\C^*}}{s^m}[1]\arrow[r] &
\frac{\ov\cL_{U}}{s^m}[1]; & & &
f^{-1}\frac{\ov\cL_{\C^*}}{s^m}[1]\arrow[r] &
\frac{\ov\cL_{U}}{s^m}[1].
\end{tikzcd}
\]
The bottom arrow comes from the (dual of) the base change isomorphism $f^{-1}\cL_{\C^*} = \linebreak f^{-1}\exp_!\uR_{\C} \cong \pi_!f_\infty^{-1}\uR_{\C} \cong \pi_! \uR_{U^f} = \cL_{U}$. The isomorphism $f^{-1}\uR_{\C^*}\to \uR_U$ sends $1$ to $1$.

In particular, there is a constant $c\in \R\setminus \{0\}$ for which
\[
c A_U = \phi_{m,U} \circ \pi_{\cL,U}^\vee.
\]
These are equal as maps in the derived category of sheaves of $\R$-vector spaces. In particular, $A_{U}$ comes from a map in the derived category of sheaves of $R$-modules, after restriction of scalars.
\end{lem}

The proof of all the lemmas in order can be found below. Before, let us show how to prove the theorem.

By Lemma~\ref{lem:theMapA}, $A$ induces a MHS morphism $H^i(U;\R)\to H^{i+1}(U;\frac{\ov\cL}{s^m})$, and by Lemma~\ref{lem:pullbacks}, $\phi_{m} \circ \pi_{\cL}^\vee$ is its multiple. Therefore, $\phi_{m} \circ \pi_{\cL}^\vee$ induces a MHS morphism in cohomology as well. Now, since the image of a torsion module must lie in the torsion, there is a unique map as follows, which we will also denote $H^{i}(\pi_{\cL}^\vee)$:
\[
\begin{tikzcd}
H^{i-1}(U;\R) \arrow[r,"H^{i}(\pi_{\cL}^\vee)"]\arrow[dr,dashrightarrow,"\exists!"] 
& H^i(U;\ov\cL) \arrow[r,"H^i(\phi_m)"] 
& H^i(U;\frac{\ov\cL}{s^m}) \\
& \Tors_R H^i(U;\ov\cL)\arrow[ur,hookrightarrow,"\iota"]\arrow[u,hookrightarrow] &
\end{tikzcd}
\]
By Remark~\ref{remk:mhsSummary}, $\iota$ is a MHS morphism. Since $$H^i(\pi_{\cL}^\vee)\colon H^{i-1}(U;\R)\to \Tors_R H^i(U;\ov\cL)$$ amounts to restricting the codomain of $H^i(\phi_m\circ \pi_{\cL}^\vee)\colon H^{i-1}(U;\R)\to H^i(U;\frac{\ov\cL}{s^m})$ to a sub-MHS, it is itself a MHS morphism.

\begin{proof}[Proof of Lemma~\ref{lem:theMapA}]

Let us recall the filtrations:
\begin{align*}
\tau_k\cE_U^\bullet\left(\darg, m\right)[1] &= \bigoplus_j \tau_{k+2j+2}\cE_U^\bullet\otimes \R\langle s^j\rangle; \\
W_k\Omega_X^\bullet(\log D)\left(\frac{1}{i}\frac{df}{f}, m\right)[1] &= \bigoplus_j W_{k+2j+2}\Omega_X^\bullet(\log D)\otimes \C\langle s^j\rangle; \\
F^p\Omega_X^\bullet(\log D)\left(\frac{1}{i}\frac{df}{f}, m\right)[1] &= \bigoplus_j F^{p+j+1}\Omega_X^\bullet(\log D)\otimes \C\langle s^j\rangle .
\end{align*}

One can easily check that both components of $A_{\calH dg}$ respect the filtrations above.

We must define a morphism of pseudomorphisms that makes the definition of $A_{\calH dg}$ on the real and complex part agree. Let us start by recalling the definitions. The pseudomorphism in the definition of $\Hdg XD$ is given in (\ref{eqmorphisms}). According to the construction in Theorem~\ref{mhsthickened}, the pseudomorphism used in Theorem~\ref{logthickenedmhs} is:
\begin{align*}
j_*\calE^{\bullet}_{U}\left(\Im\frac{df}{f},m\right)&
\hookrightarrow j_*\calE^{\bullet}_{U}\left(\Im\frac{df}{f},m\right)\otimes_{\R} \C\\ &\xrightarrow{\exp\left(\frac{\Log(|f|)}{-i} \otimes s\right)} j_*\calE^{\bullet}_{U}\left(\frac{1}{ i}\frac{df}{f},m\right)\otimes_{\R} \C
\xleftarrow{\cong} \logdr{X}{D}\left(\frac{1}{ i}\frac{df}{f},m\right).
\end{align*}
This is followed by the identity map of $\logdr{X}{D}\left(\frac{1}{ i}\frac{df}{f},m\right)$, where the domain has the filtration induced by $\tau_{\lc}$ and the target has the filtration induced by $W_{\lc}$. If we extend the pseudomorphism in the obvious way it will not define a morphism of pseudomorphisms in the sense of \cite[3.16]{peters2008mixed}, namely the following diagram does not commute:
\begin{equation}\label{eq:nastyDiagram}
\begin{tikzcd}[column sep = 5em]
\calE^{\bullet}_{U}\otimes_\R \C \arrow[r,equals]\arrow[d,"\Im\frac{df}{f}\wedge \cdot "] 
&
\calE^{\bullet}_{U}\otimes_\R \C \arrow[d,"\frac{1}{i}\frac{df}{f}\wedge \cdot "] 
\\
\calE^{\bullet}_{U}\left(\Im \frac{df}{f},m\right)\otimes_\R \C[1] \arrow[r,"\exp\left(\frac{\Log(|f|)}{-i} \otimes s\right)"] 
&
\calE^{\bullet}_{U}\left(\frac{1}{i}\frac{df}{f},m\right)\otimes_\R \C[1] .
\end{tikzcd}
\end{equation}
Instead, it commutes up to a homotopy $h:\calE^{\bullet}_{U}\otimes_\R \C \to \calE^{\bullet}_{U}\left(\frac{1}{i}\frac{df}{f},m\right)\otimes_\R \C$, given by the following equation.
\[
h(\alpha) = \frac{\exp\left(\frac{\Log(|f|)}{-i} \otimes s\right) \alpha - \alpha}{s} = \sum_{j\ge 0} \frac{s^j \left( \frac{\Log(|f|)}{-i}\right)^{j+1}}{(j+1)!} \alpha.
\]
A direct computation shows that indeed we have the desired homotopy (recall that the shifted differential on the target is $d[1] = -d$):
\[
(-dh+hd)(\alpha) = \frac{1}{i} \frac{df}{f}\wedge \alpha - \exp\left(\frac{\Log(|f|)}{-i} \otimes s\right) \Im \frac{df}{f}\wedge \alpha .
\]

Hence, the diagram (\ref{eq:nastyDiagram}) commutes in the (filtered) homotopy category, so in particular it commutes in the filtered derived category. Therefore, we have defined a morphism of mixed Hodge complexes in the sense of \cite[8.1.5]{De3}, but not necessarily in the stricter sense of \cite[3.16]{peters2008mixed}. Nonetheless, since the diagram commutes up to homotopy, the diagram in hypercohomology will commute, which is enough to show that $A_\R$ induces a MHS morphism in hypercohomology following \cite[3.18]{peters2008mixed}.

Note that the sheaves involved in the $\R$ part are soft. We now take global sections in the $\R$ part, to get an induced map in cohomology for all $i$:
$$
H^i(A)\colon H^i(U;\R)\rightarrow H^{i+1}\Gamma\left(U;\calE_{U}^\bullet\left(\darg, m\right)\right)\cong H^{i+1}\Gamma\left(U;\calE_{U}^\bullet(m)_{\log}\right).
$$
Since it comes from a map of mixed Hodge complexes, it is a morphism of MHS. Equivalently, we can define the composition	$A\colon \uR_U\to \cE^\bullet_U\to \cE_U^\bullet(m)_{\log} \onarrow{\eta_m^{-1}} \frac{\ov\cL}{s^m}$, which gives us a morphism of MHS:
\[
H^i(U;\R)\to H^i\Gamma(U;\cE^\bullet_U)\to
H^{i+1}\Gamma(U;\cE_U^\bullet(m)_{\log}) \onarrow{\eta_m^{-1}} H^{i+1}\left(U;\frac{\ov\cL}{s^m}\right).
\]
Indeed, the first and third arrows are MHS morphisms since they are used to define the MHS on $H^i(U;\R)$ and $H^{i+1}(U;\ov\cL/s^m)$, respectively. Note that $\eta_m^{-1}$ is the inverse of a quasi-isomorphism, so it exists in the derived category and it induces maps in cohomology.
\end{proof}

\begin{proof}[Proof of Lemma~\ref{lem:proofInCircle}]
First, recall that $R\Hom_{D^b_{\R}(\C^*)}^\bullet(\uR_{\C^*},\cdot)$ (morphisms in the derived category of sheaves of $\R$-vector spaces on $\C^*$) and $H^\bullet(\C^*,\cdot)$ are both the derived functor of global sections, so $\Hom_{D^b_{\R}(\C^*)}\left(\uR_{\C^*},\frac{\ov \cL_{\C^*}}{s^m}[1]\right) \cong H^1\left({\C^*};\frac{\ov\cL_{\C^*}}{s^m} \right)$. Let us show that this space is one-dimensional.

Using \cite[Example 2.5.7]{dimca2004sheaves}, $H^1(\C^*;\ov\cL_{\C^*}/s^m)$ is isomorphic to the cokernel of $T-\Id$ acting on the stalk of $\ov\cL_{\C^*}/s^m$, where $T$ is the monodromy action of the generator of $\pi_1(\C^*)$. This generator acts as $t^{-1} = (1+s)^{-1}\in R_m$. We can see directly that $t^{-1}-1$ equals $s$ up to multiplication by a unit in $R_m$ (namely $-t$), so its cokernel is one-dimensional.

This shows that $(\phi_{m,\C^*}\circ \pi_{\cL,\C^*}^\vee)[1]$ and $A_{\C^*}$ are scalar multiples of each other as long as they are both nonzero (as classes of maps in the derived category).

Let us show that $\phi_{m,\C^*}\circ \pi_{\cL,\C^*}^\vee\neq 0$. First, recall that $\pi_{\cL,\C^*}^\vee\colon \uR_{\C^*}[-1]\to \ov\cL_{\C^*}$ is the roof diagram:
\[
\begin{tikzcd}
0 \arrow[d]& \ov\cL_{\C^*}\arrow[d,"s"]\arrow[r,"="]\arrow[l] & \ov\cL_{\C^*}\arrow[d] \\
\uR_{\C^*} & \ov\cL_{\C^*}\arrow[l,twoheadrightarrow] \arrow[r]& 0.
\end{tikzcd}
\]
Seeing $\pi_{\cL,\C^*}^\vee\in
 \Hom_{D^b_{\R}(\C^*)}\left(\uR_{\C^*},\ov \cL_{\C^*}[1]\right)
 \cong 
 \Ext^1_{\R}(\uR_{\C^*},\ov \cL_{\C^*})$ as the class of an extension, the roof diagram shows that it corresponds to the short exact sequence $0\to \ov\cL_{\C^*}\onarrow{s} \ov\cL_{\C^*} \to \uR_{\C^*}\to 0$. Note that for our purposes it doesn't matter which surjection $\ov\cL_{\C^*}\to \uR_{\C^*}$ we use. Using the Yoneda product (see \cite[III.5, III.6]{mac}), $\phi_{m,\C^*}\circ \pi_{\cL,\C^*}^\vee$ is the extension class in $\Ext^1_R(\uR,\frac{\ov\cL}{s^m})$ given by the pushout of this short exact sequence by $\phi_{m,\C^*}$, i.e. the class of the following short exact sequence:
\[
\begin{tikzcd}
0 \arrow[r]&
\ov\cL \arrow[r,"s"]\arrow[d,"\phi_{m,\C^*}"]&
\ov \cL \arrow[r,twoheadrightarrow]\arrow[d,"\phi_{m+1,\C^*}"] &
\uR \arrow[r]\arrow[d,equals] & 0 \\
0 \arrow[r]&
\frac{\ov\cL}{s^m} \arrow[r,"s"]&
\frac{\ov \cL}{s^{m+1}} \arrow[r,twoheadrightarrow]\arrow[ul, phantom, "\scalebox{1.25}{$\ulcorner$}" , very near start] &
\uR \arrow[r] & 0. 
\end{tikzcd}
\]
Since $\frac{\ov\cL}{s^{m+1}}$ is indecomposable, the sequence is not split. Therefore $\phi_{m,\C^*}\circ \pi_{\cL,\C^*}^\vee\neq 0$.

To see that $0\neq A_{\C^*}$, we can see that the image of $1\in H^0(\C^*;\R)$ is $\Im\frac{dz}{z} \neq 0 \in H^1\Gamma\left(\C^*,\cE^\bullet_{\C^*}\left(\Im\frac{dz}{z},m\right)\right) \cong H^1(\C^*;\frac{\ov\cL}{s^m})$, so $A_{\C^*}\neq 0$.
\end{proof}

\begin{proof}[Proof of Lemma~\ref{lem:pullbacks}]
We want to show that the following diagram is commutative:
\begin{equation}
\begin{tikzcd}[column sep = 4em]
f^{-1}\uR_{\C^*} \arrow[r,"f^{-1}\pi_{\cL,\C^*}^\vee{[1]}"]\arrow[d]&
f^{-1} \ov\cL_{\C^*}[1] \arrow[r,"f^{-1}\phi_{m,\C^*}{[1]}"] \arrow[d]&
f^{-1} \frac{\ov\cL_{\C^*}}{s^m}[1]\arrow[d]\\
\uR_{U} \arrow[r,"\pi_{\cL,U}^\vee{[1]}"] &
 \ov\cL_{U}[1] \arrow[r,"\phi_{m,U}{[1]}"] &
\frac{\ov\cL_{U}}{s^m}[1].
\end{tikzcd}
\end{equation}
For the right hand square, we see that it commutes because it amounts to tensoring the isomorphism $f^{-1}\ov\cL_{\C^*}\to \ov\cL_U$ with the map $R\to R_m$.

Let us apply $\cdot^\vee= R\Homm_R(\cdot,\ul R)$ to the left hand square. Note that, for locally constant sheaves of $R$-modules, pullback is restriction of scalars on the stalks from $R[\pi_1(\C^*)]$ to $R[\pi_1(U)]$, so there is a natural isomorphism of $\pi_1(U,x)$-modules:
\begin{align*}
(f^{-1}\Homm_R(\calF,\calG))_x &\cong (\Hom_{R}(\calF_{f(x)},\calG_{f(x)}))_{R[\pi_1(U)]} \\ &\cong \Hom_{R}((\calF_{f(x)})_{R[\pi_1(U)]},(\calG_{f(x)})_{R[\pi_1(U)]})
\\ &\cong\Hom_{R}((f^{-1}\calF)_x,(f^{-1}\calG)_x) \\ &\cong \Homm_{R}(f^{-1}\calF,f^{-1}\calG)_x.
\end{align*}
Applying this to resolutions by locally constant sheaves of free $R$-modules, we have that there is also a natural isomorphism $$f^{-1}R\Homm_R^\bullet(\calF^\bullet,\calG^\bullet)\cong R\Homm_R^\bullet(f^{-1}\calF^\bullet,f^{-1}\calG^\bullet)$$ for bounded complexes of locally constant sheaves. Therefore, showing that the left hand square above commutes is equivalent to showing that the following square commutes, obtained by applying $\cdot^\vee[1]$:
\[
\begin{tikzcd}
f^{-1}\uR_{\C^*} &
f^{-1}\cL_{\C^*} = f^{-1}\exp_!\uR_{\C} \arrow[l,"f^{-1}\pi_{\cL,\C^*}"] \\
\uR_{U} \arrow[u]&
\cL_{U} = \pi_!\uR_{U^f} \arrow[l,"\pi_{\cL,U}"] \arrow[u," \pi_!f_{\infty}^{-1}\cong f^{-1}\exp_!"'].
\end{tikzcd}
\]
We can see that on a stalk $(\cL_{U})_x$ for some $x\in U$, the generators $\delta_{(x,z)}$ (as in Remark~\ref{rem:oppositeDual}) are all mapped to $f^{-1}1\in f^{-1}\uR_{\C^*}$ via both paths. Therefore, the square commutes and so does its dual.

We consider now the following diagram, where $\eta_m$ is the map in Remark~\ref{remk:thickenedComplex}:
\[
\begin{tikzcd}[column sep = 5em]
f^{-1}\uR_{\C^*} \arrow[r,"f^{-1}(\eta_{m,\C^*}\circ A_{\C^*})"]\arrow[d] &
f^{-1} \cE^\bullet_{\C^*}(m)_{\log}[1] \arrow[d,"f^*"]&
f^{-1} (\ov\cL_{\C^*}/s^m)[1]\arrow[l,"\eta_{m,\C^*}"] \arrow[d,"\pi_!f^{-1}_{\infty}\cong f^{-1}\exp_!"] \\
\uR_{U} \arrow[r,"\eta_{m,U}\circ A_U"] &
 \cE_{U}^\bullet(m)_{\log} [1]&
\ov\cL_{U}/s^m[1]\arrow[l,"\eta_{m,U}"] 
\end{tikzcd}
\]

The arrow $f^*$ is induced by sending $\alpha\otimes s^j\in \cE_{\C^*}^i(m)_{\log}$ defined on an open set $V$ to $f^*\alpha\otimes s^j\in \cE_{U}^i(m)_{\log}$ defined on $f^{-1}(V)$. Just from the definitions we can see that the left hand square commutes, taking into account that $f^*z= f$.

Let us see that the right hand square commutes. Note that applying Section~\ref{ssAlex} to the special case $f= \Id_{\C^*}$, the stalk of $\ov\cL_{\C^*}$ at some $y\in \C^*$ is generated by sections of the form $\ov\delta_{\wt y}$ for $\wt y\in \C$ such that $e^{\wt y} = y$. Let us look at the stalk at $x\in U$, and let $z\in \C$ be such that $(x,z)\in U^f$. Then $\ov\delta_{z}\in (\ov\cL_{\C^*}/s^m)_{f(x)}$. Since $f_{\infty}$ maps $(x,z)\in U^f$ to $z\in \C$, the image of $f^{-1}\ov\delta_{z}$ in the stalk of $\ov\cL_{U}/s^m$ is $\ov\delta_{(x,z)}$. Now, let $\wt\log$ be a local branch of the logarithm such that $\wt\log(f(x)) = z$. Then, $(\Id_U,\wt\log\circ f)\colon U\to U^f\subset U\times \C$ is a local section of $\pi$ mapping $x$ to $(x,z)$ (it is only defined on a neighborhood $V$ of $x$ small enough that $f(V)$ is contained in the domain of $\wt\log$). Then, from Remark~\ref{remk:thickenedComplex}, $\eta_{m,\C^*}(\ov\delta_{z}) = \exp\left(-\frac{\Im\wt\log }{2\pi}\otimes \log(1+s)\right)$, so taking the image by $f^*$ we obtain
$$
f^*\circ \eta_{m,\C^*}(\ov\delta_z)=\exp\left(-\frac{\Im\wt\log \circ f}{2\pi}\otimes \log(1+s)\right)=\eta_{m,U}(\ov\delta_{(x,z)}),
$$
as desired.

We have shown that both diagrams in the lemma commute. Now, we claim that there is some $c\in \R\setminus \{0\}$ for which the vertical arrows in this commutative diagram coincide:
\[
\begin{tikzcd}[column sep = 6em]
f^{-1}\uR_{\C^*}[-1] \arrow[d,"c\cdot f^{-1}A_{\C^*}", shift left = 0.5ex] 
\arrow[d,"f^{-1}(\phi_{m,\C^*}\circ \pi_{\cL,\C^*}^\vee)"', shift right = 0.5ex]\arrow[r,"\sim"] &
\uR_{U}[-1] \arrow[d,"c\cdot A", shift left = 0.5ex] 
\arrow[d,"\phi_{m}\circ \pi_{\cL}^\vee"', shift right = 0.5ex] \\
f^{-1}\frac{\ov\cL_{\C^*}}{s^m}\arrow[r,"\sim"] &
\frac{\ov\cL_{U}}{s^m}.
\end{tikzcd}
\]

By Lemma~\ref{lem:proofInCircle}, there exists some $c\in \R\setminus \{0\}$ for which the left hand vertical arrows coincide. By the discussion above, this implies that the right hand vertical arrows coincide as well.
\end{proof}

This concludes the proof of Theorem \ref{thm:geomIsMHS}.
\end{proof}

\begin{remk}\label{rem:holyGrail}
Suppose that the action of $t$ on all  $\Tors_R H^{i+1}(U;\ov\cL)$ is unipotent and that $m$ is large enough that $(t-1)^m$ annihilates $\Tors_R H^{i+1}(U;\ov\cL)$. The map $A\colon \uR\to \ov\cL/s^m [1]$ used in the proof of Theorem~\ref{thm:geomIsMHS} induces a map in cohomology $H^i(A)\colon H^i(U;\R)\to H^{i+1}(U;\ov\cL/s^m)$ which factors (up to some $c\in \R\setminus \{0\}$) into the following $R$-linear morphisms of MHS:
\[
\begin{tikzcd}
H^i(U;\R) \arrow[rr,"cH^i(A)"] \arrow[dr,"H^i(\pi)"']& &
H^{i+1}\left(U;\frac{\ov\cL}{s^m}\right).\\
&
\Tors_R H^{i+1}(U;\ov\cL) \arrow[ur,"H^{i+1}(\phi_m)"',hookrightarrow]
\end{tikzcd}
\]
Indeed: by Lemma~\ref{lem:pullbacks}, there is a $c$ for which $cA = \phi_m\circ \pi_{\cL}^\vee$. By Proposition~\ref{prop:mapsAreEqual}, $\pi_{\cL}^\vee$ induces the map $H^i(\pi)$ in cohomology (which is a MHS morphism by Theorem~\ref{thm:geomIsMHS}); and by Remark~\ref{remk:mhsSummary}, $H^{i+1}(\phi_m)$ is the injection used to define the MHS on $\Tors_R H^{i+1}(U;\ov\cL)$, so in particular it is also a morphism of MHS.

From the relation between $A$ and $A_\R$ described in Lemma~\ref{lem:theMapA}, we have that $H^i(A_\R)=H^i(A)$, up to identifying the cohomology of the domain and target local systems with that of the global sections of the corresponding sheaf of cdgas. Recall that these identifications are morphisms of MHS by definition.
\end{remk}


	\chapter{The Geometric Map is a MHS Morphism: Consequences}\label{sec:consequences}

We start by identifying some properties of the maps induced by covering spaces in homology and cohomology.

Let $N\in\Z_{>0}$, and let $p\colon U_N\rightarrow U$ be the $N$-fold cover described in Section \ref{sscover}. The covering space $\pi\colon U^f\rightarrow U$ factors through $p\colon U_N\rightarrow U$ by $\pi_N\colon U^f\rightarrow U_N$. Recall that the deck transformation of $\pi_N$ is $t^N$.

\begin{prop}\label{prop:kerim}
The kernel of the map $H_j(\pi_N)\colon \Tors_R H_{j}(U;\cL)\to H_j(U_N;k)$ is $(t^N-1)\cdot \Tors_R H_{j}(U;\cL)$. The image of $H^j(\pi_N)\colon H^j(U_N;k)\rightarrow \Tors_R H^{j+1}(U;\ov\cL)$ is the $(t^N-1)$-torsion of the target, namely
$$
\Tors_{(t^N-1)} H^{j+1}(U;\ov\cL)\coloneqq\{a\in \Tors_R H^{j+1}(U;\ov\cL)\ \vert\ (t^N-1)a=0\}.
$$
\end{prop}

\begin{proof}
We use a similar argument to \cite[Assertion 5]{milnor}. We consider the short exact sequence of chain complexes of $R$-modules
$$
0\to C_\bullet(U^f;k)\xrightarrow{\cdot (t^N-1)}C_\bullet(U^f;k)\xrightarrow{(\pi_N)_*}C_\bullet(U_N;k)\to 0,
$$
where $t$ acts on both $U^f$ and $U_N$ by deck transformations. It gives us a long exact sequence in homology (the Milnor sequence), from which we get the exact sequence
$$
H_j(U;\cL)\xrightarrow{\cdot(t^N-1)}H_j(U;\cL)\xrightarrow{H_j(\pi_N)} H_j(U_N;k).
$$
This in turn tells us that the following sequence is exact
$$
\Tors_R H_j(U;\cL)\xrightarrow{\cdot(t^N-1)}\Tors_R H_j(U;\cL)\xrightarrow{H_j(\pi_N)} H_j(U_N;k),
$$
which implies the assertion about the kernel of $H_j(\pi_N)$. We take $\Hom_k(\cdot,k)$ of the above exact sequence and get the exact sequence
\[
H^j(U_N;k)\xrightarrow{H^j(\pi_N)} \Hom_k(\Tors_R H_j(U;\cL),k)\xrightarrow{\cdot(t^N-1)} \Hom_k(\Tors_R H_j(U;\cL),k).
\]
from which we get the claim about the image of $H^j(\pi_N)$.
\end{proof}

The following corollary is a direct consequence of Proposition \ref{prop:kerim}. It says that, if the monodromy action is semisimple, the MHS on $\Tors_R H^{j+1}(U;\ov\cL)$ (respectively, $\Tors_R H_{j}(U;\cL)$) is determined by $H^j(\pi_N)$ and Deligne's MHS on $H^j(U_N;k)$ (respectively, by $H_j(\pi_N)$ and Deligne's MHS on $H_j(U_N;k)$, which is the dual MHS of $H^j(U_N;k)$).

\begin{cor}\label{cor:surj}\index{semisimple|(}	
Suppose that the $t$-action on $\Tors_R H^{j+1}(U;\ov \cL)$ (equivalently, on \linebreak $\Tors_R H_{j}(U;\cL)$) is semisimple for some $j$, and let $N$ be such that the action of $t^N$ on $\Tors_R H_j(U;\cL)$ is unipotent. Then $H^j(\pi_N)\colon H^j(U_N;k)\to \Tors_R H^{j+1}(U;\ov\cL)$ is a surjective morphism of MHS. Equivalently, $H_j(\pi_N)\colon \Tors_R H_{j}(U;\cL)\rightarrow H_j(U_N;k)$ is an injective morphism of MHS.
\end{cor}

\begin{remk}\label{rem:decomposition}
Let $A$ be a torsion $R$-module which is annihilated by $t^N-1$. Then, there is a canonical isomorphism
$$
A_1\oplus A_{\neq 1}\cong A
$$
where $A_1=\ker(A\xrightarrow{\cdot(t-1)}A)$ and $A_{\neq 1}=\ker(A\xrightarrow{\cdot(t^{N-1}+\ldots+t+1)}A)$.
\end{remk}

\begin{cor}[of Corollary \ref{cor:surj}]\label{cor:t}
Suppose that the $t$-action on $\Tors_R H^{j+1}(U;\ov \cL)$ (equiv. on $\Tors_R H_{j}(U;\cL)$) is semisimple for some $j$. Then, multiplication by $t$ induces a MHS morphism from $\Tors_R H^{j+1}(U;\ov \cL)$ (equiv. on $\Tors_R H_{j}(U;\cL)$) to itself. In particular, we have MHS isomorphisms
$$
\Tors_R H^{j+1}(U;\ov \cL)\cong\Tors_R H^{j+1}(U;\ov \cL)_1\oplus\Tors_R H^{j+1}(U;\ov \cL)_{\neq 1}
$$
and
$$
\Tors_R H_j(U;\cL)\cong\Tors_R H_j(U;\cL)_1\oplus\Tors_R H_j(U;\cL)_{\neq 1}
$$
\end{cor}

\begin{proof}
We write the proof in the cohomology case. Let $N$ be such that the action of $t^N$ on $\Tors_R H_j(U;\cL)$ is unipotent. Since $t$ acts on $U_N$ by deck transformations, and $U_N\rightarrow U$ is an algebraic covering map, we get that the deck transformation corresponding to $t$ is actually an algebraic map, so multiplication by $t$ on $H^j(U_N)$ is a morphism of MHS. The result follows from the commutativity of the following diagram and the fact that all the arrows except for the bottom one are known to be morphism of MHS.
$$
\begin{tikzcd}
H^j(U_N;k)\arrow[r,"t"]\arrow[d, "H^j(\pi_N)",twoheadrightarrow]& H^j(U_N;k)\arrow[d, "H^j(\pi_N)",twoheadrightarrow]\\
\Tors_R H^{j+1}(U;\ov \cL)\arrow[r,"t"] & \Tors_R H^{j+1}(U;\ov \cL).
\end{tikzcd}
$$
The last statement about the direct sum decomposition follows from Remark \ref{rem:decomposition}.
\end{proof}

The converse to Corollary~\ref{cor:t} also holds.

\begin{prop}\label{prop:semisimpleIsEasyConverse}
Suppose that the $t$-action on $\Tors_R H^{j+1}(U;\ov \cL)$ (equivalently, on \linebreak $\Tors_R H_{j}(U;\cL)$) is a MHS morphism for some $j$. Then, the action of $t$ is semisimple.
\end{prop}

\begin{proof}
Let $H =\Tors_R H^{j+1}(U;\ov \cL)$. Let $N$ be such that $t^N$ is unipotent (and a morphism of MHS on $H$). By Corollary~\ref{alexandermhs}, $\log(t^N)$ is a morphism $H\to H(-1)$ (viewed as a power series around $t^N=1$). Therefore, we have a MHS $H$ with a nilpotent morphism of MHS $\theta \coloneqq \log(t^N)\colon H\to H(-1)$ such that $\exp(\theta)\colon H\to H$ is also a MHS morphism. We want to prove that $\theta=0$. For this it is enough to show that $\ker \theta = \ker \theta^2$, so we can assume that $\theta^2 (H) = 0$. Therefore, $\exp(\theta) = 1+\theta$. Since $1+\theta$ is a MHS morphism, then $\theta$ is as well.

Therefore, we have to show that if $\theta:H\to H$ and $\theta:H\to H(-1)$ are MHS morphisms, then $\theta=0$. This is a direct consequence of \cite[Corollary 3.6]{peters2008mixed}, where it is shown that a MHS morphism preserves the weight filtration strictly. Consider the image MHS $\theta (H)\subseteq H$. Then,
\begin{multline*}
W_k(\theta (H))\coloneqq W_kH\cap \theta (H) = \theta (W_kH) = \theta (W_{k+2}H) \quad\forall k \\ \Longrightarrow W_k(\theta (H)) = W_{k+1}(\theta (H))\quad\forall k.
\end{multline*}
Since the filtration $W_k$ is finite, for $k\gg 0$, $0=W_{-k}(\theta(H)) = W_{k}(\theta (H)) = \theta (H)$.
\end{proof}

Corollary~\ref{cor:t} and Proposition~\ref{prop:semisimpleIsEasyConverse} together give us Theorem \ref{tsemisimpleIntro} in the introduction, which is stated in homological notation.
\index{semisimple|)}

\section{Relationship with the Cup and Cap Products.}\label{ss:cupcap}

Consider the real component $j_*\calE^\bullet_U\left(\Im\, \frac{df}{f},m\right)$ of the $\R$-mixed Hodge complex of sheaves $\calH dg^\bullet(X\,\log D)\left(\frac{1}{i}\frac{df}{f},m\right)$ used to endow the torsion part of the Alexander modules with a MHS. The differential of $\calE^\bullet_U\left(\Im\, \frac{df}{f},m\right)$ involves a wedge with $\Im\, \frac{df}{f}	$. Wedging real differential forms corresponds to cup products in cohomology, which loosely suggests a relation between Alexander modules and the cup (and cap) products arising from the thickened complexes. In this section, we explore this relation and make it explicit.

\begin{prop}\label{prop:cup}
Let $N$ be such that the action of $t^N$ on $\Tors_R H^{j+1}(U;\ov\cL)$ is unipotent. Let $\gen\in H^1(\C^*;\Z)\cong \Z$ be a generator. We consider $\gen$ as an element of $H^1(\C^*;k)$, and $f_N^*(\gen)\in H^1(U_N;k)$. There exists an arrow (the dashed one) that makes the following diagram a commutative diagram of morphisms of MHS (and of $R$-modules, if $t$ acts on $U_N$ by deck transformations). Up to multiplication by a nonzero constant, it is induced by $\pi_{\ov\cL}\colon \ov\cL\to\ul k$.
$$
\begin{tikzcd}
H^j(U_N;k)\arrow[rr, " f_N^*(\gen)\smile"]\arrow[dr, "H^j(\pi_N)"]&\ & H^{j+1}(U_N;k)(1).\\
\ &\Tors_R H^{j+1}(U;\ov\cL)\arrow[ru, dashed]&\ 
\end{tikzcd}
$$
Here $(1)$ denotes the Tate twist, and $f_N^*(\gen)\smile $ denotes the cup product by $f_N^*(\gen)$. Moreover, if the $t$-action on $\Tors_R H^{j+1}(U;\ov \cL)$ is semisimple, the dashed arrow is injective.
\end{prop}

\begin{proof}
To make notation simpler, we may replace $U_N$ by $U$, $f_N$ by $f$, $\pi_N$ by $\pi$ and assume that the action of $t$ on $\Tors_R H^*(U;\cL)$ is unipotent. Let $m\in \N$ be the minimum natural number such that $(t-1)^m$ annihilates $\Tors_R H^{i+1}(U;\calL)$. Note that the $t$-action on $\Tors_R H^{i+1}(U;\calL)$ is semisimple if and only if $m=1$.

We continue the rest of the proof in the case when $k=\R$, from which the case $k=\Q$ will follow. Let $A_{\calH dg}\colon \Hdg{X}{D}\rightarrow \Hdg{X}{D}\left(\frac{1}{i}\frac{df}{f}, m \right)[1]$ be as in Lemma \ref{lem:theMapA}, which induces a map of MHS $$H^j(A_\R)\colon H^j(U)\rightarrow H^{j+1}\Gamma(U;\cE_U^\bullet(\Im df/f,m)).$$

Note that $\Im df/f=f^*(\Im dz/z)$, so there exists a non-zero constant $b\in\R^*$ such that $f^*(\gen)=b\Im df/f$. Let $\phi_{m-1,1}^*$ be the map induced in cohomology, corresponding to the map $\cE_U^\bullet(\Im df/f, m)\rightarrow \cE_U^\bullet(\Im df/f, 1)\cong \cE_U^\bullet$ coming from the projection $R_m\rightarrow R_1$. Note that $\phi_{0,1}^*$ is the identity map. By Lemma~\ref{inducedmhsmaps}, $\phi_{m-1,1}^*$ is a MHS morphism. Let $c\in\R^*$ be as in Remark \ref{rem:holyGrail}. By Remark \ref{rem:holyGrail}, we get the following commutative diagram of morphisms of MHS, which finishes the proof.
\[
{\footnotesize
\begin{tikzcd}[column sep = 7em]
H^j(U)\cong H^j\Gamma(U;\cE_U^\bullet)\arrow[r,bend left = 10,end anchor={[xshift=-8ex]}," f^*(\gen)\smile =b\cdot( \Im df/f \smile )"]\arrow[dr, "H^j(A_\R)", end anchor={[xshift=-7ex]}]\arrow[d, "H^j(\pi)"'] & H^j\Gamma(U;\cE_U^\bullet(\Im df/f,1)[1])\cong H^{j+1}(U)(1)\\
\Tors_R H^{j+1}(U;\ov\cL)\arrow[r,"\frac{1}{c}\cdot H^{j+1}(\phi_m)"',hookrightarrow] & H^j\Gamma(U;\cE_U^\bullet(\Im df/f,m)[1])\cong H^{j+1}(U;\ov\cL/s^m)\arrow[u, "b\cdot\phi_{m-1,1}^*{[}1{]}"'].
\end{tikzcd}
}
\]
\end{proof}

\begin{remk}\label{rem:cup}
We have not used it in the proof above, but the cup product $H^m(Y;k)\otimes H^1(Y;k)\rightarrow H^{m+1}(Y;k)$ is a morphism of MHS for every algebraic variety $Y$ (\cite[Corollary 5.45]{peters2008mixed}). The Tate twist in the target of the horizontal map in the commutative diagram above agrees with the fact that $\gen$ has weight $2$ in $H^1(\C^*;\R)$ (the MHS on $H^1(\C^*;\R)$ is pure of type $(1,1)$).
\end{remk}

The following result follows from Corollary \ref{cor:surj} and Proposition \ref{prop:cup}. The result in homology, which already appeared in the introduction as Theorem~\ref{finiteIntro}A, follows from taking $\Hom_k(\cdot,k)$ in the commutative diagram of Proposition \ref{prop:cup}.

\begin{cor}\label{cor:cup}
Let $N$ be such that  $t^N$ acts unipotently on $\Tors_R H^{j+1}(U;\ov\cL)$. Let $\gen\in H^1(\C^*;\Z)\cong \Z$ be a generator. We consider $\gen$ as an element of $H^1(\C^*;k)$, and $f_N^*(\gen)\in H^1(U_N;k)$. Suppose that the $t$-action on $\Tors_R H_j(U;\cL)$ (equivalently, on $\Tors_R H^{j+1}(U,\ov\cL)$) is semisimple. Then,
\begin{itemize}
\item $\Tors_R H^{j+1}(U,\ov\cL)$ is isomorphic, both as MHS and as $R$-modules, to the image of the cup product map $$H^j(U_N;k)\xrightarrow{ f_N^*(\gen)\smile } H^{j+1}(U_N;k)(1).$$

\item $\Tors_R H_j(U;\cL)$ is isomorphic, both as MHS and as $R$-modules, to the image of the cap product map 
$$
H_{j+1}(U_N;k)(-1)\xrightarrow{\frown f_N^*(\gen) } H_j(U_N;k)
$$
\end{itemize}
Here $H^*(U_N;k)$ are endowed with Deligne's MHS, and $H_*(U_N;k)$ are endowed with the dual of Deligne's MHS in cohomology.
\end{cor}


\section{Relationship with Deligne's MHS on the Generic Fiber.}\label{ss:genFiber}\index{fiber!generic|(}
Let $F$ be a generic fiber of $f\colon U\rightarrow \C^*$ (as in Definition~\ref{def:genericFiber}), and let $N$ be chosen such that the action of $t^N$ on $\Tors_R H^{j+1}(U;\ov\cL)$ is unipotent. Let $i\colon F\hookrightarrow U$ be the inclusion. As in Section~\ref{genfiber}, $F$ lifts to $U_N$ and $U^f$ via maps $i_N$ and $i_\infty$, making the following diagram commutative, where the vertical arrows are covering space maps.
$$
\begin{tikzcd}
\ & U^f\arrow[d, "\pi_N"]\arrow[dd, "\pi", bend left=90]\\
\ & U_N\arrow[d,"p"]\\
F\arrow[r,"i",hook]\arrow[ru, "i_N" pos=1, hook]\arrow[ruu, "i_{\infty}",hook] & U.
\end{tikzcd}
$$
In homology, the composition $i_N=\pi_N\circ i_{\infty}$ factors through $\Tors_R H_j(U^f;k)\cong \Tors_R H_j(U;\cL)$. Hence, we get
\begin{equation}\label{eq:fiberhom}
\begin{tikzcd}
H_j(F;k)\arrow[r,"H_j(i_\infty)",two heads]\arrow[rr, "H_j(i_N)", bend right = 20] & \Tors_R H_j(U;\cL)\arrow[r, "H_j(\pi_N)"] & H_j(U_N;k),
\end{tikzcd}
\end{equation}
and taking $\Hom_k(\cdot,k)$ and the appropriate identifications, we get
\begin{equation}\label{eq:fibercohom}
\begin{tikzcd}
H^j(U_N;k)\arrow[r,"H^j(\pi_N)"]\arrow[rr, "H^j(i_N)", bend right = 20] & \Tors_R H^{j+1}(U;\ov\cL)\arrow[r, "H^j(i_\infty)",hook] & H^j(F;k).
\end{tikzcd}
\end{equation}
Here $H_j(i_\infty)$ and $H^j(i_\infty)$ are surjective and injective respectively by Proposition~\ref{genf} and Remark~\ref{remk:tube}.

Note that $H_j(i_N)$ and $H^j(i_N)$ are MHS morphisms, since $i_N$ is an algebraic map. Recall that $H_j(\pi_N)$ and $H^j(\pi_N)$ are MHS morphisms by Theorem \ref{thm:geomIsMHS}. If the monodromy is semisimple, Corollary \ref{cor:surj} tells us that $H_j(\pi_N)$ and $H^j(\pi_N)$ are injective and surjective respectively. The information in this paragraph and the commutativity of the two diagrams above tell us that $H_j(i_\infty)$ and $H^j(i_\infty)$ are also MHS morphisms.

\begin{cor}\label{cor:fiber}
Let $N$ be such that the action of $t^N$ on $\Tors_R H^{j+1}(U;\ov\cL)$ is unipotent. Suppose that the $t$-action on $\Tors_R H_j(U;\cL)$ (equivalently, on $\Tors_R H^{j+1}(U,\ov\cL)$) is semisimple. Then, we have the following commutative diagrams, where all the arrows are morphisms of MHS.\index{semisimple}
$$
\begin{tikzcd}
H_j(F;k)\arrow[r,"H_j(i_\infty)",two heads]\arrow[rr, "H_j(i_N)", bend right = 20] & \Tors_R H_j(U;\cL)\arrow[r, "H_j(\pi_N)", hook] & H_j(U_N;k),
\end{tikzcd}
$$
and 
$$
\begin{tikzcd}
H^j(U_N;k)\arrow[r,"H^j(\pi_N)", two heads]\arrow[rr, "H^j(i_N)", bend right = 20] & \Tors_R H^{j+1}(U;\ov\cL)\arrow[r, "H^j(i_\infty)",hook] & H^j(F;k).
\end{tikzcd}
$$
Therefore, $\Tors_R H_j(U;\cL)$ (resp. $\Tors_R H^{j+1}(U;\ov\cL)$) is isomorphic as MHS to the image of the MHS morphism $H_j(i_N)$ (resp. $H^j(i_N)$).
\end{cor}

Note that Theorem~\ref{finiteIntro}B is an immediate consequence of the statement in homology of Corollary ~\ref{cor:fiber}.

\begin{remk}\label{rem:indepF}
In general, the MHS on $F$ depends on the specific choice of the fiber, even if $F$ is generic. Corollary \ref{cor:fiber} above tells us that the map induced by a lift of the inclusion of the generic fiber $F$ of $f$ into the infinite cyclic cover $U^f$ induces a MHS morphism in both homology and cohomology for \textbf{any} choice of generic fiber $F$.
\end{remk}

\begin{cor}\label{cor:fiberConverse}
The $t$-action on the module $\Tors_R H_j(U;\cL)$ (equivalently, on \linebreak $\Tors_R H^{j+1}(U;\ov\cL)$) is semisimple if and only if for any generic fiber $F\subset U^f$, the induced map in homology $H_j(i_\infty)\colon H_j(F;k)\to \Tors_R H_j(U^f;k)$ is a MHS morphism (equivalently, the dual map $H^j(i_\infty)\colon \Tors_R H^{j+1}(U;\ov\cL)\to H^j(F;k)$ is a MHS morphism).
\end{cor}
\begin{proof}
The forward direction is Corollary~\ref{cor:fiber}. For the converse, consider an inclusion or a fiber $i_\infty\colon F\to U^f$ and the deck transformation $t\colon U^f\to U^f$. Consider the induced maps in homology:
\[
\begin{tikzcd}
H_i(F;k) \arrow[r,"i_\infty",twoheadrightarrow]\arrow[d,equals]&
\Tors_R H_i(U^f;k)\subseteq H_i(U^f;k)\arrow[d,"t",shift right = 2ex] \\
H_i(F;k) \arrow[r,"t\circ i_\infty",twoheadrightarrow]&
\Tors_R H_i(U^f;k)\subseteq H_i(U^f;k).
\end{tikzcd}
\]
The horizontal arrows are surjections by Proposition~\ref{genf}, and by hypothesis they are MHS morphisms. Therefore, the map ``$t$'' must be a MHS morphism as well. By Proposition~\ref{prop:semisimpleIsEasyConverse}, this implies that $t$ is semisimple.
\end{proof}

Let $f\in\C[x_1,\ldots,x_n]$ be a weighted homogeneous polynomial, and let $U=\linebreak\C^n\setminus\{f=0\}$. In this case, we have a global Milnor fibration $f\colon U\to\C^*$. Moreover, assume that $f\colon U\to\C^*$ induces an epimorphism on fundamental groups, which happens if and only if the greatest common divisor of the exponents of the distinct irreducible factors of $f$ is $1$. 
Let $F$ be a fiber of $f\colon U\to\C^*$, and let $i_\infty\colon F\hookrightarrow U^f$ be a lift of the inclusion $i\colon F\hookrightarrow U$. Note that, since $f\colon U\rightarrow \C^*$ is a fibration, $i_\infty$ is a homotopy equivalence, so it induces isomorphisms $H_j(F;k)\rightarrow H_j(U^f;k)$ for all $j$, which are compatible with the $t$-action (see Lemma~\ref{lem:fiberMonodromy}). Moreover, the $t$-action on $F$ comes from an algebraic map $F\rightarrow F$ of finite order, so the $t$-action on $H_j(F)$ and $H^j(F)$ is semisimple. Applying Corollary \ref{cor:fiber}, we have arrived at the following result.

\begin{cor}[The Alexander modules recover the MHS on the global Milnor fiber]\label{cor:quasihom}
Let $f\in\C[x_1,\ldots,x_n]$ be a weighted homogeneous polynomial, and let $U=\C^n\setminus\{f=0\}$. Assume that $f\colon U\rightarrow \C^*$ induces an epimorphism in fundamental groups. Let $F$ be a fiber of $f\colon U\rightarrow \C^*$, and let $i_\infty\colon F\hookrightarrow U^f$ be a lift of the inclusion $i\colon F\hookrightarrow U$. The map
$$
\Tors_R H^{j+1}(U;\ov\cL)\rightarrow H^j(F;k)
$$
induced by $i_\infty$ is a MHS isomorphism, where $H^j(F;k)$ is endowed with Deligne's MHS.
\end{cor}
\index{fiber!generic|)}

\section{Dimca-Libgober and Liu's Mixed Hodge Structure}\label{ss:DLandLiu}
In this section, we recall other MHS on Alexander modules in the literature, and we show that the MHS that we consider in this paper generalizes them.

Let $F$ be a generic fiber of $f\colon U\rightarrow \C^*$, as in Definition~\ref{def:genericFiber}. Let $\pi\colon U^f\rightarrow U$ be the covering space induced by $f_*\colon \pi_1(U)\twoheadrightarrow \Z$, Let $i\colon F\hookrightarrow U$ be the inclusion, and let $i_{\infty}\colon F\hookrightarrow U^f$ be a lift of $i$, as in Section~\ref{genfiber}.

We make the following assumption in this section.
\begin{assumption} Fix $j\geq 0$, and suppose that $\pi^*\colon H^j(U;\Q)\rightarrow H^j(U^f;\Q)$ is an epimorphism, and $(i_{\infty})^*\colon H^j(U^f;\Q)\rightarrow H^j(F;\Q)$ is a monomorphism.
\label{a1}
\end{assumption}

Endow $H^j(U;\Q)$ and $H^j(F;\Q)$ with Deligne's MHS. Under Assumption \ref{a1}, we have the following commutative diagram:

\begin{center}
 \begin{tikzcd}
 H^j(U;\Q) \arrow[rd, "i^*"] \arrow[r, two heads, "\pi^*"] & H^j(U^f;\Q)\arrow[d, hook,"(i_{\infty})^*"] \\
\ & H^j(F;\Q)
\end{tikzcd}
\end{center}

\begin{remk}
\label{nota}
Since $i^*$ is a map of MHS, we can endow $H^j(U^f;\Q)$ with a unique MHS such that both $\pi^*$ and $(i_{\infty})^*$ are maps of MHS. Indeed, the image of $i^*$ is a sub-MHS of $H^j(F;\Q)$ that is identified through $(i_{\infty})^*$ with $H^j(U^f;\Q)$.
\end{remk}

\begin{set}\label{set:DL}
In their paper \cite{DL}, Dimca and Libgober consider the following setting: Let $W'=W_0'\cup\ldots\cup W_m'$ be a hypersurface arrangement in $\bP^N$ for $N > 1$, where $W_j'$ is a hypersurface of degree $d_j$ defined by the equation $g_j=0$, where $g_j$ is a reduced homogeneous polynomial. Let $Z\subset \bP^N$ be a smooth complete intersection of dimension $n>1$ which is not contained in $W'$, and let $W_j =W_j'\cap Z$ for $j=0,...,m$ be the corresponding hypersurface in $Z$. Let $W = W_0\cup\ldots\cup W_m$ denote the corresponding hypersurface arrangement in $Z$. Moreover, assume that the following three conditions hold:
\begin{itemize}
\item All the hypersurfaces $W_j$ are distinct, reduced and irreducible; moreover $W_0$ is smooth.
\item The hypersurface $W_0$ is transverse in the stratified sense to $T = W_1\cup\ldots\cup W_m$, i.e., if $\mathcal{S}$ is a Whitney regular stratification of $T$ , then $W_0$ is transverse to any stratum $S\in\mathcal{S}$.
\item $d_0$ divides the sum $\sum_{j=1}^m d_j$ , say $dd_0 = \sum_{j=1}^m d_j$.
\end{itemize}
The complement $M= Z \backslash W_0$ is a smooth affine variety. Let $Y$ be the hypersurface $Y=M\cap T$ in $M$. The authors consider the variety $U =M\backslash Y$, with the map
$$
\f{f}{U}{\C^*}{x}{\frac{g_1(x)\cdot\ldots\cdot g_m(x)}{g_0(x)^d}}.
$$
By \cite[Theorem 1.2]{DL}, the generic fiber of this map is connected.
\end{set}

\begin{remk}
Note that, in this setting, $U$ is an affine variety of dimension $n$, so it has the homotopy type of a finite $n$-dimensional CW-complex, which implies that $H^j(U^f,\Q)$ is a free $R$-module for $j=n$, and $0$ if $j>n$. Moreover, by \cite[Corollary 1.6]{DL}, $H^j(U^f,\Q)$ is annihilated by $t^d-1$ (and therefore is $R$-torsion and semisimple) for $j<n$.
\label{affine}
\end{remk}

Let $p\colon U_d\rightarrow U$ be the $d$-fold cover of $U$ obtained as the pullback by $f$ of the map $g\colon \C^*\rightarrow \C^*$ given by $z\mapsto z^d$, as exemplified by the following commutative diagram:
\begin{center}
 \begin{tikzcd}
U_d \arrow[d, "p"] \arrow[r, "f_d"] & \C^*\arrow[d, "g"] \\
U \arrow[r, "f"] & \C^*.
\end{tikzcd}
\end{center}
Note that $F$ is the generic fiber of $f_d$ as well. Let $\pi_d\colon U^f\rightarrow U_d$ be the covering space such that $\pi=p\circ \pi_d$, and let $i_d$ be a lift of $i$ to $U_d$. Note that $i_{\infty}$ can be considered to be a lift of $i_d$.

\begin{remk}[Dimca and Libgober's MHS]
By \cite[Theorem 1.5]{DL}, we have that $$(\pi_d)^*\colon H^j(U_d;\Q)\rightarrow H^j(U^f;\Q)$$ is an epimorphism for all $j<n$. By \cite[Corollary 1.3]{DL}, $(i_{\infty})^*\colon H^j(U^f;\Q)\rightarrow H^j(F;\Q)$ is an isomorphism if $j<n-1$ and a monomorphism if $j=n-1$. This means that if we switch $U$ for $U_d$ and $\pi$ for $\pi_d$, we are under the conditions of Assumption \ref{a1}. Hence, there is a unique MHS on $H^j(U^f;\Q)$ that makes $\pi_d^*$ and $(i_{\infty})^*$ into morphisms of MHS for $j<n$, and this MHS is the one that Dimca and Libgober construct in \cite{DL} (recall Remark \ref{nota}).
\label{remDL}
\end{remk}

In \cite{Liu}, Liu considers a particular case of the situation described above, which was initially studied in \cite{Max06}. In his setting, $N=n$, $W_0'$ is the hyperplane at infinity and $Z=\bP^{n}$. Then, $U$ is the affine complement of a hypersurface transversal at infinity, and $f$ induces the linking number homomorphism in fundamental groups. That is, the image by $f_*$ of a positively oriented meridian around each of the irreducible components $W_j$ of $X$ gets mapped to $1\in\Z$, for $j=1,\ldots,m$. He recovers Dimca and Libgober's MHS with a completely different approach, using nearby cycles.

\begin{cor}[of Corollary \ref{cor:fiber}]\label{cor:recoverMHS}
Suppose $U$ and $f$ are as in Setting \ref{set:DL}. Under the isomorphism given by Corollary~\ref{isocohom}, the MHS on $\Tors_R H^{j+1}(U;\ov\calL)$ constructed in this paper coincides with the MHS on $H^j(U^f;\Q)$ constructed in \cite{DL} , for $0\leq j\leq n-1$.
\end{cor}


\section{Bounding the Weights}
In this section we prove the following result, which was already stated in the introduction as Theorem~\ref{boundsIntro} (using homological instead of cohomological notation).

\begin{thm}\label{thm:boundedWeights}\index{filtration!weight}
If $k\notin [i,2i]\cap [i,2n-2]$, then
\[
\Gr^W_k \Tors_RH^{i+1}(U;\ov\cL) = 0,
\]
where $n$ is the dimension of $U$.
\end{thm}

Before the proof, we will discuss some consequences. In the context of Alexander modules realized by an algebraic map (as considered in this paper), the following corollary improves on \cite[Proposition 1.10]{BudurLiuWang}, as the bound below is the ceiling of half the bound in loc. cit.

\begin{cor}\label{cor:jordan}
Every Jordan block of the action of $t$ on $\Tors_R H_i(U^f;k)$ (respectively, on $\Tors_R H^{i+1}(U;\ov\cL)$) has size at most $\min\{ \lceil (i+1)/2\rceil, n- \lfloor (i+1)/2\rfloor\}$. (Here, Jordan blocks are considered over the algebraic closure of the field $k$.)
\end{cor}
\begin{proof}
Let $K=\min\{ \lceil (i+1)/2\rceil, n-\lfloor(i+1)/2\rfloor\}$. We will start by showing the corollary in the case where the action of $t$ is unipotent. By Corollary~\ref{alexandermhs}, $\log(t) = \sum_{j=0}^\infty (-1)^{j-1}\frac{(t-1)^j}{j}$ is a morphism of MHS into the $(-1)$st Tate twist, i.e. $$\log(t)W_k\Tors_R H^{i+1}(U;\ov\cL)\subseteq W_{k-2}\Tors_R H^{i+1}(U;\ov\cL).$$ Theorem~\ref{thm:boundedWeights} implies: 
\[
\log(t)^{K}\Tors_R H^{i+1}(U;\ov\cL) = \log(t)^{K}W_{\min\{2i,2n-2\}}\Tors_R H^{i+1}(U;\ov\cL)\subseteq\]\[
\subseteq W_{\min\{2i-2K,2n-2-2K\}}\Tors_R H^{i+1}(U;\ov\cL)\subseteq W_{i-1}\Tors_R H^{i+1}(U;\ov\cL)= 0.
\] 
This shows that $\log(t)^{K}$ annihilates the whole module $\Tors_R H^{i+1}(U;\ov\cL)$. Since $\log(t)/(t-1)$ is a unit in $k[[t-1]]$, this implies that $(t-1)^{K}=0$ as well. Therefore, all its Jordan blocks have size at most $K$, as desired.

The general case can be reduced to the unipotent case as described in Section \ref{sscover} through the following isomorphism of $R(N)$-modules, where $R(N) = k[t^{\pm N}]$
\[
\Tors_{R}H^{i+1}(U;\cL)=\Tors_{R(N)}H^{i+1}(U;\cL) \cong\Tors_{R(N)}H^{i+1}(U_N;\cL_N).
\] 
The action of $t^N$ is then unipotent on $\Tors_{R(N)}H^{i+1}(U_N;\cL_N) $, so it has Jordan blocks of size at most $K$, i.e. $(t^N-1)^{K}=0$. Therefore, the minimal polynomial of $t$ acting on $\Tors_{R}H^{i+1}(U;\cL)$ divides $(t^N-1)^K$, so the multiplicity of any root is at most $K$.
\end{proof}

Using the canonical isomorphisms of $R$-modules $$\Tors_R H^{i+1}(U;\ov\cL)\cong \Hom_k(\Tors_R H_i(U;\cL),k)$$ (obtained from the composition $\Res^{-1} \circ \UCT_R$), together with  $\Tors_R H_i(U;\cL)\cong \linebreak \Tors_R H_i(U^f;k)$, we obtain the analogous statement to Corollary \ref{cor:jordan} for the $t$ action by deck transformations on $\Tors_R H_i(U^f;k)$. In particular, if $i=1$, we obtain the following.

\begin{cor}\label{cor:semisimple}\index{semisimple}
If $U$ is an algebraic variety and $f\colon U\to \C^*$ is an algebraic map inducing epimorphisms on fundamental groups and inducing an infinite cyclic cover $U^f$, then the $t$-action on $H_1(U^f;k)$ is semisimple.
\end{cor}

\begin{remk}\label{remk:A1Semisimple}
Let $f\colon \C^2\rightarrow \C$ be a polynomial function such that $f^{-1}(0)$ is reduced and connected. Let $U=\C^2\setminus f^{-1}(0)$, and we use the same letter $f$ to denote the induced map $f\colon U\rightarrow \C^*$. By \cite[Corollary 1.7]{DL}, the action of $t$ on $H_1(U^f;k)$ is semisimple. The corollary above generalizes this result, not only to algebraic varieties $U$ that are not affine connected curve complements, but also to connected curve complements in which the corresponding $f$ is given by a non-reduced polynomial, or even a rational function.
\end{remk}

\begin{remk}
The bound in Theorem~\ref{thm:boundedWeights} is sharp in the sense that for any $k\in [i,2i]\cap [i,2n-2]$, $\Gr^W_k \Tors_RH^{i+1}(U;\ov\cL)$ can be nonzero. If $U=F\times \C^*$, where $F$ is a smooth algebraic variety and $f$ is the projection, then $U^f\cong F\times \C$, which is homotopy equivalent to $F$. Therefore, it is straightforward to check that $t-1$ annihilates the Alexander module, and in this case the map $F\to U^f$ induces a MHS isomorphism in cohomology by Corollary~\ref{cor:fiber}. 

For example, we can let $F = E^{n_1}\times (\C^*)^{n_2}$, where $E$ is an elliptic curve. In this case, using the K\"unneth formula (which is a MHS morphism by \cite[Theorem 5.44]{peters2008mixed}), we will have:
\begin{align*}
\Gr^W_{k} H^{i}(F;k)\neq 0 &\Longleftrightarrow
\left\{
\begin{array}{rl}
\Gr^W_{2i-k} H^{2i-k}(E^{n_1};k) & \neq 0;\\
\Gr^W_{2k-2i} H^{k-i}((\C^*)^{n_2};k) &\neq 0
\end{array}
\right\} \\
& \Longleftrightarrow k \in [i,2i] \cap [2i-2n_1,n_2+i]. 
\end{align*}
Varying $n_1,n_2\in \Z_{\ge 0}$ with $n_1+n_2 = \dim F = n-1$, one can have nonzero graded pieces for any weight $k\in [i,2i]\cap [i,2n-2]$.
\end{remk}

\begin{proof}[Proof of Theorem~\ref{thm:boundedWeights}]

First, we claim that it suffices to prove the theorem assuming that $t$ acts unipotently on $\Tors_R H^{i+1}(U;\ov\cL)$. For a general $U$, we let $N$ be such that $t^N$ acts unipotently on $\Tors_R H^{i+1}(U;\ov\cL)$ and we proceed as in Lemma~\ref{233} to obtain $U_N,f_N,\cL_N$. By Remark~\ref{remk:mhsSummary}, there is a MHS isomorphism $\Tors_R H^{i+1}(U;\ov\cL) \cong \Tors_R H^{i+1}(U_N;\ov\cL_N)$. Therefore, if the statement of Theorem~\ref{thm:boundedWeights} holds for $(U_N,f_N)$, it will hold for $(U,f)$. From now on, we will assume that $U=U_N$, and that $t$ acts unipotently on $\Tors_R H^{i+1}(U;\ov\cL)$.

Let $s=t-1\in R$, and let $R_m = \frac{R}{(s^m)}$. We work with local systems of $R$-modules on $U$. For such a local system $\cF$, we will abbreviate $\cF\otimes_R R_m$ to $\cF/s^m$. We make $\ul k$ into a sheaf of $R$-modules by letting $s$ act as $0$. Recall the notation $\ov M$ from Remark~\ref{rem:conjugate}.

Consider the following map of short exact sequences, with $\pi_{\ov\cL}$ defined in Remark~\ref{rem:Rvee}. Later we will assume that $m$ is large enough that $s^m$ annihilates $\Tors_R H^{i}(U;\ov\cL)$ for all $i$.
\begin{equation}\label{eq:twoSessBoundingWeights}
\begin{tikzcd}[row sep = 1.2em]
0\arrow[r] &
 \ov\cL\arrow[r,"s"]\arrow[d,"\phi_m"] &
 \ov\cL\arrow[r,"\pi_{\ov\cL}"]\arrow[d,"\phi_{m+1}"] &
\ul k \arrow[r]\arrow[d,"="]&
0 \\
0\arrow[r] &
 \frac{\ov\cL}{s^m}\arrow[r,"\psi_{m1}"] &
 \frac{\ov\cL}{s^{m+1}}\arrow[r,"\phi_{m1}"]&
 \ul k \arrow[r] &
0.
\end{tikzcd}
\end{equation}
Here $\phi_j$ is induced by the projection $R\to R_j$ after tensoring with $\ov\cL$. Since the vertical arrows are surjections, the arrows in the bottom row are uniquely determined as the arrows closing the diagram. By the discussion in Section~\ref{sec:maps}, concretely from (\ref{eq:picLTriangles}), the top row can be completed to an exact triangle with connecting map $-\pi_{\cL}^\vee[1]\colon \ul k\to \ov\cL[1]$.

Let us see that the bottom row can be completed to an exact triangle with the connecting map $-(\phi_m\circ \pi_{\cL}^\vee)[1]$. Recall that $\phi_m$ is the result of tensoring the quotient map $R\to R_m$ by $\ov\cL$. Using the map $\phi_j$, $\overset{(-1)}{\ov\cL}\xrightarrow{s^j}\overset{(0)}{\ov\cL}$ becomes a free resolution of $\ov\cL/s^j$, as in the following diagram:
\[
\begin{tikzcd}[row sep = 1.2em]
\ov\cL \arrow[d,"s^j"'] \arrow[r]& 0 \arrow[d]\\
\ov\cL\arrow[r,"\phi_j"] & \frac{\ov\cL}{s^j}. 
\end{tikzcd}
\]
Recall from the definition of $\pi_{\cL}^\vee$ that $\phi_m\circ\pi_{\cL}^\vee$ is given by the diagram below, on the left. Therefore, $-\phi_m\circ\pi_{\cL}^\vee[1]$ is given by the middle diagram. We apply an isomorphism of resolutions which induces the identity on $\ul k$, we have $-\phi_m\circ\pi_{\cL}^\vee[1]$ represented in a different way by the rightmost diagram:
\[
{
\hspace{-0.25cm}\begin{tikzcd}[row sep = 1.3em, column sep = 1.4em]
& 0 \arrow[r]\arrow[d]& \ov\cL\arrow[d,"s^m"] & &
0 \arrow[r]\arrow[d]& \ov\cL\arrow[d,"-s^m"]& &
0 \arrow[r]\arrow[d]& \ov\cL\arrow[d,"-s^m"] \\
\deg = 0 \arrow[r,dashrightarrow] & \ov\cL \arrow[r,"1"]\arrow[d,"s"]& \ov\cL\arrow[d] &
&
\ov\cL \arrow[r,"-1"]\arrow[d,"-s"]\arrow[rrr,bend right = 20,"-1"',dashrightarrow]&
 \ov\cL\arrow[d] 
&
&
 \ov\cL \arrow[r,"1"]\arrow[d,"s"]& \ov\cL\arrow[d] 
\\
& \ov\cL\arrow[r] & 0 & \deg = 0\arrow[r,dashrightarrow]
& \ov\cL\arrow[r]\arrow[rrr,bend right = 20,"1"',dashrightarrow] & 0& 
& \ov\cL\arrow[r] & 0
\end{tikzcd}
}
\]

We can check that the short exact sequence together with the connecting map can be lifted (uniquely up to homotopy) to the following diagram:
\[
\begin{tikzcd}[column sep = 5em, row sep = 0.4em]
\frac{\ov\cL}{s^m} \arrow[r,"\psi_{m1}"]\arrow[d,phantom,"\rotatebox{-90}{$\cong$}"]
&
\frac{\ov\cL}{s^{m+1}} \arrow[r,"\phi_{m1}"]\arrow[d,phantom,"\rotatebox{-90}{$\cong$}"]
&
\ul k \arrow[r,"-(\phi_m\circ \pi_{\cL}^\vee){[1]}" ]\arrow[d,phantom,"\rotatebox{-90}{$\cong$}"]
&
\frac{\ov\cL}{s^m} [1] \arrow[d,phantom,"\rotatebox{-90}{$\cong$}"]
\\
0\arrow[r]\arrow[dd] &
0\arrow[r]\arrow[dd] &
0\arrow[r]\arrow[dd] &
\ov\cL \arrow[dd,"-s^m"] 
\\\\
\ov\cL \arrow[dd,"s^m", very near start]\arrow[r,"1"] &
\ov\cL \arrow[dd,"s^{m+1}", very near start] \arrow[r,"s^m"] & 
\ov \cL \arrow[dd,"s", near start] \arrow[r,"1"] & \ov\cL\arrow[dd] \\\\
\ov\cL \arrow[r,"s"] &
\ov\cL \arrow[r,"1"] & 
\ov \cL \arrow[r] & 0.
\end{tikzcd}
\]
To see that this is an exact triangle we need to find a quasiisomorphism between the free resolution of $\ov\cL/s^{m+1}$ and the cone $C:=\Cone(-\phi_m\circ \pi_{\cL}^\vee\colon \ul k[-1]\to \ov\cL/s^m)$, compatible with the maps coming from $\ov\cL/s^m$ and going to $\ul k$ (i.e. the maps using the resolutions). We are resolving $\ul k[-1]$ by $\ov\cL\xrightarrow{-s} \ov\cL$, the shift of the resolution above. Here is the quasiisomorphism: the map $\star$ represents the matrix $\scriptsize\arraycolsep=1.0pt
\Big(\begin{array}{cc}
s^m & -1\\
0 & s
\end{array}\Big)$.
\[
\begin{tikzcd}[column sep = 5em, row sep = 0.7em]
\ov\cL/s^m \arrow[d,phantom,"\rotatebox{-90}{$\cong$}"]
&
C \arrow[d,phantom,"\rotatebox{-90}{$\cong$}"]
&
\ov\cL/s^{m+1} \arrow[d,phantom,"\rotatebox{-90}{$\cong$}"]
&
\ul k \arrow[d,phantom,"\rotatebox{-90}{$\cong$}"]
\\
\ov\cL \arrow[dd,"s^m"]\arrow[r,"(1{,}0)^t"]
&
\ov\cL\oplus \ov\cL \arrow[dd,"\star"]
\arrow[r,"(1{,}0)"]
\arrow[ddr,phantom,"\cong"]
&
 \ov\cL \arrow[r,"s^m"]\arrow[dd,"s^{m+1}"]
&
\ov\cL \arrow[dd,"s"]\\\\
\ov\cL \arrow[r,"(1{,}0)^t"] 
& 
\ov\cL\oplus \ov\cL \arrow[r,"(s{,}1)"]
&
\ov\cL \arrow[r,"1"]
&
\ov\cL
\end{tikzcd}
\]
From this diagram, we automatically see that the quasiisomorphism is compatible with the map from $\ov\cL_m$. For the map to $\ul k$, the map $C\to \ul k$ above should coincide with the projection on the first coordinate, which is realized by the following homotopy:
\[
\begin{tikzcd}[column sep = 12em, row sep =2em]
\ov\cL\oplus \ov\cL \arrow[r,"(s^m{,}0)",shift left = 1.5ex]\arrow[r,"(0{,}1)"', shift right = -0.5 ex]\arrow[d,"\star"] &
\ov\cL \arrow[d,"s"] \\
\ov\cL\oplus \ov\cL \arrow[r,"(s{,}1)",shift left = 0.5ex]\arrow[r,"(0{,}1)"',shift right = 0.5ex]\arrow[ur,"(-1{,}0)",near start] &
\ov\cL
\end{tikzcd}
\]
So the short exact sequences in (\ref{eq:twoSessBoundingWeights}) can be filled in to exact triangles in the derived category, using the maps $-\pi_{\cL}^\vee[1]\colon \ul k\to \ov\cL$ and $-(\phi_m\circ \pi_{\cL}^\vee)[1]\colon \ul k\to \ov\cL/s^m$. Note that clearly, these will create a map between triangles.

Therefore, we have the following map of cohomology long exact sequences, given in the two bottom rows below.\begin{equation}\label{eq:bigExactSeqs}
{
\begin{tikzcd}[column sep = 1.5em, row sep = 1.5em]
H^i(U;k)\arrow[r,"\pi_{\cL}^\vee"]\arrow[d,"="] &
 \Tors_R H^{i+1}(U;\ov\cL) \arrow[r,"s"]\arrow[d,hookrightarrow] &
\Tors_R H^{i+1}(U;\ov\cL) \arrow[r,"\pi_{\ov\cL}"]\arrow[d,hookrightarrow] &
H^{i+1}(U;k) \arrow[d,"="]\\
H^i(U;k)\arrow[r,"\pi_{\cL}^\vee"]\arrow[d,"="] &
 H^{i+1}(U;\ov\cL) \arrow[r,"s"]\arrow[d,"\phi_m"] &
H^{i+1}(U;\ov\cL) \arrow[r,"\pi_{\ov\cL}"]\arrow[d,"\phi_{m+1}"] &
H^{i+1}(U;k)\arrow[d,"="] \\
H^i(U;k)\arrow[r,"\phi_m\circ \pi_{\cL}^\vee"] &
 H^{i+1}(U;\frac{\ov\cL}{s^m}) \arrow[r,"s"] &
H^{i+1}(U;\frac{\ov\cL}{s^{m+1}}) \arrow[r,"\phi_{m1}"] &
H^{i+1}(U;k) .
\end{tikzcd}
}
\end{equation}
We are abusing notation by using the same letters to denote the maps of (the derived category of) sheaves and their induced maps in cohomology.

From now on, we assume that $m$ is large enough so that $s^m$ annihilates $\Tors_R H^{i+1}(U;\ov\cL)$. By Remark~\ref{remk:torsionForSheaves}, the maps $\Tors_R H^*(U;\cL)\rightarrow H^*(U;\cL/s^m)$ induced by $\phi_m$ (the composition of the two central vertical arrows above) are injective. Also, note that on $\Tors_R H^{i+1}(U;\ov\cL)$, $s$ and $\log(t) = \log(1+s)$ differ by a unit, namely the Taylor polynomial of degree $m-1$ of $\frac{s}{\log(1+s)}$, and therefore their kernel and cokernel coincide. This means that the bottom row in the diagram above will remain exact if we replace $s$ by $\log(t)$. We have the following commutative diagram, formed by the first and last rows of equation (\ref{eq:bigExactSeqs}).
\begin{equation}\label{eq:kerOfLogS}
{
\begin{tikzcd}[column sep =1.5em, row sep=1.3em]
H^i(U;k)\arrow[r,"\pi_{\cL}^\vee"]\arrow[d,"="] &
 \Tors_R H^{i+1}(U;\ov\cL) \arrow[r,"\log t"]\arrow[d,"\phi_m",hookrightarrow] &
\Tors_R H^{i+1}(U;\ov\cL)(-1) \arrow[r,"\pi_{\ov\cL}"]\arrow[d,"\phi_{m+1}",hookrightarrow] &
H^{i+1}(U;k) \arrow[d,"="]\\
H^i(U;k)\arrow[r,"\pi_{\cL}^\vee"] &
 H^{i+1}(U;\frac{\ov\cL}{s^m}) \arrow[r,"\log t"] &
H^{i+1}(U;\frac{\ov\cL}{s^{m+1}})(-1)\arrow[r,"\phi_{m1}"] &
H^{i+1}(U;k) .
\end{tikzcd}
}
\end{equation}
We make the following claims:
\begin{enumerate}
\item $\pi_{\cL}^\vee$ is a MHS morphism (Theorem~\ref{thm:geomIsMHS} together with Proposition~\ref{prop:mapsAreEqual}).
\item $\log t$ is a MHS morphism (note the Tate twist), by Corollary~\ref{alexandermhs}.
\item $\phi_m$ and $\phi_{m+1}$ are MHS morphisms (Remark~\ref{remk:mhsSummary}).
\item $\phi_{m1}$ is a MHS morphism. This follows from Lemma~\ref{inducedmhsmaps}. The Tate twist is due to the fact that the MHS on the cohomology $\ov\cL/s$ comes from the shifted Hodge-de Rham complex, as in Remark~\ref{remk:mhsSummary}. Remark~\ref{transvstate} tells us that shifts of mixed Hodge complexes result in Tate twists in their cohomologies.
\item It follows that $\pi_{\ov\cL} = \phi_{m1}\circ \phi_{m+1}$ is a MHS morphism.
\item The top row is exact. This is a straightforward verification.
\end{enumerate}

We have arrived at the following exact sequence of mixed Hodge structures:
\[
H^i(U;k) \xrightarrow{\pi_{\cL}^\vee} \Tors_R H^{i+1}(U;\ov\cL) \xrightarrow{\log(t)} \Tors_R H^{i+1}(U;\ov\cL)(-1) \xrightarrow{\pi_{\ov\cL}} H^{i+1}(U;k).
\]
Taking the associated graded sequence for the weight filtration, we obtain the following exact sequence of pure Hodge structures for any $k$:
\begin{multline*}
\Gr^W_k H^i(U;k) \xrightarrow{\pi_{\cL}^\vee} \Gr^W_k \Tors_R H^{i+1}(U;\ov\cL) \xrightarrow{\log(t)} \Gr^W_{k-2} \Tors_R H^{i+1}(U;\ov\cL)\\ \xrightarrow{\pi_{\ov\cL}} \Gr^W_k H^{i+1}(U;k).
\end{multline*}
By \cite[Corollaire 3.2.15]{De2}, the nonzero graded pieces of the weight filtration of $H^i(U;k)$ are in the interval $[i,\min\{2i,2n\}]$. Therefore, if either $k\ge \min\{2i+3,2n+1\}$ or $k\le i-1$, we have an isomorphism
\[
\Gr^W_k \Tors_R H^{i+1}(U;\ov\cL) \overset{\log(t)}{\cong} \left(\Gr^W_{k-2} \Tors_R H^{i+1}(U;\ov\cL)\right)(-1).
\]
Since the MHS are finite dimensional, these groups must be $0$ for $k\ll 0$ and $k\gg 0$, which implies that
\[
\Gr^W_k \Tors_R H^{i+1}(U;\ov\cL) = 0 \text{ unless } k\in [i,\min\{2i,2n-2\}].
\]
\end{proof}


	\chapter{Semisimplicity for Proper Maps}\label{sect:ss}

In this chapter, assuming that $f\colon U \to \C^*$ is proper, we will prove that the torsion part $A_i(U^f; \Q)$ of the Alexander module $H_i(U^f; \Q)$ is a semisimple $R$-module, that is, the $t$-action on $A_i(U^f; \Q)$ is semisimple, for all $i \geq 0$. 

\begin{thm}\label{thmsimple}\index{semisimple}
Let $U$ be a smooth complex algebraic variety, and let $f\colon U\to \C^*$ be a proper algebraic map. Let $\calL$ be the rank one $\Q[t^{\pm 1}]$-local system on $U$ defined as in Section \ref{ssAlex}. Then $\Tors_R {H^i(U;\calL)}$ is a semisimple $R$-module. 
\end{thm}
Since the operations of taking the conjugate $\Q[t^{\pm 1}]$-module structure and taking the $\Q$-vector space dual preserve semisimplicity, the following is an immediate consequence of Proposition~\ref{propcanon} and the above theorem. 
\begin{cor}\label{corss}
Under the assumptions of Theorem \ref{thmsimple}, the torsion part $A_i(U^f; \Q)$ of the homology Alexander module $H_i(U^f; \Q)$ is a semisimple $\Q[t^{\pm 1}]$-module, for all $i \geq 0$. 
\end{cor}
In the rest of this chapter, we prove Theorem \ref{thmsimple}.

Note that $\Tors_{\C[t^{\pm 1}]}H^i(U,\cL\otimes_\Q \C)$ is the base change of $\Tors_{\Q[t^{\pm 1}]}H^i(U,\cL)$ from $\Q$ from $\C$. In particular, the minimal polynomial of $t$ acting on either is the same, so one is semisimple if and only if the other is (if and only if the minimal polynomial is separable). Thus, it is sufficient to prove Theorem~\ref{thmsimple} over $\C$. In this chapter, let $R=\C[t^{\pm 1}]$, and $\cL$ will denote the tautological $R$-local system from Section~\ref{ssAlex}, over $k=\C$.

Let $\calL_{\C^*}$ be the tautological rank one $R$-local system on $\C^*$. Under the natural isomorphism $R\cong \C[\pi_1(\C^*)]$, the monodromy action is the natural $\pi_1(\C^*)$-action on $\C[\pi_1(\C^*)]$. Then $\calL\cong f^{-1}(\calL_{\C^*})$ as $R$-local systems. 

By the projection formula, we have
\[
H^i(U;\calL)\cong H^i(U; f^{-1}(\calL_{\C^*}))\cong \mathbb{H}^i(\C^*, Rf_*(\C_U)\otimes_\C \calL_{\C^*}).
\]

Since $U$ is smooth and $f$ is a proper map, the decomposition theorem of \cite{BBD} yields that the push-forward $Rf_*(\C_U)$ decomposes as\index{decomposition theorem}\index{perverse sheaf|(}
\[
Rf_*(\C_U)\cong \bigoplus_{\lambda\in \Lambda}\mathcal{P}_\lambda[d_\lambda],
\]
where $\Lambda$ is a finite index set, $d_\lambda\in \Z$, and each $\mathcal{P}_\lambda$ is a simple ($\C$-)perverse sheaf on $\C^*$. Therefore, to prove Theorem \ref{thmsimple}, it suffices to show the following proposition, which is essentially a special case of the theory of the Mellin transformation developed by Gabber-Loeser \cite{GL}. 

\begin{prop}\label{propMellin}
If $\mathcal{P}$ is a simple perverse sheaf on $\C^*$, we have:
\begin{enumerate}
\item $\mathbb{H}^i(\C^*; \mathcal{P}\otimes \calL_{\C^*})=0$ for $i\neq 0$;
\item $\mathbb{H}^0(\C^*; \mathcal{P}\otimes \calL_{\C^*})$ is a simple $R$-module when $\mathcal{P}$ is smooth (i.e., the shift of a local system on $\C^*$), and a free $R$-module when $\mathcal{P}$ is not smooth. 
\end{enumerate}
\end{prop}
Here, we recall that an $R$-module is {\em simple} if and only if it is one-dimensional, since we are working over $\C$. Before proving the above proposition, we need a lemma.

\begin{lem}\label{lemH1}
Let $\mathcal{P}$ be a simple perverse sheaf on $\C^*$. If $\mathbb{H}^{-1}(\C^*; \mathcal{P})\neq 0$, then $\mathcal{P}\cong \ul{\C}_{\C^*}[1]$. 
\end{lem}
\begin{proof}
By the natural isomorphism
\[
\mathbb{H}^{-1}(\C^*; \mathcal{P})\cong \Hom(\ul{\C}_{\C^*}[1], \mathcal{P}),
\]
if $\mathbb{H}^{-1}(\C^*; \mathcal{P})\neq 0$, then there exists a nontrivial map between $\ul\C_{\C^*}[1]$ and $\mathcal{P}$. Since both $\ul{\C}_{\C^*}[1]$ and $\mathcal{P}$ are simple perverse sheaves, it follows that they are isomorphic. 
\end{proof}

\begin{proof}[Proof of Proposition \ref{propMellin}] 
\textbf{Case 1.} Assume $\mathcal{P}$ is smooth. In this case, $\mathcal{P}\cong L[1]$ for a simple (i.e. rank 1) $\C$-local system $L$ on $\C^*$. As in \cite[Example 2.5.7]{dimca2004sheaves}, we have natural isomorphisms of $R$-modules
\[
\mathbb{H}^i(\C^*; \mathcal{P}\otimes \calL_{\C^*})\cong H^{i+1}(\C^*; L\otimes_\C \calL_{\C^*})\cong 
\begin{cases} 
0 \quad &\text{if}\quad i\neq 0 \\
\overline{V_L}\quad &\text{if} \quad i=0,
\end{cases}
\]
where $V_L$ is the $R$-module associated to the monodromy representation of $L$, and $\overline{\cdot}$ denotes the conjugate $R$-module structure, with $t$ acting by $t^{-1}$ (see Remark \ref{rem:conjugate}). In this case, $L$ has rank 1 if and only if its stalk is 1-dimensional, i.e. $L$ is simple if and only if $V_L$ is a simple representation. Thus, in this case, the proposition follows. 

\noindent\textbf{Case 2.} Assume $\mathcal{P}$ is not smooth. Let $p: \C^*\to \mathrm{pt}$ be the projection to a point and let $\mathfrak m$ be any maximal ideal of $R$. We claim that there is a natural isomorphism
\begin{equation}\label{eqpf}
Rp_*(\mathcal{P}\otimes_{\C} \calL_{\C^*})\stackrel{L}{\otimes}_R R/\mathfrak{m}\cong Rp_*(\mathcal{P}\otimes_{\C} \calL_{\C^*}\otimes_R R/\mathfrak{m}),
\end{equation}
where $\stackrel{L}{\otimes}$ denotes the derived tensor product. In fact, take a complex of injective sheaves $\mathcal{I}^\bullet$ on $\C^*$ representing $\mathcal{P}\otimes_{\C}\calL_{\C^*}$ and take a free resolution $F^\bullet$ of the $R$-module $R/\mathfrak{m}$. Then the total complex of $\Gamma(\C^*, \mathcal{I}^\bullet)\otimes_R F^\bullet$ is a complex representing both sides of \eqref{eqpf}. 

Let $L\coloneqq \calL_{\C^*}\otimes_R R/\mathfrak{m}$. Then $\mathcal{P}\otimes_{\C} L$ is a simple perverse sheaf on $\C^*$, since the operation of tensoring with $L$ is invertible. By the support condition for perverse sheaves we have that $\mathcal{H}^i(\mathcal{P}\otimes_{\C} L)=0$ for $i<-1$. Thus, the hypercohomology spectral sequence yields that $\mathbb{H}^i(\C^*; \mathcal{P}\otimes_{\C} L)=0$ for $i<-1$. Additionally, by Lemma \ref{lemH1}, $\mathbb{H}^i(\C^*;\mathcal{P}\otimes_{\C} L)=0$ for $i\leq -1$. On the other hand, we get by Artin's vanishing theorem that $\mathbb{H}^i(\C^*; \mathcal{P}\otimes_{\C} L)=0$ for $i\geq 1$. Therefore, 
\[
\dim \mathbb{H}^0(\C^*; \mathcal{P}\otimes_{\C} L)=\chi(\C^*; \mathcal{P})
\]
is independent of the choice of the maximal ideal $\mathfrak{m}$ of $R$. 

By the isomorphism \eqref{eqpf} and the universal coefficient theorem for cochain complexes, we have the following short exact sequence
$$
0\to
\mathbb{H}^i(\C^*; \mathcal{P}\otimes_{\C} \calL_{\C^*})\otimes_R R/\mathfrak m
\to
\mathbb{H}^i(\C^*; \mathcal{P}\otimes_{\C} L)
\to
\mathrm{Tor}^R_1(\mathbb{H}^{i+1}(\C^*; \mathcal{P}\otimes_{\C} \calL_{\C^*}), R/\mathfrak m)
\to
0.
$$
The results in the previous paragraph then yield that, for any maximal ideal $\mathfrak{m}$ of $R$, we have
\[
\mathbb{H}^i(\C^*; \mathcal{P}\otimes_{\C} \calL_{\C^*})\otimes_R R/\mathfrak{m}=0 \quad \text{for all } i\neq 0,
\]
and
\[
\mathrm{Tor}^R_1(\mathbb{H}^{i}(\C^*; \mathcal{P}\otimes_{\C} \calL_{\C^*}), R/\mathfrak{m})=0 \quad \text{for all } i\neq 1.
\]
Thus, $\mathbb{H}^i(\C^*; \mathcal{P}\otimes_{\C} \calL_{\C^*})=0$ for all $i\neq 0$ and $\mathbb{H}^0(\C^*; \mathcal{P}\otimes_{\C} \calL_{\C^*})$ is free. \index{perverse sheaf|)}
\end{proof}

As a consequence of Corollary \ref{corss}, we get the following purity result (Theorem~\ref{thm:pur} in the introduction).
\begin{cor}\label{cor:pure}
If $f\colon U\to \C^*$ is a proper algebraic map, then $A_i(U^f; \Q)$ carries a pure Hodge structure of weight $-i$. 
\end{cor}
\begin{proof}
The generic fiber $F$ of the proper morphism $f\colon U\to \C^*$ between smooth complex algebraic varieties is a complete smooth algebraic variety, whose rational homology and cohomology groups have pure Hodge structures. By Corollary \ref{corss} and Corollary \ref{cor:fiber}, $A_i(U^f; \Q)$ is a quotient of the weight $-i$ pure Hodge structure on $H_i(F; \Q)$, which proves our claim.
\end{proof}

\begin{remk}\label{rem:smf}
If $f\colon U \to \C^*$ is a projective submersion of smooth complex algebraic varieties, then $f$ is a fibration, and let $F$ denote its fiber. In this case, the semisimplicity of $A_i(U^f;\Q)=H_i(U^f;\Q)\cong H_i(F;\Q)$ is a direct consequence of Deligne's decomposition theorem \cite{De, De2} (compare with \cite[Remark 1.12]{BudurLiuWang}). Indeed, Deligne's theorem implies that the local systems $R^if_*\Q_U$ are semisimple on $\C^*$, or equivalently, the monodromy representation on $H^i(F;\Q)$ is semisimple. The semisimplicity of $A_i(U^f;\Q)$ follows then by using Lemma \ref{lem:fiberMonodromy}.
\end{remk}


	\chapter{Relation to the Limit MHS}\label{rellim}\index{Hodge!limit MHS}
In this chapter, we compare the mixed Hodge structure on $\Tors_R H^*(U;\ov\calL)$ with the limit mixed Hodge structure on the generic fiber of $f$ (as recalled in Section \ref{LimitMixedHodgeStructure}).

\textit{Throughout this chapter assume the ground field $k = \Q$, that the algebraic map $f\colon U \rightarrow \C^*$ is in addition a proper map with generic fiber $F$, and that the good compactification $\bar{f}\colon X \rightarrow \C P^1$ of $f$ has the property that $\bar{f}^{-1}(0)$ is reduced.}

Let $i\colon E \rightarrow X$ denote the inclusion of the reduced divisor $\bar{f}^{-1}(0)$.

\begin{remk}\label{unipotent}
Under the above assumptions, we have that the monodromy of $\Tors_R H^*(U,\ov\calL)$ is unipotent, see \cite[Corollary 11.19]{peters2008mixed}.
\end{remk}

\begin{remk}\label{proper}
Let $f\colon U\rightarrow \C^*$ be a proper map. Recall the notations of the beginning of Section \ref{sscover}. By \cite[Semi-stable Reduction Theorem]{kempf2006toroidal}, there is a finite cover $U_N\rightarrow U$ for some $N$ such that $f_N\colon U_N\rightarrow \C^*$ is also a proper map, and there is a compactification $\bar{f_N}\colon X \rightarrow \C P^1$ satisfying the conditions that we assume throughout this chapter. Further, by loc. cit. $X\setminus U_N = (\ov f_N)^{-1}(\{0,\infty\})$. 

Note that the eigenvalues of the action of $t^N$ in $\Tors_{R(N)}H^*(U_N,\ov\calL_N)$ will be all $1$ by Remark \ref{unipotent}. Recall also that the way we induced $\Tors_R H^*(U;\ov\calL)$ with a MHS was using the natural isomorphism of $R(N)$-modules $\Tors_R H^*(U;\ov\calL)\cong \Tors_{R(N)}H^*(U_N,\ov\calL_N)$ of Lemma \ref{lemLocal}. Hence, the only meaningful conditions we are assuming in this chapter is that the map $f\colon U\rightarrow\C^*$ be \textit{proper}.
\end{remk}

We wish to relate the mixed Hodge structure on $\mathrm{Tors}_R\, H^*(U; \ov\calL)$ with the limit mixed Hodge structure on $\bH^*(E; \psi_{\bar{f}} \underline{\Q})$ discussed in Section \ref{LimitMixedHodgeStructure}.
In fact, we will identify morphisms of mixed Hodge complexes of sheaves that realize this relationship.\index{psi-f-Hdg@$\psi_f^{\mathrm{Hdg}}$}

We compare the subsequent pair of mixed Hodge complexes.
On the one hand, Theorem \ref{Qalexandermhs} provides us for $m \geq 1$:
\begin{align*}
   & \Hdg{X}{D}\left(\frac{1}{2\pi i}\frac{df}{f},m\right)\\
    &= \left(\left[\calK^\bullet_\infty\left(1 \otimes f,m\right), \tilde{W}_{\lc}\right], \left[\logdr{X}{D}\left(\frac{1}{2\pi i}\frac{df}{f},m\right), W_{\lc}, F^{\lc}\right], \varphi_{\infty\#}\right)
\end{align*}
On the other hand, Theorem \ref{rationallimitmhs} offers us, independent of $m$:
\begin{align*}
    \psi_f^{\textnormal{Hdg}} = i^{-1}\left(\left[\Tot\tilde{\calC}^{\bullet,\bullet}, \tilde{W}(M)_{\lc}\right], \left[\vphantom{\tilde{A}} \Tot\calA^{\bullet,\bullet}, W(M)_{\lc}, F^{\lc}\right], \varphi_{\infty}\right).
\end{align*}
Our first goal is to show that, pre-application of $i^{-1}$, the second triple is a quotient of a translation of the first, as long as $m$ is large enough.

We obtain our map of triples from the quotient maps (up to a sign):
\begin{align*}
&(-1)^\ell\Phi^\ell_\Q(m)\colon \bigoplus_{j=0}^{m-1} \calK^{\ell+1}_\infty \otimes \Q \langle s^j \rangle \rightarrow \bigoplus_{j=0}^{m-1} \calK^{\ell+1}_\infty/\tilde{W}_j \calK^{\ell+1}_\infty\\
&(-1)^\ell\Phi^\ell_\C(m)\colon \bigoplus_{j=0}^{m-1} \Omega_X^{\ell+1}(\log D) \otimes_\C \C \langle s^j \rangle  \rightarrow \bigoplus_{j=0}^{m-1} \Omega_X^{\ell+1}(\log D)/ W_j \Omega_X^{\ell+1}(\log D)
\end{align*}
where $\ell$ and $m \geq 1$ denote integers. We define $\Phi^\ell_\Q(\infty),\Phi^\ell_\C(\infty)$ using the same formulas.
If $m \geq \dim X$ then the codomains of the above maps are $(\Tot \tilde{\calC}^{\bullet,\bullet})^\ell$ and $(\Tot\calA^{\bullet, \bullet})^\ell$ respectively, because the filtrations $\tilde{W}_{\lc}$ and $W_{\lc}$ fill out at index $\dim X$ (this is where boundedness of filtrations is important).
This enables us to, assuming $m \geq \dim X$ (or even $m=\infty$), define natural surjections 
\[\Phi_\Q(m)\colon \calK^\bullet_\infty(1 \otimes f, m)[1] \rightarrow \Tot\tilde{\calC}^{\bullet, \bullet}
\]
and
\[
\Phi_\C(m)\colon \logdr{X}{D}\left(\frac{1}{2 \pi i} \frac{df}{f}, m\right)[1] \rightarrow \Tot\calA^{\bullet, \bullet}.
\]
The compatibility with differentials $\Phi(m) \circ d[1] = (d'+d'') \circ \Phi(m)$ follows by definition, where $d'$ and $d''$ are the differentials on the double complexes.

\begin{lem}\label{surj}
Suppose $m \geq \dim X$.
Then: 
\begin{align*}
\Phi(m)\colon \Hdg{X}{D}&\left(\frac{1}{2\pi i}\frac{df}{f},m\right)[1] \rightarrow \\ &\rightarrow\left(\left[\Tot\tilde{\calC}^{\bullet,\bullet}, \tilde{W}(M)_{\lc}\right], \left[\vphantom{\tilde{A}} \Tot\calA^{\bullet,\bullet}, W(M)_{\lc}, F^{\lc}\right], \varphi_{\infty}\right)  	
\end{align*}
is a (surjective) map of triples.	
\end{lem}
\begin{proof}
The commutativity $\Phi_\C(m) \circ \varphi_{\infty\#} = \varphi_\infty \circ \Phi_\Q(m)$ follows from the definitions.

We next verify that filtrations are preserved for the $\C$-component of the triple. 
The same arguments apply to the $\Q$-component as well.

The $W[1]_{\lc}$-filtered subcomplex indexed by integer $i$ is:
\begin{align*}
\left[W_{i+1}\logdr{X}{D}\left(\frac{1}{2\pi i}\frac{df}{f},m\right)\right][1] =  \bigoplus_{j = 0}^{m-1} \left[\vphantom{\tilde{A}}W_{i+2j+1} \Omega^{\bullet}_X(\log D)\right][1] \otimes_\C \C\langle s^j \rangle	
\end{align*}
which is mapped under $\Phi_\C(m)$ into $W(M)_i \left(\Tot\calA^{\bullet,\bullet}\right)$.
The $F[1]^{\lc}$-filtered subcomplex indexed by integer $p$ is:
\begin{align*}
\left[F^{p+1}\logdr{X}{D}\left(\frac{1}{2\pi i}\frac{df}{f},m\right) \right][1] &= \bigoplus_{j=0}^{m-1}\left[F^{p+j+1}\logdr{X}{D}\right][1] \otimes_\C \C \langle s^j \rangle\\
&= \bigoplus_{j=0}^{m-1}\left[\Omega^{\geq p + j + 1}_X (\log D)\right][1] \otimes_\C \C \langle s^j \rangle  
\end{align*}
which is mapped under $\Phi_\C(m)$ into $F^p \left(\Tot\calA^{\bullet,\bullet}\right)$. 
\end{proof}

The adjunction $\Id\rightarrow i_*i^{-1}$ applied to $$\left(\left[\Tot\tilde{\calC}^{\bullet,\bullet}, \tilde{W}(M)_{\lc}\right], \left[\vphantom{\tilde{A}} \Tot\calA^{\bullet,\bullet}, W(M)_{\lc}, F^{\lc}\right], \varphi_{\infty}\right)$$ yields a surjective map of triples
$$\Ad\colon \left(\left[\Tot\tilde{\calC}^{\bullet,\bullet}, \tilde{W}(M)_{\lc}\right], \left[\vphantom{\tilde{A}} \Tot\calA^{\bullet,\bullet}, W(M)_{\lc}, F^{\lc}\right], \varphi_{\infty}\right)\rightarrow i_*\psi_{\bar{f}}^{\textnormal{Hdg}}$$
Note that, since $i$ is proper, $i_*=i_!$ is an exact functor, so we may identify $i_*$ with $Ri_*$. Composing $\Phi(m)$ with the adjunction $\Ad$, we immediately get the following.

\begin{cor}\label{relnlimitmhc}
Suppose $m \geq \dim X$. Then:
\begin{align*}
\Ad\circ\Phi(m)\colon \Hdg{X}{D}&\left(\frac{1}{2\pi i}\frac{df}{f},m\right)[1] \rightarrow i_*\psi_{\bar{f}}^{\textnormal{Hdg}} 	
\end{align*}
is a surjective morphism of mixed Hodge complexes of sheaves.
It satisfies the equality:
\begin{align*}
\Phi(m) \circ S_m[1] = \Theta \circ \Phi(m) 	
\end{align*}
where $S_m$ is multiplication by $s$ (as in Lemma \ref{inducedmhsmaps}), and $\Theta$ is defined in Section \ref{LimitMixedHodgeStructure}.
\end{cor}
\begin{proof}
The equality is obtainable by expanding definitions.  	
\end{proof}

Recall that we defined the MHS on $H^{*+1}(U; \ov\calL \otimes R_m)$ using the $[1]$ shift of the thickened Hodge complex $\Hdg{X}{D}\left(\frac{1}{2\pi i}\frac{df}{f},m\right)_{\widetilde{s}_2}$. Twisting the $R$-module structure and taking global hypercohomology in Corollary \ref{relnlimitmhc}, we get:

\begin{cor}\label{relnlimitmhs}
For sufficiently large $m$, $$\Ad\circ\Phi(m)\colon \Hdg{X}{D}\left(\frac{1}{2\pi i}\frac{df}{f},m\right)_{\widetilde{s}_2}[1] \rightarrow i_*\psi_{\bar{f}}^{\textnormal{Hdg}}$$ induces a morphism of $\Q$-mixed Hodge structures:
\begin{align*}
	\Phi^*\colon \mathrm{Tors}_R H^{*+1}(U;\ov\calL) \rightarrow \mathbb{H}^*(E; \psi_{\bar{f}}\underline{\Q})
\end{align*}
which is also an $R$-module homomorphism.
\end{cor}
\begin{proof}
Let $k=\Q$, and $m\geq \dim X$. Since $\Theta$ corresponds to multiplication by $\log(t)$, and multiplication by $s$ on $\Hdg{X}{D}\left(\frac{1}{2\pi i}\frac{df}{f},m\right)$ corresponds to multiplication by $\log(t)$ on $\Hdg{X}{D}\left(\frac{1}{2\pi i}\frac{df}{f},m\right)_{\widetilde{s}_2}$, Corollary \ref{relnlimitmhc} tells us that
\[\Ad\circ\Phi(m)\colon \Hdg{X}{D}\left(\frac{1}{2\pi i}\frac{df}{f},m\right)_{\widetilde{s}_2}[1] \rightarrow i_*\psi_{\bar{f}}^{\textnormal{Hdg}}\]
induces the following morphism of mixed Hodge structures and $R$-modules in hypercohomology:
\begin{align*}
H^{*+1}(U;\ov\calL \otimes_R R_m) \rightarrow \mathbb{H}^*(E; \psi_{\bar{f}}\underline{\Q}).
\end{align*}
which we denote by $\Phi(m)^*$.

Note that the torsion submodule $\Tors_{R_\infty}H^{*+1}(U; \ov\calL \otimes_R R_\infty)$ is contained in $H^{*+1}(U;\ov\calL \otimes_R R_m)$ as the kernel of $\psi_{mj}[1]^*$ for sufficiently large $m$ and $j$ (Remark \ref{remQ}).
And because the monodromy action is unipotent (Remark \ref{unipotent}), the torsion submodule is isomorphic to $\Tors_R H^{*+1}(U; \ov\calL)$ as $R$-modules.  
Stitching these observations together, for sufficiently large $m$ the morphism $\Ad\circ\Phi(m)$ induces an $R$-module morphism: 
\begin{align*}
	\Phi^*\colon \mathrm{Tors}_R\,H^{*+1}(U;\ov\calL) \rightarrow \mathbb{H}^*(E; \psi_{\bar{f}}\underline{\Q})
\end{align*}
Recall Remark \ref{torsionrelations} and the proof of Corollary \ref{torsionmhs}, where we see that the MHS constructed on $\mathrm{Tors}_R\,H^{*+1}(U;\ov\calL)$ is independent of $m$, for $m$ sufficiently large. Independence of the map $\Phi^*$ of large enough $m$ is determinable through examination of the maps between the various thickened complexes associated to their inverse limit structure. 		
\end{proof}

Let $D$ and $D^*$ be an open disk and a punctured open disc centered at $0$ in $\C$, res\-pec\-tive\-ly. Let $T\coloneqq (\overline f)^{-1}(D)\subset X$ and $T^*\coloneqq f^{-1}(D^*)\subset U$. Assume further that $D$ is sufficiently small such that $$f_{|T^*}\colon T^*\rightarrow D^*$$ is a fibration. Note that by Remark~\ref{proper}, we can assume that $X\backslash U$ contains only vertical divisors, so $T\backslash E=T^*$. Recall that by definition, $\psi_{\ov f}=i^{-1} \circ R(j\circ \pi)_*\circ (j\circ \pi)^{-1} $, where, abusing notation, $\pi$ is seen as a map $\pi\colon (T^*)^f\rightarrow T^*$, $j$ is seen as $j:T^*\hookrightarrow T$, and $i$ is seen as $i:E\hookrightarrow T$. We will use this notation in the rest of this chapter.

\begin{remk}
In \cite[\S 11.2.2]{peters2008mixed}, a different infinite cyclic cover is chosen in place of $(T^*)^f$, which we will denote $\wt{(T^*)^f}$, namely replacing the exponential map in (\ref{eq:fiberProductIntro}) by the map $z\mapsto e^{2\pi i z}$. We can fix the canonical isomorphism
\[
{\footnotesize
\begin{array}{rcl}
(T^*)^f = \{(x,z)\in T^*\times \C \mid  f(x) = e^z   \} & \longrightarrow & \wt{(T^*)^f} = \{(x,z)\in T^*\times \C \mid  f(x) = e^{2\pi iz }   \} \\
(x,z) & \longmapsto & \left(x,\frac{z}{2\pi i}\right).
\end{array}
}
\]
This allows us to identify $(T^*)^f$ with $\wt{(T^*)^f}$, and we will do so implicitly for the remainder of the chapter.
\end{remk}

In the following definition, we describe a map which has the same domain and target as the MHS morphism $\Phi^*$ of Corollary \ref{relnlimitmhs}, but is defined explicitly in cohomology and in a more geometric and down-to-earth way.

\begin{dfn}\label{def:relLimit}
The map $r^*:\Tors_R H^{*+1}(U;\ov\cL)\hookrightarrow \mathbb{H}^*(E; \psi_{\bar{f}}\underline{\Q})$ is defined as the following composition
$$
\Tors_R H^{*+1}(U;\ov\cL)\hookrightarrow H^{*+1}(T^*;\ov\cL)\cong H^{*}((T^*)^f;\Q)\xrightarrow{\cong} \mathbb{H}^*(E; \psi_{\bar{f}}\underline{\Q}).
$$
Here, the first map is given by the restriction from $U$ to $T^*$ (hence the name $r^*$), and was shown to be injective in Remark \ref{remk:tube}. The isomorphism $H^{*+1}(T^*;\ov\cL)\cong H^{*}((T^*)^f;\Q)$ is given by Corollary \ref{isocohom} and was already described in Remark \ref{remk:tube}. Finally, the map $H^{*}((T^*)^f;\Q)\to \mathbb{H}^*(E; \psi_{\bar{f}}\underline{\Q})$ is given by the adjunction $\Id\to i_*i^{-1}$ applied to $R(j\circ \pi)_*(j\circ\pi)^{-1}\underline \Q_T=R(j\circ \pi)_*\underline \Q_{(T^*)^f}$. The homomorphism $H^{*}((T^*)^f;\Q)\to \mathbb{H}^*(E; \psi_{\bar{f}}\underline{\Q})$ is an isomorphism because by \cite[Remark 2.6.9]{KS},
\[
\mathbb{H}^*(E; \psi_{\bar{f}}\underline{\Q}) \cong \varinjlim_{V\supseteq E} \mathbb{H}^*(V; R(j\circ \pi)_*\underline{\Q}_{(T^*)^f}),
\]
where the limit is taken over open sets $V$. Since $f$ is proper, every such $V$ contains an open set of the form $T=(\ov f)^{-1}(D)$, for $D$ a small enough disk around $0$. The isomorphism follows from the fact that all sufficiently small tubes $T$ are fibrations over the disk, so they are homotopy equivalent to each other.
\end{dfn}

Note that, up to the natural identification $$H^{*}((T^*)^f;\Q)\cong \mathbb{H}^*(E; \psi_{\bar{f}}\underline{\Q})$$ given by adjunction, and the natural identification $$(\Tors_R H_*(U^f;\Q))^{\vee_{\Q}}\cong \Tors_R H^{*+1}(U;\ov\cL)$$ given in Proposition \ref{propcanon}, the map $r^*$ is just the dual as $\Q$-vector spaces of the map $H_*((T^*)^f;\Q)\twoheadrightarrow \Tors_R H_*(U^f;\Q)$ given by the inclusion of infinite cyclic covers $(T^*)^f\subset U^f$. Moreover, $r^*$ allows us to see the MHS on the $\Tors_R H^{*+1}(U;\ov\calL)$ as a sub-MHS of the limit MHS on $\bH^*(E;\psi_{\bar{f}}\underline{\Q})$ for all $*$, as exemplified in the following result.

\begin{thm}\label{thm:limitMap}
The maps $r^*$ and $\Phi^*$ are equal up to multiplication by a rational constant. In particular,
\[
r^*:\Tors_R H^{*+1}(U;\ov\cL)\hookrightarrow \bH^*(E; \psi_{\bar{f}}\underline{\Q})
\]
is a morphism of MHS and of $R$-modules.
\end{thm}

\begin{proof}
In order to prove this theorem, we must first recall how the hypercohomology of $\Tot \calA^{\bullet,\bullet}$ is identified with the nearby cycle cohomology in \cite{peters2008mixed}, Sections 11.2.4 and 11.2.5.

First, note that since $\pi$ is a covering map, we can see that $\pi_*$ is exact by checking over an open cover of $U$. Therefore, there are natural isomorphisms $
H^*((T^*)^f;\C) \cong 
\bH^*(T^*;R\pi_*\C) \cong 
\bH^*(T^*;\pi_*\C) $. We define the complex $$\Omega_T^\bullet(\log E)[f_\infty]\coloneqq \Omega_T^\bullet(\log E)\otimes_{\C} \C[f_\infty],$$ where the differential is given by  letting $df_\infty = \frac{df}{f}$ and using the Leibniz rule. By \cite[Theorem 11.16]{peters2008mixed} and the discussion preceding it, we have quasi-isomorphisms:
\[
i^{-1}Rj_*\pi_*\ul \C_{(T^*)^f} \xhookrightarrow{\sim} i^{-1}Rj_*\pi_*\Omega_{(T^*)^f}^\bullet \xhookleftarrow{\sim}  i^{-1}\Omega_T^\bullet(\log E)[f_\infty] .
\]
The right hand arrow is given by pulling back forms and seeing $f_\infty$ as $$f_\infty\colon(T^*)^f\subseteq U^f\to \C.$$

Next, combining Theorem 11.16 and the discussion in 11.2.5 in loc. cit., we conclude that there is a quasiisormophism as follows:
\begin{equation}\label{eqn:qisoSteenbrink}
\begin{split}
\displaystyle
{i_*i^{-1}\Omega_T^\ell(\log E)[f_\infty]} & \overset{\sim}{\longrightarrow} 
\displaystyle
{(\Tot \calA^{\bullet,\bullet})^\ell = \bigoplus_{p=0}^{\ell+1} \frac{\Omega^{\ell + 1}_T(\log E)}{W_p\Omega^{\ell + 1}_T(\log E)}}\\
\displaystyle
{\sum_{j} (f_\infty)^j\omega_j } & \longmapsto 
\displaystyle
{(-1)^\ell\frac{1}{2\pi i}\frac{df}{f}\wedge \omega_0\in \frac{\Omega^{\ell + 1}_T(\log E)}{W_0\Omega^{\ell + 1}_T(\log E)}.}
\end{split}
\end{equation}
We should note that the $\frac{1}{2\pi i}$ in the formula above does not appear in \cite{peters2008mixed} but it needs to appear here due to the difference in our conventions regarding Tate twists. As mentioned in Section \ref{ss:MHSsAndComplexes}, Tate twists in loc. cit. are defined using powers of $2\pi i$. The rational part of $\psi_{\bar{f}}^{\textnormal{Hdg}}$ is defined in loc. cit. using Tate twists, and the extra $\frac{1}{2\pi i}$ in (\ref{eqn:qisoSteenbrink}) is necessary for that quasiisomorphism to also be defined with rational coefficients, following our conventions. Needless to say, both conventions give rise to the same MHS.

Let us fill the following commutative diagram. The diagonal arrow is defined above.
\begin{equation}\label{eqn:commPhi}
{
\begin{tikzcd}[column sep=1.7em,row sep = 1.3em]
 \Omega^\bullet_T(\log E)[f_\infty]\arrow[r,"i^{-1}\vdash i_*"]\arrow[d,dashrightarrow] 
&
i_*i^{-1}\Omega^\bullet_T(\log E)[f_\infty]\arrow[dr, bend left = 20]\arrow[d,dashrightarrow] 
&
 \\
\Omega_T^{\bullet+1}(\log E)\left(\frac{1}{2\pi i}\frac{df}{ f},\infty\right)[1] \arrow[r,"i^{-1}\vdash i_*"]
&
i_*i^{-1}\Omega_T^{\bullet+1}(\log E)\left(\frac{1}{2\pi i}\frac{df}{ f},\infty\right)[1] \arrow[r,"\Phi^\ell_{\C}(\infty)"]
&
\Tot \calA^{\bullet,\bullet}.
\end{tikzcd}
}
\end{equation}
One can check directly that the map $\sum_{j} (f_\infty)^j\omega_j \mapsto \frac{1}{2\pi i}\frac{df}{f}\wedge \omega_0 \otimes 1\in \Omega_T^\bullet(\log E)\otimes R_\infty$ makes the diagram commute and it induces a map of complexes (taking into account that $d[1]=-d$ on the target).

Consider now the following diagram, which includes the maps explained above. The thickened complex is a resolution of $\ov\cL\otimes_R R_\infty$ using Proposition~\ref{propLocal}. The arrow $\pi_{\cL}^\vee$ is the one in Proposition~\ref{prop:mapsAreEqual}.
\begin{equation}\label{eqn:commLimit}
	\begin{tikzcd}[column sep = 0.1em, row sep = 1.4em]
		Rj_*\ul \C_{T^*} 
		\arrow[d,"\pi_{\cL}^\vee",bend right = 10,start anchor={[xshift=-8em]},
		end anchor={[xshift=-7em]}]
		 \to
		 \arrow[r,"\pi^{-1}\vdash \pi_*", start anchor={[yshift=0.3em,xshift=-18em]},
		 end anchor={[yshift=0.3em, xshift=-17em]}
		 ,draw=none
		 ]
		Rj_*\pi_*\ul\C_{(T^*)^f} \to
		i_*i^{-1}Rj_*\pi_*\Omega_{(T^*)^f}^\bullet
		&
		i_*i^{-1}\Omega^\bullet_T(\log E)[f_\infty]\arrow[d,dashrightarrow]\arrow[l,"\sim"']\\
		Rj_*\ov\cL[1]\otimes_R R_\infty \overset{\sim}{\to} \Omega_T^{\bullet}(\log E)\left(\frac{1}{2\pi i}\frac{df}{f},\infty\right)[1]
		\arrow[r]
		&
		i_*i^{-1}\Omega_T^{\bullet}(\log E)\left(\frac{1}{2\pi i}\frac{df}{f},\infty\right)[1].
	\end{tikzcd}
\end{equation}

We will show that it commutes up to multiplication by a nonzero real constant. Before we show this, let us explain why the commutativity (up to constant) of this diagram finishes the proof. Identifying $\bH^j(E;\psi_{\ov f}\C)$, $\bH^j(E;\Omega^\bullet_T(\log E)[f_\infty])$ and $\bH^j(T;\Tot \calA^{\bullet,\bullet})$ through the quasiisomorphisms between the corresponding complexes that we have described above, taking hypercohomology in the commutative diagrams (\ref{eqn:commPhi}) and (\ref{eqn:commLimit}) yields the following commutative (up to a non-zero constant) diagram. Note that we can identify $H^{j+1}(T^*;\ov\cL\otimes_R R_\infty)$ with $H^{j+1}(T^*;\ov\cL)$ naturally, since the latter is annihilated by $s^m$ for some $m$, using Remark~\ref{proper}.
\[
{
\begin{tikzcd}[column sep = {8em,between origins}, row sep = 1.4em]
H^j(T^*;\C) \arrow[dr,"H^j(\pi_{\cL}^\vee)"',bend right = 10] \arrow[r,"\pi^*"]
&
H^j((T^*)^f;\C) \arrow[r, "i^{-1}\vdash i_*"]&
\bH^j(E;\psi_{\ov f}\C) \arrow[r,equals]\arrow[dr,dashrightarrow]
& \bH^j(T;\Tot \calA^{\bullet,\bullet})
\\
&
H^{j+1}(T^*;\ov\cL)
\arrow[rr, "i^{-1}\vdash i_*"]
&
&
\bH^{j+1}\left(E;\Omega_T^{\bullet}(\log E)\left(\frac{1}{2\pi i}\frac{df}{ f},\infty\right)\right)\arrow[u,"H^j(\Phi^\ell_{\C}(\infty))",swap].
\end{tikzcd}
}
\]
If the above commutes (up to a scalar), we incorporate to this diagram arrows given by the inclusion of $T^*$ in $U$ in the left side.
\begin{equation}\label{eqn:commCohom}
{\footnotesize
\begin{tikzcd}[column sep = 0.8em]
H^j(U;\C)\arrow[r]\arrow[dr,"H^j(\pi_{\cL}^\vee)",bend right = 10, two heads] &H^j(T^*;\C) \arrow[dr,"H^j(\pi_{\cL}^\vee)",bend right = 10] \arrow[r,"\pi^*"]
&
H^j((T^*)^f;\C) \arrow[r, "i^{-1}\vdash i_*"]&
\bH^j(E,\psi_{\ov f}\C)
\\
&
\Tors_R H^{j+1}(U;\ov\cL)\arrow[r]
&
H^{j+1}(T^*;\ov\cL)
\arrow[r]
&
\bH^{j+1}\left(E;\Omega_T^{\bullet}(\log E)\left(\frac{1}{2\pi i}\frac{df}{ f},\infty\right)\right)\arrow[u,"H^j(\Phi^\ell_{\C}(\infty))",swap].
\end{tikzcd}
}
\end{equation}

Let us recall why  $H^j(\pi_{\cL}^\vee)\colon H^j(U;\C)\to \Tors_R H^{j+1}(U;\ov\cL)$ is surjective. By Theorem \ref{thmsimple}, $\Tors_R H^{j+1}(U;\ov\cL)$ is a semisimple $R$-module. Recall that, in this chapter, we assumed  that the monodromy action on $\Tors_R H^{j+1}(U;\ov\cL)$ is unipotent (Remark \ref{unipotent}). The surjectivity now follows from Proposition~\ref{prop:mapsAreEqual}, which shows that $H^j(\pi_{\cL}^\vee)$ is the dual of the map induced by $U^f\to U$, and Corollary \ref{cor:surj}, which shows that if $t$ acts as the identity on $\Tors_R H^{j+1}(U;\ov\cL)$, this map is surjective.

We also note that, by Proposition \ref{prop:mapsAreEqual}, we have the following commutative diagram, where the vertical arrow comes from Proposition~\ref{propcanon}. The vertical isomorphism appears in Definition~\ref{def:relLimit}.
\[
\begin{tikzcd}[row sep = 1.4em]
H^j(T^*;\C) \arrow[dr,"H^j(\pi_{\cL}^\vee)"',bend right = 10] \arrow[r,"\pi^*"]
&
H^j((T^*)^f;\C)\\
&
H^{j+1}(T^*;\ov\cL).\arrow[u,"\cong"]
\end{tikzcd}
\]

Using the diagram above, the map $H^j(U;\C)\rightarrow \bH^j(E;\psi_{\ov f}\C)$ in (\ref{eqn:commCohom}) given by the composition of all the arrows in the top row is just $(r^*\otimes\C)\circ H^j(\pi_{\cL}^\vee)$. The map  $\Tors_R H^{j+1}(U;\ov\cL)\rightarrow \bH^j(E;\psi_{\ov f}\C)$ in (\ref{eqn:commCohom}) given by the composition of all the arrows in the bottom row and the vertical arrow on the right is none other that $\Phi^*\otimes\C$: recall that when we defined $\Phi^*$ in the proof of Corollary~\ref{relnlimitmhs}, we showed that by definition it is induced by $\Phi_{\Q}(\infty)$ taking cohomology and using our identifications. Further, by Lemma~\ref{surj} $\Phi_{\C}(\infty)$ and $\Phi_{\Q}(\infty)\otimes \C$ induced the same morphism in cohomology $\Tors_R H^{j+1}(U;\ov\cL\otimes \C)\to H^j(E;\psi_{\ov f}\C) $. The commutativity (up to non-zero constant) of (\ref{eqn:commCohom}) implies that for some $c\in \R\setminus\{0\}$, $c \cdot (r^*\otimes\C)\circ H^j(\pi_{\cL}^\vee)=(\Phi^*\otimes\C)\circ H^j(\pi_{\cL}^\vee)$.
The surjectivity of $H^j(\pi_{\cL}^\vee)$ implies that
$c \cdot (r^*\otimes_{\Q} \C)=(\Phi^*\otimes_{\Q} \C)$. Finally, the fact that $\C$ is fully faithful over $\Q$ implies that $c\cdot r^*=\Phi^*$ and $c$ must actually be a rational number, as desired.

To conclude the proof, we want to show that diagram (\ref{eqn:commLimit}) commutes up to a real constant. We use the de Rham resolution to identify $\ul \C_{T^*}$ with $\Omega_{T^*}^\bullet$, and similarly on $(T^*)^f$. Then, the commutativity of (\ref{eqn:commLimit}) is reduced to the commutativity of \eqref{eqn:commLimit-deRham} below. We are denoting by $\star$ the map of de Rham complexes induced by $\pi_{\cL}^\vee$, and we explain below why the dashed arrow $\omega\mapsto \frac{1}{i}\frac{df}{ f} \wedge \omega$ is indeed induced by $\pi_\cL^\vee$.

\begin{figure}[h]
\begin{equation}\label{eqn:commLimit-deRham}
{
\begin{tikzcd}[column sep = 0.5em, row sep = 1.2 em]
i_*i^{-1}\Omega_T^\bullet(\log E) 
\arrow[dr,dashrightarrow,"\Id_{\Omega_T^\bullet(\log E)}", bend left = 10,end anchor={[xshift=9ex, yshift = -0.5ex]}]
\\
\Omega_T^\bullet(\log E)\xrightarrow{\pi^{-1}\vdash \pi_*}
\arrow[u,"i^{-1}\vdash i_*",start anchor={[xshift=-5ex, yshift = -1ex]}, end anchor={[xshift=-2ex]}] 
\arrow[dd,"(\Id \to i_*i^{-1})\circ \star"',start anchor={[xshift=-5ex, yshift = 1ex]}, bend right = 15]
 \arrow[ddr,dashrightarrow," \omega\mapsto \frac{1}{i}\frac{df}{ f} \wedge \omega",near end  ,start anchor={[xshift=-12ex, yshift = 1ex]}] 
Rj_*\pi_*\Omega_{(T^*)^f}^\bullet 
\arrow[r]
&
i_*i^{-1}Rj_*\pi_*\Omega_{(T^*)^f}^\bullet 
\xleftarrow{\sim}
i_*i^{-1}\Omega^\bullet_T(\log E)[f_\infty]
\arrow[dd," \sum (f_\infty)^j \omega_j\mapsto  \frac{1}{2\pi i}\frac{df }{f}\wedge \omega_0",start anchor={[xshift=5ex, yshift = 1ex]}, bend left = 10]
\\
\\
i_*i^{-1}\Omega_T^{\bullet}(\log E)\left(\frac{1}{ i}\frac{df}{ f},\infty\right)[1]
\arrow[r,"G_\infty^{-1}"]
&
i_*i^{-1}\Omega_T^{\bullet}(\log E)\left(\frac{1}{2\pi i}\frac{df}{ f},\infty\right)[1].
\end{tikzcd}
}
\end{equation}

\end{figure}

First, a direct computation shows that the inclusion of $\Omega_T^\bullet(\log E)$ into the complex $\Omega^\bullet_T(\log E)[f_\infty]$ makes the top portion of the diagram below commute. Therefore, the composition (in the derived category) of the middle row is the inclusion $\Omega_T^\bullet(\log E)\hookrightarrow\Omega^\bullet_T(\log E)[f_\infty]$ composed with the adjunction map $\Id\to i_*i^{-1}$.

Next, consider the map $\star$ induced by $\pi_{\cL}^\vee$, which is the complexification of the map $\pi_{\cL}^\vee$ appearing in Lemma~\ref{lem:pullbacks}. By this Lemma (up to a real constant and taking the inverse limit as $m\to \infty$), $\pi_{\cL}^\vee$ corresponds to the map of real de Rham complexes $A:\cE_T^\bullet\to \cE_T^{\bullet} \left(
\Im\frac{df}{f},\infty
\right)[1]$ given by $\alpha\mapsto \Im \frac{df}{f}\wedge \alpha$. By Lemma~\ref{lem:theMapA} and its proof, the complexification of this map of real de Rham complexes induces in the logarithmic de Rham complexes the map $\omega\mapsto \frac{1}{i} \frac{df}{f}\wedge \omega:\Omega_{T^*}^\bullet(\log E)\to \Omega_{T^*}^{\bullet}(\log E)\left(\frac{1}{ i}\frac{df}{ f},\infty\right)[1]$ (these are maps in the derived category, in particular, two homotopic maps are equal, as in the proof of Lemma~\ref{lem:theMapA}).

The map $G_\infty$ is the inverse limit of the maps $G_m$ defined in the proof of Theorem~\ref{Qalexandermhs}. Namely, its inverse is given by
\[
G_\infty^{-1}(\omega\otimes s^j) = (2\pi)^{-j}\omega\otimes s^j.
\]
As in the proof of Theorem~\ref{Qalexandermhs}, a straightforward computation shows that the following diagram commutes. Recall that $\nu_{\Q}$ (resp. $\nu$) is defined in Remark~\ref{rem:nuQ} (resp. Remark~\ref{rem:nu}).
\[
\begin{tikzcd}[row sep = 1.2em]
\cL \arrow[r] \arrow[dr,"\nu_\Q"']&
\cL\otimes_{\Q} \C \otimes_R R_\infty \arrow[r,"\nu\otimes \C"]& \Omega_T^{\bullet}(\log E)\left(\frac{1}{ i}\frac{df}{ f},\infty\right) \arrow[d,"G_\infty^{-1}"]\\
&
\calK^\bullet_\infty(1\otimes f,m)\arrow[r,"\varphi_\infty"]
& \Omega_T^{\bullet}(\log E)\left(\frac{1}{ 2 \pi i}\frac{df}{ f},\infty\right)
\end{tikzcd}
\]
Using this diagram, we see that the lower left triangle in (\ref{eqn:commLimit-deRham}) commutes, since up to the resolutions above it is the following diagram:
\[
\begin{tikzcd}[row sep = 0.8 em]
\ul {\C}_{T^*} \arrow[d,"\pi_{\cL}^\vee"'] \arrow[dr,"\pi_{\cL}^\vee"]\\
\ov\cL \arrow[r,equals] &
\ov\cL.
\end{tikzcd}
\]
Finally, it follows that  (\ref{eqn:commLimit-deRham}) commutes up to real constant, by a computation on the remaining triangle.
\end{proof}

\begin{cor}\label{limit}
Let $f\colon U\rightarrow \C^*$ be a fibration. Then, the map
$$
r^*\colon \mathrm{Tors}_R\,H^{*+1}(U;\ov\calL) \rightarrow \mathbb{H}^*(E; \psi_{\bar{f}}\underline{\Q})
$$
of Definition~\ref{def:relLimit} is an isomorphism of MHS for all $*$. In other words, the MHS described in Corollary \ref{alexandermhs} coincides with the limit MHS.
\end{cor}
\begin{proof}
If $f$ is a fibration, the inclusion $T^*\hookrightarrow U$ is a homotopy equivalence, so $H^{*}(U;\ov\calL)$ is $R$-torsion, and the monomorphism $\Tors_R H^{*+1}(U;\ov\cL)\hookrightarrow H^{*+1}(T^*;\ov\cL)$ in Definition~\ref{def:relLimit} is actually an isomorphism.
\end{proof}

The content of Theorem~\ref{comp} is the result of combining Theorem~\ref{thm:limitMap} and Corollary~\ref{limit} in homological notation (dualizing via Corollary~\ref{isocohom}).


	\chapter{Examples and Open Questions}\label{sec:examples}

\section{Hyperplane Arrangements}\label{sec:hyp}

Let $n\geq 2$. Let $f_1,\ldots,f_d$ be degree $1$ polynomials in $\C[x_1,\ldots,x_n]$ defining $d$ distinct hyperplanes and let $f=f_1\cdot\ldots\cdot f_d$. The zeros of $f$ define a hyperplane arrangement $\cA$ of $d$ hyperplanes in $\C^n$. Let $U\subset \C^n$ be the corresponding arrangement complement.

\begin{dfn}\label{def:rankessential}
The rank of a hyperplane arrangement $\cA$ in $\C^n$ is the maximal codimension of a non-empty intersection of some subfamily of $\cA$. We say that $\cA$ is an essential hyperplane arrangement if its rank is equal to $n$.
\end{dfn}

\begin{remk}[Reducing to the case where $\cA$ is essential]\label{rem:essential}
Suppose that $\cA$ has rank $l<n$. By \cite[Proposition 6.1]{kohnopajitnov}, there exists an affine subspace $L$ of dimension $l$ such that $U_L:=U\cap L$ (seen in $\C^l$) is an essential hyperplane arrangement complement, and the inclusion $U_L\hookrightarrow U$ is a homotopy equivalence. Let $f_L$ be the restriction of $f\colon U\rightarrow \C^*$ to $U_L$.

The functoriality of the MHS (Theorem \ref{functorial}), as well as the homotopy invariance of cohomology with local systems, ensures that the map
$$
\Tors_R H^j(U;\overline \cL)\rightarrow \Tors_R H^j(U_L;\overline{\cL_L})
$$
induced by inclusion is a MHS isomorphism between cohomology Alexander modules for all $j$, where $\cL_L$ is the local system induced by $f_L$. Therefore, the study of the MHS on $\Tors_R H^j(U;\overline \cL)$ for $\cA$ a hyperplane arrangement can be reduced to the case where $\cA$ is an essential hyperplane arrangement.
\end{remk}

\begin{remk}[The cohomology groups of the infinite cyclic cover.]\label{rem:cohomologyhyperplanes}
Let $\cA$ be an essential hyperplane arrangement in $\C^n$ defined by the zeroes of a reduced polynomial $f$. Let $U$ be the corresponding arrangement complement, and $\cL$ the local system on $U$ induced by $f\colon U\rightarrow \C^*$. By \cite[Theorem 4]{thesiseva}, we have that
$
H_j(U;\cL)$ is a torsion $R$-module for all $j<n$, a free $R$-module for $j=n$, and $0$ for $j>n$. Hence, by Proposition \ref{propcanon} and Corollary \ref{isocohom}, we have canonical isomorphisms
$$\Tors_R H^{j+1}(U;\overline{\cL})\cong H^{j}(U^{f};k)$$
for $0\leq j<n$, and
$$\Tors_R H^{j+1}(U;\overline{\cL})\cong 0$$
for $j\geq n$. We use this canonical isomorphism to endow $H^j(U^{f};k)$ with a MHS, for $0\leq j\leq n-1$. In this section, we will talk about the MHS on $H^j(U^{f};k)$ instead of the isomorphic MHS on $\Tors_R H^{j+1}(U;\overline{\cL})$, both to simplify the notation and to highlight the geometric nature of the situation.
\end{remk}

\begin{remk}[Connectivity of the fiber]\label{rem:connectedfiber}
Let $\cA$ be an essential arrangement of $d$ hyperplanes in $\C^n$ defined by the zeros of a reduced polynomial $f$ of degree $d$, with $n>1$. Then, by \cite[Theorem 2.1]{DimcaTame} (and the discussion following it), the generic fiber of $f$ is connected.
\end{remk}

By Corollary \ref{cor:semisimple}, the $t$-action on $H^1(U^f;k)$ is semisimple, so by Corollary \ref{cor:t}, we have an isomorphism of MHS $H^1(U^f;k)\cong H^1(U^f;k)_1\oplus H^1(U^f;k)_{\neq 1}$. The goal of this section is to arrive at the following result.

\begin{thm}\label{thm:hyperplanes}
Let $\cA$ be an arrangement of $d$ hyperplanes in $\C^n$ defined by the zeros of a reduced polynomial $f$ of degree $d$, for $n\geq 2$. Assume that not all the hyperplanes of $\cA$ are parallel, or equivalently, that the rank of $\cA$ is greater than or equal to $2$. Then,
\begin{enumerate}
\item $H^1(U^f;k)_1$ is a pure Hodge structure of type $(1,1)$, and has dimension $d-1$.
\item $H^1(U^f;k)_{\neq 1}$ is a pure Hodge structure of weight $1$.
\end{enumerate}
\end{thm}

\begin{remk}[The Alexander polynomial of an essential line arrangement]\label{rem:alexpoly}
The first Alexander polynomial $\Delta_1(t)$ of an essential line arrangement is defined as the order of the torsion $R$-module $H_1(U;\cL)\cong H_1(U^f;k)$, or equivalently, as a generator of the $0$-th Fitting Ideal of the $R$-module $H_1(U;\cL)$. Hence, it is well defined up to multiplication by a unit of $R$. Since $H_1(U;\cL)$ is semisimple, $\Delta_1(t)$ determines the $R$-module structure of both $H_1(U;\cL)$ and $H^1(U^f;k)$, its dual as a vector space. 

Note that Theorem \ref{thm:hyperplanes} shows that the first Alexander polynomial of an essential line arrangement complement completely determines the Hodge numbers of $H^1(U^f;k)$.
\end{remk}

If $\cA$ is a \textit{central} hyperplane arrangement ($f$ is a homogeneous polynomial), $f$ determines a global Milnor fibration with fiber $F$, so $H^{j}(U^f;k)\cong H^j(F;k)$ for all $j$. In particular, $H^{j}(U^f;k)$ is a finite dimensional vector space for all $j$, so by Remarks \ref{rem:essential} and \ref{rem:cohomologyhyperplanes}, $H^j(U^f;k)=0$ for all $j\geq \text{ rank }\cA$. Hence, Corollary \ref{cor:quasihom} and Remark \ref{rem:cohomologyhyperplanes} tell us that the isomorphism $H^{j}(U^f;k)\cong H^j(F;k)$ is a MHS isomorphism for all $j$. Thus, Theorem \ref{thm:hyperplanes} is a direct generalization of parts (1) (in the case $j=1$) and (3) of the following result regarding \textit{central} hyperplane arrangements, which can be found in \cite[Theorem 7.7]{dimca2017hyperplanes}. Note that the last assertion in the result below follows from the second to last one.

\begin{thm}[\cite{dimca2017hyperplanes}, Theorem 7.7]\label{thm:central}
Let $\cA$ be a central hyperplane arrangement in $\C^{n}$ defined by a homogeneous reduced polynomial $f$. Let $F$ denote
its (global) Milnor fiber, given by the equation $f=1$. 
\begin{enumerate}
\item $H^j(F;k)_1$ is a pure Hodge structure of type $(j,j)$ for any $j\leq n-1$.
\item $\Gr_{2j}^W H^j(F,k)_{\neq 1}=0$ for any $j\leq n-1$. 
\item $H^1(F;k)_{\neq 1}$ is a pure Hodge structure of weight $1$.
\end{enumerate}
\end{thm}

Before we prove Theorem \ref{thm:hyperplanes}, we need the following lemma regarding the MHS on the generic fiber of $f$.

\begin{lem}\label{lem:genFiberLines}
Let $\cA$ be an essential line arrangement in $\C^2$, given by the zeros of a reduced polynomial $f$ of degree $d$. Let $c\in \C$ be generic, and let $F = f^{-1}(c)\subset \C^2$. Then
\[
\dim \Gr^W_2 H^1(F;k) = d-1.
\]
\end{lem}

\begin{proof}
By generic smoothness and Remark \ref{rem:connectedfiber}, $F$ is a smooth connected curve, whose genus we will denote $g$. Let $\ov F$ be its closure in $\C P^2$, and let $\wt F$ be the normalization of $\ov F$. By \cite[Corollaire 3.2.15 and Corollaire 3.2.17]{De2}, the mixed Hodge structure on $H^1(F;k)$ has $W_0H^1(F;k)=0$, $W_2H^1(F;k) = H^1(F;k)$ and $W_1H^1(F;k)$ is the image of $H^1(\wt F;k)$. The latter has dimension $2g$, and the map $H^1(\wt F;k)\to H^1(F;k)$ is injective, since $F$ is a punctured genus $g$ orientable surface. Let $\#p$ be the number of punctures. Then:
\[
\dim \Gr^W_2 H^1(F;k) = \dim \frac{ H^1(F;k)}{W_1 H^1(F;k)} = (2g + \#p -1) - 2g = \# p-1.
\]
So all we need to show is that $\#p = d$. Take the set of points at infinity $\{p_i\}\coloneqq \ov{f^{-1}(0)}\setminus \C^2$. Let $r_i$ be the number of lines of the arrangement passing through $p_i$. Locally around $p_i$, the closure of $\{f=0\}$ has an ordinary singularity of multiplicity $r_i$. The fibers $\ov F$ are the curves in the pencil generated by $\{f=0\}$ and $d\cdot L_\infty$, where $L_\infty$ denotes the line at infinity. Since the multiplicity of $d\cdot L_\infty$ at $p_i$ is $d>r_i$ (because the arrangement is essential), all the fibers $\ov F$ have ordinary singularities of multiplicity $r_i$ at $p_i$. This means that $\wt F$ has $r_i$ many branches over $p_i$, and this is all we need: $\#p = \sum_i r_i = d$.
\end{proof}

Now, we can finally prove Theorem \ref{thm:hyperplanes}.

\begin{proof}[Proof of Theorem \ref{thm:hyperplanes}]
This result deals with the MHS on $H^1(U^f;k)\cong (H_1(U;\cL))^{\vee_k}$. In light of Remark \ref{rem:essential}, we see that to study $H^1(U^f;k)$, it suffices to consider the case in which $\cA$ is an essential hyperplane arrangement. After intersecting with enough generic hyperplanes, we can and will assume in this proof that $\cA$ is an essential line arrangement in $\C^2$, by a Lefschetz type argument.

Let us start by proving that $H^1(U^f;k)_1$ is a pure Hodge structure of type $(1,1)$. By Proposition \ref{prop:kerim}, the map $H^1(U;k)\rightarrow H^{2}(U;\ov\cL)\cong H^1(U^f;k)$ induced by the covering space map $\pi\colon U^f\rightarrow U$ is surjective onto the $(t-1)$-torsion of $H^1(U^f;k)$. Since the $t$-action on $H^1(U^f;k)$ is semisimple, we get that $H^1(U;k)\rightarrow H^1(U^f;k)_1$ is a surjective MHS morphism, and the purity result follows from the fact that $H^j(U;k)$ is a pure Hodge structure of type $(j,j)$ for all $U$ affine hyperplane arrangement complement, by \cite{shapiro}.

Now, we prove that $\dim_k H^1(U^f;k)_1=\dim_k H_1(U^f;k)_1= d-1$. We start with the Milnor long exact sequence (already discussed in Proposition \ref{prop:kerim})
$$\cdots \to H_1(U^f;k) \overset{t-1}{\to} H_1(U^f;k) \to H_1(U;k) \overset{\partial}{\to} H_0(U^f;k) \overset{t-1}\to H_0(U^f;k) \to H_0(U;k) \to 0.$$
Since $U^f$ and $U$ are connected, $\dim_k H_0(U^f;k)=1=\dim_k H_0(U;k)$, so the homomorphism $H_0(U^f;k) \overset{t-1}\to H_0(U^f;k)$ is the zero map.
Also, since $H_1(U^f;k)$ is a semisimple $R$-module, then $H_1(U^f;k)/(t-1) H_1(U^f;k)\cong H_1(U^f;k)_1$. Hence, the Milnor long exact sequence gives us the short exact sequence
$$
0\to H_1(U^f;k)_1 \to H_1(U;k) \to H_0(U^f;k)\to 0.
$$
Since $\dim_k H_1(U;k)=d$, this finishes our proof of the equality 
\begin{center}$\dim_k H^1(U^f;k)_1=\dim_k H_1(U^f;k)_1= d-1$.\end{center}

Recall that, by Theorem \ref{thm:boundedWeights}, $\Gr_j^W H^1(U^f;k)=0$ for all $j\neq 1,2$. Let $F\subset U$ be a generic fiber of $f$. By Corollary \ref{cor:fiber}, the map
$$
H^1(U^f;k)\hookrightarrow H^1(F;k)
$$
induced by inclusion is a morphism of MHS. We know $H^1(U^f;k)_1$ is pure Hodge structure of weight $2$ and dimension $d-1$. By Lemma \ref{lem:genFiberLines} and the inclusion above, we get that $\Gr_2^W H^1(U^f;k)_{\neq 1}=0$, concluding our proof.
\end{proof}

In light of Theorem \ref{thm:hyperplanes} and Remark \ref{rem:alexpoly}, one might wonder in which cases the MHS of $H^1(U^f;k)$ is pure. If $d>1$, this amounts to $H^1(U^f;k)_{\neq 1}=0$, or equivalently, $\Delta_1=(t-1)^{d-1}$, where $\Delta_1$ is the first Alexander polynomial of the line arrangement complement. One can find sufficient conditions for $\Delta_1=(t-1)^{d-1}$ in \cite[Theorem 6]{thesiseva}, for example, which translated to the notation of this paper reads as follows.

\begin{prop}[\cite{thesiseva}, Theorem 6]\label{prop:purity}
Let $\cA=\{L_1,\ldots,L_d\}$ be an essential line arrangement of $d$ lines in $\C^2$, and, after reordering, let $\cB=\{L_1,\ldots,L_l\}$ be the set of lines in $\cA$ such that for each line in $\cB$ no other line in
$\cA$ is parallel to it, where $0\leq l\leq d$. Suppose that $\cB\neq\emptyset$. If for every $m>2$, there exists a line in $\cB$ with no points of multiplicity divisible by $m$, then $\Delta_1$ is a power of $t-1$, or equivalently,
$$
H^1(U^f;k)=H^1(U^f;k)_1\cong k^{d-1}
$$
is pure of type $(1,1)$.

In particular, if there exists a line in $\cB$ with only double points, the hypotheses of this proposition are satisfied.
\end{prop}

By Proposition \ref{prop:purity}, we know that there are many examples of essential line arrangements such that $H^1(U^f;k)_{\neq 1}=0$ and $H^1(U^f;k)$ is pure of type $(1,1)$. Here are some examples illustrating Theorem \ref{thm:hyperplanes} and Remark \ref{rem:alexpoly}, in which the MHS is not pure.

\begin{ex} Consider a central line arrangement of $d$ lines, i.e., defined by the equation $x^d=y^d$. A simple application of the Thom-Sebastiani theorem yields that the Alexander polynomial of the complement is 
$\prod_{\alpha,\beta}(t-\alpha \beta)=(t^d-1)^{d-2}(t-1)$, where the product runs over $n$th roots of unity $\alpha, \beta$, with $\alpha\neq 1, \beta \neq 1$. Hence, the non-zero Hodge numbers of $H^1(U^f;k)$ are $h^{1,1}=d-1$, $h^{0,1}=h^{1,0}=\frac{(d-1)(d-2)}{2}$.

By Corollary \ref{cor:quasihom}, we have an isomorphism of MHS $H^1(U^f;k)\cong H^1(F;k)$, where $F$ is the global Milnor fiber of the homogeneous polynomial $f$. If $d=3$, the MHS (not just the Hodge numbers) on $H^1(U^f;k)$ is determined as follows. The closure $\ov F$ of $F$ in $\C P^2$ is the elliptic curve whose $j$-invariant is $0$, and $F$ is $\ov F$ with three points removed. Following the proof of Lemma \ref{lem:genFiberLines}, we have the MHS isomorphism $H^1(U^f;k)_{\neq 1}\cong H^1(\ov F;k)$.
\end{ex}

\begin{ex}[A non-central line arrangement with nontrivial $H^1(U^f;k)_{\neq 1}$]
Let $\cA$ be the line arrangement defined by the zeros of $f(x,y)=x(x-1)y(y-1)(x+y-1)$.
$$
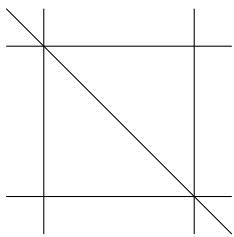
\begin{tikzpicture}
\draw (-1,1.5) -- (-1,-1.5);
\draw (1,1.5) -- (1,-1.5);
\draw (-1.5,1.5) -- (1.5,-1.5);
\draw (-1.5,1) -- (1.5,1);
\draw (-1.5,-1) -- (1.5,-1);
\end{tikzpicture}$$

The Alexander polynomial is $(t-1)^4(t^2+t+1)$. Hence, the non-zero Hodge numbers of $H^1(U^f;k)$ are $h^{1,1}=4$, $h^{0,1}=h^{1,0}=1$.
\end{ex}

We end this section with several open questions regarding the mixed Hodge structure on Alexander modules for hyperplane arrangement complements.

\begin{open}\label{q:semisimpleHyperplane}
Is $H^j(U^f;k)$ a semisimple $R$-module for $j>1$?
\end{open}

A positive answer to this question would yield a MHS isomorphism $$H^j(U^f;k) \cong H^j(U^f;k)_1 \oplus H^j(U^f;k)_{\neq 1}.$$ However, this consequence has subsequently been proved in \cite{tsemisimple}, independently of the question above. In particular, $H^j(U^f;k)_1$ is semisimple (see \cite{BudurLiuWang}) and pure of type $(j,j)$, as in part (1) of Theorem \ref{thm:central}, and the proof is the same as the $j=1$ case in Theorem \ref{thm:hyperplanes} (it is made explicit in \cite{tsemisimple}). A related question is whether the generalization of part (2) of Theorem \ref{thm:central} holds, namely,

\begin{open}
Is $\Gr_{2j}^W H^j(U^f;k)_{\neq 1}=0$?
\end{open}

If this question had a positive answer for $j=2$, it would follow that $H^2(U^f;k)$ is a semisimple $R$-module by a similar reasoning to the one in Corollary~\ref{cor:jordan} (see \cite{tsemisimple} for an explicit proof). Conversely, if these modules were known to be semisimple, it would open many more approaches to this question, similar to the one taken in this section.


\section{Future Directions. Open Questions.}\label{sec:open}
In addition to the questions already mentioned in Section \ref{sec:hyp}, we list here several open problems we hope to address in the future. Most of these are motivated by corresponding results for the (co)homology of the Milnor fiber $F_x$ associated to a complex hypersurface singularity germ $f\colon (\C^n,x) \to (\C,0)$. We aim to globalize such statements by replacing the (co)homology of the Milnor fiber $F_x$ with the torsion part $$A_*(U^f;\Q):=\Tors_R H_*(U^f;\Q) $$
of the homology Alexander modules associated to an algebraic map $f\colon U\to \C^*$.

\subsection{Semisimplicity}\index{semisimple}
In subsequent work \cite{tsemisimple}, two of the authors of this paper have proved the following:
\begin{thm}\label{ts}
Let $t_s$ denote the semisimple part of the $t$-action on $A_*(U^f;\Q)$. It is a morphism of MHS.
\end{thm}
This is a generalization of our result from Corollary \ref{cor:t}, which proves the above statement in the case $t=t_s$. It is motivated by the corresponding results for the semisimple part of the monodromy operator acting on the Milnor fiber cohomology, and, respectively, on the cohomology of the generic fiber of a proper family $f\colon U \to \Delta^*$ over a punctured disc (e.g., see \cite[Th\'eor\`eme 15.13]{NA}).

\begin{open}\label{q:semisimple}
Find examples of pairs $(U,f)$, with $U$ a smooth connected complex algebraic manifold and $f\colon U \to \C^*$ an algebraic map, for which $A_i(U^f;\Q)$ is not a semisimple $R$-module for some $i$
\end{open}

In many of the algebraic situations considered in this paper, we have in fact that $A_i(U^f;\Q)$ is semisimple for all $i$. This applies, in particular, to the following cases:
\begin{itemize}
\item When $f\colon U=\C^n \setminus \{f=0\} \to \C^*$ is induced by a complex polynomial $f\colon \C^n \to \C$ which is transversal at infinity (see \cite{DL, Max06}), e.g., $f$ could be a homogeneous polynomial. More generally, $A_i(U^f;\Q)$ is semisimple in the case of Setting \ref{set:DL} (see \cite{DL}).
\item When $f\colon U \to \C^*$ is a projective submersion of smooth complex algebraic varieties, the semisimplicity of $A_i(U^f;\Q)$ is a consequence of Deligne's decomposition theorem (see Remark \ref{rem:smf}).
\item When $f\colon U \to \C^*$ is a proper algebraic map, the semisimplicity of $A_i(U^f;\Q)$ is a consequence of the decomposition theorem of Beilinson-Bernstein-Deligne \cite{BBD} (see Corollary \ref{corss}).
\end{itemize}

Let us also point out here that the semisimplicity property does not hold in general in the local situation, that is, for the monodromy operator acting on the (co)homology of the Milnor fiber associated to a complex hypersurface singularity germ (see, e.g., the discussion in \cite[Section I.9]{Ku}).

There have also been some recent developments in relation to Question~\ref{q:semisimple}. In \cite{Lib21}, Libgober has found such an example $(U,f)$, where $U$ is a smooth variety, but $f$ is a continuous map to $S^1$. In \cite{EHMW}, we describe an example $(U,f)$, where $f$ is an algebraic map to $\C^*$, but this time $U$ is singular. No examples of $(U,f)$ are known to us where $U$ is smooth and $f$ is an algebraic map to $\C^*$.

\subsection{Finite Type Invariants} 
Despite the fact that (unlike the Milnor fiber of a hypersurface singularity germ) the infinite cyclic cover $U^f$ is not in general a CW complex of finite type, one can associate finite type invariants to $U^f$ (or better said, to the pair $(U,f)$) in terms of the $R$-torsion part $A_*(U^f;\Q)$ of the Alexander modules. For instance, one can define:
\begin{itemize}
\item Betti numbers: $b_i(U,f):=\dim_\Q A_i(U^f;\Q)$.
\item mixed Hodge numbers: $h^{p,q,i}(U,f):=\dim_\C Gr^p_FGr_{p+q}^WA_i(U^f;\C)$.
\item spectral pairs: since the semisimple part $t_s$ of the $t$-action is a MHS morphism (cf. Theorem \ref{ts}), we let $h_\alpha^{p,q,i}(U,f)$ denote the dimension of the $\lambda$-eigenspace for the $t_s$-action on $Gr^p_FGr_{p+q}^WA_i(U^f;\C)$, where $\lambda=\exp(2\pi i \alpha)$ and $\alpha \in [0,1)$. The collection $\{h_\alpha^{p,q,i}(U,f)\}$ forms the {\it spectral pairs} of the $t_s$-action on the MHS $A_i(U^f;\Q)$.
\end{itemize}
In future work, we aim to investigate such finite-type invariants of the pair $(U,f)$; compare with \cite{LiMa} for a special case. 

In the case when $U$ is the complement of an essential hyperplane arrangement, it is also natural to ask about the combinatorial nature of such finite type invariants on $A_i(U^f;\Q)$.
This question is motivated by similar open problems in the case of central arrangements, where, for instance, it is still unknown if the Betti numbers of the associated Milnor fiber are determined by the intersection lattice of the arrangement. See, e.g., \cite{LiMa} for results on the case of complements of line arrangements which are transversal at infinity, and also \cite{budursaito}, \cite{DiLe}, \cite{Yo} for the combinatorial invariance of the {\it Hodge spectrum} of a central arrangement and variants of this result.

\subsection{Motivic Realization} Motivated by connections between the Igusa zeta functions, Bernstein--Sato polynomials and the topology of hypersurface singularities, Denef and Loeser introduced in \cite{DeLo} the {\it motivic Milnor fiber} of a hypersurface singularity germ. This is a virtual variety endowed with an action of the group scheme of roots of unity, from which one can retrieve several invariants of the (topological) Milnor fiber, such as the Euler characteristic, Hodge spectrum, etc. More generally, to any (finite type) infinite cyclic cover associated to a punctured neighborhood of a divisor on a smooth quasiprojective variety, one attached in \cite{GLM1, GLM2} a {\it motivic infinite cyclic cover}. This is an element in the Grothendieck ring $K_0(\text{\rm Var}^{\hat \mu}_\C) $ of complex algebraic varieties endowed with a good action of the group scheme $\hat \mu$ of roots of unity, whose Betti realization recovers (upon taking degrees) the Euler characteristic of the (topological) infinite cyclic cover of the punctured neighborhood (see \cite[Proposition 4.2]{GLM1}). In the terminology of \cite[Section 4]{GLM1}, we can therefore ask the following.
\begin{open}
With $U$ and $f\colon U \to \C^*$ as above, does there exist an element in (a certain localization of) $K_0(\text{\rm Var}^{\hat \mu}_\C) $, whose Betti realization is given by 
$\sum_i (-1)^i [A_i(U^f;\Q)] \in K_0(V_\Q^{\text{aut}})$? Under the semisimplicity assumption for the $t$-action, a similar question can be asked about the Hodge realization of such a motive.

\end{open}
Here, $K_0(V_\Q^{\text{aut}})$ denotes the Grothendieck ring of the category of finite dimensional $\Q$-vector spaces endowed with a finite order automorphism (which in our case is given by the semisimple part of the $t$-action).

\subsection{Mixed Hodge Module Realization}
To each complex algebraic variety $X$, M. Saito \cite{Sai} associated an abelian category $\text{MHM}(X)$ of algebraic mixed Hodge modules on $X$, in such a way that Deligne's category of mixed Hodge structures is recovered as mixed Hodge modules over a point space. Mixed Hodge modules are extensions in the singular context of (admissible) variations of mixed Hodge structures, and can be regarded, informally, as sheaves of mixed Hodge structures. Hypercohomology groups of a variety, with coefficients in a complex of mixed Hodge modules, are naturally endowed with mixed Hodge structures.

In recent decades, Saito's theory has been very successful at constructing mixed Hodge structures on new entities (e.g., on intersection cohomology groups of complex algebraic varieties), as well as recovering previously known such structures (see, e.g., \cite[Chapter 11]{Max-book} for an overview). It is therefore natural to ask the following.
\begin{open}
Given the pair $(U,f)$ as before, can one recover the mixed Hodge structures on $A_*(U^f;\Q)$ via Saito's mixed Hodge module theory?
\end{open}
For instance, motivated by Corollary~\ref{torsion}, one can try to define a (complex of) mixed Hodge module(s) whose underlying rational complex is $\ov\cL\otimes_R R_m$, and such that the map $\ov\cL\otimes_R R_m \hookrightarrow \ov\cL\otimes_R R_{2m}$ given by multiplication by $\log(t)^m$ comes from a map in $D^b\text{MHM}(U)$. Then for $m\gg 0$, the torsion part of the cohomology Alexander modules would inherit the MHS on the kernel of the map induced in cohomology by $\cdot \log(t)^m$. See also \cite{EHMW} for very recent developments regarding this question.

Let us just note here that the limit mixed Hodge structure of Chapter~\ref{rellim} has such a mixed Hodge module realization. This is due to the fact that the nearby cycle functor of constructible sheaves lifts to the derived category of bounded complexes of mixed Hodge modules. Therefore, one can ask the following.
\begin{open}
Is the comparison map of Theorem \ref{thm:limitMap} induced by a map of complexes of mixed Hodge modules?
\end{open}

\subsection{Comparison to the Limit Mixed Hodge Structure in the Nonproper Case}
Under certain assumptions on $f$, a limit mixed Hodge structure can be defined even if the map $f\colon U \to \C^*$ is not proper, see \cite[Section 5]{SZ} and also \cite{EZ1,EZ2}. It would therefore be interesting to see if Theorem \ref{comp}  holds without the properness assumption; we hope to address this general situation in future work.

\subsection{Generalizations to other Algebraic Maps}

Let $f\colon X\to Y$ be an algebraic map of connected complex algebraic varieties. Any such map is homotopy equivalent to a fibration over $Y$. The fiber of this fibration, denoted by $E_f$, is called the homotopy fiber of $f$. For instance, if $Y$ is an aspherical space and $f$ induces an epimorphism on $\pi_1$, then $E_f$ is the covering space of $X$ defined by the kernel of $\pi_1(X) \to \pi_1(Y)$.
In \cite{hain1987rham}, Hain proved that if the cohomology groups of $E_f$ are finite dimensional and $\pi_1(Y)$ acts unipotently on them, then $H^*(E_f;\Q)$ have natural mixed Hodge structures. This suggests generalizations of our results to allow singularities as well as to more general algebraic maps. 
\begin{open}
Can Theorem \ref{mhsexistence} be generalized to arbitrary complex algebraic varieties $U$?
\end{open}
\begin{open}
Does Theorem \ref{mhsexistence} generalize to other algebraic maps $f\colon X\to Y$?
\end{open}

Let $f\colon X\to Y$ and $E_f$ be defined as above, where $Y$ is semi-abelian (e.g., $f$ is the Albanese map for $X$). Let $\pi:E_f\to X$ be the corresponding covering space. Each cohomology group $H^i(X; \pi_!\underline{\Q}_{E_f})$ has a natural $\Q[\pi_1(Y)]$-module structure. Moreover, $H^i(X; \pi_!\underline{\Q}_{E_f})$ is a finitely generated $\Q[\pi_1(Y)]$-module, so it contains a unique maximal $\Q[\pi_1(Y)]$-submodule that is a finite dimensional $\Q$-vector space and which is a natural generalization to the torsion part of the Alexander module considered in this paper. We can therefore ask whether this submodule admits a natural mixed Hodge structure. In the case where $Y=(\C^*)^r$ is an algebraic torus (and $X$ is arbitrary), a mixed Hodge structure has been recently constructed in \cite{EHMW} by different methods  (see \cite[Section 1.4]{EHMW} for a comparison with results of this paper).

\subsection{Coverings which are not Realized by Algebraic Maps}

The fact that the epimorphism $\xi\colon \pi_1(U) \to \Z$ is realized by an algebraic map $f\colon U \to \C^*$ plays an essential role in proving the results of this paper. It is however natural to investigate (variants of) our original Question \ref{conj} in more topological contexts when an algebraic realization of $\xi$ is not readily available (e.g., if $U$ is a smooth complex projective variety).


\phantomsection
\addcontentsline{toc}{chapter}{Bibliography}

\let\oldaddcontentsline\addcontentsline
\renewcommand{\addcontentsline}[3]{}

\let\addcontentsline\oldaddcontentsline

	\printindex
	
\end{document}